\documentclass[reqno]{amsart}

\usepackage{enumitem}
\usepackage{mathtools}
\usepackage{float}
\usepackage{xfrac}
\usepackage{embrac}
\usepackage[hyphens]{url}
\usepackage{tikz}
\usetikzlibrary{calc,arrows,shapes,shapes.misc,shapes.multipart,patterns}
\usepackage{tikz-cd}
\definecolor{leafgreen}{rgb}{0.5,0.8,0.6}
\definecolor{lightgreen}{rgb}{0.8,1.0,0.9}
\definecolor{lightgray}{rgb}{0.95,0.95,0.98}
\mathtoolsset{showonlyrefs}
\usepackage[pagebackref]{hyperref}
\hypersetup{
  colorlinks,
  linkcolor={blue!50!black},
  citecolor={green!40!black},
  urlcolor={blue!80!black}
}
\usepackage[hyphenbreaks]{breakurl}
\newtheorem{theorem}{Theorem}[subsection]
\newtheorem{lemma}[theorem]{Lemma}
\newtheorem{proposition}[theorem]{Proposition}
\newtheorem{corollary}[theorem]{Corollary}
\newtheorem{claim}[theorem]{Claim}

\newtheorem{fact}[theorem]{Fact}
\newtheorem{conjecture}[theorem]{Conjecture}
\theoremstyle{definition}
\newtheorem{remark}[theorem]{Remark}

\newtheorem{definition}[theorem]{Definition}
\newtheorem{example}[theorem]{Example}
\newtheorem{convention}[theorem]{Convention}
\numberwithin{equation}{section}
\numberwithin{figure}{section}
\allowdisplaybreaks{}
\usepackage{todonotes}

\usepackage{amssymb}

\newcommand{\strikeout}[1]{\mathrlap{\stretchrel*{\smallsetminus}{#1}}#1}

\usepackage{scalerel}
\DeclareMathOperator*{\bigplus}{\scalerel*{+}{\sum}}

\newcommand{\uppar}[1]{\textup{(}{#1}\textup{)}}  %chktex 9

\newcommand{\babydots}[0]{\scalebox{0.8}{$\dots$}}

\newcommand\opr[1]{\operatorname{#1}}

\newcommand{\RR}[0]{\mathbb{R}}
\newcommand{\CC}[0]{\mathbb{C}}
\newcommand{\FF}[0]{\mathbb{F}}
\newcommand{\QQ}[0]{\mathbb{Q}}
\newcommand{\ZZ}[0]{\mathbb{Z}}

\newcommand{\cR}[0]{\mathcal{R}}
\newcommand{\cC}[0]{\mathcal{C}}

\newcommand{\cG}[0]{\mathcal{G}}
\newcommand{\cM}[0]{\mathcal{M}}
\newcommand{\cA}[0]{\mathcal{A}}

\newcommand{\cP}[0]{\mathcal{P}}

\newcommand{\fs}[0]{\mathfrak{s}}
\newcommand{\fM}[0]{\mathfrak{M}}

\newcommand{\PP}[0]{\mathbb{P}}
\newcommand{\EE}[0]{\mathop{{}\mathbb{E}}}

\newcommand{\eps}[0]{\varepsilon}

\newcommand{\im}[0]{\opr{im}}
\newcommand{\id}[0]{\opr{id}}
\newcommand{\spn}[0]{\opr{span}}

\newcommand{\entails}[0]{\vdash}
\newcommand{\zeroi}[0]{\circleddash}
\newcommand{\leafs}[0]{\mathcal{L}}
\newcommand{\pins}[0]{\mathcal{P}}
\newcommand{\toggles}[0]{\mathcal{T}}
\newcommand{\nonleafs}[0]{\mathcal{N}}

\newcommand{\cs}[0]{\opr{cs}}
\newcommand{\CS}[0]{\opr{CS}}
\newcommand{\gc}[0]{\opr{GC}}
\newcommand{\diagram}[0]{\opr{Diag}}
\newcommand{\datum}[0]{\opr{Dat}}
\newcommand{\morph}[0]{\opr{MORPHISM}}
\newcommand{\aggregate}[0]{\opr{AggreGate}}
\newcommand{\agg}[0]{\opr{Agg}}
\newcommand{\bigagg}[0]{\opr{BigAgg}}
\newcommand{\bridge}[0]{\opr{Bridge}}
\newcommand{\supera}[0]{\opr{SupAGate}}
\newcommand{\smagate}[0]{\opr{IndAGate}}
\newcommand{\agate}[0]{\opr{AGate}}
\newcommand{\bag}[0]{\opr{bAG}}
\newcommand{\opbr}[0]{\opr{Br}}
\newcommand{\abilin}[0]{\opr{ABilinear}}
\newcommand{\amulti}[0]{\opr{AMultinear}}
\newcommand{\bigsum}[0]{\opr{BigSum}}
\newcommand{\Middle}[0]{\opr{NonInitial}}
\newcommand{\Initial}[0]{\opr{Initial}}
\newcommand{\boring}[0]{\opr{Boring}}
\newcommand{\stargate}[0]{\opr{StarGate}}
\newcommand{\ag}[0]{\opr{AG}}
\newcommand{\lag}[0]{\opr{lAG}}
\newcommand{\const}[0]{\opr{Const}}
\newcommand{\trivial}[0]{\opr{Trivial}}
\newcommand{\opsum}[0]{\opr{Sum}}
\newcommand{\opcross}[0]{\opr{Cross}}
\newcommand{\crs}[0]{\opr{Crs}}
\newcommand{\sumconst}[0]{\opr{SumConst}}
\newcommand{\gw}[0]{\opr{GW}}

\newcommand{\Alpha}[0]{\mathrm{A}}

\def\wt{\widetilde}

\begin{document}

\title{True complexity and iterated Cauchy--Schwarz}

\author{Freddie Manners}
\address{Freddie Manners, UCSD Department of Mathematics, 9500 Gilman Drive \#0112, La Jolla CA 92093, USA}
\email{fmanners@ucsd.edu}

\begin{abstract}
  We prove a polynomial bound in the ``true complexity'' problem of Gowers and Wolf.
  The proof uses only repeated applications of the Cauchy--Schwarz inequality, answering negatively a question posed by Gowers and Wolf.

  To choose and reason about the sequence of Cauchy--Schwarz steps needed, we need to introduce several layers of formalism and theory.
  The highest level of abstraction in this framework concerns building what we term ``arithmetic circuits'' encoding computations in multilinear algebra.

  It is plausible this machinery could be used to generate arithmetic inequalities in greater generality, and we state some conjectures along these lines.
\end{abstract}

\maketitle
\setcounter{tocdepth}{2}
\tableofcontents

\section{Introduction}%
\label{sec:intro}

\subsection{Cauchy--Schwarz arguments and true complexity}%
\label{sub:background}

This paper is concerned with inequalities in additive combinatorics that may be proved by multiple applications of the Cauchy--Schwarz inequality.
A standard natural example is the following.
\begin{fact}%
  \label{fact:quasirandom-thing}
  If $n$ is odd, $S \subseteq \ZZ/n\ZZ$ is a set of size $\eps n$, and $S$ has $(\eps^4+o(1))n^3$ additive quadruples,
  \[
    \bigl\lvert \bigl\{ (x,a,b) \in (\ZZ/n\ZZ)^3 \colon x,\,x+a,\,x+b,\,x+a+b \in S \bigr\}  \bigr\rvert = (\eps^4 + o(1)) n^3
  \]
  then $S$ has $(\eps^3 + o(1)) n^2$ three-term arithmetic progressions:
  \[
    \bigl\lvert \bigl\{ (y,h) \in (\ZZ/n\ZZ)^3 \colon y,\,x+h,\,x+2h, \in S \bigr\}  \bigr\rvert = (\eps^3 + o(1)) n^2.
  \]
\end{fact}
More specifically, we will be concerned with proving inequalities between certain arithmetically-defined  \emph{multilinear averages} of functions, as in the following statement.
\begin{fact}%
  \label{fact:functional-inequality}
  If $n$ is odd and $f_1,f_2,f_3 \colon \ZZ/n\ZZ \to \CC$ are $1$-bounded functions\footnote{I.e., $ \lvert f_i(x) \rvert \le 1$ for all $x \in \ZZ/n\ZZ$.} then
  \begin{equation}%
    \label{eq:ex-cs}
    \biggl\lvert \EE_{y,h \in \ZZ/n\ZZ} f_1(y) f_2(y+h) f_3(y+2h) \biggr\rvert \le \lVert f_1 \rVert_{U^2}
  \end{equation}
  where
  \[
    \lVert f_1 \rVert_{U^2} = \biggl( \EE_{x,a,b \in \ZZ/n\ZZ} f_1(x) \overline{f_1(x+a)} \overline{f_1(x+b)} f_1(x+a+b) \bigg)^{1/4}.
  \]
  The same holds with $\lVert f_2 \rVert_{U^2}$, $ \lVert f_3 \rVert_{U^2}$ on the right-hand side.
\end{fact}
By standard telescoping arguments, Fact~\ref{fact:functional-inequality} applied with $f_i=1_S - \eps$ or $1_S$ implies Fact~\ref{fact:quasirandom-thing}.

The proof of Fact~\ref{fact:functional-inequality} by multiple applications of Cauchy--Schwarz is the following now-standard argument.
By change of variables, the left-hand side may be written as
\[
  \left\lvert \EE_{z,w \in \ZZ/n\ZZ} f_1(2z\!-\!w) f_2(z) f_3(w) \right\rvert
\]
and by Cauchy--Schwarz
\begin{align*}
  &\left\lvert \EE_{w\in \ZZ/n\ZZ} f_3(w) \left( \EE_{z \in \ZZ/n\ZZ} f_1(2z\!-\!w) f_2(z) \right) \right\rvert 
  \\ &\le \left( \EE_{w \in \ZZ/n\ZZ} \lvert f_3(w) \rvert^2 \right)^{1/2} \left( \EE_{w \in \ZZ/n\ZZ} \left\lvert \EE_{z \in \ZZ/n\ZZ} f_1(2z\!-\!w) f_2(z) \right\rvert^2 \right)^{1/2}.
\end{align*}
Since $f_3$ is $1$-bounded the first term is bounded by $1$.
Expanding the second term yields
\[
\EE_{z_1,z_2, w \in \ZZ/n\ZZ} f_1(2z_1\!-\!w) f_2(z_1) \overline{f_1(2z_2\!-\!w)} \overline{f_2(z_2)}
\]
which may be rewritten and bounded by Cauchy--Schwarz as follows:
\begin{align*}
  &\EE_{z_1, z_2 \in \ZZ/n\ZZ} f_2(z_1) \overline{f_2(z_2)} \left( \EE_{w \in \ZZ/n\ZZ} f_1(2z_1\!-\!w) \overline{f_1(2z_2\!-\!w)} \right) \le \\
  & \left( \EE_{z_1,z_2 \in \ZZ/n\ZZ} \left\lvert f_2(z_1) \overline{f_2(z_2)} \right\rvert^2 \right)^{1/2}
\left( \EE_{z_1,z_2 \in \ZZ/n\ZZ} \left\lvert \EE_{w \in \ZZ/n\ZZ} f_1(2z_1\!-\!w) \overline{f_1(2z_2\!-\!w)}  \right\rvert^2 \right)^{1/2}.
\end{align*}
Again bounding the first term by $1$, expanding out the second term as above and changing variables gives the right-hand side of~\eqref{eq:ex-cs}.
By tradition, verifying this is left as an exercise for the reader.

This reader could be excused for not finding such arguments inherently inspiring, but they are certainly useful.
For example, suitable generalizations of Fact~\ref{fact:functional-inequality} are an essential early step in Gowers' proof of Szemer\'edi's theorem~\cite{gowers-szemeredi} and work of Green and Tao on linear equations in the primes~\cite{gt-primes,gt-linear}.

A key feature of such arguments is that they yield good-quality (polynomial) bounds.
In a number of more recent applications, finding new and ever more elaborate Cauchy--Schwarz arguments of this type is a limiting step in proving quantitative bounds in formally ineffective results.
Examples include:---
\begin{itemize}
  \item work of Peluse~\cite{peluse-finite} on finding polynomial progressions in dense subsets of finite fields;
  \item work of Peluse and Prendiville~\cite{peluse-prendiville-1,peluse-prendiville-2} and Peluse~\cite{peluse-integers} on cases of the polynomial Szemer\'edi theorem, which in particular required progress on a quantitative version of a ``concatenation theorem'' of Tao and Ziegler;
  \item certain arguments in~\cite{me-uk} and~\cite{gowers-mil} (and very recently~\cite{szegedy}) related to the inverse theory of the Gowers norms.
\end{itemize}

As the Cauchy--Schwarz arguments being sought grow more complicated, it is essential to have a more conceptual way of thinking about them.
One useful approach\footnote{It is hard to be sure of the correct reference for this point of view, but it certainly appears in some notes of Tao~\cite{tao-functional}.} is to consider the ``100\%'' case.  
For example, a special case of Fact~\ref{fact:functional-inequality} is that if
\[
  \left\lvert \EE_{y,h \in \ZZ/n\ZZ} f_1(y) f_2(y+h) f_3(y+2h) \right\rvert =1
\]
then
\[
  \left( \EE_{x,a,b \in \ZZ/n\ZZ} f_1(x) \overline{f_1(x+a)} \overline{f_1(x+b)} f_1(x+a+b) \right)  \ge 1.
\]
Unpacking the hypothesis, we must have that $\lvert f_i(x)\rvert=1$ for all $x$, or in other words $f_i(x) = \exp(2 \pi i F_i(x))$ for some function $F_i \colon \ZZ/n\ZZ \to \RR/\ZZ$.
Moreover, the identity
\begin{equation}%
  \label{eq:first-fun-equation}
  F_1(y) + F_2(y+h) + F_3(y+2h) = A
\end{equation}
must hold for every pair of elements $y,h \in \ZZ/n\ZZ$ and some global constant $A \in \RR/\ZZ$.
Similarly, the desired conclusion is that
\begin{equation}%
  \label{eq:affine-linear}
  F_1(x) - F_1(x+a) - F_1(x+b) + F_1(x+a+b) = 0
\end{equation}
holds for all $x,a,b \in \ZZ/n\ZZ$.
In other words, our aim is to show that if the functional equation~\eqref{eq:first-fun-equation} holds then $F_1$ (and similarly $F_2$, $F_3$) must be affine-linear.

This statement is true by our proof of Fact~\ref{fact:functional-inequality}, and it is also helpful to examine the Cauchy--Schwarz proof in the 100\% setting.
Having changed variables and rearranged,~\eqref{eq:first-fun-equation} states
\[
  F_1(2z-w) + F_2(z) = -F_3(w) + A
\]
for all $z,w \in \ZZ/n\ZZ$.
Since the right-hand side is independent of $z$, so is the left-hand side: that is,
\[
  F_1(2z_1-w) + F_2(z_1) = -F_3(w) + A = F_1(2z_2-w) + F_2(z_2)
\]
for all $z_1,z_2,w \in \ZZ/n\ZZ$.
Rearranging again,
\[
  F_1(2z_1-w) - F_1(2z_2-w) = -F_2(z_1) + F_2(z_2)
\]
is independent of $w$, so
\[
  F_1(2z_1-w_1) - F_1(2z_2-w_1) = -F_2(z_1) + F_2(z_2) = F_1(2z_1-w_2) - F_1(2z_2-w_2) 
\]
for all $z_1,w_1,z_2,w_2$, and again this implies~\eqref{eq:affine-linear} by change of variables.

This gives the following helpful approach to Cauchy--Schwarz problems.
\begin{enumerate}[label=(\roman*)]
  \item Given a putative inequality, consider the 100\% case, and decide whether it is a true statement about functional equations.
    If not, the statement is false.
  \item To have a hope of proving the inequality by Cauchy--Schwarz, it is necessary to know how to solve the functional equation problem in a completely \emph{elementary} way, i.e.\ simply by rearranging equalities, changing variables etc., since a putative Cauchy--Schwarz proof would imply such a proof.
  \item The constant $C$ in the corresponding ``1\%'' inequality
    \[
      \lvert \text{multilinear average 1} \rvert \le \lvert \text{multilinear average 2} \rvert^{1/C}
    \]
    is related to the \emph{length} of such an elementary proof:\footnote{For now we will not attempt to give a formal definition of an elementary proof in this sense, but to a decent approximation we mean a proof in some suitably construed first-order language.
    An attempt at a rigorous definition occurs in Definition~\ref{def:elem}.}
    the longer the argument, the worse the bound.
\end{enumerate}
In practice, (ii) means that we are not allowed to apply any kind of structure theorem.
For example, having determined that $F_1$ is an affine-linear function $\ZZ/n\ZZ \to \CC$, we might be tempted to observe that we know exactly what such functions look like, and use this information to make further deductions (about $F_2$, say).
However, that would no longer be an ``elementary'' proof, unless the proof of the structure theorem could itself be somehow encoded using elementary operations.

A major stumbling block for this approach is that the converse to (ii) does not hold.
We may have found an elementary proof of some fact about functional equations, but be unable to turn it into a Cauchy--Schwarz argument, because proofs corresponding to Cauchy--Schwarz arguments have an extremely special and restrictive form among all elementary proofs.\footnote{%
  \label{footnote:ultra}
  If we are allowed to use stronger tools such as ultralimits, ``soft'' structured/random decompositions (i.e., using regularity decompositions but not a hard inverse theorem), polynomial approximation~\cite{gowers-hb}, etc., it might be possible to remove these restrictions and prove a sort of dictionary from arbitrary elementary proofs to inequalities.
  This is perhaps similar to the approach to ``concatenation'' taken in~\cite{tao-ziegler-concat}.
  However, such techniques do not offer polynomial-strength bounds.
  While the author thinks these possibilities are interesting, we will not pursue them in this paper.
}
We highlight a few key features.
\begin{enumerate}[label=(\Alph*)]
  \item Essentially the only inference rule we may use, besides changes of variable etc., is the ``duplication step'': if $[E_1] = [E_2]$ for two expressions $[E_1]$ and $[E_2]$, then $[E_1] = [E_1]'$ where $[E_1]$, $[E_1]'$ are copies of the same expression with certain variables held the same (namely, those mentioned in $E_2$).
  \item In particular, in a normal proof we would be free to use any of the lines $\{1,\dots,n\}$ to deduce line $n+1$.
    In a Cauchy--Schwarz argument, we can only get line $n+1$ by applying some operation to line $n$.
    In other words, usual proofs may be parallel, whereas Cauchy--Schwarz proofs are serial.
  \item Similarly, in an ordinary proof, if we have shown $F_1(x)+F_2(y) = [E_1]$ and $F_1(x)+F_2(y) = [E_2]$ for some expressions $[E_1]$ and $[E_2]$, we can deduce $[E_1] = [E_2]$.
    In Cauchy--Schwarz world, this is typically not permitted: we can only achieve ``cancellation'' when $[E_1]$, $[E_2]$ are copies of the same expression, as in (A).

    In other words, we are working in some strange logic where we cannot assume ``$=$'' is transitive.
  \item Another way of stating the previous points is that Cauchy--Schwarz arguments have essentially no memory of the fact that two instances of the same symbol---say, $F_1$---in the same equation, actually refer to the same function.
    We have no direct mechanism to replace, say, $F_1(x)-F_1(x)$ by zero: the second $F_1$ behaves, for all practical purposes, as if it were a brand new object.
\end{enumerate}

It is natural to wonder when, if ever, these difficulties are fatal.
That is: given a putative inequality to be proved by iterated Cauchy--Schwarz, and an elementary proof of the associated functional equation, can we always turn this into a carefully constructed proof of Cauchy--Schwarz type as in (A)--(D), and hence get an associated 1\% inequality with good bounds? Or are there examples where no such proof exists, and/or the associated polynomial 1\% inequality is not true?

At the moment this question is imprecise, although by the end of the paper (Section~\ref{sec:conjectures}) we will be able to state a precise version.
In the author's view, any general affirmative answer, or negative answer (or proof of undecidability), would be of very significant interest.

In this paper, we only consider a certain sub-problem, corresponding to Cauchy--Schwarz proofs in the \emph{true complexity} problem of Gowers and Wolf, which we will now describe.
We also give one example not of this type, discussed in Section~\ref{sub:repr} below, as it is a good vehicle for explaining the methods we will use.
The hope, however, is that the techniques developed in these special cases could eventually be used to prove inequalities in the much greater generality discussed above and in Section~\ref{sec:conjectures}.

The ``true complexity'' problem of Gowers and Wolf~\cite{gw1} asks for best-possible statements of the same precise form as Fact~\ref{fact:functional-inequality}.
Specifically, suppose $p$ is a prime, $d \ge 1$ is an integer and $\Phi = (\phi_i)_{i=1}^k$ is a collection of linear forms $\phi_i \colon \FF_p^d \to \FF_p$.
The associated multilinear average of functions $f_1,\dots,f_k \colon \FF_p^n \to \CC$ is given by
\begin{equation}%
  \label{eq:system-of-forms}
  \Lambda_{\Phi}(f_1,\dots,f_k) = \EE_{x_1,\dots,x_d \in \FF_p^{n}} f_1\bigl(\phi_1(x_1,\dots,x_d)\bigr) \dots f_k\bigl(\phi_k(x_1,\dots,x_d)\bigr)
\end{equation}
where we abuse notation to write $\phi_i$ for the linear map $\phi_i^n \colon \big(\FF_p^n\big)^d \to \FF_p^n$.
For example, the system of linear forms in Fact~\ref{fact:functional-inequality} has $d=2$, $k=3$, and linear forms  $\Phi=(\phi_1,\phi_2,\phi_3)$ where $\phi_1(y,h) = y$, $\phi_2(y,h)=y+h$ and $\phi_3(y,h)=y+2h$.
Alternatively, if $d=3$, $k=4$ and $\Psi=(\psi_1,\dots,\psi_4)$ where $\psi_1(x,a,b) = x$, $\psi_2(x,a,b)=x+a$, $\psi_3(x,a,b) = x+b$ and $\psi_4(x,a,b)=x+a+b$ then
\[
  \Lambda_{\Psi}(f,\overline{f}, \overline{f}, f) = \lVert f \rVert_{U^2}^4.
\]

The analogue of Fact~\ref{fact:functional-inequality} is then a statement of the following type: given a system of linear forms $\Phi$, an index $i \in [k]$ and $\eps>0$, there exists $\delta>0$ such that for any $n \ge 1$ and $1$-bounded functions $f_1,\dots,f_k \colon \FF_p^n \to \CC$:
\begin{equation}%
  \label{eq:gw-bound}
  \lVert f_i \rVert_{U^{s+1}} \le \delta\ \ \Rightarrow \ \ \lvert \Lambda_{\Phi}(f_1,\dots,f_k) \rvert \le \eps.
\end{equation}
In other words, $\Lambda_\Phi$ is controlled by a Gowers norm\footnote{We assume the reader is familiar with Gowers norms; see e.g.~\cite[Chapter 11]{tao-vu} or~\cite[Appendix B]{gt-linear}.} $\lVert f_i \rVert_{U^{s+1}}$, where $s$ depends on $\Phi$ and possibly $i$.
Ideally, the dependence of $\delta$ on $\eps$ would be polynomial, as in Fact~\ref{fact:functional-inequality} itself: i.e.,
\begin{equation}%
  \label{eq:gw-poly}
  \lvert \Lambda_{\Phi}(f_1,\dots,f_k) \rvert \le \lVert f_i \rVert_{U^{s+1}}^{1/C}
\end{equation}
for some constant $C=C(\Phi,i)$.

When $\Phi$ is a $k$-term progression, i.e., $d=2$ and $\phi_i(x,h) = x+(i-1) h$, Gowers~\cite{gowers-szemeredi} proved such a bound with $s=k-2$ and $C=1$, by a direct generalization of the proof above when $k=3$.
A natural extension of this process to general systems of forms was formulated by Green and Tao~\cite{gt-linear}:\footnote{They actually prove something more complicated; the simplified version relevant to us is given explicitly in~\cite[Theorem 2.3]{gw1}.}
they show $\lvert \Lambda_{\Phi}(f_1,\dots,f_k) \rvert \le \lVert f_i \rVert_{U^{s+1}}$ for a value $s=s_{\cs}(\Phi, i)$ called the \emph{Cauchy--Schwarz complexity} of $\Phi$ at $i$.
This is defined as follows: it is the smallest non-negative integer $s$ such that $[k] \setminus \{i\}$ may be partitioned into $s+1$ sets $S_1,\dots,S_{s+1}$, such that $\phi_i \notin \spn\bigl((\phi_j)_{j \in S_r} \bigr)$ for each $r \in [s+1]$.
Provided $\phi_i$ is not a scalar multiple of $\phi_j$ for some $j\ne i$, we have $s_{\cs} \le k-2$; if it is, the whole situation is hopeless and foolish, and we say $s_{\cs}=\infty$.

In~\cite{gw1}, Gowers and Wolf consider the question: what is the \emph{smallest} value $s$ such that~\eqref{eq:gw-bound} holds?
We call this $s$ the%
\footnote{This is not actually the definition used by Gowers and Wolf: they impose an extra condition that $f_1=f_2=\cdots=f_k$ in the definition.
The multilinear version here appears to have been introduced in~\cite{hl}, and is in some ways more natural.
We caution that statements including the condition $f_1=f_2=\cdots=f_k$ are logically weaker than their multilinear variants, but since the multilinear versions are now known to be true we will not worry too much about this distinction.}
\emph{Gowers--Wolf complexity} $s_{\gw}(\Phi, i)$.
We also define a symmetric version
\[
  s_{\gw}(\Phi) = \max_{i \in [k]} s_{\gw}(\Phi,i) ;
\]
i.e., the smallest $s$ such that~\eqref{eq:gw-bound} holds for all $i \in [k]$.

Certainly $s_{\gw} \le s_{\cs}$.
On the other hand we may investigate $s_{\gw}(\Phi,i)$ by examining the corresponding statements about functional equations, as above.
The 100\% case of~\eqref{eq:gw-bound} asserts the following.
Suppose $F_1,\dots,F_k \colon \FF_p^n \to \RR/\ZZ$ are functions such that
\begin{equation}%
  \label{eq:functional-gw-hypo}
  F_1\bigl(\phi_1(v)\bigr) + F_2\bigl(\phi_2(v)\bigr) + \cdots + F_k\bigl(\phi_k(v)\bigr) = A
\end{equation}
for some global $A \in \RR/\ZZ$ and all $v \in \FF_p^{dn}$.
For $h \in \FF_p^{n}$ and $F \colon \FF_p^n \to \RR/\ZZ$, write $\partial_h F \colon x \mapsto F(x) - F(x+h)$ for the discrete derivative of $F$.
Then we wish to conclude (unpacking the definition of the Gowers norm $U^{s+1}$) that
\[
  \partial_{h_1} \partial_{h_2} \dots \partial_h{h_{s+1}} F_i(x) = 0
\]
for all $x,h_1,\dots,h_{s+1} \in \FF_p^n$.
This last condition is equivalent to the assertion that $F_i$ is a \emph{polynomial map} $\FF_p^n \to \RR/\ZZ$.
When $s<p$ this means that $F_i$ has the form of a multivariate polynomial,
\[
  (y_1,\dots,y_n) \mapsto \sum_{r_1+\cdots+r_n\le s} c_{r_1,\dots,r_n} y_1^{r_1} \dots y_n^{r_n}
\]
for some $c_{r_1,\dots,r_n} \in \RR/\ZZ$;
for $s \ge p$ we must also allow ``non-classical'' polynomials (see e.g.~\cite{tao-blog}) but we ignore this subtlety for now.

To summarize, in the 100\% setting, we want to know that for any solution to~\eqref{eq:functional-gw-hypo}, the function $F_i$ must be a polynomial of degree at most $s$.
Hence, if we can find a solution to~\eqref{eq:functional-gw-hypo} where $F_i$ is a polynomial of degree $t$, we can deduce\footnote{The 100\% statement still follows logically from the weaker inequality~\eqref{eq:gw-bound}, using the tensor power trick.}
the lower bound $s_{\gw}(\Phi,i) \ge t$.

Determining whether such a solution exists is an exercise in multilinear algebra. 
Considering the tensor powers\footnote{We write ${(\FF_p^d)}^\ast$ to denote the vector-space dual of $\FF_p^d$.  It is, of course, canonically isomorphic to $\FF_p^d$, but maintaining a distinction between vector spaces and their duals makes some statements easier to parse.} $\phi_i^{\otimes t} \in \left((\FF_p^d)^\ast\right)^{\otimes t }$,
a solution to~\eqref{eq:functional-gw-hypo} where $F_i$ is a polynomial of degree exactly $t$ exists if and only if%
\footnote{This characterization remains true in the non-classical regime $t \ge p$.}
\begin{equation}%
  \label{eq:true-criterion}
  \phi_i^{\otimes t} \in \spn\left( (\phi_j)^{\otimes t} \colon j \ne i \right) \le \left((\FF_p^d)^\ast\right)^{\otimes t }.
\end{equation}
Hence, the solution to the 100\% problem is captured by the following definition.

\begin{definition}%
  \label{def:true-complexity}
  For a system of linear forms $\Phi = (\phi_1,\dots,\phi_k)$ and an index $i \in [k]$, define $s(\Phi,i)$ to be the largest non-negative integer $t \ge 0$ such that~\eqref{eq:true-criterion} holds.

  We also write $s(\Phi)$ for $\max_{i \in [k]} s(\Phi,i)$, or equivalently, the largest non-negative integer $t \ge 0$ such that
  \[
    \phi_1^{\otimes t}, \dots, \phi_k^{\otimes t} \in \left((\FF_p^d)^\ast\right)^{\otimes t }
  \]
  are linearly dependent.
\end{definition}

Gowers and Wolf conjectured that these 100\% calculations correspond to the truth in the original inequality problem: that is, that $s_{\gw}(\Phi,i)=s(\Phi,i)$, or the weaker symmetric statement $s_{\gw}(\Phi)=s(\Phi)$.
In the latter case, this is equivalent to proving the following inequality, as in~\eqref{eq:gw-bound}.

\begin{conjecture}[Gowers--Wolf]%
  \label{conj:gw}
If $\Phi = (\phi_1,\ldots,\phi_k)$ is a system of linear forms $\phi_i \colon \FF_p^d \to \FF_p$ such that $(\phi_i)^{\otimes (s+1)}$ are linearly independent in $\bigl((\FF_p^d)^\ast\bigr)^{\otimes (s+1)}$, then for all $\eps>0$ there exists $\delta>0$ such that whenever $n \ge 1$, $i \in [k]$ and $f_1,\ldots,f_k \colon \FF_p^n \to \CC$ are $1$-bounded functions such that $\|f_i\|_{U^{s+1}} \le \delta$, we have $|\Lambda_{\Phi}(f_1,\ldots,f_k)| \le \eps$.
\end{conjecture}

This conjecture is now a theorem\footnote{
To be accurate, some of these results proved the original non-multilinear conjecture, with $f_1=\cdots=f_k$.
However, Hatami and Lovett~\cite{hl} (in finite fields) and Altman's argument (for $n=1$) prove the multilinear version discussed here.}
in the two regimes of significant interest, namely (a) $n=1$ and $p$ is large, or (b) $p$ is fixed and $n$ is large.  Specifically:---
\begin{itemize}
  \item in a series of papers~\cite{gw1,gw2,gw3} Gowers and Wolf resolved Conjecture~\ref{conj:gw} in regime (b), provided $p$ is not too small (i.e., avoiding the ``non-classical regime'');
  \item Hatami, Hatami and Lovett~\cite{hhl} resolved the remaining cases in regime (b), i.e., when $p$ is small; 
  \item Hatami and Lovett~\cite{hl} further proved the ``asymmetric'' conjecture $s_{\gw}(\Phi,i)=s(\Phi,i)$, again for $p$ fixed and $n$ large;
  \item Gowers and Wolf~\cite{gw4} also resolved the regime (a) cases of Conjecture~\ref{conj:gw} where $s_{\cs}(\Phi)=2$ and $s(\Phi)=1$;
  \item Green and Tao~\cite{gt} solved the remaining cases with $n=1$ and $p$ large, subject to the system of linear forms obeying a technical ``flag condition'' (which holds in many cases of interest but can fail, sometimes generically);%
\footnote{The original version of~\cite{gt} did not mention this extra hypothesis, but it was later observed by Altman that the proofs assumed it implicitly.}
  \item the remaining (symmetric) cases $n=1$, $p$ large (with or without the flag condition) were resolved by Altman~\cite{altman}, building on the Green--Tao argument.
\end{itemize}
The ``asymmetric'' statement $s(\Phi,i)=s_{\gw}(\Phi,i)$ remains open in regime (a), although it is unclear how much attention it has received.

We sketch the rough form taken by these arguments.
\begin{enumerate}[label=\arabic*.]
  \item Note we have control of $\Lambda_{\Phi}$ by \emph{some} Gowers norm $\|\cdot\|_{U^{s_{\cs}+1}}$, by the Cauchy--Schwarz complexity argument~\cite[Theorem~2.3]{gw1}.
    Hence we are free to modify $f_i$ by small errors in the $\|\cdot\|_{U^{s_{\cs}+1}}$-norm.
  \item Apply an inverse theorem for the Gowers $\|\cdot\|_{U^{s_{\cs}+1}}$-norm to the $f_i$ (\cite{inverse-fp,inverse-gtz}, or for quantitative bounds~\cite{gowers-mil,me-uk}).
  \item Taking steps 1 and 2 together, we may assume WLOG that $f_i$ are ``$U^{s_{\cs}+1}$ structured functions'': i.e., nilsequences (if $n=1$ and $p$ is large) or phase polynomials (if $p$ is fixed and $n$ is large).
  \item Solve the problem for such $f_i$.  This is still almost all the work.
\end{enumerate}
We will not say any more about step 4.  However, we note that step 2 limits the quality of the quantitative dependence of $\delta$ on $\eps$ in Conjecture~\ref{conj:gw} obtainable by these methods to that offered by the inverse theorems.  Using the best currently known bounds, this is still no better than $1/\delta \approx \exp^{O(1)}(1/\eps)$ when $s_{\cs} \ge 3$; i.e., a tower of exponentials of fixed height.
Certainly we do not get polynomial bounds as in~\eqref{eq:gw-poly}.

Gowers and Wolf posed two further related questions (\cite[Problem 7.8]{gw4}):---
\begin{enumerate}[label=(\roman*)]
  \item Could the dependence of $\delta$ on $\eps$ be polynomial as in~\eqref{eq:gw-poly}?
  \item Could Conjecture~\ref{conj:gw} be proven in general by some sufficiently complicated sequence of applications of the Cauchy--Schwarz inequality, or are ``higher tools'' such as the inverse theorem somehow essential?
\end{enumerate}
Gowers and Wolf state that they suspect that a proof as in (ii) does not exist, and that a good way to rule it out would be to show that the answer to (i) is also no (since pure Cauchy--Schwarz arguments give polynomial bounds).
We now state the main result of this paper, which resolves this question in the opposite direction.

\begin{theorem}%
  \label{thm:main}
  Suppose $p$ is a prime and $\Phi=(\phi_1,\dots,\phi_k)$ is a system of linear forms $\phi_i \colon \FF_p^d \to \FF_p$ with complexity $s(\Phi)$.
  Then for some constant $M\ge0$ depending on $\Phi$, the following holds: for any $n \ge 1$ and any $1$-bounded functions $f_1,\dots,f_k \colon \FF_p^n \to \CC$, and any $i_0 \in [k]$, we have
  \[
    \left\lvert \Lambda_{\Phi}(f_1,\dots,f_k) \right\rvert \le \|f_{i_0}\|_{U^{s(\Phi)+1}}^{2^{-M}}.
  \]
  Moreover, suppose $\phi_i(x_1,\dots,x_d) = \sum_{j=1}^d a_{ij} x_j$ where $a_{ij}$ are integer coefficients with $|a_{ij}| \le L$ for all $i,j$.
  Then we may take $M \ll k^{3} \bigl(\log k + \log \log (10 L)\bigr)$.
\end{theorem}
Moreover, the proof proceeds only by $M$ applications of the Cauchy--Schwarz inequality (as well as other elementary tools such as the triangle inequality).

Clearly it is always valid to take $L=p$, so for fixed $p$ we get a bound depending only on $s(\Phi)$, $k$ and $p$.  However, some dependence on either $p$ or the size of the coefficients $L$ is necessary: see~\cite[Theorem~1.7]{me}.  In particular, this theorem shows that for fixed $k$ the bound $M \ll \log \log (10 L)$ is in some cases best possible up to constants.

In the special case $k=6$, $d=3$ this was proved in~\cite{me}.
The general approach---finding a systematic scheme for repeatedly applying Cauchy--Schwarz---is the same here as in~\cite{me}, but the specifics of the two schemes are quite different.
In particular, the reader wishing to understand the proof of Theorem~\ref{thm:main} in general will lose almost nothing by keeping in mind the following model case of~Theorem~\ref{thm:main}, even though it was already covered in~\cite{me}.

\begin{example}[Example case of Theorem~\ref{thm:main}]%
  \label{ex:gw-ex}
  There exists $M\ge0$ such that for any prime $p>5$, integer $n \ge 1$ and $1$-bounded functions $f_1,\dots,f_6 \colon \FF_p^n \to \CC$, we have
  \[
    \biggl\lvert \EE_{x,y,z \in \FF_p^n} f_1(x) f_2(x+z) f_3(x+y) f_4(x+y+z) f_5(x+2y+3z) f_6(2x+3y+6z) \biggr\rvert \le \|f_1\|_{U^2}^{2^{-M}}.
  \]
\end{example}
This is an explicit example of a system with $s(\Phi)=1$ but $s_{\cs}=2$.
Note this is sensitive to the coefficients: for example, there does not exist $C>0$ such that
\[
  \biggl\lvert \EE_{x,y,z \in \FF_p^n}\!\! f_1(x) f_2(x+z) f_3(x+y) f_4(x+y+z) f_5(x+2y+3z) f_6(13x+12y+9z) \biggr\rvert \le \|f_1\|_{U^2}^{1/C}
\]
under the same hypotheses, even for $p$ sufficiently large, as this system of forms has $s(\Phi)=2$.

We note that Theorem~\ref{thm:main} handles the symmetric form of true complexity, dealing with $s(\Phi)$ rather than $s(\Phi,i)$ for each $i \in [k]$.
It seems plausible that our methods should resolve the full asymmetric version, but we have not been able to do this.
The problem appears to highlight a genuinely difficult case for the techniques of this paper.
Hence the following conjecture is a useful test case in improving the power of Cauchy--Schwarz arguments in general.

\begin{conjecture}%
  \label{conj:asymmetric}
  For any system of linear forms $\Phi = (\phi_1,\dots,\phi_k)$, $\phi_i \colon \FF_p^d \to \FF_p$,
  and any index $i_0 \in [k]$,
  there exists $M\ge0$ such that for any $n \ge 1$ and $f_1,\dots,f_k \colon \FF_p^n \to \CC$ $1$-bounded functions,
  \[
    \left\lvert \Lambda_{\Phi}(f_1,\dots,f_k) \right\rvert \le \|f_{i_0}\|_{U^{s(\Phi,i_0)+1}}^{2^{-M}}.
  \]
  Moreover, this equality can be proved using only multiple applications of Cauchy--Schwarz.   
\end{conjecture}

Our methods do however give the following slight strengthening of Theorem~\ref{thm:main} in the direction of this conjecture.
\begin{theorem}%
  \label{thm:main-ext}
  If $\Phi$ is a system of $k$ linear forms, $i_0 \in [k]$ and $s(\Phi) \le s(\Phi,i_0)+1$ then Conjecture~\ref{conj:asymmetric} holds for $\Phi$ and $i_0$, with the same value $M$ as in Theorem~\ref{thm:main}.
\end{theorem}

It is not too hard to find natural systems of linear forms where Conjecture~\ref{conj:asymmetric} applies but is open.
A simple one is the following.
\begin{example}%
  \label{ex:bad-asymmetric}
  According to Conjecture~\ref{conj:asymmetric}, it should be true that for some $M \ge 0$ and $p$ large enough,
  \[
    \left\lvert \EE_{x,y,z \in \FF_p^n} f_1(y+z) \prod_{r=0}^{7} f_{r+2}(rx+r^2 y+z) \right\rvert \le \|f_1\|_{U^2}^{2^{-M}}
  \]
  but we cannot currently prove this using our methods.

  In this example, the linear forms $(x,y,z) \mapsto rx+r^2 y + z$, thought of as points $(r:r^2:1)$ of the projective space $\PP^2(\FF_p) \cong \PP\bigl({(\FF_p^3)}^\ast\bigr)$, all lie on a conic that does not contain the point $(0:1:1)$.
  That is, there is a linear map $\bigl(\bigl(\FF_p^3\bigr)^\ast\bigr)^{\otimes 2} \to \FF_p$ vanishing on $\phi_i^{\otimes 2}$ for $2 \le i \le 9$ but not on $\phi_1^{\otimes 2}$.
  It follows that $s(\Phi,1)=1$.
  However, $s(\Phi)=3$ so Theorem~\ref{thm:main-ext} does not help.
\end{example}

\subsection{A simpler model inequality}%
\label{sub:repr}

The following inequality is separate from Theorem~\ref{thm:main}, although it emerges naturally from the proof.
We include it for two reasons.
First, it is a natural statement and could be of independent interest, although we do not have any applications in mind.\footnote{Related inequalities appear in~\cite[Lemma~5.5.4]{me-uk}.
There is also an analogy with problems about ``sums of dilates'', as in~\cite{bukh-sums}.}
Second, it is useful as a model problem on the way to Theorem~\ref{thm:main}, allowing us to introduce most of the key techniques of proof in a less complicated setting.

\begin{theorem}%
  \label{thm:baby-thm}
  Suppose $p>2$ is a prime, $n \ge 1$ is an integer,
  $f \colon \FF_p^n \to \CC$ is a $1$-bounded function and $\|f\|_{U^3}^8 \ge \delta$.
  Equivalently, writing $b \colon \FF_p^n \times \FF_p^n \to \CC$ for the function
  \[
    b(h,h') = \EE_{x \in \FF_p^n} f(x) \overline{f(x+h)} \overline{f(x+h')} f(x+h+h')
  \]
  we have $\|f\|_{U^3}^8 = \EE_{h,h' \in \FF_p^n} |b(h,h')|^2 \ge \delta$.
  
  Then for any integer $a$ we have
  \[
    \EE_{h,h' \in \FF_p^n} b(ah, h') \overline{b(h, a h')} \ge \delta^{2^M}
  \]
  where $M = 6 + \lfloor \log_2 \lfloor \log_2 |a| - 1 \rfloor \rfloor$ if $|a| \ge 4$ and $M=5$ for $|a|<4$.
\end{theorem}
By taking a bracket-quadratic function $f \colon \FF_p \to \CC$ such as
\[
  f(x) = \exp\bigl(2 \pi i \sqrt{2} x\bigr) \text{ for }x=0,1,\dots,p-1
\]
for $p$ large, one can show that the exponent $O(\log_2 |a|)$ is best possible as a function of $|a|$, up to constants.
The proof is similar to that of~\cite[Theorem~1.7]{me}.

We briefly explain why this statement is natural.
Analysing the 100\% case as above, our hypothesis is that we have a function $F \colon \FF_p^n \to \RR/\ZZ$ satisfying
\begin{equation}%
  \label{eq:hypo}
  \begin{aligned}
    &F(x) - F(x+h) - F(x+h') + F(x+h+h') \\
    &= F(y) - F(y+h) - F(y+h') + F(y+h+h') 
  \end{aligned}
\end{equation}
for all $x,y,h,h' \in \FF_p^n$.
In other words, the function
\[
  x,h,h' \mapsto F(x) - F(x+h) - F(x+h') + F(x+h+h')  = \partial_h \partial_{h'} F(x)
\]
is independent of $x$;
i.e., the second-order discrete derivatives $\partial_h \partial_{h'} F$ are all constant functions.
Functions $F$ with this property are exactly the quadratic polynomials $\FF_p^n \to \RR/\ZZ$, i.e., functions
\[
  (x_1,\dots,x_n) \mapsto \sum_{i,j=1}^n c_{ij} x_i x_j
\]
for coefficients $c_{ij} \in \RR/\ZZ$ with $p c_{ij} = 0$.

Write $B(h,h')$ for the constant value of $\partial_h \partial_{h'} F$.
For $F$ quadratic as above, we may compute
\[
  B(h,h') = \sum_{i,j}^n (c_{ij} + c_{ji}) h_i h'_j,
\]
the ``polarization'' of the quadratic form $F$.
In particular, $B$ is bilinear in the sense that
\begin{equation}%
  \label{eq:bilinear}
  B(h_2-h_1, h') + B(h_1,h') = B(h_2,h')
\end{equation}
for all $h_1,h_2,h'$, and similarly exchanging the first and second arguments.
Moreover, for a fixed non-zero integer $a$ we have $B(a h, h') = B(h, a h')$ for any $h,h' \in \FF_p^n$.
This is exactly the conclusion of Theorem~\ref{thm:baby-thm}, as a 100\% functional equation.

The above proof is not ``elementary'' because it uses a classification of quadratic polynomials, which we did not even prove, but it turns out that bilinearity as in~\eqref{eq:bilinear} may be derived from the hypothesis~\eqref{eq:hypo} in by a short, elementary, and even Cauchy--Schwarz-friendly argument.
Indeed, by definition of $B$ and rearranging~\eqref{eq:hypo}, we get
\[
  F(x+h) - F(x+h+h') + B(h,h') = F(x) - F(x+h')
\]
for all $x,h,h'$, and since the right-hand side does not depend on $h$,
\begin{align*}
  F(x+h_1) - F(x+h_1+h') + B(h_1,h') &= F(x) - F(x+h') \\
  &= F(x+h_2) - F(x+h_2+h') + B(h_2,h').
\end{align*}
Since
\[
  F(x+h_1) - F(x+h_1+h') - F(x+h_2) + F(x+h_2+h') = B(h_2-h_1, h')
\]
by change of variables, rearranging gives~\eqref{eq:bilinear}.

To complete the elementary proof, we would have to show that if $B$ is ``bilinear'' in the sense of~\eqref{eq:bilinear} then it is ``bilinear'' in the sense that for any fixed integer $a$, $B(a h, h') = B(h, a h')$ holds for all $h,h' \in \FF_p^n$.
In the world of usual algebra, this is straightforward: e.g., when $a=2$ we can apply~\eqref{eq:bilinear} twice:
\[
  B(h+h, h') = B(h,h') +B(h,h') = B(h, h'+h')
\]
and for general $a>0$ we could apply~\eqref{eq:bilinear} a total of $(a-1)$ times in each argument to show
\begin{equation}%
  \label{eq:unary-arg}
  \begin{aligned}
    B(ah, h') &= B((a-1)h, h') + B(h,h') \\
    &= \ldots = B(h,h') + \cdots + B(h,h') \\ %chktex 11
    &= \ldots = B(h, (a-1) h') + B(h, h') = B(h, ah'). %chktex 11
  \end{aligned}
\end{equation}
It is clear by definition that $B(-h,h') = -B(h,h')$, so the case $a<0$ is also straightforward.

\begin{remark}%
  \label{rem:big-proofs}
  Notice that the number of times we need to invoke~\eqref{eq:bilinear} grows with $|a|$.
  This is a feature of any argument proving $B(a h, h') = B(h, a h')$, not just the particular one given here.
  In the regime where $p$ is fixed but $a \in \FF_p$ is arbitrary, we are saying that although logically any $\ZZ$-bilinear function is $\FF_p$-bilinear, the ``cost'' of proving this grows with $p$.
  This is also the source of the dependence of proof-length on coefficient size in Theorem~\ref{thm:main}.
\end{remark}

Two issues prevent us from claiming an easy victory in Theorem~\ref{thm:baby-thm}.
\begin{enumerate}[label=\arabic*.]
  \item This last argument is not phrased in a Cauchy--Schwarz-friendly way.
    Steps such as ``quoting a previous result'' and ``substituting expressions'' were noted to be troublesome in points (A)--(D) in Section~\ref{sub:background}.
  \item Less fundamentally, this argument makes about $2(a-1)$ appeals to bilinearity, each of which was proved using one step of Cauchy--Schwarz type.
    So if we could turn this into a full Cauchy--Schwarz proof, we might expect to apply Cauchy--Schwarz at least $2(a-1)$ times, leading to a lower bound no better than $\delta^{2^{-O(|a|)}}$.
    This exponent is worse than what was claimed in Theorem~\ref{thm:baby-thm} by two exponentials.
\end{enumerate}

We first give a partial answer to the second issue.
We note that~\eqref{eq:unary-arg} uses a certain algorithm for multiplying by a fixed integer $a$ using repeated addition: namely, the trivial ``unary'' method
\[
  a h = h + h + \cdots + h
\]
that uses $a-1$ additions.
A more refined algorithm would use a ``binary'' or double-and-add method; for example,
\[
  13 h = h + 2( 2(h + 2 h))
\]
which uses $O(\log |a|)$ addition operations in general.
Applied to this example~\eqref{eq:unary-arg}, this would look like
\begin{equation}
  \label{eq:bilinearity-arg}
  \begin{aligned}%
    B(13h, h') &= B(12 h, h') + B(h,h') \\
               &= B(6 h, 2 h') + B(h,h') \\
               &= B(3h, 4h') + B(h,h') \\
               &= B(2h, 4h') + B(h, 4h') + B(h,h') = B(2h, 4h') + B(h, 5h') \\
               &= B(h, 8h') + B(h, 5h') = B(h, 13h')
  \end{aligned}
\end{equation}
where each step appeals either to~\eqref{eq:bilinear} or to the $a=2$ case (which uses~\eqref{eq:bilinear} twice).
Hence we may save one logarithm, at the expense of a slightly more complicated proof.

Meanwhile, Issue 1 motivates most of the main techniques of this paper.
In vague terms, our approach is to encode a calculation such as~\eqref{eq:bilinearity-arg} as a kind of ``arithmetic circuit''.
These are akin to logical circuits, except instead of ``wires'' carrying $0$ or $1$ values they carry one or more values in $\FF_p$, and instead of ``gates'' encoding logical equalities (such as $y= x_1 \vee x_2$) they encode additive arithmetic equalities (such as $y=x_1 + x_2$ in $\FF_p$).
Then, it suffices to show that we can ``build'' this circuit using Cauchy--Schwarz steps, or rather, whatever Cauchy--Schwarz steps look like in terms of circuits.

In the case of~\eqref{eq:bilinearity-arg}, a possible corresponding ``circuit'' is shown in Figure~\ref{fig:bilinear}.
We briefly explain the (informal) notation.
Each ``wire'' annotated $(\alpha,\beta)$ refers to an argument of the function $B$.
Some ``gates'' are annotated $+,=$ or $=,+$, which asserts that the incident wires $(\alpha_1, \beta_1)$, $(\alpha_2,\beta_2)$, $(\alpha_3,\beta_3)$ obey, respectively, $\alpha_1+\alpha_2=\alpha_3$ and $\beta_1=\beta_2=\beta_3$, or $\alpha_1=\alpha_2=\alpha_3$ and $\beta_1+\beta_2=\beta_3$.
To signify which of the three wires is the distinguished one $(\alpha_3,\beta_3)$ in these equations, we place a black dot.
Gates annotated $(-1,1)$ require their incident wires $(\alpha_1,\beta_1)$ and $(\alpha_2, \beta_2)$ to obey $\alpha_1=-\alpha_2$ and $\beta_1=\beta_2$.
Gates with no annotations, and dotted lines, do nothing at all, but are placeholders for where additional gates would go if the corresponding binary digit of $13$ were changed from $0$ to $1$.

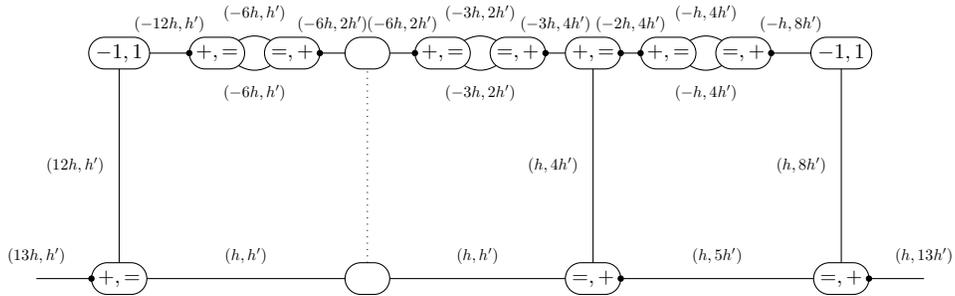
\begin{figure}[htbp]
  \begin{center}
    \begin{tikzpicture}[
      subnode/.style={rounded rectangle, draw, inner sep=2pt, outer sep=0pt,minimum size=15pt},
      mydot/.style={circle, draw, fill=black, inner sep=0pt, outer sep=0pt, minimum size=2pt},
      scale=1.0]

      \node[subnode,scale=0.8] (B0)  at (-0.3,3) {$-1,1$};
      \node[subnode,scale=0.8] (A11) at (1,3) {$+,=$};
      \node[mydot] at (A11.west) {};
      \node[subnode,scale=0.8] (A12) at (2,3) {$=,+$};
      \node[mydot] at (A12.east) {};
      \node[subnode,scale=0.8] (B1)  at (3,3) {$\phantom{--}$};
      
      \node[subnode,scale=0.8] (A21) at (4,3) {$+,=$};
      \node[mydot] at (A21.west) {};
      \node[subnode,scale=0.8] (A22) at (5,3) {$=,+$};
      \node[mydot] at (A22.east) {};
      \node[subnode,scale=0.8] (B2)  at (6,3) {$+,=$};
      \node[mydot] at (B2.east) {};

      \node[subnode,scale=0.8] (A31) at (7,3) {$+,=$};
      \node[mydot] at (A31.west) {};
      \node[subnode,scale=0.8] (A32) at (8,3) {$=,+$};
      \node[mydot] at (A32.east) {};
      \node[subnode,scale=0.8] (B3) at (9.3,3) {$-1,1$};
                                            
      \node[subnode,scale=0.8] (C0)  at (-0.3,0) {$+,=$};
      \node[mydot] at (C0.west) {};
      \node[subnode,scale=0.8] (C1)  at (3,0) {$\phantom{--}$};
      \node[subnode,scale=0.8] (C2)  at (6,0) {$=,+$};
      \node[mydot] at (C2.east) {};
      \node[subnode,scale=0.8] (C3)  at (9.3,0) {$=,+$};
      \node[mydot] at (C3.east) {};

      \draw (-1.4,0) node ["{$(13h,h')$}" {above,scale=0.6}] {} -- (C0.west);
      \draw (C0) -- (B0) node [midway, "{$(12h,h')$}" {left,scale=0.6}] {};
      \draw (B0) -- (A11) node [midway, "{$(-12h,h')$}" {above,scale=0.6,label distance=3.8}] {};
      \draw (A32) -- (B3) node [midway, "{$(-h,8 h')$}" {above,scale=0.6, label distance=3.8}] {};
      \draw (B3) -- (C3) node [midway, "{$(h,8h')$}" {left,scale=0.6}] {};

      \draw (C0) -- (C1) node [midway, "{$(h,h')$}" {above,scale=0.6}] {};
      \draw (C1) -- (C2) node [midway, "{$(h,h')$}" {above,scale=0.6}] {};
      \draw (C2) -- (C3) node [midway, "{$(h,5h')$}" {above,scale=0.6}] {};

      \draw[dotted] (C1) -- (B1);
      \draw (C2) -- (B2) node [midway, "{$(h,4h')$}" {left,scale=0.6}] {};

      \draw (C3) -- (10.4,0) node ["{$(h,13h')$}" {above,scale=0.6}] {};

      \draw (A11) to[bend left]  node[midway, "{$(-6h,h')$}" {above,scale=0.6}] {} (A12);
      \draw (A11) to[bend right] node[midway, "{$(-6h,h')$}" {below,scale=0.6}] {} (A12);
      \draw (A21) to[bend left]  node[midway, "{$(-3h,2h')$}" {above,scale=0.6}] {} (A22);
      \draw (A21) to[bend right] node[midway, "{$(-3h,2h')$}" {below,scale=0.6}] {} (A22);
      \draw (A31) to[bend left]  node[midway, "{$(-h,4h')$}" {above,scale=0.6}] {} (A32);
      \draw (A31) to[bend right] node[midway, "{$(-h,4h')$}" {below,scale=0.6}] {} (A32);

      \draw (A12) -- (B1) node[midway, "{$(-6h,2h')$}" {above,scale=0.6,label distance=3.8}] {};
      \draw (B1) -- (A21) node[midway, "{$(-6h,2h')$}" {above,scale=0.6,label distance=3.8}] {};
      \draw (A22) -- (B2) node[midway, "{$(-3h,4h')$}" {above,scale=0.6,label distance=3.8}] {};
      \draw (B2) -- (A31) node[midway, "{$(-2h,4h')$}" {above,scale=0.6,label distance=3.8}] {};
    \end{tikzpicture}
  \end{center}
  \caption{An ``arithmetic circuit'' for proving $B(13h, h') = B(h, 13h')$.}%
  \label{fig:bilinear}
\end{figure}

The use of negative values is essentially spurious, except that it allows the placement of ``dots'' to be more symmetrical, which will turn out to be important.
The reader can, if they wish, remove all minus signs and move dots around to get a simpler circuit.
We put them there so that we can refer back to this circuit without amendment in Section~\ref{sub:agate}.

The dictionary between systems of linear forms and Cauchy--Schwarz on the one hand, and these arithmetic circuits on the other, is spread over several stages and will take a large part of the paper to set up rigorously.
For now, we make a few general comments.
\begin{itemize}
  \item Very roughly, ``applying Cauchy--Schwarz'' to a circuit $\cC$ will mean replacing $\cC$ with two copies of itself, where any ``free wires'' (such as those on the far left and right in Figure~\ref{fig:bilinear}) in one copy may optionally be fused with their identical twins in the other copy.
    Thinking of a circuit as a kind of graph, one could say the resulting circuit will be ``reflection symmetric''.

    This symmetry is our enemy because it constrains the circuits we can build.
    Having instead the freedom to assemble \emph{arbitrary} circuits would roughly correspond to allowing arbitrary ``elementary proofs'' instead of Cauchy--Schwarz-friendly proofs.

  \item To defeat this symmetry it is essential that our circuits be \emph{programmable}.
    The ``build phase'' uses Cauchy--Schwarz steps to construct a highly symmetric circuit geometry, but if we can later freely ``assign'' which gates are $+,=$ gates, which are $=,+$ and which are $(-1,1)$ or ``null gates'' (i.e., the non-annotated blobs in Figure~\ref{fig:bilinear}), we can break symmetry and accomplish fairly general calculations.

  \item The primary advantage of this circuit formalism is that it is far easier for a human user to design and build arithmetic circuits than to manipulate systems of linear forms directly using Cauchy--Schwarz.
    The difference is to some extent heuristic rather than logical.
\end{itemize}

In this language it is not too unreasonable to say that our task is to ``build a computer'' using the Cauchy--Schwarz inequality.
This was also the stated purpose in~\cite{me}, so we might further clarify that our present aim is to build, as far as possible, a \emph{general-purpose} computer.
By contrast the arguments in~\cite{me} implemented only a very specialized set of operations, which happened to suffice for the problem being considered but do not generalize.

The points above also explain, in passing, how we save another logarithm as in Issue 2 above.
A single Cauchy--Schwarz operation on circuits can double the number of gates.
If we can perform multiplication by $a$ using $O(\log |a|)$ suitably programmable gates, then it only takes $\log_2 \log |a| + O(1)$ Cauchy--Schwarz steps to build a circuit large enough.

\subsection{Further techniques for the true complexity problem}%
\label{sub:true-summary}

We make some partial comments about the remaining ingredients in the proof of Theorem~\ref{thm:main}, or more precisely Theorem~\ref{thm:main-ext}, that do not appear during the discussion of the model problem in Section~\ref{sub:repr}.

It is reasonable to ask at this stage for an elementary proof of the 100\% functional equation corresponding to Theorem~\ref{thm:main-ext}: that is, for an elementary proof that if
\begin{equation}%
  \label{eq:ex-fun}
  \sum_{i=1}^k F_i(\phi_i(y_1,\dots,y_d)) = A
\end{equation}
for all $y_1,\dots,y_d \in \FF_p^n$ then $\partial_{h_1} \partial_{h_2} \dots \partial_{h_{s+1}} F_{i_0}(x) = 0$ for all $x,h_1,\dots,h_{s+1} \in \FF_p^n$, where $s=s(\Phi,i_0) \ge s(\Phi)-1$.
This is not immediately straightforward because we are forbidden from using any knowledge of the structure theory of polynomial maps.

For simplicity we first consider the model problem Example~\ref{ex:gw-ex}, where $d=3$, $k=6$, $i_0=1$, $\phi_1(x,y,z)=x$, $\phi_2(x,y,z)=x+z$, $\phi_3(x,y,z)=x+y$, $\phi_4(x,y,z)=x+y+z$, $\phi_5(x,y,z) = x+2y+3z$ and $\phi_6(x,y,z)=2x+3y+6z$,
and so we want to show $\partial_{h_1} \partial_{h_2} F_1(x) = 0$ for all $x,h_1,h_2$.

Our first observation is that we can show that $\partial_{h_1} \partial_{h_2} \partial_{h_3} F_i(x) = 0$ for all $i$, because the system has Cauchy--Schwarz complexity $2$.
The proof is by three duplication steps and is fairly standard, so we omit it.
Equivalently, we have shown that $(x,h_1,h_2) \mapsto \partial_{h_1} \partial_{h_2} F_i(x)$ is independent of $x$, and as in Section~\ref{sub:repr} we write $B_i(h_1,h_2)$ for this constant value.
Our goal is exactly to show $B_1$ is the zero function.

Crucially, as we showed in Section~\ref{sub:repr}, these functions $B_i$ are bilinear, and we may use the bilinear identities derived in that section freely.

Next let $h_1=(u_1,v_1,w_1) \in \FF_p^{3n}$ and $h_2=(u_2,v_2,w_2) \in \FF_p^{3n}$.
By differentiating~\eqref{eq:ex-fun} twice, in directions $h_1$ and $h_2$,
we obtain
\[
  \sum_{i=1}^6 B_i\bigl(\phi_i(h_1), \phi_i(h_2)\bigr) = 0.
\]
We can use several copies of this equation with different specializations of $h_1$ and $h_2$.
A reasonable if not optimally efficient choice is to take multiples of standard basis vectors: e.g., if $h_1 = (t,0,0)$ and $h_2 = (t',0,0)$ we obtain
\[
  B_1(t,t') + B_2(t,t') + B_3(t,t') + B_4(t,t') + B_5(t,t') + B_6(2t, 2t') = 0
\]
and we could replace the last term with $B_6(4t,t')$. 
If $h_1=(0,t,0)$ and $h_2 = (0,0,t')$ we get
\[
  B_1(0,0) + B_2(0,t') + B_3(t,0) + B_4(t,t') + B_5(2t,3t') + B_6(3t, 6t') = 0
\]
or equivalently
\[
  B_4(t,t') + B_5(6t,t') + B_6(18t, t') = 0.
\]
There are four more genuinely different choices.
It suffices to sum up these equations with appropriate rescalings $\alpha_1,\dots,\alpha_6$ of $t$, to deduce that $B_1(t,t') = 0$.
The problem is exactly to find $(\alpha_1,\dots,\alpha_6) \in \FF_p^6$ such that
\begin{equation}
  \begin{pmatrix} \alpha_1 & \cdots & \alpha_6 \end{pmatrix}
  \begin{pmatrix}
    1 &  1 &  1 &  1 &  1 &  4 \\
    0 &  0 &  1 &  1 &  2 &  6 \\
    0 &  1 &  0 &  1 &  3 & 12 \\
    0 &  0 &  1 &  1 &  4 &  9 \\
    0 &  0 &  0 &  1 &  6 & 18 \\
    0 &  1 &  0 &  1 &  9 & 36
  \end{pmatrix}
  = \begin{pmatrix} 1 & 0 & 0 & 0 & 0 & 0 \end{pmatrix}
  \label{eq:matrix-eq}
\end{equation}
where the columns of the matrix correspond to $B_1,\dots,B_6$ and our calculations above produced the first and fifth rows.
The solution is
\[
  (\alpha_1,\dots,\alpha_6) = (1, -1/5, -14/15, -4/5, 1, -1/15)
\]
over $\QQ$, which gives a valid solution if $p>5$.

In general, the linear algebra problem we have set up is exactly to find an element $T \in (\FF_p^d)^{\otimes (s+1)}$ such that $\phi_i^{\otimes (s+1)}(T) = 1$ and $\phi_j^{\otimes (s+1)}(T) = 0$ for $j \ne i$.
In the above, where $s=1$ and $d=3$, we would take
\[
  T = e_1 \otimes e_1 - \frac15 e_1 \otimes e_2 - \frac{14}{15} e_1 \otimes e_3 - \frac{4}{5} e_2 \otimes e_2 + e_2 \otimes e_3 - \frac{1}{15} e_3 \otimes e_3.
\]
This is possible if and only if the true complexity of $\Phi$ at $i$ is at most $s$.
For example, for the second system of linear forms from Example~\ref{ex:gw-ex}, the analogous matrix equation would not have a solution.

In this case, the Cauchy--Schwarz complexity and the true complexity differ by one.
In general the difference could be arbitrarily large.
Hence to extend this argument to the general case, it needs to be applied iteratively, a process called ``degree reduction'' in~\cite{peluse-finite,peluse-integers}.
In other words:---
\begin{itemize}
  \item we first show $\partial_{h_1} \dots \partial_{h_{t+1}} F_j(x) = 0$ for each $j$, where $t$ is the (maximum) Cauchy--Schwarz complexity of the system;
  \item we do degree $t$ multilinear algebra on the functions $B_j(h_1,\dots,h_t)$ and prove $B_j=0$ for all $j$;
  \item this means that $\partial_{h_1} \dots \partial_{h_t} F_j(x) = 0$ for each $j$; and
  \item we proceed inductively until we have established $B_{i_0}(h_1,\dots,h_{s+1}) = 0$ where $s=s(\Phi,i_0)$.
\end{itemize}
Note that we can only perform the last step if $B_j(h_1,\dots,h_{s+1})$ for $j \ne i_0$ are well-defined functions, which is the origin of the hypothesis $s(\Phi) \le s(\Phi,i_0)$ in Theorem~\ref{thm:main-ext}.

These arguments serve as the inspiration for the proof of Theorem~\ref{thm:main-ext}, but again we are far from done because they do not have Cauchy--Schwarz type.
Instead we must implement these multilinear calculations in terms of an arithmetic circuit, which is itself built by Cauchy--Schwarz steps, as discussed in Section~\ref{sub:repr}.

There is one further crucial technique we will need to implement this degree-lowering strategy, which is not required explicitly for Theorem~\ref{thm:baby-thm}.
As noted in Section~\ref{sub:background}, one weakness of Cauchy--Schwarz style proofs is their serial nature: it is not obvious how to ``save'' an intermediate equation for future use.
However under some circumstances it is possible to simulate this.

For example, to implement the above argument for Example~\ref{ex:gw-ex} in inequalities, an early step is to show, say,
\begin{equation}
  \label{eq:u3-control}
  \lvert \Lambda_{\Phi} (f_1,\dots,f_6) \rvert \le \|f_6\|_{U^3}
\end{equation}
as this is a translation of the statement $\partial_{h_1} \partial_{h_2} \partial_{h_3} F_6(x) = 0$ for all $x,h_1,h_2,h_3$.
Having done this, we claim the following consequence: assuming $\lvert \Lambda_\Phi(f_1,\dots,f_6) \rvert \ge \eps$, there exists a $1$-bounded function $f'_6$ such that $\lvert \Lambda(f_1,\dots,f_5,f_6') \rvert \ge \eps^{16}$, and moreover $f'_6$ is itself a \emph{$U^3$-dual function}
\[
  f_6'(x) = D_{U^3}(g) = \EE_{h_1,h_2,h_3 \in \FF_p^n} \begin{aligned}[t]
    &g(x+h_1) g(x+h_2) \overline{g(x+h_1+h_2)} g(x+h_3) \\
    &\overline{g(x+h_1+h_3)} \overline{g(x+h_2+h_3)} g(x+h_1+h_2+h_3)
  \end{aligned}
\]
for some $1$-bounded function $g \colon \FF_p^n \to \CC$.

The proof of this fact is very closely analogous to arguments in~\cite{peluse-finite,peluse-integers}, which use a Hahn--Banach decomposition in the style of~\cite{gowers-hb}.
Conveniently we can phrase the argument here purely in terms of the Cauchy--Schwarz inequality.

If we write
\[
  D_{\Phi,6}(f_1,\dots,f_5) = x \mapsto \EE_{\substack{v \in \FF_p^{3n} \\ \phi_6(v)=x}} f_1(\phi_1(v)) \dots f_5(\phi_5(v))
\]
for the ``$\Phi$-dual function at position $6$'' of $f_1,\dots,f_5$,
then
\begin{align*}
  \eps \le \lvert \Lambda_{\Phi}(f_1,\dots,f_6) \rvert &= \biggl\lvert \EE_{x \in \FF_p^n} f_6(x) D_{\Phi,6}(f_1,\dots,f_5)(x) \biggr\rvert \\
                                              &\le \left( \EE_{x \in \FF_p^n} \bigl\lvert D_{\Phi,6}(f_1,\dots,f_5)(x) \bigr\rvert^2 \right)^{1/2}
\end{align*}
by Cauchy--Schwarz and $1$-boundedness, and
\[
  \EE_{x \in \FF_p^n} \bigl\lvert D_{\Phi,6}(f_1,\dots,f_5)(x) \bigr\rvert^2 = \Lambda_{\Phi}\bigl(f_1,\dots,f_5, \overline{D_{\Phi,6}(f_1,\dots,f_5)}\bigr).
\]
Writing $g = \overline{D_{\Phi}(f_1,\dots,f_5)}$, 
by~\eqref{eq:u3-control} we have
\[
  \lvert \Lambda_{\Phi}(f_1,\dots,f_5, g) \rvert \le \lVert g \rVert_{U^3}
\]
and by definition
\[
  \lVert g \rVert_{U^3}^8 = \EE_{x \in \FF_p^n} \bigl( D_{U^3}(g)(x) \bigr) \overline{g(x)}.
\]
Setting $f'_6 = D_{U^3}(g)$, we deduce
\[
  \lVert g \rVert_{U^3}^8 = \EE_{x \in \FF_p^n} \bigl( D_{\Phi,6}(f_1,\dots,f_5)(x)  \bigr) f'_6(x) = \Lambda_{\Phi}\bigl(f_1,\dots,f_5, f_6'\bigr) \ge \eps^{16}.
\]

The rough summary is as follows.
The first Cauchy--Schwarz step creates two copies of $\Phi$.
In one of them, we apply our proof of~\eqref{eq:u3-control}.
But then by rediscovering the other copy of $\Phi$ (in the form of a $\Phi$-dual function) we reinterpret the multilinear average we have created as $\Lambda_{\Phi}(f_1,\dots,f_6')$.

The effect is that the proof that $\Lambda_{\Phi}(f_1,\dots,f_6)$ is controlled by $\lVert f_6 \rVert_{U^3}$ has now been ``banked'' as an extra hypothesis that $f_6$ is itself a $U^3$-dual function.
This information can \emph{only ever help us}: if at any stage it is not useful to think of some copy of $f_6$ as a $U^3$-dual function, we can revert to treating it as an arbitrary $1$-bounded function, and have lost nothing except a power of $16$.
On the other hand, sometimes it will be extremely useful to know that $f_6$ is a $U^3$-dual function, because we can start to apply bilinearity inequalities as in Section~\ref{sub:repr}.

Finally, we note that the quantity ``$\Lambda_{\Phi}(f_1,\dots,f_6)$ where $f_6 = D_{U^3}(g)$'' can be encoded as a larger multilinear average $\Lambda_{\Phi'}(f_1,\dots,f_5,g, \overline{g}, \dots, g)$ for some larger system of linear forms $\Phi'$ obtained by gluing together $\Phi$ with the system of linear forms corresponding to the $U^3$-norm. 
Hence all these arguments will stay neatly within the Cauchy--Schwarz formalism we will introduce in Section~\ref{sec:framework}.

We refer to this technique as \emph{stashing} or \emph{stashed structure},
because we ``stash'' some hypothesis or intermediate conclusion as a dual-function hypothesis on some function, until we are ready to use it later in the proof.
In a way, the argument above uses stashing twice: at the start we stash a $\Phi$-dual function hypothesis for later use, then at the end we stash a $U^3$-dual function hypothesis.

\subsection{Relation to previous work}%
\label{sub:previous-work}

We have already discussed the points of comparison between this paper and previous work of Gowers~\cite{gowers-szemeredi}, Gowers and Wolf~\cite{gw1,gw2,gw3,gw4}, Green and Tao~\cite{gt-linear}, and Peluse~\cite{peluse-finite,peluse-integers}, as well as~\cite{gt,hhl,hl,altman,szegedy}.

We make one further remark about the comparison between this work and~\cite{me}.
As we have said, the latter solved a particular case of Theorem~\ref{thm:main-ext} by building a specialized kind of ``circuit'' that does not work in general, whereas here we try to build ``general-purpose'' circuits.

It is worth emphasizing, though, the increased \emph{cost} of the latter approach.
The key ``circuit'' used in~\cite{me} to prove the case Example~\ref{ex:gw-ex} could be rendered in the formal notation we will develop in Section~\ref{sub:labelling}, as
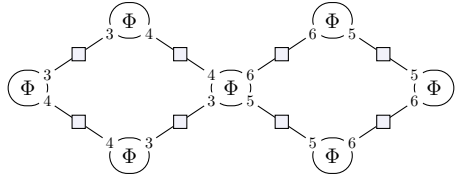
\begin{center}
  \begin{tikzpicture}[
      defnode/.style={rectangle,draw,fill=lightgray,inner sep=2pt,outer sep=0pt,minimum size=10pt},
      gatelabel/.style={circle, inner sep=1pt, outer sep=0pt, fill=white,scale=0.6},
      subnode/.style={rounded rectangle, draw, inner sep=5pt, outer sep=0pt,minimum size=15pt},
      leaf/.style={rectangle,draw,fill=leafgreen,inner sep=2pt,outer sep=0pt,minimum size=10pt},
      scale=0.90
    ]
    \node[subnode, scale=0.8] (A1) at (-3.0,   0) {$\Phi$};
    \node[subnode, scale=0.8] (A2) at (-1.5, 1.0) {$\Phi$};
    \node[subnode, scale=0.8] (A3) at (-1.5,-1.0) {$\Phi$};
    \node[subnode, scale=0.8] (A4) at (   0,   0) {$\Phi$};

    \node[subnode, scale=0.8] (B2) at ( 1.5, 1.0) {$\Phi$};
    \node[subnode, scale=0.8] (B3) at ( 1.5,-1.0) {$\Phi$};
    \node[subnode, scale=0.8] (B1) at ( 3.0,   0) {$\Phi$};

    \draw (A1) -- node[gatelabel,pos=0.05] {$3$} node[midway, defnode,scale=0.5] {} node[gatelabel,pos=0.95] {$3$} (A2);
    \draw (A1) -- node[gatelabel,pos=0.05] {$4$} node[midway, defnode,scale=0.5] {} node[gatelabel,pos=0.95] {$4$} (A3);
    \draw (A4) -- node[gatelabel,pos=0.05] {$4$} node[midway, defnode,scale=0.5] {} node[gatelabel,pos=0.95] {$4$} (A2);
    \draw (A4) -- node[gatelabel,pos=0.05] {$3$} node[midway, defnode,scale=0.5] {} node[gatelabel,pos=0.95] {$3$} (A3);

    \draw (B1) -- node[gatelabel,pos=0.05] {$5$} node[midway, defnode,scale=0.5] {} node[gatelabel,pos=0.95] {$5$} (B2);
    \draw (B1) -- node[gatelabel,pos=0.05] {$6$} node[midway, defnode,scale=0.5] {} node[gatelabel,pos=0.95] {$6$} (B3);
    \draw (A4) -- node[gatelabel,pos=0.05] {$6$} node[midway, defnode,scale=0.5] {} node[gatelabel,pos=0.95] {$6$} (B2);
    \draw (A4) -- node[gatelabel,pos=0.05] {$5$} node[midway, defnode,scale=0.5] {} node[gatelabel,pos=0.95] {$5$} (B3);
  \end{tikzpicture}
\end{center}
which describes a system of $26$ linear forms.
To solve the \emph{same problem} Example~\ref{ex:gw-ex} in this paper, we use the circuit shown in Figure~\ref{fig:stargate}, where in Russian doll fashion every blob marked $A(-,-)$ is a copy of the circuit in Definition~\ref{def:large-agate}, and every blob there is a copy of the circuit in Definition~\ref{def:agate}.
Unpacked in full, this circuit has on the order of thousands of forms.

In~\cite{me} it was plausible if unwieldy to analyse the system of linear forms by hand.
Here that would be impractical, which is why we spend a significant portion of the paper introducing a formalism to automate many of the steps.
This formalism is rather heavy, and it is likely that it could be improved to require somewhat less set-up or give somewhat cleaner proofs, but having \emph{no formalism at all} is simply not an option.

We also briefly remark on the relationship of these methods to Razborov's \emph{flag algebras}.
There is a reasonably close analogy, where 
\begin{center}
  \begin{tabular}{rcl}
    flags / small finite graphs $G$ &$\leftrightarrow$& linear data $\Phi$ \\
    homomorphism densities $p(G,-)$ &$\leftrightarrow$& $\Lambda$-values $\Lambda_{\Phi}(-)$.
  \end{tabular}
\end{center}
Both frameworks make heavy use of the Cauchy--Schwarz inequality, implicitly or explicitly, and have something to do with first-order logic.

However there are a few key differences.
\begin{itemize}
  \item We always allow negative- or complex-valued functions, so many positivity arguments do not arise.
  \item We are asking much less delicate questions in the setting of linear data than are typically asked using flag algebras.
    For instance, many flag algebra proofs are searching for an optimal constant $C$ such that $p(G_1,-) \ge C$ subject to some constraints on $p(G_2,-), p(G_3,-)$, etc..
  Here we only care about statements of the form: if $\Lambda_{\Phi}(-)$ is close to zero then $\Lambda_{\Psi}(-)$ must also necessarily be fairly close to zero.
  \item On the other hand, the underlying linear-algebraic theory of linear data is much richer than the basic theory of graphs, and statements which would be trivial to prove or disprove for graphs require complicated arguments for linear data.
\end{itemize}
As a result, in practice there is not much cross-over between the tools that are useful in each case.

\subsection{Outline of the remainder of the paper}%
\label{sub:outline}

A large part of the work of this paper is in building a sufficiently powerful formalism that the Cauchy--Schwarz proofs of Theorem~\ref{thm:main-ext} and Theorem~\ref{thm:baby-thm}, when we come to them, will be easy to state and intuitive to justify.
This process is fairly lengthy, but as we have discussed, the alternative is far worse.

We present the first part in Section~\ref{sec:framework}.
The key definitions are that of a \emph{linear datum} (also appearing in~\cite{me}), and a \emph{Cauchy--Schwarz diagram}.
We also build the necessary theory for manipulating them and proving inequalities.

Next, Section~\ref{sec:gates} attempts to bridge the gap between Cauchy--Schwarz diagrams and ``arithmetic circuits'' of the kind discussed above.

In Section~\ref{sec:baby-thm} we begin in earnest building such circuits, and use them to prove Theorem~\ref{thm:baby-thm}.
We prove Theorem~\ref{thm:main} (and Theorem~\ref{thm:main-ext}) in Section~\ref{sec:thm-main}, after introducing some further general theory relating to stashing in Section~\ref{sec:stashing} (see Section~\ref{sub:true-summary}). 

Finally, Section~\ref{sec:conjectures} discusses some general conjectures.

\subsection{Acknowledgements}%
\label{sub:ack}

The author is grateful to Sarah Peluse, Ben Green, Julia Wolf, Sean Prendiville and Daniel Altman for informative discussions on these topics.

\section{The Cauchy--Schwarz framework}%
\label{sec:framework}

We introduce some definitions and notation that constitute the language of iterated Cauchy--Schwarz proofs we will use.
There are several layers which are introduced in turn.

\subsection{Linear data}%
\label{sub:linear-datum}

The notion of a \emph{linear datum} was introduced in~\cite{me} as a conservative generalization of the notion of a system of linear forms.
We recall a definition, with some minor changes relative to~\cite{me}.

\begin{definition}%
  \label{def:datum}
  Let $p$ be a prime.  A \emph{linear datum} over $\FF_p$ consists of a tuple $\Phi = \big(I, V, (W_i)_{i \in I}, (\phi_i)_{i \in I}\big)$ where $I$ is a finite index set, $V$ and $W_i$ are finite-dimensional vector spaces over $\FF_p$ and $\phi_i \colon V \to W_i$ are linear maps.
\end{definition}
Suppose $n \ge 1$ and $f_i \colon W_i^n \to \CC$ for $i \in I$ are functions.  Exactly as for a system of linear forms, we may define the multilinear average
\begin{equation}%
  \label{eq:lambda-def}
  \Lambda_{\Phi}\big( (f_i)_{i \in I} \big) = \EE_{v \in V^n} \prod_{i \in I} f_i\bigl(\phi_i(v)\bigr).
\end{equation}
Again we abuse notation to write $\phi_i$ for the map $\phi_i^n \colon V^n \to W_i^n$ applied entry-wise.
When $W_i = \FF_p$ for each $i\in I$, then $\bigl(\phi_i \colon V \to \FF_p\bigr)_{i \in I}$ is a system of linear forms, and this definition of $\Lambda_{\Phi}$ coincides with~\eqref{eq:system-of-forms}.

\begin{convention}%
  \label{conv:datum-tuple}
  To avoid running out of variable names, we adopt the convention that $I^\Phi$, $V^\Phi$, $W^{\Phi}_i$, $\phi^{\Phi}_i$ refer to the constituent parts of the linear datum tuple $\Phi$.  We will use similar notation for other definitions.
\end{convention}

Clearly any system of linear forms is an example of a linear datum.
We take the opportunity to give a few other standard examples.

\begin{definition}%
  \label{def:trivial}
  Suppose $I$ is an index set and $(W_i)_{i \in I}$ is a collection of finite-dimensional vector spaces.
  The linear datum $\trivial\bigl(I, (W_i)_{i \in I}\bigr)$ has these spaces as $W_i$, $V = \bigoplus_{i \in I} W_i$ and $\phi_i \colon V \to W_i$ the corresponding projection map.

  If all the $W_i$ are $\FF_p$, we just write $\trivial(I)$ in place of $\trivial\bigl(I, (\FF_p)_{i \in I}\bigr)$ for the trivial system of linear forms.
\end{definition}

\begin{definition}%
  \label{def:const}
  Suppose $W$ is a single finite-dimensional vector space and $I$ is an index set.
  We write $\const(I,W)$ for the linear datum with $W_i = W$ for all $i \in I$, $V = W$ and $\phi_i \colon V \to W_i$ the identity map for each $i$.

  Again if $W = \FF_p$ this is abbreviated to $\const(I)$.
\end{definition}

\begin{definition}%
  \label{def:uk}
  Let $k \ge 1$ and let $W$ be a finite-dimensional vector space.
  We write $U^k(W)$ for the linear datum with index set $I = \{0,1\}^k$, $W_i = W$ for $i \in I$, $V = W^{k+1}$ and $\phi_\omega \colon V \to W_i$ for $\omega = (\omega_1,\dots,\omega_k) \in \{0,1\}^k$ given by
  \[
    \phi_{\omega}(x,h_1,\dots,h_k) = x + \omega_1 h_1 + \cdots + \omega_k h_k.
  \]
  Again we abbreviate $U^k(\FF_p)$ to $U^k$.  
\end{definition}

\begin{remark}%
  \label{rem:surj-degen}
  There are two natural additional hypotheses on $\Phi$ which it is convenient not to insist upon in the definition itself.

  First, consider the joint kernel of all the linear maps $\phi_i$,
  \[
    K = \bigcap_{i \in I} \ker \phi_i \subseteq V.
  \]
  We say $\Phi$ is \emph{non-degenerate} if $K = \{0\}$ and \emph{degenerate} otherwise.
  Replacing $v \in V^n$ by $v+k$ for any $k \in K^n$ leaves all the terms $f_i(\phi_i(v))$ in~\eqref{eq:lambda-def} unchanged, and so the part of the averaging $\EE_{v \in V^n}$ over $K^n$ is completely spurious.
  In other words, the same averaging operator $\Lambda_{\Phi}$ factors through a smaller average.
  For example, $\EE_{x,y \in \FF_p^n} f(x-y)$ is degenerate and equivalent to the smaller average $\EE_{t \in \FF_p^n} f(t)$.

  Second, it is not assumed that $\phi_i$ are all surjective maps.\footnote{This choice differs from the definitions in~\cite{me}.}
  We say that $\Phi$ is \emph{surjective at $i \in I^\Phi$} if $\phi_i^\Phi$ is a surjective map, and that $\Phi$ is \emph{surjective} if it is surjective at every $i \in I^\Phi$.

  It is trivial to modify $\Phi$ so that one or both of these conditions holds, without affecting the values of any multilinear averages $\Lambda_\Phi$.  Indeed, to make $\Phi$ non-degenerate we simply replace $V$ with $V/K$, and to make it surjective we replace $W_i$ with $\im \phi_i$ (and $f_i$ with its restriction to $\im \phi_i$).
\end{remark}

Our goal is to prove inequalities of the form
\begin{equation}%
  \label{eq:lambda-ineq}
  \bigl\lvert \Lambda_{\Phi}\bigl((f_j)_{j \in I}\bigr) \bigr\rvert \le \bigl\lvert \Lambda_{\Psi}\bigl((f'_j)_{j \in I'}\bigr) \bigr\rvert^{1/C}
\end{equation}
for related\footnote{%
  \label{foot:cut-norm}%
  Given a datum $\Lambda$, an index $i$ and an integer $n$, one could define an associated \emph{cut-norm} on functions $f \colon (W_i^\Phi)^n \to \CC$ by
  \[
    \lVert f \rVert_{\opr{cut},\Phi,i} = \sup_{(f_j)_{j \ne i}} \bigl\lvert \Lambda_{\Phi}\bigl( (f_j)_{j \ne i}, i \mapsto f\bigr) \bigr\rvert
  \]
  where the supremum is over all tuples of $1$-bounded functions $f_j \colon (W_j^\Phi)^n \to \CC$.
  When $(f_j)_{j \in I}$ and $(f'_j)_{j \in I'}$ in~\eqref{eq:lambda-ineq} have one function $f_i=f'_{i'}$ in common but are otherwise arbitrary $1$-bounded functions,~\eqref{eq:lambda-ineq} is exactly an inequality $\lVert f \rVert_{\opr{cut},\Phi,i} \le \lVert f \rVert_{\opr{cut}, \Psi, i'}^{1/C}$.
  This approach is expounded in~\cite{tao-cut-blog}.
  We do not adopt it here explicitly, but note that discussing families of linear data is a roughly equivalent perspective to discussing families of cut-norms.
}
tuples of functions $(f_i)$, $(f'_i)$.
Before discussing Cauchy--Schwarz, we note there are a number of elementary ways to achieve such inequalities with $C=1$.
By way of exposition we give some examples.
Suppose
\[
  \biggl\lvert \EE_{x,a,b \in \FF_p^n} f_1(x) f_2(x+a) f_3(x+b) f_4(x+a+b) \biggr\rvert = \eps
\]
for $1$-bounded functions $f_i$. Then
\begin{equation}
  \label{eq:gen-conv}
  \biggl\lvert \EE_{x,a \in \FF_p^n} f_1(x) f_2(x+a) g(a) \biggr\rvert \ge \eps
\end{equation}
holds for some $1$-bounded function $g \colon \FF_p^n \to \CC$.
Indeed, we may change variables from $b$ to $t=x+b-a$ (say) to get
\begin{equation}
  \label{eq:u2conv-cov}
  \eps = \biggl\lvert \EE_{x,a,t \in \FF_p^n} f_1(x) f_2(x+a) f_3(t+a) f_4(t+2a) \biggr\rvert
\end{equation}
then set $\gamma(t,a) = f_3(t+a) f_4(t+2a)$ (``merging'' the two functions $f_3$ and $f_4$).
We then observe (by the triangle inequality)
\begin{align*}
  \biggl\lvert \EE_{x,a,t \in \FF_p^n} f_1(x) f_2(x+a) \gamma(t,a) \biggr\rvert &\le \EE_{t \in \FF_p^n} \biggl\lvert \EE_{x,a \in \FF_p^n} f_1(x) f_2(x+a) \gamma(t,a) \biggr\rvert \\
                                                   &\le \max_{t \in \FF_p^n} \biggl\lvert \EE_{x,a \in \FF_p^n} f_1(x) f_2(x+a) \gamma(t,a) \biggr\rvert
\end{align*}
and set $g(a) = \gamma(t,a)$ for some fixed $t$ achieving the maximum value on the right-hand side.

An alternative argument is that by a change of variables $y=x+b$,~\eqref{eq:gen-conv} implies
\[
  \biggl\lvert \EE_{x,a,y \in \FF_p^n} f_1(x) f_2(x+a) f_3(y) f_4(y+a) \biggr\rvert = \eps
\]
and by the triangle inequality
\[
  \EE_{y \in \FF_p^n} |f_3(y)|\ \biggl\lvert \EE_{x,a \in \FF_p^n} f_1(x) f_2(x+a) f_4(y+a) \biggr\rvert \ge \eps
\]
which again implies that for at least one value of $y$,
\[
  \biggl\lvert \EE_{x,a \in \FF_p^n} f_1(x) f_2(x+a) f_4(y+a) \biggr\rvert \ge \eps.
\]
Setting $g(a)=f_4(y+a)$ for this $y$, we recover the original statement but with the extra information that $g$ can be taken to be some translate of $f_4$.

Other steps include introducing or removing spurious averaging, or spuriously increasing the domain of a function: for example, if~\eqref{eq:gen-conv} holds then clearly
\[
  \biggl\lvert \EE_{x,a,b \in \FF_p^n} f_1(x) f_2(x+a) g(a) \biggr\rvert \ge \eps
\]
and moreover
\[
  \biggl\lvert \EE_{x,a,b \in \FF_p^n} f_1(x) f_2(x+a) h(a,b) \biggr\rvert \ge \eps
\]
for some $1$-bounded function $h \colon \FF_p^{2n} \to \CC$, namely $h(a,b) = g(a)$.

All of these inequalities may be summarized in the following general scheme.
There is a natural notion of a \emph{morphism} of linear data, i.e., a structure-preserving map $\Phi \to \Psi$ between two linear data $\Phi$ and $\Psi$.
Proposition~\ref{prop:morphism} below
 shows that any morphism $\Phi \to \Psi$ induces an inequality $\lvert \Lambda_{\Psi}(\dots) \rvert \le \lvert \Lambda_{\Phi}(\dots)\rvert$, and all the inequalities above arise in this way.

We first introduce the following useful convention.
Note that introducing a new space $W=\{0\}$ with the zero map $V \to W$ to a linear datum $\Phi$ has essentially no effect: for instance, it changes $\Lambda_{\Phi}$ by introducing a function $\{0\} \to \CC$, i.e., a constant.
It is sometimes convenient to act as though such a zero subspace were present---essentially, to absorb constants---but it would be overly cumbersome to actually include one.
We circumvent this as follows.
\begin{convention}%
  \label{convention:ast}
  For any linear datum $\Phi$,
  the element $\zeroi$ denotes an arbitrary symbol not contained in $I^\Phi$.   
  We write $W^\Phi_{\zeroi}$ to denote the zero space $\{0\}$ and $\phi^\Phi_\zeroi$ to denote the zero map $V^\Phi \to W^\Phi_\zeroi$.
\end{convention}

Given this convention,
the definition of a \emph{morphism of linear data} is fairly clean.
\begin{definition}%
  \label{def:morphism}
  Let $\Phi$ and $\Psi$ be two linear data.
  Suppose $\alpha \colon I^{\Psi} \to I^{\Phi} \cup \{\zeroi\}$ is a function.

  A \emph{morphism of linear data} $\Theta \colon \Phi \to \Psi$ of \emph{shape $\alpha$} consists of (i) a linear map $\theta \colon V^{\Phi} \to V^{\Psi}$, and (ii) for each $i \in I^{\Psi}$, a morphism $\sigma_i \colon W^{\Phi}_{\alpha(i)} \to W^{\Psi}_i$, such that the diagram
  \begin{equation}%
    \label{eq:morph}
    \begin{tikzcd}
      V^{\Phi} \arrow[d, "\phi^{\Phi}_{\alpha(i)}"] \arrow[r,"\theta"] & %chktex 18
      V^{\Psi} \arrow[d, "\phi^{\Psi}_{i}"] \\ %chktex 18
      W^{\Phi}_{\alpha(i)} \arrow[r,"\sigma_i"] & %chktex 18
      W^{\Psi}_{i}
    \end{tikzcd}
  \end{equation}
  commutes for each $i \in I^{\Psi}$.

  For $i \in I^{\Psi}$, we say the morphism $\Theta$ \emph{respects $i$} if (a) $\alpha(i) \ne \zeroi$, (b) for all other $j \in I^{\Psi}$ we have $\alpha(j) \ne \alpha(i)$, and (c) $W^{\Psi}_i = W^{\Phi}_{\alpha(i)}$ and $\sigma_i$ is the identity map.
\end{definition}

As in Convention~\ref{conv:datum-tuple}, we may write $\alpha^{\Theta}$, $\theta^{\Theta}$ or $\sigma^{\Theta}_i$ to denote the constituent parts of a morphism $\Theta$.

We now explain how between morphisms of linear data give rise to inequalities.
This is a slight generalization of~\cite[Proposition 3.4]{me}.
\begin{proposition}%
  \label{prop:morphism}
  Suppose $\Phi$, $\Psi$ are linear data and $\Theta \colon \Phi \to \Psi$ is a morphism.
  Let $I_1 \subseteq I^{\Psi}$ denote the set of all $i$ respected by $\Theta$.

  Then the following holds: for $n \ge 1$ and any tuple of $1$-bounded functions $(f_i)_{i \in I^{\Psi}}$, $f_i \colon (W^\Psi_i)^n \to \CC$, there exist $1$-bounded functions $(g_j)_{j \in I^{\Phi}}$, $g_j \colon (W^{\Phi}_j)^n \to \CC$ such that
  \[
    \bigl\lvert\Lambda_{\Psi}\bigl((f_i)_{i \in I^{\Phi}}\bigr) \bigr\rvert \le \bigl\lvert\Lambda_{\Phi}\bigl((g_j)_{j \in I^{\Psi}}\bigr)\bigr\rvert.
  \]
  Moreover, for each $i \in I_1$, $g_{\alpha(i)}$ is a translate of $f_i$; that is, there exists $t \in (W^\Psi_i)^n$ such that $g(x) = f(x+t)$ for all $x \in (W_i^\Psi)^n$.
\end{proposition}
\begin{proof}%
  The proof is much the same as~\cite[Proposition 3.4]{me}, but we give it for completeness.
  Take $v_1,\dots,v_M \in (V^\Psi)^n$ to be a complete list of coset representatives of the subgroup $\theta^\Theta\bigl((V^{\Phi})^n\bigr) \le (V^\Psi)^n$.
  We compute
  \begin{align*}
    \bigl\lvert\Lambda_{\Psi}\bigl((f_i)_{i \in I^\Psi}\bigr)\bigr\rvert &= \left\lvert \EE_{v \in (V^{\Psi})^n} \prod_{i \in I^\Psi} f_i\bigl(\phi^\Psi_i(v)\bigr) \right\rvert \\
                                                     &= \left\lvert \EE_{\ell \in [M]} \EE_{u \in (V^{\Phi})^n} \prod_{i\in I} f_i\bigl(\phi^\Psi_i(v_\ell + \theta^\Theta(u))\bigr) \right\rvert
  \end{align*}
  and by the triangle inequality this is at most
  \[
    \max_{\ell \in [M]} \left\lvert \EE_{u \in (V^\Phi)^n} \prod_{i \in I^\Psi} f_i\bigl(\phi^\Psi_i(v_\ell + \theta^\Theta(u))\bigr) \right\rvert.
  \]
  Pick some $\ell \in [M]$ that achieves the maximum value.
  By~\eqref{eq:morph}, for each $i \in I^\Psi$,
  \[
    f_i\bigl(\phi^\Psi_i(v_\ell + \theta^\Theta(u))\bigr) = f_i\bigl(\phi^\Psi_i(v_\ell) + \sigma_i(\phi^\Phi_{\alpha(i)}(u))\bigr).
  \]
  For $j \in I^\Phi \cup \{\zeroi\}$ define
  \[
    g_j(w) = \prod_{i \in I^\Psi \colon \alpha(i) = j} f_i\bigl(\phi^\Psi_i(v_\ell) + \sigma_i(w)\bigr).
  \]
  Since $f_i$ are all $1$-bounded so are $g_j$, and
  \[
    \left\lvert \EE_{u \in (V^\Phi)^n} \prod_{i \in I^\Psi} f_i\bigl(\phi^\Psi_i(v_\ell + \theta^\Theta(u))\bigr) \right\rvert = \bigl\lvert g_\zeroi(0) \bigr\rvert  \, \bigl\lvert \Lambda_{\Phi}((g_j)_{j \in I^\Phi})\bigr\rvert.
  \]
  The bound $\bigl\lvert \Lambda_{\Psi}\bigl((f_i)_{i \in I^\Psi}\bigr) \bigr\rvert \le \bigl\lvert \Lambda_{\Phi}\bigl((g_j)_{j \in I^\Phi}\bigr) \bigr\rvert$ follows, and it is clear by construction that $g_{\alpha(i)}$ is a translate of $f_i$ when $i \in I_1$.
\end{proof}

\begin{example}%
  \label{ex:morphism-application}
  The example calculation above has $I^\Psi=[4]$, $I^\Phi=[3]$, $V^\Psi = \FF_p^3$, $W^\Psi_1,\dots,W^\Psi_4=\FF_p$, $V^{\Phi} = \FF_p^2$, $W^\Phi_1,\dots,W^\Phi_3=\FF_p$ and
  \[
    \begin{aligned}
      \phi^\Psi_1(x,a,b) &= x   & \phi^{\Phi}_1(x,a) &= x \\
      \phi^\Psi_2(x,a,b) &= x+a & \phi^{\Phi}_2(x,a) &= x+a \\
      \phi^\Psi_3(x,a,b) &= x+b & \phi^{\Phi}_3(x,a) &= a. \\
      \phi^\Psi_4(x,a,b) &= x+a+b
    \end{aligned}
  \]
  The first argument above corresponds to a morphism $\Theta_1 \colon \Phi \to \Psi$ with $\alpha(1)=1$, $\alpha(2)=2$, $\alpha(3)=\alpha(4)=3$ and $\theta(x,a) = (x,a,a-x)$, with $\sigma_1=\sigma_2=\sigma_3 = \id$ and $\sigma_4(z)=2z$.
  It respects the indices $1,2 \in I^\Psi$.

  The second argument is encoded by a different morphism $\Theta_2 \colon \Phi \to \Psi$ with $\alpha(1)=1$, $\alpha(2)=2$, $\alpha(3)=\zeroi$ and $\alpha(4)=3$, and $\theta(x,a) = (x,a,-x)$, with $\sigma_1=\sigma_2=\sigma_4=\id$ and $\sigma_3=0$.
  In this case the morphism respects $1,2,4 \in I^\Psi$ so we get the added conclusion that $g_3$ (i.e., $g$) is a translate of $f_4$.
\end{example}

The longhand arguments above actually give the stronger fact (in both cases) that both $g_1=f_1$ and $g_2=f_2$, not merely that they are translates of each other.
With more effort we could keep track of such extra information, but
in practice this becomes difficult to work with.
Since replacing a function by a translate is typically benign, we simply accept this loss of precision.

\begin{remark}%
  \label{rem:pictures}
  It is often convenient to show morphisms of linear data pictorially.
  For example, the morphism $\Theta_1$ in Example~\ref{ex:morphism-application} could be shown as follows:
  
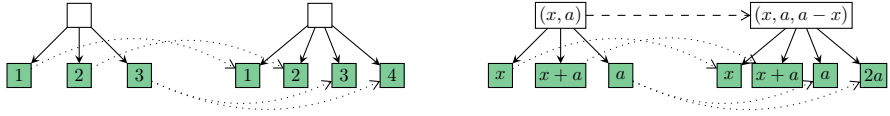
\begin{center}
    \begin{tikzpicture}[
        boxnode/.style={rectangle,draw,inner sep=2pt,outer sep=0pt,minimum size=13pt,scale=0.7},
        leaf/.style={fill=leafgreen},
        scale=0.8]
      \begin{scope}[shift={(0,0)}]
        \node[boxnode]       (V)  at ( 0  , 0) {};
        \node[leaf,boxnode] (W1) at (-1,-1) {$1$};
        \node[leaf,boxnode] (W2) at (0,-1) {$2$};
        \node[leaf,boxnode] (W3) at (1,-1) {$3$};
        \draw[-stealth] (V) -- (W1);
        \draw[-stealth] (V) -- (W2);
        \draw[-stealth] (V) -- (W3);

        \node[boxnode]       (Vp)  at (4 , 0) {};
        \node[leaf,boxnode] (Wp1) at ($(Vp) + (-1.2,-1)$) {$1$};
        \node[leaf,boxnode] (Wp2) at ($(Vp) + (-0.4,-1)$) {$2$};
        \node[leaf,boxnode] (Wp3) at ($(Vp) + ( 0.4,-1)$) {$3$};
        \node[leaf,boxnode] (Wp4) at ($(Vp) + ( 1.2,-1)$) {$4$};
        \draw[-stealth] (Vp) -- (Wp1);
        \draw[-stealth] (Vp) -- (Wp2);
        \draw[-stealth] (Vp) -- (Wp3);
        \draw[-stealth] (Vp) -- (Wp4);

        \draw[dotted,-angle 60] (W1) to[bend left = 30] (Wp1);
        \draw[dotted,-angle 60] (W2) to[bend left = 30] (Wp2);
        \draw[dotted,-angle 60] (W3) to[bend right = 25] (Wp3);
        \draw[dotted,-angle 60] (W3) to[bend right = 25] (Wp4);
      \end{scope}
      \begin{scope}[shift={(8,0)}]
        \node[boxnode]       (V)  at ( 0  , 0) {$(x,a)$};
        \node[leaf,boxnode] (W1) at (-1,-1) {$x$};
        \node[leaf,boxnode] (W2) at (0,-1) {$x+a$};
        \node[leaf,boxnode] (W3) at (1,-1) {$a$};
        \draw[-stealth] (V) -- (W1);
        \draw[-stealth] (V) -- (W2);
        \draw[-stealth] (V) -- (W3);

        \node[boxnode]       (Vp)  at (4 , 0) {$(x,a,a-x)$};
        \node[leaf,boxnode] (Wp1) at ($(Vp) + (-1.2,-1)$) {$x$};
        \node[leaf,boxnode] (Wp2) at ($(Vp) + (-0.4,-1)$) {$x+a$};
        \node[leaf,boxnode] (Wp3) at ($(Vp) + ( 0.4,-1)$) {$a$};
        \node[leaf,boxnode] (Wp4) at ($(Vp) + ( 1.2,-1)$) {$2a$};
        \draw[-stealth] (Vp) -- (Wp1);
        \draw[-stealth] (Vp) -- (Wp2);
        \draw[-stealth] (Vp) -- (Wp3);
        \draw[-stealth] (Vp) -- (Wp4);

        \draw[dotted,-angle 60] (W1) to[bend left = 30] (Wp1);
        \draw[dotted,-angle 60] (W2) to[bend left = 30] (Wp2);
        \draw[dotted,-angle 60] (W3) to[bend right = 25] (Wp3);
        \draw[dotted,-angle 60] (W3) to[bend right = 25] (Wp4);

        \draw[dashed,-angle 60] (V) -- (Vp);
      \end{scope}
    \end{tikzpicture}
  \end{center}
  and similarly $\Theta_2$ would be:
  
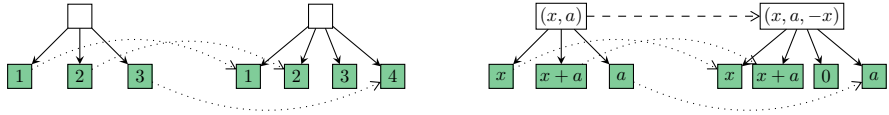
\begin{center}
    \begin{tikzpicture}[
        boxnode/.style={rectangle,draw,inner sep=2pt,outer sep=0pt,minimum size=13pt,scale=0.7},
        leaf/.style={fill=leafgreen},
        scale=0.8]
      \begin{scope}[shift={(0,0)}]
        \node[boxnode]       (V)  at ( 0  , 0) {};
        \node[leaf,boxnode] (W1) at (-1,-1) {$1$};
        \node[leaf,boxnode] (W2) at (0,-1) {$2$};
        \node[leaf,boxnode] (W3) at (1,-1) {$3$};
        \draw[-stealth] (V) -- (W1);
        \draw[-stealth] (V) -- (W2);
        \draw[-stealth] (V) -- (W3);

        \node[boxnode]       (Vp)  at (4 , 0) {};
        \node[leaf,boxnode] (Wp1) at ($(Vp) + (-1.2,-1)$) {$1$};
        \node[leaf,boxnode] (Wp2) at ($(Vp) + (-0.4,-1)$) {$2$};
        \node[leaf,boxnode] (Wp3) at ($(Vp) + ( 0.4,-1)$) {$3$};
        \node[leaf,boxnode] (Wp4) at ($(Vp) + ( 1.2,-1)$) {$4$};
        \draw[-stealth] (Vp) -- (Wp1);
        \draw[-stealth] (Vp) -- (Wp2);
        \draw[-stealth] (Vp) -- (Wp3);
        \draw[-stealth] (Vp) -- (Wp4);

        \draw[dotted,-angle 60] (W1) to[bend left = 30] (Wp1);
        \draw[dotted,-angle 60] (W2) to[bend left = 30] (Wp2);
        \draw[dotted,-angle 60] (W3) to[bend right = 25] (Wp4);
      \end{scope}
      \begin{scope}[shift={(8,0)}]
        \node[boxnode]       (V)  at ( 0  , 0) {$(x,a)$};
        \node[leaf,boxnode] (W1) at (-1,-1) {$x$};
        \node[leaf,boxnode] (W2) at (0,-1) {$x+a$};
        \node[leaf,boxnode] (W3) at (1,-1) {$a$};
        \draw[-stealth] (V) -- (W1);
        \draw[-stealth] (V) -- (W2);
        \draw[-stealth] (V) -- (W3);

        \node[boxnode]       (Vp)  at (4 , 0) {$(x,a,-x)$};
        \node[leaf,boxnode] (Wp1) at ($(Vp) + (-1.2,-1)$) {$x$};
        \node[leaf,boxnode] (Wp2) at ($(Vp) + (-0.4,-1)$) {$x+a$};
        \node[leaf,boxnode] (Wp3) at ($(Vp) + ( 0.4,-1)$) {$0$};
        \node[leaf,boxnode] (Wp4) at ($(Vp) + ( 1.2,-1)$) {$a$};
        \draw[-stealth] (Vp) -- (Wp1);
        \draw[-stealth] (Vp) -- (Wp2);
        \draw[-stealth] (Vp) -- (Wp3);
        \draw[-stealth] (Vp) -- (Wp4);

        \draw[dotted,-angle 60] (W1) to[bend left = 30] (Wp1);
        \draw[dotted,-angle 60] (W2) to[bend left = 30] (Wp2);
        \draw[dotted,-angle 60] (W3) to[bend right = 25] (Wp4);

        \draw[dashed,-angle 60] (V) -- (Vp);
      \end{scope}
    \end{tikzpicture}
  \end{center}
  The left-hand picture\footnote{We could combine both sides into a single diagram but this becomes harder to read.} identifies the indices of the datum, and the shape $\alpha$ of the morphism is implied by the dotted arrows (read backwards).
  An index $i$ with no incoming arrow means $\alpha(i) = \zeroi$.

  The right-hand picture shows the maps $\phi^\Phi_i$ and $\theta$ explicitly, and provided $\Phi$ is surjective we could uniquely infer the maps $\sigma_j$.
  Verifying~\eqref{eq:morph} corresponds to checking that the maps $\phi^{\Psi}_j$ and $\sigma_j$ do indeed send the inputs shown to the outputs shown.
\end{remark}

\begin{example}%
  \label{ex:weak-eq}
  Suppose as in Remark~\ref{rem:surj-degen} that $\Phi$ is a linear datum and
  \[
    K = \bigcap_{i \in I^\Phi} \ker \phi^\Phi_i \subseteq V^\Phi
  \]
  is non-zero.
  Let $\Psi$ be the quotient datum $I^\Psi=I^\Phi$, $V^\Psi=V/K$, $W_i^\Psi=W_i^\Phi$ and $\phi_i^\Psi(v+K) = \phi_i^\Phi(v)$ (which is well-defined).
  Then the quotient map $\theta \colon V \to V/K$, together with identity maps $\alpha$ and $\sigma_i$ for $i \in I^\Phi$, gives a morphism $\Phi \to \Psi$ which respects every $i \in I^\Psi$.

  If $\theta' \colon V/K \to V$ is any section (i.e., $\theta \circ \theta' = \id_{V/K}$) then we similarly get a morphism $\Psi \to \Phi$ respecting every index $i \in I^\Phi$.
\end{example}

\begin{example}%
  \label{ex:surj-fix}
  Similarly, for $\Phi$ a linear datum, consider the datum $\Psi$ with $I^\Psi=I^\Phi$, $V^\Psi=V^\Phi$, $W^\Psi_i=\im \phi^\Phi_i$ and $\phi^\Psi_i =\phi^{\Phi}_i$.

  Then taking $\id_V \colon V \to V$ together with $\sigma_i \colon \im \psi_i \to W_i$ the inclusion map, defines a morphism $\Psi \to \Phi$ which fixes the ``surjectivity problem'' as in Remark~\ref{rem:surj-degen}.
  It respects those $i \in I^\Phi$ for which $\phi^\Phi_i$ is already surjective.
\end{example}

We now consider applications of the Cauchy--Schwarz inequality, such as the proof of Fact~\ref{fact:functional-inequality}, in the language of linear data.
To describe such arguments in general we introduce the notion of ``joining'' two linear data together along one or more of the spaces $W_i$.
This concept will be useful in other settings so we define it in somewhat greater generality than we need right now.
\begin{definition}%
  \label{def:joining}
  Suppose $\Phi_0$, $\Phi_1$ are two linear data.
  Let $\beta \subseteq I^{\Phi_0} \times I^{\Phi_1}$ be a partial matching, specifically a bijection between $J_0 \subseteq I^{\Phi_0}$ to $J_1 \subseteq I^{\Phi_1}$,
  and assume that
  $W^{\Phi_0}_{i_0} = W^{\Phi_1}_{i_1}$ for all $(i_0,i_1) \in \beta$.

  The \emph{joining} of $\Phi_0$ and $\Phi_1$ along $\beta$, denoted $\Phi_0 +_\beta \Phi_1$, is a new datum $\Psi$ defined as follows.
  \begin{itemize}
    \item The index set $I^\Psi$ is the set
      \[
        J = \bigl\{ (L, i) \colon i \in I^{\Phi_0} \setminus J_0 \bigr\} \cup \bigl\{ (R, i) \colon i \in I^{\Phi_1} \setminus J_1 \bigr\}
      \]
      where $L$, $R$ denote any two arbitrary symbols.\footnote{%
        Using $0$ and $1$ would be more standard but can be hard to read if an index set $\{1,\dots,k\}$ also appears. 
        We will address labelling issues in more detail in Section~\ref{sub:labelling}.}
    \item For $j \in J$, the space $W^{\Psi}_j$ is $W^{\Phi_0}_i$ if $j=(L,i)$ for some $i \in I^{\Phi_0}$, or $W^{\Phi_1}_i$ if $j=(R,i)$ for some $i \in I^{\Phi_1}$.
    \item The vector space $V^\Psi \subseteq V^{\Phi_0} \oplus V^{\Phi_1}$ is the fiber product
      \begin{equation}%
        \label{eq:fiber-product}
        V^\Psi = \left\{ (v_0, v_1) \in V^{\Phi_0} \oplus V^{\Phi_1} \colon \phi^{\Phi_0}_{i_0}(v_0) = \phi^{\Phi_1}_{i_1}(v_1) \ \forall (i_0,i_1) \in \beta \right\}.
      \end{equation}
    Let $\pi_0 \colon V^\Psi \to V^{\Phi_0}$ and $\pi_1 \colon V^\Psi \to V^{\Phi_1}$ denote the two projection maps.
  \item The maps $\phi_j^{\Psi} \colon V^\Psi \to W^\Psi_j$ are given by $\phi_j^\Psi = \phi^{\Phi_0}_{i} \circ \pi_0$ if $j = (L,i)$ and $\phi_j^\Psi = \phi^{\Phi_1}_i \circ \pi_1$ if $j=(R,i)$.
  \end{itemize}

  If $I^{\Phi_0}=I^{\Phi_1}=I$ and $\beta = \{ (i,i) \colon i \in S \}$ for some $S \subseteq I$, we write $\Phi +_S \Phi$ in place of $\Phi +_\beta \Phi$.
\end{definition}

Applying Cauchy--Schwarz to eliminate the terms $\{ f_i \colon i \in S \}$ corresponds to the self-joining datum $\Phi +_S \Phi$, as the following proposition shows.
(This is closely related to~\cite[Proposition 3.6]{me}.)

\begin{proposition}%
  \label{prop:cs}
  Suppose $\Phi$ is a linear datum and $S \subseteq I^\Phi$ is a subset.  Let $\Psi = \Phi +_S \Phi$, as in Definition~\ref{def:joining}.

  For any $n \ge 1$ and functions $1$-bounded functions $(f_i)_{i \in I}$, $f_i \colon (W^\Phi_i)^n \to \CC$, define a tuple of functions $(g_j)_{j \in I^\Psi}$ by $g_{(L,i)} = f_i$ and $g_{(R,i)} = \overline{f_i}$ for all $i \in I^\Phi \setminus S$.  Then
  \[
    \bigl\lvert\Lambda_{\Phi}\bigl((f_i)_{i \in I^\Phi}\bigr)\bigr\rvert \le \bigl\lvert\Lambda_{\Psi}\bigl((g_j)_{j \in I^\Psi}\bigr) \bigr\rvert^{1/2}.
  \]
\end{proposition}
\begin{proof}%
  The proof is much more straightforward than the set-up.  Let
  \[
    \tau \colon V^\Phi \to \bigoplus_{i \in S} W^\Phi_i
  \]
  be the map $v \mapsto \bigl(\phi^\Phi_i(v)\bigr)_{i \in S}$, and let $U = \im \tau$.
  In particular any $u \in U$ is a tuple $(u_i)_{i \in S}$.
  Then
  \[
    \bigl\lvert\Lambda_{\Phi}((f_i)_{i \in I})\bigr\rvert = \biggl\lvert \EE_{u \in U} \Biggl( \prod_{i \in S} f_i(u_i) \Biggr) \Biggl( \EE_{v \in \tau^{-1}(u)} \prod_{i \notin S} f_i\bigl(\phi^\Phi_i(v)\bigr) \Biggr) \biggr\rvert.
  \]
  Applying Cauchy--Schwarz bounds this by
  \[
    \left( \EE_{u \in U} \prod_{i \in S} \bigl\lvert f_i(u_i) \bigr\rvert^2 \right)^{1/2} \left( \EE_{u \in U} \biggl\lvert\EE_{v \in \tau^{-1}(u)} \prod_{i \notin S} f_i\bigl(\phi^\Phi_i(v)\bigr) \biggr\rvert^2 \right)^{1/2}.
  \]
  The left-hand term is bounded by $1$, and expanding the square on the right yields
  \[
    \EE_{u \in U} \mathop{{}\EE}_{\substack{v_0, v_1 \in V^\Phi \\ \tau(v_0)=\tau(v_1)=u}} \prod_{i \notin S} f_i\bigl(\phi^\Phi_i(v_0)\bigr) \overline{f_i\bigl(\phi^\Phi_i(v_1)\bigr)}
  \]
  which on close inspection is exactly $\Lambda_{\Psi}\bigl((g_j)_{j \in I^\Psi}\bigr)$ as claimed.
\end{proof}

\begin{remark}%
  \label{rem:stashing}
  In Section~\ref{sub:true-summary} we considered ``stashed'' averages of the form $\Lambda_{\Phi}\bigl(f_1,\dots,f_5,f_6'\bigr)$ where $f_6' = D_{U^3}(g)$ is the $U^3$-dual function of some $g$.
  
  The multilinear average above may be expressed in terms of a single datum using the joining $\Phi +_{6 \leftrightarrow 000} U^3$, where $U^3$ is as in Definition~\ref{def:uk}: we have
  \[
    \Lambda_{\Phi}\bigl(f_1,\dots,f_5, D_{U^3}(g)\bigr) = 
    \Lambda_{\Phi +_{6 \leftrightarrow 000} U^3} \bigl(f_1,\dots,f_5, g, g, \overline{g}, \dots, g\bigr).
  \]
\end{remark}

At this stage we can say formally what a ``Cauchy--Schwarz argument'' looks like: it is simply a sequence of linear data $\Phi_0, \Phi_1, \dots, \Phi_M$ where each $\Phi_{i+1}$ is related to $\Phi_i$ either by a morphism $\Phi_{i+1} \to \Phi_i$ or by a Cauchy--Schwarz construction $\Phi_{i+1} = \Phi_{i} +_S \Phi_{i}$ for some $S \subseteq I^{\Phi_i}$.

Combining Proposition~\ref{prop:morphism} and Proposition~\ref{prop:cs} repeatedly, we get a statement of the following type: writing $I=I^{\Phi_0}$ and $J = I^{\Phi_M}$, for any tuple of $1$-bounded functions $(f_i)_{i \in I}$ there exists a tuple of $1$-bounded functions $(g_j)_{j \in J}$ such that
\[
  \bigl\lvert\Lambda_{\Phi_0}\bigl((f_i)_{i \in I}\bigr) \bigr\rvert \le \bigl\lvert\Lambda_{\Phi_M}\bigl((g_j)_{j \in J}\bigr) \bigr\rvert^{1/2^k}
\]
where $k$ counts the number of Cauchy--Schwarz steps.
Moreover, some (but not necessarily all) functions $g_j$ are known to be translates $x \mapsto f_i(x+t)$ or conjugate translates $x \mapsto \overline{f_i(x+t)}$ of some functions $f_i$.
Certainly at least one $g_j$ should have one of these form or the statement is vacuous.

By analogy with elementary proofs of functional equations, as discussed in Section~\ref{sub:background}, it is reasonable to think of such a sequence of steps as a ``Cauchy--Schwarz-friendly proof of $\Phi_M$ assuming $\Phi_0$''.
For this reason we borrow the logical notation\footnote{The notation $\Phi \entails \Psi$ is pronounced ``$\Phi$ entails $\Psi$''.}
$\Phi_0 \entails \Phi_M$ to denote the existence of such an argument.

Some extra decoration to the symbol $\entails$ is necessary to keep track of (i) the number of Cauchy--Schwarz steps, and (ii) how functions on the right relate to functions on the left.
For (ii), it is natural to introduce a partial function $\gamma \colon J \to I \times \{0,1\}$ where $\gamma(j) = (i, 0)$ means that $g_j$ is a translate of $f_i$ and $\gamma(j) = (i,1)$ means that $g_j$ is a translate of $\overline{f_i}$.
If $\gamma(j)$ is not defined, this means that, as far as we know, $g_j$ is an arbitrary $1$-bounded function.

We now set this up formally.
The details are involved but hopefully motivated by the above discussion.
\begin{definition}%
  \label{def:logic-notation}
  Suppose $\Phi$ and $\Psi$ are two linear data.
  Let $\gamma \colon I^{\Psi} \to I^{\Phi} \times \{0,1\}$ be a partial function and $k \ge 0$ an integer.
  We write
  $\Phi \entails^k_{\gamma} \Psi$
  if any of the following hold.
  \begin{enumerate}[label=(\alph*)]
    \item (One ``morphism'' step.) There is a morphism of linear data $\Theta \colon \Psi \to \Phi$, $k=0$, and $\gamma$ is the partial function
      \[
        \gamma = \bigl\{ \alpha^\Theta(i) \mapsto (i,0) \colon i \in I^\Phi,\ \Theta \text{ respects } i \bigr\}.
      \]
    \item (One Cauchy--Schwarz step.) We have $\Psi = \Phi +_S \Phi$ for some $S \subseteq I^\Phi$, and $\gamma((L,i)) = (i, 0)$ and $\gamma((R,i)) = (i,1)$ for all $i \in I \setminus S$.

    \item (Composing steps.) For some datum $\Xi$ it holds recursively that $\Phi \entails^{k_1}_{\gamma_1} \Xi$ and $\Xi \entails^{k_2}_{\gamma_2} \Psi$ where $k=k_1+k_2$ and $\gamma$ is the ``composite''
      \[
        \gamma = \bigl\{ j \mapsto \bigl(i, a_1+a_2 \bmod 2\bigr) \colon \gamma_2(j) = (\ell, a_2), \gamma_1(\ell) = (i, a_1) \bigr\}.
      \]
    \item (Discarding information.) It holds recursively that $\Phi \entails^{k'}_{\gamma'} \Psi$ where $k' \le k$ and $\gamma$ is a sub-function of the partial function $\gamma'$.
  \end{enumerate}
\end{definition}
We say that $\Phi \entails \Psi$ holds \emph{by }$\morph(\Theta)$ in case (a), and respectively \emph{by }$\CS(S)$ in case (b).
Hence whenever $\Phi \entails \Psi$ in general, there will be a sequence of steps
``$\morph(\Theta)$'' or ``$\CS(S)$'' encoding the ``proof'' of the statement $\Phi \entails \Psi$.

\begin{remark}%
  \label{rem:conj-tracking}
  Often we will not be interested in keeping track of the $\{0,1\}$ part of $\gamma(j)$.
  In this case, by abuse of notation we will simply provide a partial function $\gamma \colon I^{\Psi} \to I^\Phi$.
  For such $\gamma$, the statement $\Phi \entails^k_\gamma \Psi$ should be interpreted as saying that there exists a function $\gamma' \colon I^\Psi \to I^\Phi \times \{0,1\}$ with the same domain as $\gamma$, and with $\gamma'(j) = (\gamma(j), 0)$ or $(\gamma(j), 1)$ for all $j$, such that $\Phi \entails^k_{\gamma'} \Psi$.
\end{remark}

The following is then immediate from Proposition~\ref{prop:morphism} and Proposition~\ref{prop:cs}.
\begin{corollary}%
  \label{cor:log-entails}
  Let $\Phi$ and $\Psi$ be two linear data and write $I = I^\Phi$, $J=I^\Psi$.
  Suppose $\Phi \entails^k_\gamma \Psi$.
  Then for any $n \ge 1$ and any tuple $(f_i)_{i \in I}$ of $1$-bounded functions $f_i \colon (W_i^\Phi)^n \to \CC$ there exists a tuple of $1$-bounded functions $(g_j)_{j \in J}$, $g_j \colon (W^{\Psi}_j)^n \to \CC$, such that
  \[
    \bigl\lvert\Lambda_{\Phi}\bigl((f_i)_{i \in I}\bigr)\bigr\rvert \le \bigl\lvert\Lambda_{\Psi}\bigl((g_j)_{j \in J}\bigr)\bigr\rvert^{1/2^k}
  \]
  and such that whenever $j$ is in the domain of $\gamma$ and $\gamma(j) = (i, a)$, we have that $W^\Phi_i =W^\Psi_j$, and for some $t \in W_i^\Phi$:
  \[
    \forall x \in W^\Phi_i \colon g_j(x) = f_i(x+t) 
  \]
  if $a=0$, and if $a=1$:
  \[
    \forall x \in W^\Phi_i \colon g_j(x) = \overline{f_i(x+t)}.
  \]
\end{corollary}

\begin{remark}%
  \label{rem:computer}
  We emphasize that $\entails^k_\gamma$ is a purely \emph{syntactic} statement about linear data: its truth is independent of any mention of functions $(f_i)$ or $(g_j)$, and a sequence of $\morph$ and $\CS$ statements that purports to show that $\Phi \entails^k_\gamma \Psi$ can be checked formally just by doing computations in linear algebra.
  In particular, any given instance of such a proof can---and arguably should---be verified by a computer.
  It is (intentionally) distinct from simply asserting that the conclusion of Corollary~\ref{cor:log-entails} holds.
\end{remark}

\subsection{Cauchy--Schwarz diagrams}%
\label{sub:diagrams}

Logically, the notion of a linear datum is sufficient to describe all the Cauchy--Schwarz arguments we need.
In practice, we soon have to consider hugely unwieldy vector spaces $V$ obtained by a sequence of iterated joinings / fiber products, either from applications of Cauchy--Schwarz or to encode stashed structure (see Remark~\ref{rem:stashing}).
It is essentially impossible for a human to reason about these spaces without keeping additional data about their internal fiber product structure.

To address this we introduce the notion of a \emph{Cauchy--Schwarz diagram} as an extra layer of language sitting above the notion of a linear datum.
There will be a forgetful mapping (in fact, a functor) from diagrams to linear data, and a standard embedding of linear data as diagrams (also a functor), allowing us to import the theory from Section~\ref{sub:linear-datum}.

The basic idea is as follows.  Any time we are tempted to apply a joining of linear data---say, along $\beta = \{(i_0,i_1)\}$--- to produce a new linear datum, i.e., by taking the fiber product
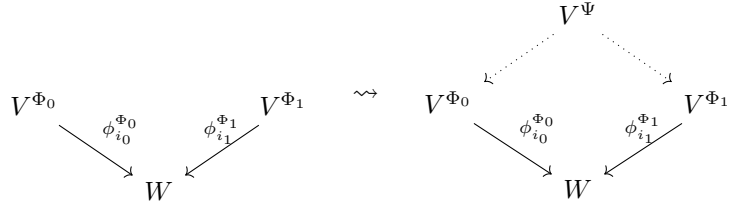
\[
  \begin{tikzcd}
    & \phantom{\wt V} & \\
    V^{\Phi_0} \arrow[dr, "\phi^{\Phi_0}_{i_0}"] & & V^{\Phi_1} \arrow[dl, "\phi^{\Phi_1}_{i_1}" above] \\ %chktex 18
                                     & W &
  \end{tikzcd}
  \hspace{10pt}\leadsto\hspace{10pt}
  \begin{tikzcd}
    & V^\Psi \arrow[dl,dotted] \arrow[dr,dotted] & \\
    V^{\Phi_0} \arrow[dr, "\phi_{i_0}^{\Phi_0}"] & & V^{\Phi_1} \arrow[dl, "\phi_{i_1}^{\Phi_1}" above] \\ %chktex 18
                                     & W &
  \end{tikzcd}
\]
as in~\eqref{eq:fiber-product}, we instead \emph{don't} take the fiber product and just ``remember'' the diagram on the left.
Iterating this process will result in a complicated digraph.
There will also be a mechanism for simplifying these digraphs using a ``morphism of diagrams'', analogously to the situation for morphisms of linear data.

A Cauchy--Schwarz diagram is not the same thing as an ``arithmetic circuit'', as discussed in Section~\ref{sub:repr}, but is designed as a step in that direction.

The related but not identical notion of a ``graph of vector spaces'' was introduced in~\cite{me}.
However it was used only as an informal tool to illustrate fairly small cases.
For our current purposes it is necessary to develop the concept systematically.

\begin{definition}%
  \label{def:diagram}
  Fix a prime $p$.  A \emph{diagram}\footnote{
  The term ``diagram'' is also chosen because it is consistent with the meaning of ``diagram'' in category theory.  
Formally, we can think of $\cG$ as (essentially) a finite category, and $V_x$ and $\phi_{xy}$ as a functor from $\cG$ to vector spaces over $\FF_p$.  

We will refer to the categorical perspective in subsequent footnotes, on the basis that the reader can take it exactly as seriously as they wish.}
  or \emph{Cauchy--Schwarz diagram} over $\FF_p$ consists of the following data:---
  \begin{itemize}
    \item a (finite) directed acyclic graph $\cG = (X, E)$ which is 
      \emph{singly connected}: i.e., for any pair $x,y \in X$ there is at most one directed path $x \to y$ in $\cG$;\footnote{%
        This condition is not actually needed anywhere, but has the effect of limiting the definition to cases where it is morally correct.
        Specifically, given two paths $x \to z \to y$ and $x \to z' \to y$ we should worry about whether the associated diagram should commute; i.e., whether or not we demand $\phi_{zy} \circ \phi_{xz} = \phi_{z'y} \circ \phi_{xz'}$.
        Either option could be natural according to context.
        Correctly we should replace $\cG$ with a finite category, which includes such information, but we wish to avoid such complications.
        However, it is possible that the extra flexibility afforded by these more general categorical diagrams could be useful in future.}
    \item a subset $\leafs \subseteq X$ of vertices, all having in-degree $1$ and out-degree $0$, termed the \emph{leaves} of the diagram;
    \item for each $x \in X$, a (finite-dimensional) vector space $V_x$ over $\FF_p$;
    \item for each directed edge $xy \in E$, a linear map $\phi_{xy} \colon V_x \to V_y$.
  \end{itemize}
  We write $\nonleafs = X \setminus \leafs$ for the non-leaf vertices.
\end{definition}

Once again, as in Convention~\ref{conv:datum-tuple} we write $\cG^D$, $X^D$, $\leafs^D$, $\nonleafs^D$, $E^D$, $V^D_x$ and $\phi^D_{xy}$ to refer to the constituent parts of a diagram $D$.

The intuition behind ``leaves'' is that they will correspond to the spaces $W^\Phi_i$ of a corresponding linear datum, where functions $f$ are actually evaluated.
By contrast, non-leaves will contribute only to the definition and structure of the space $V^\Phi$ (and the maps $\phi^\Phi_i$).

The following definition explains in full this structure-discarding mapping from diagrams to linear data.
\begin{definition}%
  \label{def:diagram-datum}
  Given a Cauchy--Schwarz diagram $D$, we define its associated linear datum $\Phi=\datum(D)$ as follows.
  \begin{itemize}
    \item The vector space $V^\Phi$ is\footnote{
      Equivalently, this is exactly the definition of the ``limit of a diagram'' in category theory, applied to this case.
    }%
      \begin{equation}%
        \label{eq:limit-space}
        V^{\Phi} = \left\{ (v_x)_{x \in X^D} \in \bigoplus_{x \in X^D} V^D_x : \forall xy \in E^D \colon \phi^D_{xy}(v_x) = v_y  \right\}.
      \end{equation}
    \item The index set $I^\Phi$ is $\leafs^D$, the set of leaves, and $W^\Phi_x = V^D_x$ for $x \in \leafs^D$.
    \item The maps $\phi^\Phi_x \colon V^\Phi \to W^\Phi_x$ for $x \in \leafs^D$ are exactly the projections 
      \begin{equation}%
        \label{eq:projections}
        \pi_x \colon \bigoplus_{y \in X^D} V^D_y \to V^D_x
      \end{equation}
      restricted to the subspace $V^\Phi$.\footnote{Alternatively, these are the maps $V^\Phi \to V^D_x$ from the universal property of the categorical limit.}
  \end{itemize}
  As an extra piece of terminology, we call an tuple $(v_x)_{x \in X}$, $v_x \in V^D_x$ such that $\phi_{xy}(v_x) = v_y$ for all $xy \in E^D$, as in~\eqref{eq:limit-space}, a \emph{compatible tuple}.
\end{definition}

Conversely, any linear datum defines a Cauchy--Schwarz diagram of a particular simple type.
\begin{definition}%
  \label{def:trivial-diagram}
  Given a linear datum $\Phi$, we define its associated diagram
  $D=\diagram(\Phi)$ as follows:---
  \begin{itemize}
    \item we take $X^D = I^{\Phi} \cup \{\diamond\}$, where $\diamond \notin I^{\Phi}$ is an extra symbol, $\leafs^D = I^\Phi$, and $E^D = \{ \diamond i \colon i \in I^{\Phi} \}$;
    \item we set $V^D_{\diamond} = V^{\Phi}$ and for $i \in I^{\Phi}$, $V^D_i = W^{\Phi}_i$ and $\phi^D_{\diamond i} = \phi^{\Phi}_i$.  
  \end{itemize}
\end{definition}
It is clear a diagram corresponds to a linear datum in this way precisely when it has exactly one non-leaf.
Pictorially, if $I^{\Phi}=\{i_1,\dots,i_k\}$ such a diagram looks like this:
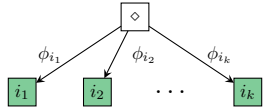
\begin{center}
  \begin{tikzpicture}[
      boxnode/.style={rectangle,draw,inner sep=2pt,outer sep=0pt,minimum size=15pt,scale=0.7},
      leaf/.style={fill=leafgreen}]
    \node[boxnode]       (V)  at ( 0  , 0) {$\diamond$};
    \node[leaf,boxnode] (W1) at (-1.5,-1) {$i_1$};
    \node[leaf,boxnode] (W2) at (-0.5,-1) {$i_2$};
    \node (dd) at ( 0.5,-1) {$\dots$};
    \node[leaf,boxnode] (Wk) at ( 1.5,-1) {$i_k$};
    \draw[-stealth] (V) -- (W1) node[midway, "{$\phi_{i_1}$}" {left,scale=0.7}] {};
    \draw[-stealth] (V) -- (W2) node[midway, "{$\phi_{i_2}$}" {right,scale=0.7}] {};
    \draw[-stealth] (V) -- (Wk) node[midway, "{$\phi_{i_k}$}" {right,scale=0.7}] {};
  \end{tikzpicture}
\end{center}
Here, as in future, we have adopted the convention that leaves are green.

\begin{remark}%
  \label{rem:datum-diag}
  For any linear datum $\Phi$, we have
  $\datum(\diagram(\Phi)) = \Phi$,
  up to canonical isomorphisms.
  Indeed, the space
  \[
    V^{\datum(\diagram(\Phi))} = \biggl\{ (v_x)_{x \in I \cup \{\diamond\}} \colon \forall i \in I : v_i = \phi^\Phi_i(v_\diamond) \biggr\} \subseteq V^\Phi \oplus \bigoplus_{i \in I} W^\Phi_i
  \]
  from~\eqref{eq:limit-space} is naturally isomorphic to $V^\Phi$.
\end{remark}

\begin{convention}%
  \label{conv:no-diagram-word}
  If $\Phi$ is a linear datum, in practice we often abuse notation and simply write $\Phi$ to denote the corresponding diagram $\diagram(\Phi)$.
  This is especially true of the named linear data in Section~\ref{sub:linear-datum}, such as $\trivial(I)$, $\const(I)$ etc..
\end{convention}

We give one example (for now) of a diagram that is not simply $\diagram(\Phi)$ for some linear datum $\Phi$.
\begin{example}%
  \label{ex:my-first-diagram}
  Consider a diagram $D$ with vertex set $X = \{ A, B, 1A, 1B, 2A, 2B, 3 \}$, of which $\leafs = \{1A,1B,2A,2B\}$.
  The digraph $\cG$ consists of edges $A\!\rightarrow\!1A$, $A\!\rightarrow\!2A$, $A\!\rightarrow\! 3$, $B\!\rightarrow\! 1B$, $B\!\rightarrow\! 2B$, $B\!\rightarrow\! 3$ as shown below.
  
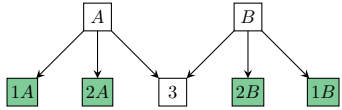
\begin{center}
    \begin{tikzpicture}[
        every node/.style={rectangle,draw,inner sep=2pt,outer sep=0pt,minimum size=15pt,scale=0.7},
        leaf/.style={fill=leafgreen}]
      \node       (0V)  at ( 0, 0) {$A$};
      \node[leaf] (0W1) at (-1,-1) {$1A$};
      \node[leaf] (0W2) at ( 0,-1) {$2A$};
      \node       (-W3) at ( 1,-1) {$3$};
      \node       (1V)  at ( 2, 0) {$B$};
      \node[leaf] (1W2) at ( 2,-1) {$2B$};
      \node[leaf] (1W1) at ( 3,-1) {$1B$};
      \draw[-stealth] (0V) -- (0W1);
      \draw[-stealth] (0V) -- (0W2);
      \draw[-stealth] (0V) -- (-W3);
      \draw[-stealth] (1V) -- (1W1);
      \draw[-stealth] (1V) -- (1W2);
      \draw[-stealth] (1V) -- (-W3);
    \end{tikzpicture}
  \end{center}
  We take $V_A = V_B = \FF_p^2$ and $V_x = \FF_p$ for all other $x$.
  The morphisms are given by $\phi_{A \rightarrow 1A}(x,h) = x$, $\phi_{A \rightarrow 2A}(x,h) = x+h$ and $\phi_{A\rightarrow 3}(x,h) = x+2h$, together with the same linear maps with $B$ replacing $A$ throughout.

  This diagram represents the state of affairs in the proof of Fact~\ref{fact:functional-inequality} after the first Cauchy--Schwarz step.
  Accordingly, it is possible to check that $\datum(D)$ is, up to relabelling, the same as $\Phi +_{\{3\}} \Phi$ where $\Phi$ is the linear datum of three-term arithmetic progressions.
  We will prove a general statement of this form in Proposition~\ref{prop:diagram-joining} below.
\end{example}

For much of the rest of this section, we adapt the material from Section~\ref{sub:linear-datum} to the augmented world of diagrams.
In so doing we prove the link between abstract diagram manipulations and honest $\Lambda$-inequalities.

We begin by remarking on the notion of surjectivity for diagrams; see Remark~\ref{rem:surj-degen}.
\begin{definition}%
  \label{def:diag-surj}
  If $D$ is a diagram and $x \in X^D$ is a vertex, we say $D$ is \emph{surjective at $x$} if the natural projection map $V^{\datum(D)} \to V_x$ (see~\eqref{eq:limit-space}) is surjective.
  When $x \in \leafs^D$, this is equivalent to saying $\datum(D)$ is surjective at $x$.
\end{definition}
It is typically not obvious that a complicated diagram is surjective at any vertex, but most of the diagrams we care about are in fact surjective at every vertex. 
In Section~\ref{sub:surj} we develop some tools for checking such statements.

As with linear data, there is a natural and useful notion of morphisms between diagrams.
As in that case, it is convenient for the definition to imagine that there is an extra space $V_\zeroi = \{0\}$ attached to every diagram.
The following convention is analogous to Convention~\ref{convention:ast}.

\begin{convention}%
  \label{convention:ast-revenge}
  For any Cauchy--Schwarz diagram $D$, we assume $\zeroi$ is an extra symbol with $\zeroi \notin X^D$,
  and write $V^D_{\zeroi} = \{0\}$, the zero space.
  
  By convention $x\zeroi$ and $\zeroi y$ for $x,y \in X^D$ are considered elements of $E^D$.
  The maps $\phi^D_{\zeroi y}$ and $\phi^D_{x \zeroi}$ are the zero map.

  Also, for any $x \in X^D$ the self-loop $xx$ is considered to be an edge\footnote{When we make statements such as ``$x$ has out-degree $0$'' we of course do not include this self-loop or the edge $x \zeroi$.} in $E^D$, and by convention $\phi^D_{xx} \colon V^D_x \to V^D_x$ denotes the identity map.\footnote{This is needed to make certain definitions work correctly, and is natural if we take the view that our directed acyclic graph $\cG$ should really be thought of as a finite category.}
\end{convention}

\begin{definition}%
  \label{def:diagram-morphism}
  Suppose $C,D$ are two diagrams, and $\alpha \colon X^D \to X^{C} \cup \{\zeroi\}$ is a function.

  Suppose also that (i) $\alpha\bigl(\leafs^{D}\bigr) \subseteq \leafs^{C} \cup \{\zeroi\}$ and $\alpha \bigl( \nonleafs^D \bigr) \subseteq \nonleafs^C$, and (ii) for all $xy \in E^D$ we have  $\alpha(x)\alpha(y) \in E^{C}$.
  By Convention~\ref{convention:ast-revenge}, (ii) is automatically satisfied if $\alpha(x)=\zeroi$, $\alpha(y)=\zeroi$ or $\alpha(x)=\alpha(y)$.

  A \emph{morphism of diagrams} $\Theta \colon C \to D$ of shape $\alpha$ consists of a collection of linear maps $\theta = (\theta_x)_{x \in X^D}$ where $\theta_x \colon V^{C}_{\alpha(x)} \to V^D_x$ such that for every $xy \in E^D$, the diagram
  \begin{equation}%
    \label{eq:diag-morph}
    \begin{tikzcd}
      V^{C}_{\alpha(x)} \arrow[d, "\phi^{C}_{\alpha(x)\alpha(y)}"] \arrow[r,"\theta_{x}"] & %chktex 18
      V^D_{x} \arrow[d, "\phi^D_{xy}"] \\ %chktex 18
      V^{C}_{\alpha(y)} \arrow[r,"\theta_y"] & %chktex 18
      V^D_{y}
    \end{tikzcd}
  \end{equation}
  commutes, where Convention~\ref{convention:ast-revenge} applies if $\alpha(x)=\zeroi$, $\alpha(y)=\zeroi$ or $\alpha(x)=\alpha(y)$.

  We say the morphism \emph{respects} a leaf $x \in \leafs^D$ if (a) $\alpha(x) \ne \zeroi$, (b) for all other $y \in \leafs^D$ we have $\alpha(x) \ne \alpha(y)$, and (c) $V^D_x = V^{C}_{\alpha(x)}$ and $\theta_x$ is the identity map.
\end{definition}

Again, as in Convention~\ref{conv:datum-tuple}, we write $\theta^\Theta$ and $\alpha^\Theta$ to refer to the constituent parts of a morphism $\Theta$.

As one might hope, a morphism of diagrams induces a morphism of the corresponding linear data (and hence an inequality of multilinear averages $\Lambda_\Phi$).

\begin{proposition}%
  \label{prop:diagram-morphism}
  Suppose $\Theta \colon C \to D$ is a morphism of diagrams of shape $\alpha$.
  Then there is a corresponding morphism $\datum(\Theta) \colon \datum(C) \to \datum(D)$ of linear data of shape $\alpha|_{\leafs^D}$ with $\sigma_i = \phi_i$ for all $i \in \leafs$.

  Moreover, if $\Theta$ respects some $i \in \leafs^D$ then so does $\datum(\Theta)$.
\end{proposition}
\begin{proof}%
  We write $\Phi = \datum(C)$ and $\Psi=\datum(D)$.
  In particular $V^\Phi$ and $V^\Psi$ are the vector spaces from~\eqref{eq:limit-space} applied to $C$ and $D$ respectively.

  With reference to~\eqref{eq:limit-space}, we may define
  $\wt \theta \colon V^\Phi \to V^\Psi$ by
  \begin{equation}%
    \label{eq:theta-tilde}
    \wt \theta \bigl( (v_z)_{z \in X^C} \bigr) = \bigl(\theta_{x}(v_{\alpha(x)})\bigr)_{x \in X^D}.
  \end{equation}
  We claim this is well-defined: that is, if the tuple $(v_z)_{z \in X^C}$ lies in $V^\Phi$ (i.e., is a compatible tuple in $C$) then the right-hand side lies in $V^\Psi$ (i.e., is a compatible tuple in $D$).
  In other words, we want to show that for all $xy \in E^D$,
  \[
    \phi_{xy}^D \bigl(\theta_{x}(v_{\alpha(x)})\bigr) = \theta_{y}(v_{\alpha(y)}).
  \]
  Indeed, we have
  \[
    \phi_{xy}^D \bigl(\theta_{x}(v_{\alpha(x)})\bigr) =\theta_y\bigl(\phi^C_{\alpha(x)\alpha(y)}(v_{\alpha(x)})\bigr)
  \]
  by~\eqref{eq:diag-morph}, and
  \[
    \phi^C_{\alpha(x)\alpha(y)}(v_{\alpha(x)}) = v_{\alpha(y)}
  \]
  by the hypothesis that $(v_z)_{z \in X^C}$ is a compatible tuple.
  
  For $i \in \leafs^D$, define $\sigma_i \colon W^{\Phi}_{\alpha(i)} \to W^\Psi_i$ by $\sigma_i = \theta_i$.
  Now~\eqref{eq:morph}, i.e.\ the equation
  \[
    \sigma_i \circ \phi^{\Phi}_{\alpha(i)} = \phi^\Psi_i \circ \wt \theta
  \]
  is immediate from the definitions of $\wt \theta$, $\phi^\Phi_{\alpha(i)}$ and $\phi^\Psi_i$.

  The statement about respecting an index $i \in \leafs^D$ is also clear from the definitions.
\end{proof}

\begin{remark}%
  \label{rem:other-functor}
  Proposition~\ref{prop:diagram-morphism} states essentially that $\datum(-)$ is a functor.
  We note that $\diagram(-)$ is also a functor: given a morphism $\Theta \colon \Phi \to \Psi$ of linear data, a corresponding morphism $\diagram(\Phi) \to \diagram(\Psi)$ is easily constructed by setting $\alpha(\diamond)=\diamond$, $\theta_{\diamond} = \theta^\Theta$ and $\theta_{j} = \sigma_j$ for $j \in I^\Psi$.
  The notion of ``respecting'' a leaf or index is preserved by this construction.

  The construction also works in reverse: given a morphism of diagrams $\diagram(\Phi) \to \diagram(\Psi)$ we necessarily have $\alpha(\diamond) = \diamond$ and recover a morphism of linear data $\Phi \to \Psi$ by taking $\theta = \theta_\diamond$ and $\sigma_j = \theta_j$ for $j \in I^\Psi$.

  We therefore have a natural bijection
  \begin{equation}%
    \label{eq:fully-faithful}
    \bigl\{ \text{diagram morphisms } \diagram(\Phi) \to \diagram(\Psi) \bigr\} \leftrightarrow \bigl\{ \text{datum morphisms } \Phi \to \Psi \bigr\}
  \end{equation}
  that preserves the notion of ``respecting'' leaves.

  Conveniently this means that  we do not have to worry about whether a morphism $\Phi \to \Psi$ is supposed to mean a morphism of linear data, or a diagram morphism $\diagram(\Phi) \to \diagram(\Psi)$ under the abuse of notation in Convention~\ref{conv:no-diagram-word}.
\end{remark}

We now state crucially how this notion of a diagram interacts with ``joinings'' of linear data and Cauchy--Schwarz.
\begin{definition}%
  \label{def:diagram-cs}
  Suppose $D_0, D_1$ are two diagrams and $\beta \subseteq \leafs^{D_0} \times \leafs^{D_1}$ is a partial matching between leaves of $D_0$ and leaves of $D_1$, specifically a bijection between $J_0 \subseteq \leafs^{D_0}$ and $J_1 \subseteq \leafs^{D_1}$.
  Assume that $V^{D_0}_{x_0} = V^{D_1}_{x_1}$ for all $(x_0, x_1) \in \beta$.

  The \emph{joining} of $D_0$ and $D_1$ along $\beta$, denoted $D_0 +_{\beta} D_1$, is a new diagram $D$ defined as follows.
  \begin{itemize}
    \item The digraph $\cG^D$ consists of a copy of each of $\cG^{D_0}$ and $\cG^{D_1}$ with pairs of vertices $(x_0, x_1) \in \beta$ identified.
      Formally, for any two arbitrary symbols $L$, $R$, we set
      \[
        X^D = \left( \bigl\{ (L,x) \colon x \in ^{D_0} \bigr\} \cup \bigl\{ (R, x) \colon x \in X^{D_1} \bigr\} \right) / \sim
      \]
      where $\sim$ is the equivalence relation given by $(L, x_0) \sim (R, x_1)$ for $(x_0, x_1) \in \beta$ together with $y \sim y$ for all $y \in X^D$, and\footnote{
        Correctly we should say ``the $\sim$-equivalence class of $(L,x)$'' in place of ``$(L,x)$'', etc, in the definition of $E^D$.
      We note though that all the edges appearing in the definition are distinct under $\sim$: for $(L,x_0)(L,y_0)$ and $(R,x_1)(R,y_1)$ to coincide we would need $(L,x_0) \sim (R,x_1)$ and so $x_0$ and $x_1$ would be leaves, but leaves have out-degree $0$.
      }
      \[
        E^D = \bigl\{ (L, x) (L, y) \colon xy \in E^{D_0} \bigr\} \cup \bigl\{ (R,x) (R,y) \colon xy \in E^{D_1} \bigr\}.
      \]
    \item The leaves $\leafs^D$ are given by the disjoint union
      \[
        \leafs^D = \bigl\{ (L, x_0) \colon x_0 \in \leafs^{D_0} \setminus J_0 \bigr\} \cup \bigl\{ (R, x_1) \colon x_1 \in \leafs^{D_1} \setminus J_1 \bigr\}.
      \]
      In particular, ``matched'' vertices $(L,x_0) \sim (R, x_1)$ are \emph{not} leaves of $D$.
    \item The vector spaces are $V^D_{(L,x)}=V^{D_0}_x$ for $x \in X^{D_0}$ and $V^D_{(R,x)} = V^{D_1}_x$ for $x \in X^{D_1}$.
      Note this is well-defined under $\sim$ by our hypotheses.
    \item The linear maps are $\phi^D_{(L,x)(L,y)} = \phi^{D_0}_{xy}$ for $xy \in E^{D_0}$ and $\phi^D_{(R,x)(R,y)} = \phi^{D_1}(xy)$ for $xy \in E^{D_1}$.
  \end{itemize}
  As with linear data, if $\leafs = \leafs^{(0)} = \leafs^{(1)}$ and $\beta = \{ (i,i) \colon i \in S \}$ is diagonal for some $S \subseteq \leafs$, we write $D_0 +_S D_1$ to mean the same as $D_0 +_\beta D_1$.
\end{definition}

By a slight abuse of notation,
we will refer to elements of $X^D$ as simply $(L,x_0)$ or $(R,x_1)$ and write $(L,x_0)=(R,x_1)$ if they are equivalent; i.e., the equivalence relation $\sim$ is always left implicit.
If $(x,x) \in \beta$ and so $(L,x)=(R,x)$, to preserve symmetry we occasionally write $(\#, x)=(L,x)=(R,x)$ to denote this element.

Again, as we might hope, joining two diagrams $D_0$, $D_1$ is compatible with joining their associated linear data $\datum(D_0)$, $\datum(D_1)$.
\begin{proposition}%
  \label{prop:diagram-joining}
  If $D_0$, $D_1$ are two diagrams and $\beta \subseteq \leafs^{D_0} \times \leafs^{D_1}$ is a partial matching as in Definition~\ref{def:diagram-cs}, then
  \[
    \datum(D_0) +_{\beta} \datum(D_1) = \datum \bigl( D_0 +_{\beta} D_1 \bigr)
  \]
  are the same linear datum, up to canonical isomorphisms.
\end{proposition}
The proof is essentially a case of unpacking the definitions, but for completeness we give the details.
\begin{proof}%
  Write $D = D_0 +_\beta D_1$ and $\Phi = \datum(D_0) +_{\beta} \datum(D_1)$.
  Hence we wish to show $\datum(D) = \Phi$.
  It is clear by definition that $I^\Phi = \leafs^D$.

  Recall that
  \begin{align*}
    V^{\datum(D_0)} &\subseteq \bigoplus_{x_0 \in X^{D_0}} V_{x_0}^{D_0} \\
    V^{\datum(D_1)} &\subseteq \bigoplus_{x_1 \in X^{D_1}} V_{x_1}^{D_1}  \\
    V^{\datum(D)} &\subseteq \bigoplus_{z \in X^{D}} V_{z}^{D} 
  \end{align*}
  are the subspaces of compatible tuples, as in~\eqref{eq:limit-space}.
  Also recall that $V^\Phi$ as defined in~\eqref{eq:fiber-product} is the subspace
  $
  V^\Phi \subseteq V^{\datum(D_0)} \oplus V^{\datum(D_1)}
  $
  consisting of pairs $(u,v)$ such that $u_{i_0} = v_{i_1}$ for all $(i_0,i_1) \in \beta$ (unwrapping the definitions of $\phi_{i_0}^{\datum(D_0)}$ and $\phi_{i_1}^{\datum(D_1)}$).

  The embeddings $X^{D_0} \to X^D$ and $X^{D_1} \to X^D$ induce natural maps
  \begin{align*}
    \tau_0 \colon \bigoplus_{z \in X^D} V^D_z &\to \bigoplus_{x_0 \in X^{D_0}} V_{x_0}^{D_0} \\
    \tau_1 \colon \bigoplus_{z \in X^D} V^D_z &\to \bigoplus_{x_1 \in X^{D_1}} V_{x_1}^{D_1}
  \end{align*}
  sending $(v_z)_{z \in X^D}$ to $(v_{(L,x_0)})_{x_0 \in X^{D_0}}$ and $(v_{(R,x_1)})_{x_1 \in X^{D_1}}$ respectively.
  Combining these gives a map
  \[
    \tau \colon \bigoplus_{z \in X^D} V^D_z \to \left( \bigoplus_{x_0 \in X^{D_0}} V_{x_0}^{D_0} \right) \oplus \left( \bigoplus_{x_1 \in X^{D_1}} V_{x_1}^{D_1} \right).
  \]
  We make the following claims:---
  \begin{enumerate}[label=(\alph*)]
    \item $\tau_0\bigl(V^{\datum(D)}\bigr) \subseteq V^{\datum(D_0)}$ and 
      $\tau_1\bigl(V^{\datum(D)}\bigr) \subseteq V^{\datum(D_1)}$;
    \item $\tau\bigl(V^{\datum(D)}\bigr) \subseteq V^\Phi$;
    \item $\tau$ restricted to $V^{\datum(D)}$ is an isomorphism onto $V^\Phi$.
  \end{enumerate}
  Indeed, (a) says that if $(v_z)_{z \in X^D}$ is a compatible tuple and $x_0 y_0 \in E^{D_0}$ then $\phi^{D_0}_{x_0y_0}(v_{(L,x_0)}) = v_{(L,y_0)}$ (and similarly for $D_1$).
  But this is just the definition of compatibility applied to the edge $(L,x_0)(L, y_0) \in E^{D}$.

  In (b), for $(u,v)=\tau(w)$ and $(i_0,i_1) \in \beta$ we have $(L,i_0) = (R,i_1)$ and so
  \[
    u_{i_0} = w_{(L,i_0)} = w_{(R,i_1)} = v_{i_1}.
  \]

  For (c), we note in reverse that any pair of tuples $(u,v)$ obeying $u_{i_0} = v_{i_1}$ for all $(i_0,i_1) \in \beta$ has a unique preimage $(w_z)_{z \in X^D}$ under $\tau$, by setting $w_{(L,x_0)} = u_{x_0}$ and $w_{(R,x_1)} = v_{x_1}$.
  By reversing the argument in (a), this element lies in $V^{\datum(D)}$.

  Hence $V^{\Phi}$ and $V^{\datum(D)}$ are canonically isomorphic under $\tau$.
  To complete the proof we consider the spaces $W_i$ and maps $\phi_i$.

  For $i \in I^\Phi = \leafs^D$, if $i=(L,i_0)$ from the definitions we have
  \[
    W_i^{\Phi} = V_{i_0}^{D_0} = V_i^D = W_i^{\datum(D)}
  \]
  and similarly if $i=(R,i_1)$.
  Finally, if $i=(L,i_0)$ and $w \in V^{\datum(D)}$, writing $\tau(w)=(u,v)$ the definitions give
  \[
    \phi_i^{\datum(D)}(w) = w_i = u_{i_0} = \phi_{i_0}^{\datum(D_0)}(u) = \phi_i^{\Phi}(u,v) = \phi_i^{\Phi}(\tau(w))
  \]
  and similarly if $i=(R,i_1)$, so the maps $\phi_i$ coincide under $\tau$.
\end{proof}

Hence, applying Cauchy--Schwarz to a diagram, as in Proposition~\ref{prop:cs}, corresponds to replacing $D$ with $D +_S D$ for some set of leaves $S \subseteq \leafs^D$.

We can now define an ``$\entails$'' notation for diagrams analogous to Definition~\ref{def:logic-notation}.
\begin{definition}%
  \label{def:logic-notation-revenge}
  Suppose $C$ and $D$ are two diagrams.
  Let $\gamma \colon \leafs^{D} \to \leafs^C \times \{0,1\}$ be a partial function and $k \ge 0$ an integer.
  We write $C \entails^k_\gamma D$ if any of conditions (a')--(d') hold, where (a')--(d') are obtained from the exactly analogous conditions (a)--(d) in Definition~\ref{def:logic-notation} by replacing $\Phi$ with $C$, $\Psi$ with $D$ and $I^\Phi$, $I^\Psi$ with $\leafs^{C}$, $\leafs^D$ wherever they appear.

  We also use the notation $\morph(\Theta)$ to denote a step of type (a') and $\CS(S)$ to denote a step of type (b').
\end{definition}

The following is then immediate from Proposition~\ref{prop:diagram-morphism} and Proposition~\ref{prop:diagram-joining}.
\begin{corollary}%
  \label{cor:entails-entails}
  If $C$ and $D$ are diagrams and $C \entails^k_\gamma D$ then $\datum(C) \entails^k_\gamma \datum(D)$.
\end{corollary}
Clearly this may be combined with Corollary~\ref{cor:log-entails} to obtain a statement involving inequalities.

By way of example, we explain what the proof of Fact~\ref{fact:functional-inequality} looks like in the language of Cauchy--Schwarz diagrams.
\begin{example}%
  \label{ex:diag-ex1}
  We start with the linear datum $\Phi$ with $V^\Phi=\FF_p^2$, $I^\Phi=[3]$, $W^\Phi_i = \FF_p$, $\phi^\Phi_1(x,h)=x$, $\phi^\Phi_2(x,h)=x+h$ and $\phi^\Phi_3(x,h)=x+2h$.
  The diagram $D_0=\diagram(\Phi)$ looks like this:
  
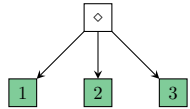
\begin{center}
    \begin{tikzpicture}[
        every node/.style={rectangle,draw,inner sep=2pt,outer sep=0pt,minimum size=15pt,scale=0.7},
        leaf/.style={fill=leafgreen}]
      \node       (V)  at ( 0, 0) {$\diamond$};
      \node[leaf] (W1) at (-1,-1) {$1$};
      \node[leaf] (W2) at ( 0,-1) {$2$};
      \node[leaf] (W3) at ( 1,-1) {$3$};
      \draw[-stealth] (V) -- (W1);
      \draw[-stealth] (V) -- (W2);
      \draw[-stealth] (V) -- (W3);
    \end{tikzpicture}
  \end{center}
  We also have the linear datum $U^2$ from Definition~\ref{def:uk}: that is, $I^{U^2} = \{00,01,10,00\}$, $V^{U^2} = \FF_p^3$, and $\phi_{ij}^{U^2}(x,a,b) = x\!+\!ia\!+\!jb$.
  We again have a diagram $\diagram(U^2)$ (also called $U^2$; see Convention~\ref{conv:no-diagram-word}):
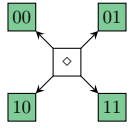
  \begin{center}
    \begin{tikzpicture}[
        every node/.style={rectangle,draw,inner sep=2pt,outer sep=0pt,minimum size=15pt,scale=0.7},
        leaf/.style={fill=leafgreen}]
      \node       (V)  at ( 0, 0) {$\diamond$};
      \node[leaf] (W00) at (-0.6, 0.6) {$00$};
      \node[leaf] (W01) at ( 0.6, 0.6) {$01$};
      \node[leaf] (W10) at (-0.6,-0.6) {$10$};
      \node[leaf] (W11) at ( 0.6,-0.6) {$11$};
      \draw[-stealth] (V) -- (W00);
      \draw[-stealth] (V) -- (W01);
      \draw[-stealth] (V) -- (W10);
      \draw[-stealth] (V) -- (W11);
    \end{tikzpicture}
  \end{center}
  Our goal is to show that $D_0 \entails^2_\gamma \diagram(U^2)$ where $\gamma(ij) = (1,i+j \bmod 2)$ for $i,j \in \{0,1\}$.
  Hence $\Phi \entails^2_\gamma U^2$.
  This is weaker than Fact~\ref{fact:functional-inequality} in that it implies
  \[
    \lvert \Lambda_{\Phi}(f_1,f_2,f_3) \rvert \le \lvert \Lambda_{\Psi}(g_{00},g_{01},g_{10},g_{11}) \rvert^{1/4}
  \]
  for some functions $g_{ij}$ where $g_{ij}$ is a translate of $f_1$ (if $ij=00,11$) or $\overline{f_1}$ (if $ij=01,10$).
  By contrast, in Fact~\ref{fact:functional-inequality} we got the same statement where $g_{ij}$ was equal to $f_1$ or $\overline{f_1}$.
  However, as we have said, it is somewhat burdensome to build this stronger observation into our framework.
  In many cases including this one, we could recover the stronger form by the Gowers--Cauchy--Schwarz inequality:
  \[
    \lvert \Lambda_{\Psi}(g_{00},g_{01},g_{10},g_{11}) \rvert^{1/4} \le 
    \lVert g_{00} \rVert_{U^2}^{1/4}\,
    \lVert g_{01} \rVert_{U^2}^{1/4}\,
    \lVert g_{10} \rVert_{U^2}^{1/4}\,
    \lVert g_{11} \rVert_{U^2}^{1/4}
    = 
    \lVert f \rVert_{U^2}.
  \]

  We now prove $D_0 \entails^2_\gamma \diagram(U^2)$. 
  We first apply $\CS(\{3\})$ to give $D_1 =D_0 +_{\{3\}} D_0$, which looks like this:
  
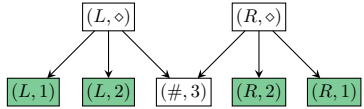
\begin{center}
    \begin{tikzpicture}[
        every node/.style={rectangle,draw,inner sep=2pt,outer sep=0pt,minimum size=15pt,scale=0.7},
        leaf/.style={fill=leafgreen}]
      \node       (0V)  at ( 0, 0) {$(L,\diamond)$};
      \node[leaf] (0W1) at (-1,-1) {$(L,1)$};
      \node[leaf] (0W2) at ( 0,-1) {$(L,2)$};
      \node       (-W3) at ( 1,-1) {$(\#,3)$};
      \node       (1V)  at ( 2, 0) {$(R,\diamond)$};
      \node[leaf] (1W2) at ( 2,-1) {$(R,2)$};
      \node[leaf] (1W1) at ( 3,-1) {$(R,1)$};
      \draw[-stealth] (0V) -- (0W1);
      \draw[-stealth] (0V) -- (0W2);
      \draw[-stealth] (0V) -- (-W3);
      \draw[-stealth] (1V) -- (1W1);
      \draw[-stealth] (1V) -- (1W2);
      \draw[-stealth] (1V) -- (-W3);
    \end{tikzpicture}
  \end{center}
  Moreover $D_0 \entails^1_{(L,i) \mapsto (i,0), (R,i) \mapsto (i,1)} D_1$ by this Cauchy--Schwarz step.
  Next, applying $\CS(\{(L,2),(R,2)\})$ we obtain $D_2 = D_1 +_{\{(L,2),(R,2)\}} D_1$:
  
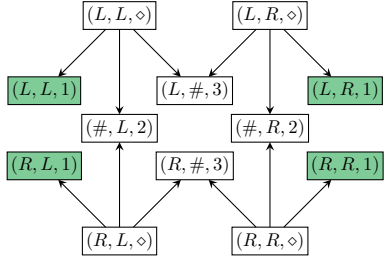
\begin{center}
    \begin{tikzpicture}[
        every node/.style={rectangle,draw,inner sep=2pt,outer sep=0pt,minimum size=15pt,scale=0.7},
        leaf/.style={fill=leafgreen}]
      \node       (00V)  at ( 0, 0)   {$(L,L,\diamond)$};
      \node[leaf] (00W1) at (-1,-1)   {$(L,L,1)$};
      \node[leaf] (01W1) at ( 3,-1)   {$(L,R,1)$};
      \node       (0-W3) at ( 1,-1)   {$(L,\#,3)$};
      \node       (01V)  at ( 2, 0)   {$(L,R,\diamond)$};

      \node       (-0W2) at ( 0,-1.5) {$(\#,L,2)$};
      \node       (-1W2) at ( 2,-1.5) {$(\#,R,2)$};

      \node       (10V)  at ( 0,-3) {$(R,L,\diamond)$};
      \node[leaf] (10W1) at (-1,-2) {$(R,L,1)$};
      \node[leaf] (11W1) at ( 3,-2) {$(R,R,1)$};
      \node       (1-W3) at ( 1,-2) {$(R,\#,3)$};
      \node       (11V)  at ( 2,-3) {$(R,R,\diamond)$};

      \draw[-stealth] (00V) -- (00W1);
      \draw[-stealth] (01V) -- (01W1);
      \draw[-stealth] (10V) -- (10W1);
      \draw[-stealth] (11V) -- (11W1);
      \draw[-stealth] (00V) -- (-0W2);
      \draw[-stealth] (01V) -- (-1W2);
      \draw[-stealth] (10V) -- (-0W2);
      \draw[-stealth] (11V) -- (-1W2);
      \draw[-stealth] (00V) -- (0-W3);
      \draw[-stealth] (01V) -- (0-W3);
      \draw[-stealth] (10V) -- (1-W3);
      \draw[-stealth] (11V) -- (1-W3);
    \end{tikzpicture}
  \end{center}
  Again $D_1 \entails^1_{\gamma} D_2$ where $\gamma(a,b,1) = (1,\ell)$ for $a,b \in \{L,R\}$, where $\ell=0$ when $a=b$ and $\ell=1$ otherwise.
  
  Finally, we claim there is a morphism $\Theta \colon \diagram(U^2) \to D_2$ as follows.
  We choose the natural bijection on leaves, $\alpha^\Theta(L,L,1)=00$, $\alpha^\Theta(L,R,1)=01$, $\alpha^\Theta(R,L,1)=10$ and $\alpha^\Theta(R,R,1)=11$, and also set $\alpha^\Theta(x) = \diamond$ for every non-leaf $x$ of $D_2$ (the only possible choice).
  The maps $\theta^\Theta_{(L,L,1)} \colon V_{00}^{U^2} \to V_{(L,L,1)}^{D_2}$, etc., are all the identity map $\FF_p \to \FF_p$.
  Therefore $\Theta$ respects every leaf.
  We have the following picture:
  
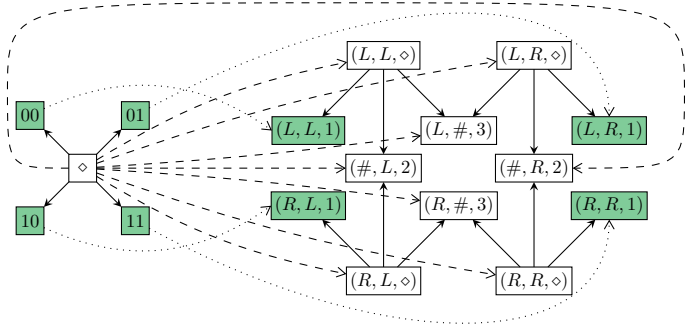
\begin{center}
    \begin{tikzpicture}[
        every node/.style={rectangle,draw,inner sep=2pt,outer sep=0pt,minimum size=15pt,scale=0.7},
        leaf/.style={fill=leafgreen}]
      \begin{scope}[shift={(-4,-1.5)}]
        \node       (V)  at ( 0, 0) {$\diamond$};
        \node[leaf] (W00) at (-0.7, 0.7) {$00$};
        \node[leaf] (W01) at ( 0.7, 0.7) {$01$};
        \node[leaf] (W10) at (-0.7,-0.7) {$10$};
        \node[leaf] (W11) at ( 0.7,-0.7) {$11$};
        \draw[-stealth] (V) -- (W00);
        \draw[-stealth] (V) -- (W01);
        \draw[-stealth] (V) -- (W10);
        \draw[-stealth] (V) -- (W11);
      \end{scope}

      \node       (00V)  at ( 0, 0)   {$(L,L,\diamond)$};
      \node[leaf] (00W1) at (-1,-1)   {$(L,L,1)$};
      \node[leaf] (01W1) at ( 3,-1)   {$(L,R,1)$};
      \node       (0-W3) at ( 1,-1)   {$(L,\#,3)$};
      \node       (01V)  at ( 2, 0)   {$(L,R,\diamond)$};

      \node       (-0W2) at ( 0,-1.5) {$(\#,L,2)$};
      \node       (-1W2) at ( 2,-1.5) {$(\#,R,2)$};

      \node       (10V)  at ( 0,-3) {$(R,L,\diamond)$};
      \node[leaf] (10W1) at (-1,-2) {$(R,L,1)$};
      \node[leaf] (11W1) at ( 3,-2) {$(R,R,1)$};
      \node       (1-W3) at ( 1,-2) {$(R,\#,3)$};
      \node       (11V)  at ( 2,-3) {$(R,R,\diamond)$};

      \draw[-stealth] (00V) -- (00W1);
      \draw[-stealth] (01V) -- (01W1);
      \draw[-stealth] (10V) -- (10W1);
      \draw[-stealth] (11V) -- (11W1);
      \draw[-stealth] (00V) -- (-0W2);
      \draw[-stealth] (01V) -- (-1W2);
      \draw[-stealth] (10V) -- (-0W2);
      \draw[-stealth] (11V) -- (-1W2);
      \draw[-stealth] (00V) -- (0-W3);
      \draw[-stealth] (01V) -- (0-W3);
      \draw[-stealth] (10V) -- (1-W3);
      \draw[-stealth] (11V) -- (1-W3);

      \draw[dotted,-angle 60] (W00) to[bend left] (00W1.west);
      \draw[dotted, -angle 60] (W01) .. controls (0,1) and (3,1) .. (01W1.north);
      \draw[dotted,-angle 60] (W10) to[bend right] (10W1.west);
      \draw[dotted, -angle 60] (W11) .. controls (0,-4) and (3,-4) .. (11W1.south);

      \draw[dashed, -angle 60] (V) to[bend left=10] (00V);
      \draw[dashed, -angle 60] (V) to[bend left=5] (01V);
      \draw[dashed, -angle 60] (V) to[bend right=10] (10V);
      \draw[dashed, -angle 60] (V) to[bend right=5] (11V);

      \draw[dashed, -angle 60] (V) to (-0W2);
      \draw[dashed, -angle 60] (V) to[bend right=2] (0-W3);
      \draw[dashed, -angle 60] (V) to[bend left=2]  (1-W3);
      \draw[dashed, -angle 60] (V) .. controls (-5,-1.5) .. (-5, -0.5) .. controls (-5, 0.7) .. (-0.5, 0.7) .. controls (4,0.7) .. (4,-0.5) .. controls (4,-1.5) .. (-1W2.east);
    \end{tikzpicture}
  \end{center}
  and it remains to find maps $\theta_x \colon V^{U^2}_\diamond \to V^{D_2}_x$ for every non-leaf $x \in X^{D_2}$ (corresponding to the dashed arrows) such that~\eqref{eq:diag-morph} holds.
  This is the step that in the original proof was summarized by the last ``changing variables'' instruction.
  The choice is not unique, but an explicit family of maps is as follows:---
  \begin{align*}
    \theta_{(L,L,\diamond)}(y,a,b) &= (y,0) & 
    \theta_{(\#,L,2)}(y,a,b) &= y\\
    \theta_{(L,R,\diamond)}(y,a,b) &= (y+b,-b/2) & 
    \theta_{(\#,,R,2)}(y,a,b) &= y+b/2\\
    \theta_{(R,L,\diamond)}(y,a,b) &= (y+a,-a) & 
    \theta_{(L,\#,,3)}(y,a,b) &= y\\
    \theta_{(R,R,\diamond)}(y,a,b) &= (y+a+b,-a-b/2) & 
    \theta_{(R,\#,3)}(y,a,b) &= y-a.
  \end{align*}
  As in Remark~\ref{rem:pictures} it is most conceptual to show this information pictorially,
  by inserting these linear functions of $(y,a,b)$ into the diagram above in place of the vertex labels.
  
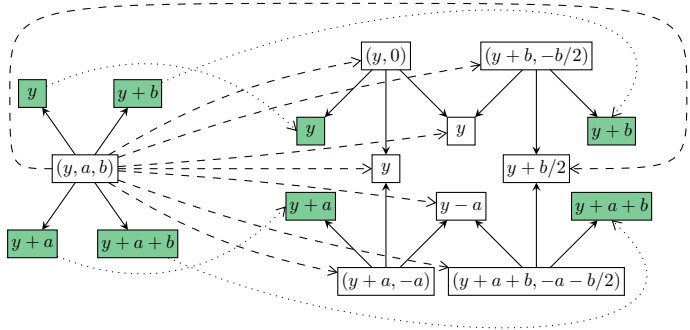
\begin{center}
    \begin{tikzpicture}[
        every node/.style={rectangle,draw,inner sep=2pt,outer sep=0pt,minimum size=15pt,scale=0.7},
        leaf/.style={fill=leafgreen}]
      \begin{scope}[shift={(-4,-1.5)}]
        \node       (V)  at ( 0, 0) {$(y,a,b)$};
        \node[leaf] (W00) at (-0.7, 1.0) {$y$};
        \node[leaf] (W01) at ( 0.7, 1.0) {$y+b$};
        \node[leaf] (W10) at (-0.7,-1.0) {$y+a$};
        \node[leaf] (W11) at ( 0.7,-1.0) {$y+a+b$};
        \draw[-stealth] (V) -- (W00);
        \draw[-stealth] (V) -- (W01);
        \draw[-stealth] (V) -- (W10);
        \draw[-stealth] (V) -- (W11);
      \end{scope}

      \begin{scope}[yscale=1.0]
        \node       (00V)  at ( 0, 0)   {$(y,0)$};
        \node[leaf] (00W1) at (-1,-1)   {$y$};
        \node[leaf] (01W1) at ( 3,-1)   {$y+b$};
        \node       (0-W3) at ( 1,-1)   {$y$};
        \node       (01V)  at ( 2, 0)   {$(y+b,-b/2)$};

        \node       (-0W2) at ( 0,-1.5) {$y$};
        \node       (-1W2) at ( 2,-1.5) {$y+b/2$};

        \node       (10V)  at ( 0,-3) {$(y+a,-a)$};
        \node[leaf] (10W1) at (-1,-2) {$y+a$};
        \node[leaf] (11W1) at ( 3,-2) {$y+a+b$};
        \node       (1-W3) at ( 1,-2) {$y-a$};
        \node       (11V)  at ( 2,-3) {$(y+a+b,-a-b/2)$};
      \end{scope}

      \draw[-stealth] (00V) -- (00W1);
      \draw[-stealth] (01V) -- (01W1);
      \draw[-stealth] (10V) -- (10W1);
      \draw[-stealth] (11V) -- (11W1);
      \draw[-stealth] (00V) -- (-0W2);
      \draw[-stealth] (01V) -- (-1W2);
      \draw[-stealth] (10V) -- (-0W2);
      \draw[-stealth] (11V) -- (-1W2);
      \draw[-stealth] (00V) -- (0-W3);
      \draw[-stealth] (01V) -- (0-W3);
      \draw[-stealth] (10V) -- (1-W3);
      \draw[-stealth] (11V) -- (1-W3);

      \draw[dotted,-angle 60] (W00) to[bend left=35] (00W1.west);
      \draw[dotted, -angle 60] (W01) .. controls (0,1) and (4.6,1) .. (01W1.north);
      \draw[dotted,-angle 60] (W10) to[bend right=35] (10W1.west);
      \draw[dotted, -angle 60] (W11) .. controls (0,-4) and (5.0,-4) .. (11W1.south);

      \draw[dashed, -angle 60] (V) to[bend left=10] (00V);
      \draw[dashed, -angle 60] (V) to[bend left=5] (01V);
      \draw[dashed, -angle 60] (V) to[bend right=10] (10V);
      \draw[dashed, -angle 60] (V) to[bend right=5] (11V);

      \draw[dashed, -angle 60] (V) to (-0W2);
      \draw[dashed, -angle 60] (V) to[bend right=2] (0-W3);
      \draw[dashed, -angle 60] (V) to[bend left=2]  (1-W3);
      \draw[dashed, -angle 60] (V) .. controls (-5,-1.5) .. (-5, -0.5) .. controls (-5, 0.7) .. (-0.5, 0.7) .. controls (4,0.7) .. (4,-0.5) .. controls (4,-1.5) .. (-1W2.east);
    \end{tikzpicture}
  \end{center}
  It is clear that each solid arrow on the left is the given map $\phi^{\diagram(U^2)}_{\diamond i}$,
  each dotted arrow (from leaf to leaf) is the map $\theta_x = \id_{\FF_p}$
  and each dashed arrow (from non-leaf to non-leaf) specifies the map $\theta_x$ given above.
  To verify that~\eqref{eq:diag-morph} holds,
  it suffices to check that each solid arrow $\phi^{D_2}_{xy}$ on the right does indeed map the input shown to the output shown;
  equivalently, that the right-hand side is a compatible tuple in $D_2$.
  In pictorial form, this is clear by inspection.

  We note that this last picture is very close to what we mean by an ``arithmetic circuit'', up to notational changes.
  In particular we could redraw the right-hand side to better resemble Figure~\ref{fig:bilinear}:
  \begin{center}
    \begin{tikzpicture}[
        subnode/.style={rounded rectangle, draw, inner sep=2pt, outer sep=0pt,minimum size=15pt},
        pinnode/.style={circle,fill=white,scale=0.6, inner sep=0.7pt, outer sep=0pt},
        wirelabel/.style={scale=0.7, label distance=-0.4em}]

        \node[subnode]  (00V)  at ( 0, 0)       {$\ \Phi\ $};
        \node[subnode]  (01V)  at ( 2.0, 0)     {$\ \Phi\ $};
        \node[subnode]  (10V)  at ( 0, -1.3)    {$\ \Phi\ $};
        \node[subnode]  (11V)  at ( 2.0, -1.3)  {$\ \Phi\ $};

        \draw (00V) -- node[midway, "{$y$}" {above,wirelabel}] {} (01V);
        \draw (10V) -- node[midway, "{$y-a$}" {below,wirelabel}] {} (11V);
        \draw (00V) -- node[midway, "{$y$}" {left ,wirelabel}] {} (10V);
        \draw (01V) -- node[midway, "{$y+b/2$}" {right,wirelabel}] {} (11V);

        \draw (00V) -- ++(-1.5,0)      node["{$y$}" {above,wirelabel}] {};
        \draw (01V) -- ++(1.5,0)     node["{$y+b$}" {above,wirelabel}] {};
        \draw (10V) -- ++(-1.5,0)   node["{$y+a$}" {below,wirelabel}] {};
        \draw (11V) -- ++(1.5,0)  node["{$y+a+b$}" {below,wirelabel}] {};

        \draw (00V.west)  node[pinnode] {$1$};
        \draw (00V.south) node[pinnode] {$2$};
        \draw (00V.east)  node[pinnode] {$3$};
        \draw (01V.west)  node[pinnode] {$3$};
        \draw (01V.south) node[pinnode] {$2$};
        \draw (01V.east)  node[pinnode] {$1$};
        \draw (10V.west)  node[pinnode] {$1$};
        \draw (10V.north) node[pinnode] {$2$};
        \draw (10V.east)  node[pinnode] {$3$};
        \draw (11V.west)  node[pinnode] {$3$};
        \draw (11V.north) node[pinnode] {$2$};
        \draw (11V.east)  node[pinnode] {$1$};
    \end{tikzpicture}
  \end{center}
  Again the notation here is informal: we will define a slightly different precise notation in Convention~\ref{conv:labelling-other}.
  We briefly explain how the diagram above should be interpreted.
  \begin{itemize}
    \item Each vertex $x$ with $V_x = \FF_p$ has become a ``wire''.
      In particular the four leaves have become the four ``free wires'' in the NW, NE, SW and SE corners, and the four non-leaf vertices with $V_y = \FF_p$ have become the four ``connecting wires'' forming the middle square.
    \item Each copy of $V^\Phi$ has become a ``$\Phi$-gate''.
      Such a gate has three ``pins'', called $1$, $2$ and $3$, and comes with an assertion that the values $z_1$, $z_2$ and $z_3$ on the wires at these pins satisfy $z_1=\phi_1^\Phi(v)$, $z_2 = \phi_2^\Phi(v)$ and $z_3=\phi_3^\Phi(v)$ for some $v \in \FF_p^2$, or equivalently that $(z_1,z_2,z_3)$ form an arithmetic progression, i.e., $z_1-2 z_2 + z_3 = 0$.
  \end{itemize}
  Hence the picture above corresponds to building a ``$U^2$-gate'' out of four $\Phi$-gates.

  Note the ``circuit'' here is doing purely linear calculations, as opposed to the multilinear calculations in Section~\ref{sub:repr}.
  We will explain this key missing feature in Section~\ref{sec:gates}.
  Otherwise, though, the remarks above will remain valid when we consider how to to realize the sketch proof in Section~\ref{sub:repr} in terms of diagrams.
  In particular, we emphasize that the values appearing on wires in arithmetic circuits arise precisely as compatible tuples of some diagram, related to some morphism.

  Here we started by defining the maps $\theta_x$ and used that definition to write down these pictures.
  In practice it is much more useful to go the other way: given the last ``arithmetic circuit'' picture, it is routine to work backwards and recover the linear maps $\theta_x$.
  In what follows we will sometimes leave this step to the reader.
\end{example}

\begin{example}%
  \label{ex:cache-example}
  We recall the description of ``stashing'' from Section~\ref{sub:true-summary}.
  We saw in Remark~\ref{rem:stashing} that stashing a dual function corresponds to joining two linear data at a single pair $(i_0,i_1)$.
  In terms of diagrams this will look like a cut-vertex with in-degree $2$ and out-degree $0$.
  For example, in the setting of Remark~\ref{rem:stashing} the diagram $\diagram(\Phi) +_{6 \leftrightarrow 000} \diagram(U^3)$ looks like this:
  
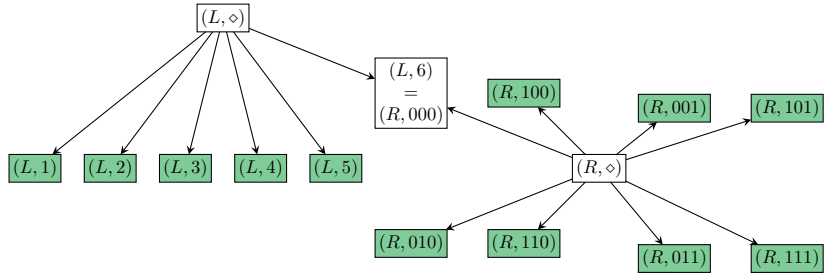
\begin{center}
    \begin{tikzpicture}[
        every node/.style={rectangle,draw,inner sep=2pt,outer sep=0pt,minimum size=15pt,scale=0.7},
        leaf/.style={fill=leafgreen}]
      \node       (V)  at ( 0, 0) {$(L,\diamond)$};
      \node[leaf] (W1) at (-2.5,-2) {$(L,1)$};
      \node[leaf] (W2) at (-1.5,-2) {$(L,2)$};
      \node[leaf] (W3) at (-0.5,-2) {$(L,3)$};
      \node[leaf] (W4) at ( 0.5,-2) {$(L,4)$};
      \node[leaf] (W5) at ( 1.5,-2) {$(L,5)$};
      \node[align=center] (W6) at ( 2.5,-1) {$(L,6)$\\ $=$\\ $(R,000)$};

      \coordinate (X) at (1.5,-0);
      \coordinate (Y) at (0.0,-2.0);
      \coordinate (Z) at (3.5,-0.2);

      \node[leaf] (R100) at ($(W6) + (X)$) {$(R,100)$};
      \node[leaf] (R010) at ($(W6) + (Y)$) {$(R,010)$};
      \node[leaf] (R001) at ($(W6) + (Z)$) {$(R,001)$};
      \node[leaf] (R110) at ($(W6) + (X) + (Y)$) {$(R,110)$};
      \node[leaf] (R101) at ($(W6) + (X) + (Z)$) {$(R,101)$};
      \node[leaf] (R011) at ($(W6) + (Y) + (Z)$) {$(R,011)$};
      \node[leaf] (R111) at ($(W6) + (X) + (Y) + (Z)$) {$(R,111)$};
      \node (U) at (5,-2) {$(R,\diamond)$};

      \draw[-stealth] (V) -- (W1);
      \draw[-stealth] (V) -- (W2);
      \draw[-stealth] (V) -- (W3);
      \draw[-stealth] (V) -- (W4);
      \draw[-stealth] (V) -- (W5);
      \draw[-stealth] (V) -- (W6);

      \draw[-stealth] (U) -- (W6);
      \draw[-stealth] (U) -- (R001);
      \draw[-stealth] (U) -- (R010);
      \draw[-stealth] (U) -- (R011);
      \draw[-stealth] (U) -- (R100);
      \draw[-stealth] (U) -- (R101);
      \draw[-stealth] (U) -- (R110);
      \draw[-stealth] (U) -- (R111);
    \end{tikzpicture}
  \end{center}
  We can think of this either as a copy of $\diagram(\Phi)$ with a stashed $U^3$-dual function, or as a copy of $\diagram(U^3)$ with a stashed $\Phi$-dual function.
  Visually these correspond to ``collapsing'' respectively the right-hand or left-hand halves of the diagram.
  We will formalize this in Section~\ref{sec:stashing}.
\end{example}

\subsection{Some further constructions and properties}%
\label{sub:constructions}

In the above we focussed on constructions and properties of linear data and diagrams that correspond to inequalities.
We now mention some others that do not, but which will nonetheless be useful.

First we make the rather fundamental observation that morphisms of linear data and diagrams can be composed.
\begin{definition}%
  \label{def:compose-morph}
  Given two morphisms of linear data $\Theta_1 \colon \Phi_0 \to \Phi_1$ and $\Theta_2 \colon \Phi_1 \to \Phi_2$, their composite is a morphism $\Theta_2 \circ \Theta_1 \colon \Phi_0 \to \Phi_2$, given by $\alpha = \alpha^{\Theta_1} \circ \alpha^{\Theta_2}$, $\theta = \theta^{\Theta_2} \circ \theta^{\Theta_1}$ and $\sigma_i = \sigma^{\Theta_2}_{i} \circ \sigma^{\Theta_1}_{\alpha^{\Theta_2}(i)}$.

  Similarly, for two diagram morphisms $\Theta_1 \colon D_0 \to D_1$ and $\Theta_2 \colon D_1 \to D_2$ their composite $\Theta_2 \circ \Theta_1 \colon D_0 \to D_2$ has $\alpha = \alpha^{\Theta_1} \circ \alpha^{\Theta_2}$ and 
  $\theta_x = \theta^{\Theta_2}_x \circ \theta^{\Theta_1}_{\alpha^{\Theta_2}(x)}$.

  Where necessary we use the further convention that $\alpha(\zeroi) = \zeroi$.
\end{definition}
We leave it to the reader to verify that these are indeed morphisms.

We briefly consider various notions of isomorphism among linear data or diagrams.

\begin{definition}%
  \label{def:strong-iso}
  As usual, a morphism $\Phi \to \Psi$ of linear data or $C \to D$ of diagrams is called an \emph{isomorphism} if it has an inverse morphism.

  These is readily seen to be equivalent to the added conditions that (i) $\alpha$ is a bijection (for linear data) or graph isomorphism (for diagrams), and (ii) all the component maps $(\theta, (\sigma_i)_{i \in I})$ (for linear data) or $(\theta_x)_{x \in X}$ (for diagrams) are isomorphisms.

  An isomorphism $\Phi \to \Psi$ is called a \emph{strong isomorphism} if furthermore it respects every index; i.e, if all maps $\sigma_i$ are the identity map.
  Similarly an isomorphism $C \to D$ of diagrams is called a \emph{strong isomorphism} if it respects every leaf.
\end{definition}

Next, we note a key property relating morphisms of linear data and diagrams.
\begin{proposition}%
  \label{prop:adjunction}
  Let $\Phi$ be a linear datum and $D$ a diagram.
  There is a natural bijection
  \[
    \bigl\{ \text{diagram morphisms } \diagram(\Phi) \to D \bigr\} \leftrightarrow \bigl\{ \text{datum morphisms } \Phi \to \datum(D) \bigr\}.
  \]
  The diagram morphism $\diagram(\Phi) \to D$ respects a leaf $i \in \leafs^D$ if and only if the linear datum morphism $\Phi \to \datum(D)$ respects $i \in I^{\datum(D)}$.
\end{proposition}
In other words, $\diagram(-)$ is the left adjoint of $\datum(-)$.
Recalling Remark~\ref{rem:datum-diag}, this means that when considering morphisms with domain $\diagram(\Phi)$, it makes no difference whether we work in the category of linear data or diagrams.
\begin{proof}[Proof of Proposition~\ref{prop:adjunction}]
  Write $\Psi = \datum(D)$, so $V^\Psi$ is the space in~\eqref{eq:limit-space}.
  For $x \in X^D$, write $\pi_x \colon V^\Psi \to V_x$ for the coordinate projection map $(v_y)_{y \in X^D} \mapsto v_x$, as in~\eqref{eq:projections}.
  Note that by the definition of $V^\Psi$, for all $xy \in E^D$ we have
  \begin{equation}
    \label{eq:pi-compat}
    \pi_y = \phi^D_{xy} \circ \pi_x.
  \end{equation}
  
  Given a morphism $\Theta \colon \Phi \to \Psi$, we may define a diagram morphism $\wt \Theta \colon \diagram(\Phi) \to D$ by setting $\alpha^{\wt \Theta}(x) = \diamond$ for all $x \in \nonleafs^D$ and $\theta^{\wt \Theta}_x = \pi_x \circ \theta^\Theta$ for all $x \in X^D$.
  Then~\eqref{eq:diag-morph} follows from~\eqref{eq:pi-compat}.

  Conversely, given $\wt \Theta \colon \datum(\Phi) \to D$ we necessarily have $\alpha^{\wt \Theta}(x) = \diamond$ for all $x \in \nonleafs^D$ (as it is the only non-leaf) and we set $\alpha^\Theta =\alpha^{\wt \Theta}|_{\leafs^D}$.
  We also set $\sigma^\Theta_j = \theta^{\wt \Theta}_j$ for $j \in \leafs^D$, and define $\theta^\Theta \colon V^\Phi \to V^\Psi$ by\footnote{In other words we are using the universal property of the limit space~\eqref{eq:limit-space}.}
  \[
    \theta^\Theta(v) = \bigl(\theta^{\wt \Theta}_x(v)\bigr)_{x \in X^D}.
  \]
  The fact that the right-hand side lies in $V^\Psi$, i.e., is a compatible tuple in $D$, is immediate from the hypothesis~\eqref{eq:diag-morph} on $\wt \Theta$.
\end{proof}

Next, we consider natural notions of ``sub-datum'' and ``sub-diagram''.
\begin{definition}%
  \label{def:subdatum}
  If $\Phi$ is a linear datum and $J \subseteq I^\Phi$ is a subset of its indices, the \emph{restriction} or \emph{induced sub-datum} $\Phi[J]$ is simply $\bigl(J, V^\Phi, (W^\Phi_i)_{i \in J}, (\phi^\Phi_i)_{i \in J}\bigr)$, i.e., the linear datum obtained by discarding spaces $W_i$ and maps $\phi_i$ for $i \notin J$.
\end{definition}

\begin{definition}%
  \label{def:subdiagram}
  If $D$ is a diagram and $Y \subseteq X^D$ is a subset of its vertices, the \emph{restriction} or \emph{induced sub-diagram} on $Y$, denoted $D[Y]$, is the diagram obtained by removing all vertices, edges, vector spaces and linear maps mentioning vertices of $X \setminus Y$.
\end{definition}

\begin{remark}%
  \label{rem:subdiagram-madness}
  The restriction $D[Y]$ might not actually be a legal diagram, since we require any surviving leaves to have in-degree $1$.
  We require the added constraint on $Y$ that if $y \in Y$ is a leaf and $x \in X^D$ is its unique parent $xy \in E^D$, then $x \in Y$.
\end{remark}

\begin{convention}%
  If, as is often the case, we wish to delete some leaves of a diagram $D$ but keep all non-leaves, for $Y \subseteq \leafs^D$ we write $D[\nonleafs \cup Y]$ in place of the redundant $D\bigl[\nonleafs^D \cup Y\bigr]$.

  If $\Phi$ is a linear datum and $J \subseteq I^\Phi$ then it is clear that $\diagram(\Phi[J]) = \diagram(\Phi)[\{\diamond\} \cup J]$.
  However, if we are using Convention~\ref{conv:no-diagram-word} to write $\Phi$ to mean $\diagram(\Phi)$, we have an ambiguity because $\Phi[J]$ and $\diagram(\Phi)[J]$ are not the same: indeed, the latter is never legal in the sense of Remark~\ref{rem:subdiagram-madness}.
  Hence we lose nothing by deciding that $\Phi[J]$ always defaults to the restriction with $\Phi$ considered as a linear datum, if applicable.

  Applying the same abuse of notation again, $\Phi[J]$ can also refer to $\diagram(\Phi[J])$.
\end{convention}

\begin{remark}%
  \label{rem:restriction}
  We note that in both cases there is a natural \emph{restriction map}, i.e., a morphism $\Phi \to \Phi[J]$ or $D \to D[Y]$.
  We take $\alpha(x)=x$ for all $x$ in the domain, and all linear maps to be the identity.
\end{remark}

\begin{remark}%
  \label{rem:restriction-morphs}
  Given two linear data $\Phi$, $\Psi$, a morphism $\Theta \colon \Phi \to \Psi$ and a set $Y \subseteq I^\Phi$, let $Y' = (\alpha^\Theta)^{-1}(Y)$.
  Then it is clear $\Theta$ defines a morphism $\Phi[Y] \to \Phi[Y']$, by simply discarding the maps $\sigma^\Theta_i$ for $i \notin Y'$.

  Similarly, given two diagrams $C$, $D$, a morphism $\Theta \colon C \to D$ and a set $Y \subseteq X^C$, we get a morphism $C[Y] \to D[Y']$ where $Y' = (\alpha^{\Theta})^{-1}(Y)$.
\end{remark}

We should check, because it is not obvious, that the notions of sub-datum and sub-diagram are compatible under the $\datum(-)$ construction.

\begin{proposition}%
  \label{prop:sub-sub}
  If $D$ is a linear datum and $Y \subseteq \leafs^D$ is a subset, then
  \[
   \datum(D[\nonleafs \cup Y]) = (\datum(D))[Y]
  \]
  up to canonical isomorphisms.
  More precisely, there is a canonical strong isomorphism between them, as in Definition~\ref{def:strong-iso}.
\end{proposition}
The proof is similar in spirit to Remark~\ref{rem:datum-diag}.
\begin{proof}%
  Write $\Phi = \datum(D[\nonleafs \cup Y])$ and $\Psi = (\datum(D))[Y]$.
  Since they have the same index set $Y$ and the same spaces $(W_y)_{y \in Y}$, it suffices to find an isomorphism
  $\tau \colon V^\Phi \xrightarrow{\cong} V^\Psi$ such that $\sigma_y^\Psi \circ \tau = \sigma_y^\Phi$ for all $y \in Y$.
  
  Every leaf $i \in \leafs^D$ has in-degree $1$,\footnote{This is possibly the only place we use this hypothesis.}
  and we write $n_i \in \nonleafs^D$ for the unique vertex with $n_i i \in E^D$.
  Write $Y' = \nonleafs^D \cup Y$, $Z = \leafs^D \setminus Y$ and
  \[
    E' = \{ xy \in E^D \colon x, y \in Y' \} = E^D \setminus \{ n_i i \colon i \in Z \}.
  \]
  By~\eqref{eq:limit-space} we have
  \[
    V^{\datum(D)} 
    = \left\{ (v_x)_{x \in X^D} \in \bigoplus_{x \in X^D} V^D_x : \forall xy \in E^D \colon \phi^D_{xy}(v_x) = v_y  \right\}
  \]
  and
  \[
    V^{\datum(D[\nonleafs \cup Y])} 
    = \left\{ (v_x)_{x \in Y'} \in \bigoplus_{x \in Y'} V^D_x : \forall xy \in E' \colon \phi^D_{xy}(v_x) = v_y  \right\}.
  \]
  There is a natural projection map $V^{\datum(D)} \to V^{\datum(D[\nonleafs \cup Y])}$ by omitting coordinates.
  In the other direction, define $\tau \colon V^{\datum(D[\nonleafs \cup Y])} \to V^{\datum(D)}$ by
  \[
    (v_x)_{x \in Y'} \mapsto \biggl(v_x \colon x \in Y',\ \phi^D_{n_i i}(v_{n_i}) \colon i \in Z\biggr).
  \]
  It is clear this is a well-defined linear map and an inverse to the projection map.
  Recalling that $V^{\datum(D)[Y]} = V^{\datum(D)}$ and that $\sigma^\Phi_y$, $\sigma^\Psi_y$ are just the coordinate projections, we are done.
\end{proof}

There is also a useful dual notion to the sub-datum construction, where instead of simply deleting certain indices we both delete them and set them to zero.

\begin{definition}%
  \label{def:cores}
  Suppose $\Phi$ is a linear datum and $J \subseteq I^\Phi$ is a subset.
  Write $I' = I^\Phi \setminus J$ for the complement.
  The \emph{co-restriction} $\Phi\langle J \rangle$ is the linear datum $\bigl(I', V', (W^\Phi_i)_{i \in I'}, (\phi^\Phi_i|_{V'})_{i \in I'}\bigr)$, where
  \[
    V' = \bigcap_{j \in J} \ker \phi_j^\Phi.
  \]
\end{definition}

\begin{remark}%
  \label{rem:cores}
  The natural \emph{co-restriction map} is a morphism $\Phi\langle J \rangle \to \Phi$.
  It has $\alpha(i) = i$ for all $i \notin J$ and $\alpha(j) = \zeroi$ for all $j \in J$.
  The maps $\sigma_i$ are the identity for $i \notin J$, and $\theta \colon V' \to V^\Phi$ is the inclusion map.

  Co-restriction has the following universal property: the morphisms $\Psi \to \Phi$ which factor through the co-restriction map as $\Psi \to \Phi\langle J \rangle \to \Phi$ are exactly those with $\alpha(j) = \zeroi$ for all $j \in J$, and moreover those factor uniquely.
\end{remark}

There is an analogous notion for diagrams (defining a suitable subspace of every space $V_x$ for $x$ a non-leaf) but it is awkward to state and we don't need to.
However, in the special case of diagrams $D$ of the form $D=\diagram(\Phi)$, we can write $D\langle J \rangle$ to mean the same as $\diagram(\Phi\langle J \rangle)$.

One particularly straightforward operation on diagrams is to make a leaf into a non-leaf or vice-versa.
Again, by itself this has no valid interpretation in terms of inequalities but it will be useful in the abstract.

\begin{definition}%
  \label{def:leaf-non-leaf}
  Let $D$ be a diagram.
  For a subset $Y \subseteq \leafs^D$, we write $D(Y \leadsto \nonleafs)$ to denote the diagram $D'$ which is identical to $D$ except that $\leafs^{D'} = \leafs^D \setminus Y$ and $\nonleafs^{D'} = \nonleafs^D \cup Y$.
\end{definition}

\begin{remark}%
  \label{rem:joining-sub}
  Consider a joining $C = D_0 +_{\beta} D_1$ where $\beta$ is a matching between $J_0 \subseteq \leafs^{D_0}$ and $J_1 \subseteq \leafs^{D_1}$, as in Definition~\ref{def:diagram-cs}.
  Then consider a sub-diagram such as $C[Y]$ where $Y = \bigl\{ (L, x) \colon x \in X^{D_0} \bigr\}$.
  This is exactly $D_0(J_0 \leadsto \nonleafs)$: i.e., it is a copy of $D_0$ where all elements of $J_0$ have become non-leaves.
  This is the main motivation for Definition~\ref{def:leaf-non-leaf}.
\end{remark}

\begin{remark}%
  \label{rem:leaf-non-leaf-morph}
  If $C$, $D$ are two diagrams, $\Theta \colon C \to D$ is a morphism, $Y \subseteq \leafs^C$ is a subset and we set $Y' = (\alpha^\Theta)^{-1}(Y)$, then the same value $\Theta$ defines a morphism $C(Y \leadsto \nonleafs) \to D(Y' \leadsto \nonleafs)$.
\end{remark}

Finally, we state a useful but straightforward ``patching'' or ``glueing'' lemma for constructing morphisms into a diagram.

\begin{lemma}%
  \label{lem:patching}
  Let $C$ and $D$ be two diagrams.
  Suppose $X^D = Y_1 \cup \dots \cup Y_k$ is a cover of the vertices of $D$ which furthermore covers the edges $E^D$: i.e., for every edge $xy \in E^D$ there is some $i \in [k]$ such that $x,y \in Y_i$.

  Then there is a one-to-one correspondence between:---
  \begin{enumerate}[label=(\alph*)]
    \item morphisms $\Theta \colon C \to D$;
    \item collections of morphisms $\Theta_1,\dots,\Theta_k$, $\Theta_i \colon C \to D[Y_i]$, such that
      $\alpha^{\Theta_i}(x) = \alpha^{\Theta_j}(x)$ and $\theta^{\Theta_i}_x = \theta^{\Theta_j}_x$ whenever $x \in Y_i \cap Y_j$.
  \end{enumerate}
\end{lemma}
\begin{proof}%
  Given $\Theta$ we can define $\Theta_i$ by composing with the restriction map $D \to D[Y_i]$ from Remark~\ref{rem:restriction}.

  Conversely, given $\Theta_1,\dots,\Theta_k$ we can define a single morphism $\Theta$ by setting $\alpha^\Theta(x) = \alpha^{\Theta_i}(x)$ and $\theta^{\Theta}_x = \theta^{\Theta_i}_x$ for any $i$ with $x \in Y_i$: by hypothesis this does not depend on the choice of $i$.

  For every edge $xy \in E^D$, the compatibility condition~\eqref{eq:diag-morph} holds for $\Theta$, by choosing $i$ such that $x,y \in Y_i$ and invoking~\eqref{eq:diag-morph} for $\Theta_i$.
\end{proof}

\subsection{Labelling conventions and pictorial representations}%
\label{sub:labelling}

A trivial but very frustrating problem when manipulating diagrams is handling index sets.
If we start with a reasonable linear datum or Cauchy--Schwarz diagram and apply ten Cauchy--Schwarz / self-joining steps in succession, we will have a diagram with thousands of vertices.
We need to be able to refer to each of them individually in a way that is both systematic and human-readable.

At the cost of layering yet more conventions, we present a labelling scheme that seems to be workable in practice.
A typical diagram we will encounter admits a top-level description as being several smaller diagrams glued together.  
If we label each of these top-level smaller diagrams, we can refer to vertices in a recursive, tree-like fashion by a tuple $i_1; i_2;\dots; i_k$ where $i_1$ names one of the top-level sub-diagrams, $i_2$ names a top-level sub-diagram of the sub-diagram $i_1$, and so on, until $i_k$ names an honest vertex of the sub-diagram $i_{k-1}$.

Two types of ambiguity are encouraged.  The top-level pieces need not be disjoint, so there may be multiple ways to refer to the same vertex.  Also, a single diagram may admit many natural labelling schemes, by breaking apart its structure in different ways or using different label names.

\begin{definition}%
  \label{def:labelling}
  A \emph{labelling structure} on a Cauchy--Schwarz diagram $D$ is a prefix-free\footnote{I.e., for distinct tuples $t_1, t_2 \in T$, $t_1$ is never a prefix of $t_2$.} collection $T$ of tuples, together with a surjective mapping $\imath \colon T \to X^D$.
  We denote elements of $T$ by sequences separated by semi-colons, $i_1;i_2;\dots;i_k$ (with or without parentheses).  

  Given a sequence $i_1;i_2;\dots;i_k$, we write $i_1;i_2;\dots;i_k;\ast$ to denote the (possibly empty) set of all elements of $T$ which have $i_1;i_2;\dots;i_k$ as a (not necessarily proper) prefix.
  We write $\imath(i_1;i_2;\dots;i_k)$ for the image of this set in $X^D$,\footnote{When $i_1;i_2;\dots;i_k \in T$, we have overloaded $\imath(i_1;i_2;\dots;i_k)$ to mean either a vertex $x \in X^D$ or the singleton $\{x\}$, but in practice this will not cause confusion.} or (by abuse of notation) for the sub-diagram $D[\imath(i_1;i_2;\dots;i_k)]$ of $D$ induced by this set (see Definition~\ref{def:subdiagram}).

  A \emph{labelled Cauchy--Schwarz diagram} is a diagram $D$ together with a labelling $(T,\imath)$ of $D$.

  A labelling structure on a linear datum $\Phi$, and a labelled linear datum, are defined analogously, replacing $X^D$ with $I^\Phi$.
\end{definition}

Equivalently, $T$ can be thought of as a rooted tree where every vertex is equipped with a labelling of its children.
Any leaf of this tree can be uniquely identified by the labels on the unique path from the root.
A prefix $i_1;i_2;\dots;i_k;\ast$ names an internal vertex of the tree, or the sub-tree rooted there.

This notation is more-or-less backwards-compatible with our existing scheme for labelling left-hand and right-hand copies of a joining $D_0 +_\beta D_1$, except using semicolons instead of tuples.
\begin{convention}%
  \label{convention:joining-labelling}
  If $D_0$ and $D_1$ are two diagrams with labellings $(T_0,\imath_0)$, $(T_1,\imath_1)$, then 
  we write $D_0 +_\beta^{\mathfrak{L},\mathfrak{R}} D_1$ to denote the diagram $D_0 +_\beta D_1$ together with the labelling
  \[
    T = \bigl\{ (\mathfrak{L};t_0)  \colon t_0 \in T_0 \bigr\} \cup \big\{ (\mathfrak{R};t_1) \colon t_1 \in T_1 \big\}
  \]
  with $\imath(\mathfrak{L};t_0) = ({L}, \imath_0(t_0))$ and $\imath(\mathfrak{R};t_1) = ({R}, \imath_1(t_1))$.
  In other words, the diagram has two top-level labels $\mathfrak{L}$ and $\mathfrak{R}$ corresponding to the sub-diagrams
  \[
    \bigl\{ ({L}, x_0) \colon x_0 \in X^{D_0} \bigr\},\ \big\{ ({R}, x_1) \colon x_1 \in X^{D_1} \big\} \subseteq X^D.
  \]
  Here $\mathfrak{L}$ and $\mathfrak{R}$ are symbols of the user's choice.
  If we write just $D_0 +_\beta D_1$ instead of $D_0 +_\beta^{\mathfrak{L},\mathfrak{R}} D_1$, these symbols default to $\mathfrak{L} = L$ and $\mathfrak{R}=R$.

  If both $\mathfrak{L};t$ and $\mathfrak{R};t$ name the same vertex in $X^D$ or $I^\Phi$, we continue to write $\imath(\#;t)$ to denote this vertex.

  Another possible labelling is denoted by $D_0 +_{\beta}^{,\mathfrak{R}} D_1$, where $\mathfrak{R}$ is not a top-level label of $T_0$.
  Informally, we create a new top-level label $\mathfrak{R}$ for the sub-diagram 
  $
    \big\{ (R, x_1) \colon x_1 \in X^{D_1} \big\} \subseteq X^D
  $
  and leave the labelling of 
  $
  \bigl\{ (L, x_0) \colon x_0 \in X^{D_0} \bigr\},
  $
  unchanged.
  Formally, $T$ is the disjoint union
  \[
    T = T_0 \cup \big\{ (\mathfrak{R};t_1) \colon t_1 \in T_1 \big\}
  \]
  and $\imath(t_0) = (L,\imath_0(t_0))$ and $\imath(\mathfrak{R};t_1) = (R,\imath_1(t_1))$ for $t_0 \in T_0$ or $t_1 \in T_1$.
  The labelling $D_0 +_{\beta}^{\mathfrak{L},} D_1$ is defined analogously.

  All the same conventions apply, with the obvious modifications, to a linear datum $\Phi$.
\end{convention}

In order to feed the kind of arguments given in this paper into a computer to be verified, much of the required effort is spent manipulating these labellings using formal operations.
Since the reader is not a computer, we can leave many of these manipulations implicit, or represent them pictorially, without causing confusion.
It is rarely necessary to define labellings $(T,\imath)$ formally, as we did above.

For this to work, we emphasize some further guidelines and symbology.
\begin{convention}%
  \label{conv:labelling-other}
  We make the following uses and abuses of notation.
  \begin{itemize}
    \item The symbol $\imath$ is omitted in practice: we just write $i_1;\dots;i_k$ to denote the vertex or set of vertices $\imath(i_1;\dots;i_k)$ this label represents.
    \item When we present a new interesting diagram, we will describe its top-level labels.
      The reader should assume that these labels are in important and canonical part of the definition and may be referenced later.
    \item In a graphical representation of a diagram, rectangular boxes \tikz[baseline=-3pt]{\node[rectangle,draw,inner sep=4pt,outer sep=0pt,scale=0.7] {$x$}} denote honest vertices $x \in X^D$, whereas rounded boxes such as \tikz[baseline=-3pt]{\node[rounded rectangle,draw,inner sep=4pt,outer sep=0pt,scale=0.7] {$X7 \,:\, C$}} indicate an entire copy of some smaller known diagram, here $C$, which has been assigned a top-level label, here $X7$.

      For those vertices of the sub-diagram $C$ which interact with the rest of the diagram (for example, because they are also contained in some other sub-diagram), we ``expose'' them as follows:
      \begin{center}
        \begin{tikzpicture}[
          defnode/.style={rectangle,draw,fill=lightgray,inner sep=2pt,outer sep=0pt,minimum size=10pt},
          gatelabel/.style={rectangle, inner sep=1pt, outer sep=0pt, fill=white,scale=0.7},
          subnode/.style={rounded rectangle, draw, inner sep=5pt, outer sep=0pt,minimum size=15pt}]
          \node[subnode,scale=0.8] (V) at (0,0) {$X7 : C$};
          \node[defnode,scale=0.4] (W) at (2,0) {};
          \draw (V) -- node[pos=0.2, gatelabel] {$4i2;\clubsuit$} (W);
        \end{tikzpicture}
      \end{center}
      to indicate that the vertex \tikz[defnode/.style={rectangle,fill=lightgray,draw,inner sep=2pt,outer sep=0pt,minimum size=4pt}]{\node[defnode] at (0,0) {}} is the one labelled (in this case) $4i2;\clubsuit$ within $X7$, and hence labelled $X7;4i2;\clubsuit$ overall.

    \item A visible vertex, i.e.\ those of shape \tikz[baseline=-3pt]{\node[rectangle,draw,inner sep=4pt,outer sep=0pt,scale=0.7] {$x$}} or \tikz[defnode/.style={rectangle,fill=lightgray,draw,inner sep=2pt,outer sep=0pt,minimum size=4pt}]{\node[defnode] at (0,0) {}}, is a leaf in $D$ if and only if it is colored green, as in \tikz[baseline=-3pt]{\node[rectangle,draw,fill=leafgreen,inner sep=4pt,outer sep=0pt,scale=0.7] {$x$}} or \tikz[defnode/.style={rectangle,draw,fill=leafgreen,inner sep=2pt,outer sep=0pt,minimum size=4pt}]{\node[defnode] at (0,0) {}}.
      A vertex which is only indicated implicitly by a sub-diagram \tikz[baseline=-3pt]{\node[rounded rectangle,draw,inner sep=4pt,outer sep=0pt,scale=0.7] {$X7 \,:\, C$}} is a leaf if and only if it is a leaf in $C$.
  \end{itemize}
\end{convention}

For example, if $C$ and $D$ are two diagrams, $a,b \in \leafs^C$ and $A,B \in \leafs^D$ are four vertices or vertex labels with $V^C_a = V^{D}_A$ and $V^C_b=V^{D}_B$, then the diagram $C +_{\substack{a \leftrightarrow A \\ b \leftrightarrow B}}^{\aleph,\beth} D$ could be shown graphically as
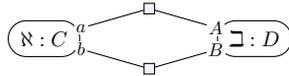
\begin{center}
  \begin{tikzpicture}[
    defnode/.style={rectangle,fill=lightgray,draw,inner sep=2pt,outer sep=0pt,minimum size=10pt},
    gatelabel/.style={rectangle, inner sep=1pt, outer sep=0pt, fill=white,scale=0.7},
    subnode/.style={rounded rectangle, draw, inner sep=5pt, outer sep=0pt,minimum size=15pt}]
    \node[subnode,scale=0.8] (V0) at (0,0) {$\aleph : C$};
    \node[defnode,scale=0.4] (Wa) at (1.4,0.4) {};
    \node[defnode,scale=0.4] (Wb) at (1.4,-0.4) {};
    \node[subnode,scale=0.8] (V1) at (2.8,0) {$\beth : D$};

    \draw (V0) -- node[pos=0.06, gatelabel] {$a$} (Wa);
    \draw (V0) -- node[pos=0.06, gatelabel] {$b$} (Wb);
    \draw (V1) -- node[pos=0.07, gatelabel] {$A$} (Wa);
    \draw (V1) -- node[pos=0.07, gatelabel] {$B$} (Wb);
  \end{tikzpicture}
\end{center}
where we have chosen top-level labels $\aleph$ for $C$ and $\beth$ for $D$.
The middle-top vertex is labelled both $\aleph;a$ and $\beth;A$, and the middle-bottom vertex is both $\aleph;b$ and $\beth;B$.

A useful property of this notation is that it allows us to denote diagrams precisely in a way that closely resembles our previous informal pictures of arithmetic circuits, as in Example~\ref{ex:diag-ex1}.

\section{Generalized convolutions, gates and circuits}%
\label{sec:gates}

The third and final layer of our Cauchy--Schwarz language, lying above both linear data and Cauchy--Schwarz diagrams, are what we call (almost interchangeably) \emph{gates} or \emph{circuits}.
These correspond to the heuristic notion of an arithmetic circuit presented in Section~\ref{sub:repr}.

To introduce them, we must first make two digressions in the next two subsections, outlining preliminary steps in the proofs of Theorem~\ref{thm:main} or Theorem~\ref{thm:baby-thm}.
The secondary purpose of these digressions is to explain how tensors and multilinear algebra come to be encoded by diagrams and linear data.  

\subsection{Cauchy--Schwarz complexity and generalized convolutions}%
\label{sub:gen-convolution}

We first recall the concept of the \emph{Cauchy--Schwarz complexity} of a system of linear forms from~\cite{gt-linear}, and extend it to linear data or diagrams.
The original statement was that a system $\Phi = (\phi_i)_{i \in [k]}$ has Cauchy--Schwarz complexity at most $s$ at $i$ if it is possible to partition the remaining indices
\[
  [k] \setminus \{i\} = S_1 \cup \dots \cup S_{s+1}
\]
such that $\phi_i$ does not lie in the subspace $\spn( \phi_j \,\colon\, j \in S_r)$ for any $r \in [s+1]$, as elements of $(\FF_p^d)^\ast$.
By duality and general linear algebra, this is equivalent to any of the following statements:---
\begin{itemize}
  \item $\{\phi_i\}^\perp \not\supseteq \{\phi_j \colon j \in S_r\}^\perp$ for any $r \in [s+1]$, as subspaces of $\FF_p^d$;
  \item we have $\phi_i \left( \bigcap_{j \in S_r} \ker \phi_j \right) = \FF_p$ for each $r \in [s+1]$;
  \item for each $r \in [s+1]$ there exists $v_r \in \FF_p^d$ such that $\phi_i(v_r) = 1$ but $\phi_j(v_r) = 0$ for all $j \in S_r$;
  \item for each $r \in [s+1]$ there exists a linear map $\mu_r \colon \FF_p \to \FF_p^d$ such that $\phi_i \circ \mu_r = \id_{\FF_p}$ but $\phi_j \circ \mu_r = 0$ for all $j \in S_r$.
\end{itemize}

We adapt the last statement to a definition for general linear data.
(The second statement would adapt to give an equivalent definition.)
\begin{definition}[Cauchy--Schwarz complexity]%
  \label{def:cs-complexity}
  Let $\Phi$ be a linear datum and $i \in I^\Phi$ and index.
  We write $s_{\cs}(\Phi, i)$, the \emph{Cauchy--Schwarz complexity} of $\Phi$ at $i$, for the smallest $s$ such that it is possible to partition the remaining indices
  \[
    I^\Phi \setminus \{i\} = S_1 \cup \dots \cup S_{s+1}
  \]
  into $s+1$ sets, such that for each $r \in [s+1]$ there exists a linear map $\mu_r \colon W^\Phi_i \to V^\Phi$ such that $\phi^\Phi_i \circ \mu_r = \id_{W_i^\Phi}$ but $\phi^{\Phi}_j \circ \mu_r = 0$ for each $j \in S_r$.

  If no such $s$ exists then we write $s_{\cs}(\Phi,i) = \infty$.
\end{definition}

\begin{remark}%
  \label{rem:def-cs-monotone}
  We should remark that the property in Definition~\ref{def:cs-complexity} is indeed monotone in $s$: that is, if a partition $S_1,\dots,S_{s+1}$ and maps $\mu_1,\dots,\mu_{s+1}$ exist for some $s \ge 0$ then they also exist for any $s' >s$.
  Indeed, we can set $S_{s+2},\dots,S_{s'+1} = \emptyset$ and $\mu_{s+2},\dots,\mu_{s'+1} = \mu_1$.

  Hence, we have $s_{\cs}(\Phi, i) \le t$ if, and only if, there exist $S_1,\dots,S_{t+1}$ and $\mu_1,\dots,\mu_{t+1}$ satisfying the conditions in Definition~\ref{def:cs-complexity}.
\end{remark}

\begin{remark}%
  \label{rem:cs-complex-equiv}
  It is possible, if a little eccentric at this point, to phrase this definition in terms of morphisms.
  Given a partition $S_1,\dots, S_{s+1}$ as in Definition~\ref{def:cs-complexity}, specifying maps $\mu_r$ is equivalent to specifying morphisms $\cM_r \colon \trivial(\{i\}, W_i) \to \Phi[\{i\} \cup S_r]$ which respect the vertex $i$ and which have $\alpha^{\cM_r}(j) = \zeroi$ for all $j \in S_r$.
  Here $\Phi[\{i\} \cup S_r]$ is the induced sub-datum, as in Definition~\ref{def:subdatum}.

  Indeed, setting $\theta^{\cM_r} = \mu_r$ and $\sigma_i^{\cM_r} = \id$, the conditions $\phi_i^\Phi \circ \mu_r = \id_{W_i}$ and $\phi_j^\Phi \circ \mu_r = 0$ for $j \in S_r$ are exactly the morphism condition~\eqref{eq:morph}.

  Writing $S_r = \{j_1,\dots,j_k\}$ this is shown graphically below.
  
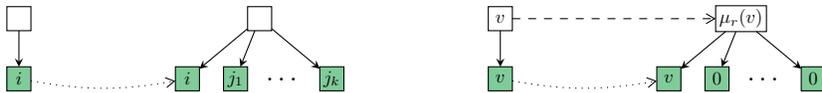
\vspace{0.5\baselineskip}
  \begin{center}
    \begin{tikzpicture}[
        boxnode/.style={rectangle,draw,inner sep=2pt,outer sep=0pt,minimum size=13pt,scale=0.7},
        leaf/.style={fill=leafgreen},
        scale=0.8]
      \begin{scope}[shift={(0,0)}]
        \node[boxnode]       (V)  at ( 0  , 0) {};
        \node[leaf,boxnode] (W) at (0,-1) {$i$};
        \draw[-stealth] (V) -- (W);

        \node[boxnode]       (Vp)  at (4 , 0) {};
        \node[leaf,boxnode] (Wp1) at ($(Vp) + (-1.2,-1)$) {$i$};
        \node[leaf,boxnode] (Wp2) at ($(Vp) + (-0.4,-1)$) {$j_1$};
        \node               (Wp3) at ($(Vp) + ( 0.4,-1)$) {$\dots$};
        \node[leaf,boxnode] (Wp4) at ($(Vp) + ( 1.2,-1)$) {$j_k$};
        \draw[-stealth] (Vp) -- (Wp1);
        \draw[-stealth] (Vp) -- (Wp2);
        \draw[-stealth] (Vp) -- (Wp4);

        \draw[dotted,-angle 60] (W) to[bend right = 10] (Wp1);
      \end{scope}
      \begin{scope}[shift={(8,0)}]
        \node[boxnode]       (V)  at ( 0  , 0) {$v$};
        \node[leaf,boxnode] (W) at (0,-1) {$v$};
        \draw[-stealth] (V) -- (W);

        \node[boxnode]       (Vp)  at (4 , 0) {$\mu_r(v)$};
        \node[leaf,boxnode] (Wp1) at ($(Vp) + (-1.2,-1)$) {$v$};
        \node[leaf,boxnode] (Wp2) at ($(Vp) + (-0.4,-1)$) {$0$};
        \node               (Wp3) at ($(Vp) + ( 0.4,-1)$) {$\dots$};
        \node[leaf,boxnode] (Wp4) at ($(Vp) + ( 1.2,-1)$) {$0$};
        \draw[-stealth] (Vp) -- (Wp1);
        \draw[-stealth] (Vp) -- (Wp2);
        \draw[-stealth] (Vp) -- (Wp4);

        \draw[dotted,-angle 60] (W) to[bend right = 10] (Wp1);
        \draw[dashed,-angle 60] (V) to (Vp);
      \end{scope}
    \end{tikzpicture}
  \end{center}
  This formulation has the advantage that it carries over naturally to diagrams.
\end{remark}

The point of Definition~\ref{def:cs-complexity} is that it readily implies a bound
\begin{equation}
  \label{eq:cs-bound}
  \lvert \Lambda_{\Phi}\bigl((f_j)_{j \in I}\bigr) \rvert \le \|f_i\|_{U^{s+1}}
\end{equation}
for any $n \ge 1$ and $1$-bounded functions $f_j \colon (W_j^{\Phi})^n \to \CC$, where $s=s_{\cs}(\Phi, i)$.
The proof is by $s+1$ applications of Cauchy--Schwarz.
Instead of recalling a direct proof of this, similar to Example~\ref{ex:diag-ex1}, we proceed in a roundabout way via the notion of a \emph{generalized convolution}.
This was introduced in~\cite{conlon-fox-zhao} to describe a class of functions: given $s \ge 1$, for an abelian group $U$ a function $f \colon U \to \CC$ was called a \emph{generalized convolution of degree $s$} if it has the form
\[
  f(y) = \EE_{\substack{x_1,\dots,x_{s+1} \in U \\ x_1+\cdots+x_{s+1} = y}} F_1(x_2,x_3,\dots,x_{s+1}) F_2(x_1,x_3,\dots,x_{s+1}) \dots F_{s+1}(x_1,\dots,x_s)
\]
for some functions $F_i \colon U^s \to \CC$, typically assumed to be $1$-bounded.
The case $s=1$ is the statement $f = F_1 \ast F_2$ where $\ast$ is convolution in the usual sense, hence the term ``generalized convolution''.

We may state this in terms of a linear datum.\footnote{More accurately, a generalized convolution is the ``dual function'' version of the linear datum we define, as discussed in Section~\ref{sub:true-summary}.
  These concepts are often discussed in terms of the associated cut-norm, in the sense of Footnote~\ref{foot:cut-norm}.
}
Typically we will have $U=\FF_p$ below, but for generality we allow any bounded-dimensional space.
\begin{definition}%
  \label{def:gen-conv}
  Let $s \ge 1$ and let $U$ be a finite-dimensional vector space over $\FF_p$.
  The \emph{generalized convolution datum} over $U$ of degree $s$ is denoted $\gc_s(U)$, where
  \begin{itemize}
    \item $V^{\gc_s(U)} = U^{s+1}$;
    \item $I^{\gc_s(U)} = \{ \triangle \} \cup \{1,\dots,s+1\}$;
    \item $W^{\gc_s(U)}_{\triangle} = U$ and $W_i = U^{s}$ for $i \in [s+1]$;
    \item $\phi^{\gc_s(U)}_\triangle(x_1,\dots,x_{s+1})=  x_1+\cdots+x_{s+1}$ and
      \[
        \phi^{\gc_s(U)}_i(x_1,\dots,x_{s+1}) = (x_1,\dots,x_{i-1},x_{i+1},\dots,x_{s+1})
      \]
      for $i \in [s+1]$ (i.e., the map omitting the $i$-th coordinate).
  \end{itemize}
  If the space $U$ is $\FF_p$, or otherwise clear from context, we abbreviate this to $\gc_s$.
\end{definition}

It turns out that~\eqref{eq:cs-bound} is a consequence of the following two steps.
\begin{lemma}%
  \label{lem:gc-cs}
  Let $\Phi$ be a linear datum and $i \in I^\Phi$ and index.
  Then $\Phi$ has Cauchy--Schwarz complexity at most $s$ at $i$ if and only if there exists a morphism of linear data $\Theta \colon \gc_s(W_i) \to \Phi$ respecting $i$ with $\alpha^\Theta(i)=\triangle$.
\end{lemma}
Note it is easy to see $\gc_s$ has Cauchy--Schwarz complexity at most $s$ at $\triangle$, by setting $S_r = \{r\}$
and $\mu_r(x) = (0, \dots, 0, x, 0, \dots, 0)$ with $x$ in the $r$-th position.
The lemma asserts that $\gc_s$ is moreover \emph{universal} for the property of having Cauchy--Schwarz complexity at most $s$.
It is also very similar in spirit to statements of Green and Tao about Cauchy--Schwarz complexity and ``normal forms'' of systems of linear forms~\cite[Definition~4.2]{gt-linear}.

\begin{lemma}%
  \label{lem:gc-bound}
  For any finite-dimensional vector space $W$, any $n \ge 1$ and any $1$-bounded functions $f_\triangle \colon W^n \to \CC$ and $f_i \colon (W^n)^s \to \CC$ for $i \in [s+1]$, we have
  \[
    \lvert \Lambda_{\gc_s(W)}(f_\triangle,f_1,\dots,f_{s+1}) \rvert \le \|f_\triangle\|_{U^{s+1}}.
  \]
\end{lemma}
Together, and with Proposition~\ref{prop:cs}, these imply~\eqref{eq:cs-bound}.

Lemma~\ref{lem:gc-bound} is somewhat standard: it appears as~\cite[Lemma 6]{tao-cut-blog}, or in many places in a more disguised form as a statement about hypergraph norms.
It proceeds by $s+1$ applications of Cauchy--Schwarz, and is in some ways the archetypal ``Cauchy--Schwarz $s+1$ times'' argument from which all others descend.

It almost, but not quite, corresponds to the statement in our language that
$\gc_s(W) \entails^{s+1}_\gamma U^{s+1}$ where $\gamma(\omega) = \left(\triangle, \sum_{i=1}^{s+1} \omega_i \bmod 2\right)$ for each $\omega \in \{0,1\}^{s+1}$: indeed, this would show
\[
  \lvert \Lambda_{\gc_s(W)}(f_\triangle, f_1,\dots, f_{s+1}) \rvert \le \bigl\lvert \Lambda_{U^{s+1}(W)}\bigl((g_\omega)_{\omega \in \{0,1\}^{s+1}}\bigr) \bigr\rvert^{2^{-(s+1)}}
\]
for some functions $g_{\omega}$ which are all translates of $f_\triangle$ or $\overline{f_{\triangle}}$ as appropriate.
These translates are not really there, and as in Example~\ref{ex:diag-ex1} we can improve this artificially weak statement with the Gowers--Cauchy--Schwarz inequality to obtain the claimed bound.
This ``$\entails$'' claim follows by $s+1$ Cauchy--Schwarz steps applied to $\gc_s$; specifically,
\[
  \CS(\{1\}),\ \CS(\{(L;2),(R;2)\}),\ \dots,\ \CS\bigl(\bigl\{ (i_1;i_2;\dots,i_s;s+1) \colon i_1,\dots,i_s \in \{L,R\} \bigr\}\bigr)
\]
using the notation in Section~\ref{sub:labelling}.
The resulting datum is exactly $U^{s+1}$, after removing some degeneracy (see Remark~\ref{rem:surj-degen}).

However, the formalism makes verifying the details harder rather than easier in this case, so although we rely on Lemma~\ref{lem:gc-bound} for our main theorem, we leave its formal proof to the references.

We now prove Lemma~\ref{lem:gc-cs}.

\begin{proof}[Proof of Lemma~\ref{lem:gc-cs}]
  By Remark~\ref{rem:def-cs-monotone}, the first condition is equivalent to the existence a partition $S_1 \cup \dots \cup S_{s+1}$ and a tuple of maps $(\mu_r)_{r \in [s+1]}$ as in Definition~\ref{def:cs-complexity}.
  In fact we will show that $(S_r)_{r \in [s+1]}$, $(\mu_r)_{r \in [s+1]}$ encode exactly the same data as a morphism $\gc_s(W_i) \to \Phi$.
  The content of the proof is therefore once again to unpack the definition of a morphism.

  First consider the shape $\alpha^\Phi$.
  We may reduce to the case where $\alpha(j) \ne \zeroi$ for all $j \in I^{\Phi}$, as otherwise modifying this to $\alpha(j) = 1$ and $\sigma_j = 0$ produces another valid morphism.
  Then, specifying a function $\alpha \colon I^{\Phi} \to I^{\gc_s}$ respecting $i \in I^\Phi$ with $\alpha(i)=\triangle$ is exactly the same thing as specifying a partition $I^{\Phi} \setminus \{i\} = S_1 \cup \dots \cup S_{s+1}$, by setting $S_r = \alpha^{-1}(r)$ and vice versa.

  Suppose first that linear maps $\theta^\Theta$ and $\sigma_j^\Theta$ defining a valid morphism are given.
  Then maps $\mu_r \colon W_i \to V^\Phi$ as in Definition~\ref{def:cs-complexity} may be defined as follows.
  For $r \in [s+1]$ define $\tau_r \colon W_i \to W_i^{s+1}$ by
  \[
    \tau_r(u) = (0, \dots, 0, u, 0, \dots, 0)
  \]
  where the $u$ on the right-hand side is in position $r$.
  Clearly $\phi^{\gc_s}_\triangle \circ \tau_r = \id_{W_i}$ and $\phi^{\gc_s}_r \circ \tau_r = 0$. 
  Then set $\mu_r = \theta^{\Theta} \circ \tau_r$.
  By~\eqref{eq:morph} we have
  \[
    \phi^{\Phi}_i \circ \mu_r = \bigl( \phi^{\Phi}_i \circ \theta^{\Theta} \bigr) \circ \tau_r = \bigl( \sigma^\Theta_i \circ \phi^{\gc_s}_{\triangle} \bigr) \circ \tau_r = \id_{W_i}
  \]
  (recalling $\sigma_i^\Theta=\id_{W_i}$) and similarly for $j \in S_r$,
  \[
    \phi^{\Phi_j} \circ \mu_r = \bigl( \phi^{\Phi}_j \circ \theta^{\Theta} \bigr) \circ \tau_r = \bigl( \sigma^\Theta_j \circ \phi^{\gc_s}_{r} \bigr) \circ \tau_r = 0
  \]
  as required in Definition~\ref{def:cs-complexity}.

  Conversely, given linear maps $\mu_r \colon W_i \to V^\Phi$ as in Definition~\ref{def:cs-complexity}, we may define $\Theta$ by $\sigma_i^\Theta = \id_{W_i}$,
  \[
    \theta^\Theta(u_1,\dots,u_{s+1}) = \sum_{\ell=1}^{s+1} \mu_\ell(u_\ell)
  \]
  and for each $j \in S_r$,
  \[
    \sigma_j^\Theta(u_1,\dots,u_{r-1},u_{r+1},\dots,u_{s+1}) = \sum_{\ell \in [s+1] \setminus \{r\}} \phi_j^{\Phi} \bigl( \mu_\ell(u_\ell) \bigr).
  \]
  It suffices to check~\eqref{eq:morph} holds.
  Indeed,
  \[
    \bigl(\phi^{\Phi}_i \circ \theta^{\Theta}\bigr)(u_1,\dots,u_{s+1})
    = \sum_{\ell=1}^{s+1} \bigl(\phi^{\Phi}_i \circ \mu_\ell\bigr)(u_\ell)
    = \sum_{\ell=1}^{s+1} u_\ell
    = \bigl(\sigma_i \circ \phi^{\gc_s}_{\triangle}\bigr)(u_1,\dots,u_{s+1})
  \]
  and for $j \in S_r$,
  \begin{align*}
    \bigl(\phi^{\Phi}_j \circ \theta^{\Theta}\bigr)(u_1,\dots,u_{s+1})
    &= \sum_{\ell=1}^{s+1} \bigl(\phi^{\Phi}_j \circ \mu_\ell\bigr)(u_\ell)
    = \sum_{\ell \in [s+1] \setminus \{r\} }^{s+1} \bigl(\phi^{\Phi}_j \circ \mu_\ell\bigr)(u_\ell)  \\
    &= \bigl(\sigma_j \circ \phi^{\gc_s}_{r}\bigr)(u_1,\dots,u_{s+1})
  \end{align*}
  as required, since $\phi^{\Phi}_j \circ \mu_r = 0$.
\end{proof}

In fact it turns out to be easier to work consistently with generalized convolution data rather than Gowers norm data.
In particular our proof of Theorem~\ref{thm:main-ext} will proceed via the following fact (and Lemma~\ref{lem:gc-bound}).

\begin{theorem}%
  \label{thm:main-gc}
  Take any system of linear forms $\Phi = (\phi_1,\dots,\phi_k)$, $\phi_j \colon \FF_p^d \to \FF_p$.
  We may consider this as a linear datum with $I^\Phi=[k]$.
  
  For any index $i \in [k]$, if $s=s(\Phi,i) < \infty$ and $s(\Phi) \le s+1$ then
  \[
    \Phi \entails_\gamma^M \gc_s(\FF_p)
  \]
  where $\gamma(\triangle) = (i,0)$ or $(i,1)$ and $M \ll k^{3} \bigl(\log k + \log \log (10 L)\bigr)$.
\end{theorem}

We close by recording the situation for diagrams.
\begin{proposition}%
  \label{prop:diag-cs-complexity}
  Let $D$ be a diagram and $i \in \leafs^D$ an index.
  The following are equivalent:---
  \begin{enumerate}[label=(\roman*)]
    \item $\datum(D)$ has Cauchy--Schwarz complexity at most $s$ at $i$;
    \item there exists a diagram morphism $\Theta \colon \gc_s(V^D_i) \to D$ with $\alpha^{\Theta}(i) = \triangle$ respecting $i$;
    \item there exists a partition $\leafs^D \setminus \{i\} = S_1 \cup \dots \cup S_{s+1}$ and diagram morphisms $\cM_r \colon \trivial(\{i\}, V^D_i) \to D\bigl[\nonleafs \cup \{i\} \cup S_r\bigr]$ for $1 \le r \le s+1$, such that $\alpha^{\cM_r}(i) = i$, $\alpha^{\cM_r}(j) = \zeroi$ for $j \in S_r$, and $\cM_r$ respects $i$.
  \end{enumerate}
  We say \emph{$D$ has Cauchy--Schwarz complexity at most $s$} if any of these conditions hold, and write $s_{\cs}(D,i) = s_{\cs}(\datum(D),i)$.
\end{proposition}
\begin{proof}%
  By Lemma~\ref{lem:gc-cs}, (i) is equivalent to (ii) except with a linear datum morphism $\gc_s(V^D_i) \to \datum(D)$, which is equivalent to (ii) by Proposition~\ref{prop:adjunction}.
  On the other hand Remark~\ref{rem:cs-complex-equiv} asserts (i) is equivalent to (iii) except again using linear datum morphisms throughout, which is equivalent to (iii) by Proposition~\ref{prop:adjunction} and Proposition~\ref{prop:sub-sub}.
\end{proof}

\subsection{A reformulation of Theorem~\ref{thm:baby-thm}}%
\label{sub:baby-gc}

Our next leading question is the following.
If the Gowers norm bound in our main theorem Theorem~\ref{thm:main} is obtained via Theorem~\ref{thm:main-gc} and the datum $\gc_s$, what analogous statement / linear datum implies the Gowers norm-type bound in our model problem, Theorem~\ref{thm:baby-thm}?
As well as being important for Theorem~\ref{thm:baby-thm} itself, this question will lead us naturally the definition of gates and circuits below.

The answer is essentially to generalize the idea of Cauchy--Schwarz complexity from one index $i$ to two or more indices.
Let $\Phi$ be a linear datum and $i,j \in I^\Phi$ two indices with, for simplicity, $W^\Phi_i=W^\Phi_j=\FF_p$.
Let $\phi_{i,j}$ denote the map $V^\Phi \to \FF_p^2$ given by $v \mapsto (\phi_i(v), \phi_j(v))$.
Suppose $I^\Phi \setminus \{i,j\} = S_1 \cup S_2$ is a partition of the other indices.
By analogy with Cauchy--Schwarz complexity, we set all the functions $\phi_\ell$ for $\ell \in S_r$ to zero and see what this does to $\phi_{i,j}$: that is, for each set $S_r$ we consider the subspace
\[
  U_r = \phi_{i,j} \left( \bigcap_{\ell \in S_r} \ker(\phi_\ell) \right) = \left\{ \bigl(\phi_i(v), \phi_j(v)\bigr) \colon v \in \bigcap_{\ell \in S_r} \ker(\phi_\ell)\right\} \subseteq \FF_p^2.
\]
For Cauchy--Schwarz complexity the two possibilities for such subspaces were $0$ or $\FF_p$, but here many one-dimensional subspaces $\{ (ax,bx) \colon x \in \FF_p \}$ are possible, for integer values $(a,b)$.
The subspace may also vary with $r$.

Finally, we also now need to consider
\[
  U_0 = \phi_{i,j}\bigl(V^\Phi\bigr) = \left\{ \bigl(\phi_i(v), \phi_j(v)\bigr) \colon v \in V^\Phi \right\} \subseteq \FF_p^2
\]
since this cannot necessarily be inferred from $U_1$ and $U_2$.

For Theorem~\ref{thm:baby-thm}, the goal turns out to be to have $U_0=\FF_p^2$ and
\begin{align*}
  U_1 &\supseteq \bigl\{ (ax, x) \colon x \in \FF_p \bigr\}, & 
  U_2 &\supseteq \bigl\{ (x, ax) \colon x \in \FF_p \bigr\}.
\end{align*}
As with Cauchy--Schwarz complexity, there are a number of equivalent ways to state this.
By linear algebra the following are all the same:---
\begin{enumerate}[label=(\alph*)]
  \item for $U_0$, $U_1$, $U_2$ defined as above, $U_0=\FF_p^2$, $U_1 \supseteq \{ (ax, x) \colon x \in \FF_p \}$ and $U_2 \supseteq \{ (x, ax) \colon x \in \FF_p \}$;
  \item there exist linear maps $\mu'_0 \colon \FF_p^2 \to V^\Phi$, $\mu'_1 \colon \{ (ax,x) \colon x \in \FF_p \} \to V^\Phi$ and $\mu'_2 \colon \{ (x,ax) \colon x \in \FF_p \} \to V^\Phi$ such that $\phi_{i,j} \circ \mu'_r = \id$ for each $r=0,1,2$, and
    $\phi_\ell \circ \mu'_r = 0$ for $r=1,2$ and $\ell \in S_r$;
  \item there exist linear maps $\mu_0 \colon \FF_p^2 \to V^\Phi$, $\mu_1 \colon \FF_p \to V^\Phi$ and $\mu_2 \colon \FF_p \to V^\Phi$ such that $\phi_{i,j}(\mu_0(u)) = u$ for all $u \in \FF_p^2$, $\phi_{i,j}(\mu_1(x)) = (ax, x)$, $\phi_{i,j}(\mu_2(x)) = (x, ax)$ for all $x \in \FF_p$, and $\phi_\ell \circ \mu_r = 0$ for $r=1,2$ and $\ell \in S_r$.
\end{enumerate}
Indeed, it is clear (b) implies (a), and conversely (a) implies (b) as every surjective linear map has a right inverse.
Also (b) and (c) are equivalent by composing with the isomorphisms $\FF_p \cong \{ (ax,x) \colon x \in \FF_p\}$ and $\FF_p \cong \{ (x,ax) \colon x \in \FF_p \}$.

Hence, we have the following analogue of Cauchy--Schwarz complexity applicable to Theorem~\ref{thm:baby-thm}.
\begin{lemma}%
  \label{lem:baby-cs-complexity}
  Suppose $a$ is an integer, $\Phi$ is a linear datum and $i,j \in I^\Phi$ are indices with $W_i=W_j=\FF_p$.
  Suppose also we have a partition
  \[
    I^\Phi \setminus \{ i,j \} = S_1 \cup S_2
  \]
  and linear maps $\mu_0$, $\mu_1$, $\mu_2$ obeying the conditions in \uppar{c} above.

  Then for any $n \ge 1$ and any tuple of $1$-bounded functions $(f_\ell)_{\ell \in I^\Phi} \colon (W_\ell^\Phi)^n \to \CC$, we have
  \[
    \lvert \Lambda_{\Phi}((f_\ell)_{\ell \in I^{\Phi}}) \rvert \le \left( \EE_{h,h' \in \FF_p^n} b_1(ah,h') b_2(h,ah') \right)^{1/4}
  \]
  where
  \begin{align*}
    b_1(h,h') &= \EE_{x \in \FF_p^n} f_i(x) \overline{f_i(x+h)} \overline{f_i(x+h')} f_i(x+h+h') \\
    b_2(h,h') &= \EE_{x \in \FF_p^n} f_j(x) \overline{f_j(x+h)} \overline{f_j(x+h')} f_j(x+h+h').
  \end{align*}
\end{lemma}

As in Section~\ref{sub:gen-convolution}, we may break the proof of this into two parts.
\begin{lemma}%
  \label{lem:gc-cs-baby}
  Let $a$ be an integer.
  Consider the datum $\Psi$ with $I^{\Psi} =[4]$, $W^\Psi_1=W^\Psi_2=\FF_p$, $W^\Psi_3=W^\Psi_4=\FF_p^3$, $V^\Psi=\FF_p^4$ and $\phi_1^\Psi(x,y,z,w)=x+az+w$, $\phi_2^\Psi(x,y,z,w)=y+z+aw$, $\phi_3^\Psi(x,y,z,w)=(x,y,w)$, $\phi_4^\Psi(x,y,z,w)=(x,y,z)$.

  Then the hypothesis of Lemma~\ref{lem:baby-cs-complexity} holds for $\Phi$, $i,j$ and $a$ if and only if there exists a morphism $\Theta \colon \Psi \to \Phi$ respecting $i,j$ and with $\alpha^\Theta(i)=1$, $\alpha^\Theta(j)=2$.
\end{lemma}

Again it is clear that $\Psi$ itself obeys the hypothesis of Lemma~\ref{lem:baby-cs-complexity}, by setting $S_1=\{3\}$ and $S_2=\{4\}$, and the lemma states that it is universal for this property.

\begin{lemma}%
  \label{lem:gc-bound-baby}
  Let $a$ be an integer and let $\Psi$ be the datum in Lemma~\ref{lem:gc-cs-baby}.
  Then for any $n \ge 1$ and $1$-bounded functions $f_1,f_2\colon \FF_p^n \to \CC$ and $f_3,f_4 \colon (\FF_p^3)^n \to \CC$, we have
  \[
    \lvert \Lambda_{\Psi}(f_1,f_2,f_3,f_4) \rvert \le \left( \EE_{h,h' \in \FF_p^n} b_1(ah,h') b_2(h,ah') \right)^{1/4}
  \]
  where $b_1,b_2$ are defined as in Lemma~\ref{lem:baby-cs-complexity} applied to $f_1,f_2$.
\end{lemma}

This lemma is again proved by a number of applications of Cauchy--Schwarz.
\begin{proof}[Proof of Lemma~\ref{lem:gc-bound-baby}]%
  The proof corresponds essentially to the steps $\CS(\{3\})$ and $\CS(\{(L;4), (R;4)\})$ in our language.
  However, to avoid the usual issues about keeping track of translates (see Example~\ref{ex:diag-ex1}, Lemma~\ref{lem:gc-bound}) we drop back to invoking Proposition~\ref{prop:cs} twice to deduce
  \[
    \lvert \Lambda_{\Psi}(f_1,f_2,f_3,f_4) \rvert \le \bigl\lvert \Lambda_{\Xi}\bigl(f_1,\overline{f_1}, \overline{f_1}, f_1, f_2, \overline{f_2}, \overline{f_2}, f_2\bigr) \bigr\rvert^{1/4}
  \]
  where $\Xi = (\Psi +_{\{3\}} \Psi) +_{\{(L;4), (R;4)\}} (\Psi +_{\{3\}} \Psi)$.
  Expanding the definitions,
  \[
    \Xi^V = \biggl\{ \bigl((x_{\omega},y_{\omega},z_{\omega},w_{\omega})_{\omega \in \{L,R\}^2}\bigr) \in \FF_p^{16} \colon
        \begin{aligned}[t]
        (x_{LL},y_{LL},w_{LL}) &= (x_{LR}, y_{LR}, w_{LR}), \\
        (x_{RL},y_{RL},w_{RL}) &= (x_{RR}, y_{RR}, w_{RR}), \\
        (x_{LL},y_{LL},z_{LL}) &= (x_{RL}, y_{RL}, z_{RL}), \\
      (x_{LR},y_{LR},z_{LR}) &= (x_{RR}, y_{RR}, z_{RR}) \ \ \smash{\biggr\}}.
  \end{aligned}
\]
  We may clearly change variables
  \begin{align*}
    x &= x_{LL} = x_{LR} = x_{RL} = x_{RR} &
    y &= y_{LL} = y_{LR} = y_{RL} = y_{RR} \\
    z_L &= z_{LL} = z_{RL} & z_R &= z_{LR} = z_{RR} \\
    w_L &= w_{LL} = w_{LR} & w_R &= w_{RL} = w_{RR}
  \end{align*}
  and then slightly less clearly change variables again to replace $z_R$ with $h=z_R-z_L$, $w_R$ with $h'=w_R-w_L$, $x$ with $x'=x+az_L+w_L$ and $y$ with $y'=y+z_L+a w_L$.
  Now $\Lambda_{\Xi}$ becomes
  \[
    \Lambda_{\Xi}\bigl(f_1,\overline{f_1}, \overline{f_1}, f_1, f_2, \overline{f_2}, \overline{f_2}, f_2\bigr) 
    = \EE_{\substack{x',y' \in \FF_p^n \\h,h' \in \FF_p^n \\z_L, w_L \in \FF_p^n}} \begin{aligned}[t]
    &f_1(x') \overline{f_1(x'\!+\!ah)} \overline{f_1(x'\!+\!h')} f_1(x'\!+\!ah\!+\!h') \\
    &f_2(y') \overline{f_2(y'\!+\!h)} \overline{f_2(y'\!+\!ah')} f_2(y'\!+\!h\!+\!ah')
  \end{aligned}
  \]
  which, after dropping the spurious $z_L$, $w_L$ averages and unwrapping the definitions of $b_1$ and $b_2$, is exactly what we wanted.
\end{proof}

We now prove Lemma~\ref{lem:gc-cs-baby}.
\begin{proof}[Proof of Lemma~\ref{lem:gc-cs-baby}]%
  As in Lemma~\ref{lem:gc-cs}, we show that the partition $S_1,S_2$ and linear maps $(\mu_0,\mu_1,\mu_2)$ encode the same data as a morphism $\Theta \colon \Psi \to \Phi$.

  Writing $\alpha=\alpha^\Theta$, as in Lemma~\ref{lem:gc-cs} we may ignore the possibility $\alpha^\Theta(\ell) = \zeroi$.
  Setting $S_1 = \alpha^{-1}(\{3\})$ and $S_2 = \alpha^{-1}(\{4\})$, specifying $\alpha$ and the partition $S_1,S_2$ are equivalent.

  Given $\Theta$, we may define $\mu_0(x,y) = \theta^\Theta(x,y,0,0)$, $\mu_1(z) = \theta^{\Theta}(0,0,z,0)$ and $\mu_2(w) = \theta^\Theta(0,0,0,w)$.
  Conversely, given $\mu_0,\mu_1,\mu_2$ we define $\theta^\Theta(x,y,z,w)=\mu_0(x,y)+\mu_1(z)+\mu_2(w)$ and 
  \begin{align*}
    \sigma_i=\sigma_j&=\id_{\FF_p}\\
    \forall \ell \in S_1:\ \ \sigma_\ell(x,y,w) &= \phi^\Phi_\ell\bigl(\mu_0(x,y) + \mu_2(w)\bigr) \\
    \forall \ell \in S_2:\ \ \sigma_\ell(x,y,z) &= \phi^\Phi_\ell\bigl(\mu_0(x,y) + \mu_1(z)\bigr).
  \end{align*}
  Under this dictionary the conditions in (c) and the compatibility relations~\eqref{eq:morph} are readily seen to be equivalent, as in Lemma~\ref{lem:gc-cs}.
\end{proof}

This reduces the task of proving Theorem~\ref{thm:baby-thm} to the following statement.
\begin{theorem}%
  \label{thm:baby-gc}
  Let $p>2$ be a prime and $a \in \ZZ$.
  Then $U^3 \entails^M_\gamma \Psi$ where $\Psi$ is the linear datum in Lemma~\ref{lem:gc-cs-baby} \uppar{depending on $a$}, 
  $M = 6 + \lfloor \log_2 \lfloor \log_2 (|a|-1) \rfloor \rfloor $ if $|a| \ge 4$ or $M=5$ if $|a|<4$, and $\gamma(1)=(000,0)$, $\gamma(2)=(000,1)$.
\end{theorem}

\begin{remark}%
  \label{rem:many-morphs}
  As in Remark~\ref{rem:cs-complex-equiv}, it is possible to reformulate condition (c) into a statement about several morphisms 
  $\cM_0$, $\cM_1$, $\cM_2$, rather than a single morphism as in Lemma~\ref{lem:gc-cs-baby}.

  Indeed, let $\Psi_0 = \trivial(\{1,2\})$, and let $\Psi_1$, $\Psi_2$ be the linear data with $I^{\Psi_r}=[2]$, $V^{\Psi_r} = W^{\Psi_r}_1=W^{\Psi_r}_2=\FF_p$ and
  \begin{align*}
    \phi^{\Psi_1}_1(t) &= a t & \phi^{\Psi_2}_1(t) &= t \\ 
    \phi^{\Psi_1}_2(t) &= t   & \phi^{\Psi_2}_2(t) &= a t.
  \end{align*}
  Then $\mu_0$ corresponds exactly to a morphism $\cM_0 \colon \Psi_0 \to \Phi[\{i,j\}]$ with shape $\alpha^{\cM_0}(i)=1$, $\alpha^{\cM_0}(j)=2$ and respecting $i,j$ (i.e., $\sigma_i=\sigma_j=\id_{\FF_p}$).
  Similarly, $\mu_1$ corresponds exactly to a morphism $\cM_1 \colon \Psi_1 \to \Phi[\{i,j\} \cup S_1]$ with shape $\alpha^{\cM_1}(i)=1$, $\alpha^{\cM_1}(j)=2$ and $\alpha^{\cM_1}(\ell)=\zeroi$ for $\ell \in S_1$, and again respecting $i,j$.
  The same holds symmetrically for $\mu_2$ and $\cM_2 \colon \Psi_2 \to \Phi[\{i,j\} \cup S_2]$.
    
  Indeed, in each case the morphism $\cM_r$ has been completely specified by these requirements except for the linear map $\theta^{\cM_r} \colon V^{\Psi_r} \to V^\Phi$, which we take to be $\mu_r$.
  Then the conditions in (c) are exactly the compatibility conditions~\eqref{eq:morph} for $\cM_0$, $\cM_1$, $\cM_2$ to be morphisms.
\end{remark}

We summarize the key lessons from this discussion.
\begin{itemize}
  \item Doing multilinear algebra calculations using linear data involves statements that generalize the traditional definition of Cauchy--Schwarz complexity: that is, we divide all indices $I^\Phi \setminus J$ into $r$ classes $S_1,\dots,S_r$, where $r$ is the number of tensor modes being considered, and see what happens to the remaining maps $\{ \phi^\Phi_i \colon i \in J\}$ when we set each class $\{\phi^\Phi_\ell \colon \ell \in S_r \}$ to zero.
  \item Positive statements take the form of saying these joint kernels $\bigcap_{\ell \in S_r} \ker(\phi^\Phi_\ell)$ are large.
  \item Such statements can be witnessed by the existence of morphisms $\cM_r \colon \Psi_r \to \Phi[J \cup S_r]$ of a certain shape, or equivalently by a single morphism $\Theta \colon \Psi \to \Phi$ where $\Psi$ is some combined datum of (loosely) generalized-convolution type.
\end{itemize}
Possibly the first point---that the parts $\{1,\dots,r\}$ should be interpreted as tensor modes---is only strongly hinted at this point, considering Lemma~\ref{lem:baby-cs-complexity} in the context of Section~\ref{sub:repr}.
This is not something we can justify precisely, but it should become clear as we construct more examples that this is what is going on.

So far we have phrased everything in the language of linear data.
However, any statement about morphisms has a clear analogue in the setting of diagrams (which is typically equivalent, by Proposition~\ref{prop:adjunction} and Proposition~\ref{prop:sub-sub}).
Combining these observations with the language of diagrams gives the notion of circuits and gates, which we now state precisely.

\subsection{Gates and circuits}%
\label{sub:circuits}

We now introduce the definitions of Cauchy--Schwarz \emph{gates}.
Since gates can be formed by joining together smaller gates, there is no logical need for a separate notion of ``circuit'' but we use the term informally.

\begin{definition}%
  \label{def:gate}
  A \emph{Cauchy--Schwarz gate} $G$ (or simply a \emph{gate}) is a diagram $D$ together with (i) a finite set $\cR$ of \emph{modes}, and (ii) a partition of its leaves $\leafs^D=\pins \cup \toggles$ into two classes called the \emph{pins} $\pins$ and the \emph{toggles} $\toggles$.

  By convention we assume $0 \notin \cR$.
\end{definition}

The modes $\cR$ should be thought of as tensor modes, as in the previous section.
The toggles $\toggles$ are the indices which we plan to partition into $|\cR|$ classes $\{ S_r \colon r \in \cR \}$ and deal with as in Sections~\ref{sub:gen-convolution} and~\ref{sub:baby-gc}.
The remaining indices, the pins, are the ones we are interested in, in the sense of $\{i\}$ in Section~\ref{sub:gen-convolution}, $\{i,j\}$ in Section~\ref{sub:baby-gc} or $J$ in the discussion above.
At the same time, the pins are the indices we can ``fuse'' or ``join by a wire'' to other pins on other gates, by performing Cauchy--Schwarz operations on the underlying diagram.

It is a crucial feature that a single kind of gate can have many possible valid partitions $\toggles = \bigcup_{r \in \cR} S_r$, which achieve different multilinear calculations but all use the same underlying diagram.
This is what was meant in Section~\ref{sub:repr} by saying the gates should be ``programmable''.
The different partitions, and morphisms to keep track of their properties, are therefore defined separately to the underlying gate.

\begin{definition}%
  \label{def:assignment}
  Given a gate $G=(D,\cR, \pins, \toggles)$, an \emph{assignment} $\cA$ of $G$ consists of the following data:---
  \begin{itemize}
    \item a partition $\toggles = \bigcup_{r \in \cR} S_r$ of the toggles, i.e., a function $S \colon \toggles \to \cR$;
    \item a collection of diagrams $(D_r)_{r \in \cR \cup \{0\}}$, with given bijections $\leafs^{D_r} \cong \pins$ for each $r \in \cR \cup \{0\}$, together with morphisms
      \[
        \cM_0 \colon D_0 \to D[\nonleafs \cup \cP]
      \]
      and for each $r \in \cR$,
      \[
        \cM_r \colon D_r \to D[\nonleafs \cup \cP \cup S_r]
      \]
      such that $\cM_r$ respects every vertex $v \in \pins$ for each $r \in \cR \cup \{0\}$, and such that $\alpha^{\cM_r}(\ell) = \zeroi$ for each $r \in \cR$ and $\ell \in S_r$.
  \end{itemize}
\end{definition}

As in Remark~\ref{rem:many-morphs}, we note that the shape $\alpha^{\cM_r}(x)$ of the morphisms $\cM_r$, and the linear maps $\theta_x^{\cM_r}$ for leaves $x \in \leafs^{D_r}$, are determined by definition.
Hence, to specify $\cM_r$ it suffices to describe the shape $\alpha^{\cM_r}(x)$ and maps $\theta_x^{\cM_r}$ for non-leaves $x \in \nonleafs^{D_r}$ (and verify the morphism condition~\eqref{eq:diag-morph}).
Typically the diagram $D_r$ has only one non-leaf (because it has the form $\diagram(\Phi_r)$ for some linear datum $\Phi_r$) so the choice of $\alpha^{\cM_r}$ is in fact completely forced.

We also note that once the diagrams $D_r$ and the partition $S$ have been specified, proving the existence of the morphisms $\cM_r$ is a mechanical linear algebra problem that could in principle be solved by a computer.
Moreover, logically we never care what the morphisms $\cM_r$ are provided we know they exist.
However, in practice, the construction of $\cM_r$ is the part of our arguments where we get to turn an informal ``circuit picture'' into a formal result, and so these proofs carry a lot of the guiding intuition and should not be completely ignored.

\begin{remark}%
  \label{rem:no-big-datum}
  In the spirit of Lemma~\ref{lem:gc-cs} and Lemma~\ref{lem:gc-cs-baby}, we could perhaps define an assignment with reference to a single ``universal'' diagram and morphism built from all of the smaller ones $(D_r)_{r \in \cR \cup \{0\}}$ and $(\cM_r)_{r \in \cR \cup \{0\}}$.
  At least in the case that $D_r = \diagram(\Psi_r)$ are diagrams corresponding to linear data, which they typically are, this can be done.
  (The case of general diagrams $D_r$ is less clear.)

  We do not pursue this because (i) there does not seem to be any advantage to doing so, and (ii) it is just more complicated than working with the component objects $D_r$, $\cM_r$.
  At the very end of the calculation, we can apply Lemma~\ref{lem:gc-cs} or Lemma~\ref{lem:gc-cs-baby} to build a combined object.
\end{remark}

Most of our remaining proofs will involve (i) defining a gate, possibly with reference to smaller gates; (ii) proving an $\entails$ assertion that it is possible to ``build'' that gate starting from some simple diagram, and (iii) defining and verifying various assignments of that gate.
This process is best illustrated by example, so we now put this into practice for Theorem~\ref{thm:baby-thm}.

\section{A proof of Theorem~\ref{thm:baby-thm}}%
\label{sec:baby-thm}

We now begin the task of actually constructing useful gates that prove interesting things.
Our first target is to prove Theorem~\ref{thm:baby-thm}, via Theorem~\ref{thm:baby-gc}.

\subsection{The AggreGate}%
\label{sub:aggre}

The basic building block of all other gates is called the \emph{AggreGate}.
It is denoted $\aggregate_s$, or $\agg_s$ for short, for some parameter $s \ge 1$.
It may be thought of as realizing the fundamental rule of tensor arithmetic (and variants of it): if $u+v=w$ in $V$ and $v_2,\dots,v_m$ are other vectors then
\[
  u \otimes v_2 \otimes \dots \otimes v_s + v \otimes v_2 \otimes \dots \otimes v_s = w \otimes v_2 \otimes \dots \otimes v_s.
\]
There are several variants of the AggreGate.
In all cases the underlying diagram is formed of four copies of the generalized convolution diagram $\gc_s$, as follows.\footnote{The intention here is that each of $00$, $01$, $10$, $11$ is a single top-level label; i.e., in the labelling tree $T$ the root has four children.
  This is distinct from naming these parts $(0;0)$, $(0;1)$ etc.: that would imply a labelling tree where the root has two children labelled $\{0,1\}$ each of which has two children labelled $\{0,1\}$. %chktex 12
Fortunately, this distinction is both superficial and completely irrelevant.}
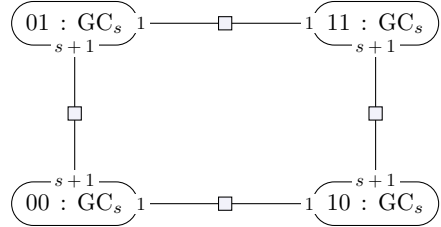
\begin{center}
  \begin{tikzpicture}[
      defnode/.style={rectangle,fill=lightgray,draw,inner sep=2pt,outer sep=0pt,minimum size=10pt},
      gatelabel/.style={rectangle, inner sep=2pt, outer sep=0pt, fill=white,scale=0.7},
      subnode/.style={rounded rectangle, draw, inner sep=5pt, outer sep=0pt,minimum size=15pt},
      scale=0.8
    ]

    \node[subnode,scale=0.9] (V00) at (0,0) {$00 \,:\, \gc_s$};
    \node[subnode,scale=0.9] (V01) at (0,3) {$01 \,:\, \gc_s$};
    \node[subnode,scale=0.9] (V10) at (5,0) {$10 \,:\, \gc_s$};
    \node[subnode,scale=0.9] (V11) at (5,3) {$11 \,:\, \gc_s$};

    \node[defnode,scale=0.5] (Ws-0) at ($(V00)!0.5!(V10)$) {};
    \node[defnode,scale=0.5] (Ws-1) at ($(V01)!0.5!(V11)$) {};
    \node[defnode,scale=0.5] (Wt0-) at ($(V00)!0.5!(V01)$) {};
    \node[defnode,scale=0.5] (Wt1-) at ($(V10)!0.5!(V11)$) {};

    \draw (V00) -- node[gatelabel, pos=0.04] {$1$} (Ws-0);
    \draw (V10) -- node[gatelabel, pos=0.04] {$1$} (Ws-0);
    \draw (V01) -- node[gatelabel, pos=0.04] {$1$} (Ws-1);
    \draw (V11) -- node[gatelabel, pos=0.04] {$1$} (Ws-1);

    \draw (V00) -- node[gatelabel, pos=0.02] {$s+1$} (Wt0-);
    \draw (V10) -- node[gatelabel, pos=0.02] {$s+1$} (Wt1-);
    \draw (V01) -- node[gatelabel, pos=0.02] {$s+1$} (Wt0-);
    \draw (V11) -- node[gatelabel, pos=0.02] {$s+1$} (Wt1-);
  \end{tikzpicture}
\end{center}
In particular, the leaves of $\aggregate_s$ are the vertices $\omega;\triangle$ and $\omega;r$ for $\omega \in \{0,1\}^2$ and $2 \le r \le s$, all of which appear only implicitly in this picture.

To allow the reader to cross-check the notation, we show the same diagram with the shorthand from Section~\ref{sub:labelling} expanded as Figure~\ref{fig:big-aggre-diag}.
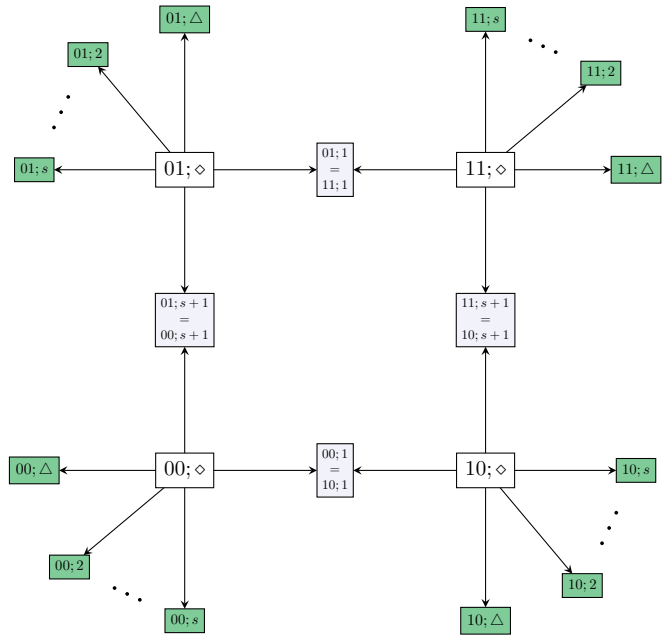
\begin{figure}[htbp]
  \begin{center}
    \begin{tikzpicture}[
        defnode/.style={rectangle,draw,inner sep=4pt,outer sep=0pt,minimum size=15pt},
        mydot/.style={circle,fill,inner sep=0.5pt},
        leaf/.style={defnode,fill=leafgreen},
        toggle/.style={defnode,fill=leafgreen}
      ]

      \node[defnode,scale=0.8] (V00) at (0, 0) {$00;\diamond$};
      \node[defnode,scale=0.8] (V01) at (0, 4) {$01;\diamond$};
      \node[defnode,scale=0.8] (V10) at (4, 0) {$10;\diamond$};
      \node[defnode,scale=0.8] (V11) at (4, 4) {$11;\diamond$};
      \node[defnode,fill=lightgray,scale=0.5,align=center] (Ws-0) at ($(V00)!0.5!(V10)$) {$00;1$   \\ $=$ \\ $10;1$};
      \node[defnode,fill=lightgray,scale=0.5,align=center] (Ws-1) at ($(V01)!0.5!(V11)$) {$01;1$   \\ $=$ \\ $11;1$};
      \node[defnode,fill=lightgray,scale=0.5,align=center] (Wt0-) at ($(V00)!0.5!(V01)$) {$01;s+1$ \\ $=$ \\ $00;s+1$};
      \node[defnode,fill=lightgray,scale=0.5,align=center] (Wt1-) at ($(V10)!0.5!(V11)$) {$11;s+1$ \\ $=$ \\ $10;s+1$};

      \node[leaf,scale=0.6]    (WX00) at ($(V00) + (180:2)$) {$00;\triangle$};
      \node[toggle,scale=0.55] (W100) at ($(V00) + (220:2)$) {$00;2$};
      \node[toggle,scale=0.55] (Wr00) at ($(V00) + (270:2)$) {$00;s$};

      \node[leaf,scale=0.6]    (WX01) at ($(V01) + (090:2)$) {$01;\triangle$};
      \node[toggle,scale=0.55] (W101) at ($(V01) + (130:2)$) {$01;2$};
      \node[toggle,scale=0.55] (Wr01) at ($(V01) + (180:2)$) {$01;s$};

      \node[leaf,scale=0.6]    (WX11) at ($(V11) + (000:2)$) {$11;\triangle$};
      \node[toggle,scale=0.55] (W111) at ($(V11) + (040:2)$) {$11;2$};
      \node[toggle,scale=0.55] (Wr11) at ($(V11) + (090:2)$) {$11;s$};

      \node[leaf,scale=0.6]    (WX10) at ($(V10) + (270:2)$) {$10;\triangle$};
      \node[toggle,scale=0.55] (W110) at ($(V10) + (310:2)$) {$10;2$};
      \node[toggle,scale=0.55] (Wr10) at ($(V10) + (000:2)$) {$10;s$};

      \draw[-stealth] (V00) -- (Ws-0);
      \draw[-stealth] (V10) -- (Ws-0);
      \draw[-stealth] (V01) -- (Ws-1);
      \draw[-stealth] (V11) -- (Ws-1);

      \draw[-stealth] (V00) -- (Wt0-);
      \draw[-stealth] (V10) -- (Wt1-);
      \draw[-stealth] (V01) -- (Wt0-);
      \draw[-stealth] (V11) -- (Wt1-);

      \draw[-stealth] (V00) -- (WX00);
      \draw[-stealth] (V00) -- (W100);
      \draw[-stealth] (V00) -- (Wr00);

      \draw[-stealth] (V01) -- (WX01);
      \draw[-stealth] (V01) -- (W101);
      \draw[-stealth] (V01) -- (Wr01);

      \draw[-stealth] (V10) -- (WX10);
      \draw[-stealth] (V10) -- (W110);
      \draw[-stealth] (V10) -- (Wr10);

      \draw[-stealth] (V11) -- (WX11);
      \draw[-stealth] (V11) -- (W111);
      \draw[-stealth] (V11) -- (Wr11);

      \path (W100) -- node[mydot, pos=0.333] {} node [mydot] {} node[mydot, pos=0.667] {} (Wr00);
      \path (W101) -- node[mydot, pos=0.333] {} node [mydot] {} node[mydot, pos=0.667] {} (Wr01);
      \path (W110) -- node[mydot, pos=0.333] {} node [mydot] {} node[mydot, pos=0.667] {} (Wr10);
      \path (W111) -- node[mydot, pos=0.333] {} node [mydot] {} node[mydot, pos=0.667] {} (Wr11);
    \end{tikzpicture}
  \end{center}
  \caption{The diagram $\agg_{s}$, shown in full.}%
  \label{fig:big-aggre-diag}
\end{figure}

\begin{definition}%
  \label{def:aggre}
  The diagram $\aggregate_s$, with standard top-level labels $00$, $01$, $10$, $11$, is the one shown by the picture above (or by Figure~\ref{fig:big-aggre-diag}).
\end{definition}

We observe that it is indeed possible to build this underlying diagram, starting with a generalized convolution diagram.

\begin{lemma}%
  \label{lem:gc-aggre}
  Let $s \ge 1$ and set $\gamma(00;\triangle)=\gamma(11;\triangle) = (\triangle,0)$ and $\gamma(01;\triangle)=\gamma(10;\triangle)=(\triangle,1)$.
  Then $\gc_s \entails^2_{\gamma} \aggregate_s$.
\end{lemma}
\begin{proof}%
  Starting with $\gc_s$, we apply $\CS(\{s+1\})$ followed by $\CS(\{(L;1),\, (R;1)\})$.
  Up to relabelling $L;L \leadsto 00$, $L;R \leadsto 01$, etc., the resulting datum is exactly $\agg_s$.
\end{proof}

We now describe the gate structure.
There are a few variants.
\begin{definition}%
  \label{def:aggregate}
  Let $P \subseteq \{00,01,10,11\}$ be a subset of size $2$, $3$ or $4$.
  Then the gate $\agg_s^P$ consists of the diagram $\agg_s$, the set of modes $\cR = [s]$, and the set of pins
  \[
    \cP = \bigl\{ (p; \triangle) \colon p \in P \bigr\}.
  \]
  All other leaves of $\agg_s$ are toggles.

  By convention we also write just $\agg_s$ for $\agg_s^{\{00,01,10,11\}}$, and $\agg_s^{\overline{\omega}}$ or $\agg_s^{\overline{\omega_1},\overline{\omega_2}}$ for $\agg_s^P$ where $P = \{00,01,10,11\} \setminus \{\omega\}$ or $\{00,01,10,11\} \setminus \{\omega_1,\omega_2\}$ respectively (i.e., with $\overline{\omega}$ denoting indices $\omega;\triangle$ which are \emph{not} pins).
\end{definition}

We now describe several key assignments of each of these variants.
First we introduce some further standard linear data.

\begin{definition}%
  \label{def:sum}
  The linear datum $\opsum$ is a synonym for $U^2$: that is, $I=\{0,1\}^2$, $W_{ij}=\FF_p$ for $i,j \in \{0,1\}$, $V=\FF_p^3$ and $\phi_{ij}(x,a,b) = x+ia+jb$.

  For $P \subseteq \{0,1\}^2$, the linear datum $\opsum(P)$ is the co-restriction $\opsum\langle \{0,1\}^2 \setminus P \rangle$.
  Explicitly, we have $I =P$, $W_{ij} = \FF_p$ for $ij \in P$,
  \[
    V = \bigl\{ (x,a,b) \in \FF_p^3 \colon x+i a + jb = 0 \ \forall ij \in \{0,1\}^2 \setminus P \bigr\}
  \]
  and $\phi_{ij}(x,a,b) = x+ia+jb$.
\end{definition}

In $\opsum(P)$, we have
\[
  \biggl\{ (\phi_{ij}(v))_{ij \in P} \colon v \in V \biggr\}  = \left\{ (z_{ij})_{ij \in P} \in \FF_p^P \colon \sum_{ij \in P} (-1)^{i+j} z_{ij} = 0 \right\}.
\]
In other words, there is a single constraint that the $\pm 1$ sum of the leaf values is $0$.
Up to strong isomorphism, we could equivalently take $V$ to be this subspace of $\FF_p^P$ and $\phi_\omega$ to be the restrictions of the projection maps $\FF_p^P \to \FF_p$.

\begin{lemma}%
  \label{lem:sum-const-assignment}
  There is an assignment $\sumconst_1$ of the gate $G=\aggregate_s^P$ having diagrams $(D_r)_{r \in [s] \cup \{0\}}$ given by:---
  \begin{itemize}
    \item $D_0 = \opsum(P)$ if $P = \{0,1\}^2$, or $D_0 = \trivial(P)$ if $|P|=3$ or $|P|=2$;
    \item $D_1=\opsum(P)$; and 
    \item $D_2=\cdots=D_s=\const(P)$.
  \end{itemize}
  In all cases, we use the natural identification $\leafs^{D_r} = P \cong \cP^G = \{ \omega;\triangle \colon \omega \in P\}$.
\end{lemma}

In other words, we may partition the toggles into $[s]$ parts such that in part $1$ the gate behaves as a ``sum gate'', where the $\pm 1$ sum of the pins is $0$, and in the other $s-1$ modes as the ``const gate'' where all the pins are the same.

\begin{proof}%
  In all cases, the mapping $\toggles \to [s]$ is given by $\omega;r \mapsto r$ for $\omega \in \{0,1\}^2$ and $r \in \{2,\dots,s\}$, and $\omega; \triangle \mapsto 1$ for $\omega \in \{0,1\}^2 \setminus P$.

  It now suffices to describe the morphisms $\cM_0,\cM_1,\dots,\cM_s$ and verify they have the properties claimed.
  As on previous occasions, the proof may be summarized by some pictures in which we write $\FF_p$-values at the vertices of a diagram.

  We first describe $\cM_1$.
  When $P = \{0,1\}^2$, this also describes $\cM_0 = \cM_1$.
  As noted above, the shape $\alpha^{\cM_1}$ is already forced: we take $\alpha^{\cM_1}(\omega;\triangle)=\omega$ for each $\omega \in P$, $\alpha^{\cM_1}(\omega;\triangle) = \zeroi$ for $\omega \in \{0,1\}^2 \setminus  P$, and $\alpha^{\cM_1}(x) = \diamond$ for each non-leaf $x \in \nonleafs^{\agg_s}$.

  First suppose $P=\{0,1\}^2$.
  We can define the non-trivial maps $\theta^{\cM_1}_x$ as follows:
  \begin{align*}
    \theta^{\cM_1}_{00;\diamond} \!=\! \big((x,a,b)\mkern-3mu & \mapsto \mkern-2mu(x,0,\babydots,0,0) \big)&
    \theta^{\cM_1}_{00;1} \!=\!                                                2  
    \theta^{\cM_1}_{10;1}        \!=\! \big((x,a,b)\mkern-3mu & \mapsto \mkern-2mu(0,\babydots,0,0) \big)\\
    \theta^{\cM_1}_{01;\diamond} \!=\! \big((x,a,b)\mkern-3mu & \mapsto \mkern-2mu(x,0,\babydots,0,b) \big)&
    \theta^{\cM_1}_{01;1} \!=\!                                                2  
    \theta^{\cM_1}_{11;1}        \!=\! \big((x,a,b)\mkern-3mu & \mapsto \mkern-2mu(0,\babydots,0,b) \big)\\
    \theta^{\cM_1}_{10;\diamond} \!=\! \big((x,a,b)\mkern-3mu & \mapsto \mkern-2mu(x\!+\!a,0,\babydots,0,0) \big)&
    \theta^{\cM_1}_{00;s+1} \!=\!                                              2  
    \theta^{\cM_1}_{01;s+1}      \!=\! \big((x,a,b)\mkern-3mu & \mapsto \mkern-2mu(x,0,\babydots,0) \big)\\
    \theta^{\cM_1}_{11;\diamond} \!=\! \big((x,a,b)\mkern-3mu & \mapsto \mkern-2mu(x\!+\!a,0,\babydots,0,b) \big)&
    \theta^{\cM_1}_{10;s+1} \!=\!                                              2  
    \theta^{\cM_1}_{11;s+1}      \!=\! \big((x,a,b)\mkern-3mu & \mapsto \mkern-2mu(x\!+\!a,0,\babydots,0) \big).
  \end{align*}
  In pictorial terms this is shown in Figure~\ref{fig:sum}.
  \begin{figure}[htbp]
    \begin{center}
      \begin{tikzpicture}[
          defnode/.style={rectangle,draw,inner sep=4pt,outer sep=0pt,minimum size=15pt},
          mydot/.style={circle,fill,inner sep=0.5pt},
          leaf/.style={defnode,fill=leafgreen},
          toggle/.style={defnode,fill=leafgreen}
        ]

        \node[defnode,scale=0.6] (U) at (-5, 1) {$(x,a,b)$};
        \node[leaf,scale=0.6]    (UX00) at ($(U) + (225:1)$) {$x$};
        \node[leaf,scale=0.6]    (UX01) at ($(U) + (135:1)$) {$x+b$};
        \node[leaf,scale=0.6]    (UX10) at ($(U) + (315:1)$) {$x+a$};
        \node[leaf,scale=0.6]    (UX11) at ($(U) + (045:1)$) {$x+a+b$};

        \draw[-stealth] (U) -- (UX00);
        \draw[-stealth] (U) -- (UX01);
        \draw[-stealth] (U) -- (UX10);
        \draw[-stealth] (U) -- (UX11);

        \node[defnode,scale=0.6] (V00) at (0, 0) {$(x,0,\dots,0,0)$};
        \node[defnode,scale=0.6] (V01) at (0, 2) {$(x,0,\dots,0,b)$};
        \node[defnode,scale=0.6] (V10) at (3.5, 0) {$(x+a,0,\dots,0,0)$};
        \node[defnode,scale=0.6] (V11) at (3.5, 2) {$(x+a,0,\dots,0,b)$};
        \node[defnode,fill=lightgray,scale=0.5] (Ws-0) at ($(V00)!0.5!(V10)$) {$(0,\dots,0,0)$};
        \node[defnode,fill=lightgray,scale=0.5] (Ws-1) at ($(V01)!0.5!(V11)$) {$(0,\dots,0,b)$};
        \node[defnode,fill=lightgray,scale=0.5] (Wt0-) at ($(V00)!0.5!(V01)$) {$(x,0,\dots,0)$};
        \node[defnode,fill=lightgray,scale=0.5] (Wt1-) at ($(V10)!0.5!(V11)$) {$(x+a,0,\dots,0)$};

        \node[leaf,scale=0.6]    (WX00) at ($(V00) + (225:1)$) {$x$};
        \node[leaf,scale=0.6]    (WX01) at ($(V01) + (135:1)$) {$x+b$};
        \node[leaf,scale=0.6]    (WX10) at ($(V10) + (315:1)$) {$x+a$};
        \node[leaf,scale=0.6]    (WX11) at ($(V11) + (045:1)$) {$x+a+b$};

        \draw[-stealth] (V00) -- (Ws-0);
        \draw[-stealth] (V10) -- (Ws-0);
        \draw[-stealth] (V01) -- (Ws-1);
        \draw[-stealth] (V11) -- (Ws-1);

        \draw[-stealth] (V00) -- (Wt0-);
        \draw[-stealth] (V10) -- (Wt1-);
        \draw[-stealth] (V01) -- (Wt0-);
        \draw[-stealth] (V11) -- (Wt1-);

        \draw[-stealth] (V00) -- (WX00);
        \draw[-stealth] (V01) -- (WX01);
        \draw[-stealth] (V10) -- (WX10);
        \draw[-stealth] (V11) -- (WX11);

        \draw[dotted, -angle 60] (UX00) to[bend right=10] (WX00);
        \draw[dotted, -angle 60] (UX01) to[bend left=10] (WX01);
        \draw[dotted, -angle 60] (UX10) to[bend right=25] (WX10);
        \draw[dotted, -angle 60] (UX11) to[bend left=25] (WX11);

        \draw[dashed, -angle 60] (U) to[bend right=10]  (V00);
        \draw[dashed, -angle 60] (U) to[bend left=10]   (V01);
        \draw[dashed, -angle 60] (U) to[bend left=5] (V10);
        \draw[dashed, -angle 60] (U) to[bend right=5]  (V11);

        \draw[dashed, -angle 60] (U) to[bend left=6] (Ws-0);
        \draw[dashed, -angle 60] (U) to[bend right=6] (Ws-1);
        \draw[dashed, -angle 60] (U) to (Wt0-);
        \draw[dashed, -angle 60] (U) .. controls ++(0,2.5) .. ++(8, 2.5) .. controls ++(2, 0) .. ++(2,-1.5) .. controls ++(0,-1) ..  (Wt1-.east);
      \end{tikzpicture}
    \end{center}
    \caption{The morphism $\cM_1$ in the $\sumconst_1$ assignment of $\agg_s$.}%
    \label{fig:sum}
  \end{figure}

  The right-hand diagram is Figure~\ref{fig:big-aggre-diag} with the toggles $\omega;r$ for $r=2,\dots,s$ removed (since these do not lie in $\cP \cup S_1$), and the left-hand diagram is a copy of $\opsum(\{0,1\}^2)$.
  The dotted arrows show the identity maps on pins, and the dashed arrows are the remaining maps $\theta_x^{\cM_1}$.
  Recalling that the maps $\phi_1^{\gc_s}$ (appearing on the horizontal sides of the square) and $\phi_{s+1}^{\gc_s}$  (appearing on the vertical sides of the square) are the maps $\FF_p^{s+1} \to \FF_p^s$ omitting the first and last coordinates respectively, and that $\phi_{\triangle}^{\gc_s} \colon \FF_p^{s+1} \to \FF_p$ (appearing in the corners) takes the sum of the coordinates, it is clear by inspection that the morphism condition~\eqref{eq:diag-morph} holds.

  When $P \subseteq \{0,1\}^2$ has size $2$ or $3$, the only modification necessary to obtain $\cM_1$ from the morphism above is to ``set some of $x$, $x+a$, $x+b$, $x+a+b$ to zero'' in Figure~\ref{fig:sum}.
  Formally, we compose the morphism $\cM_1$ above for $P=\{0,1\}^2$ with the co-restriction map $\opsum(P) \to \opsum(\{0,1\}^2)$; see Definition~\ref{def:cores}, Remark~\ref{rem:cores} and Definition~\ref{def:compose-morph}.

  Similarly, if we apply Remark~\ref{rem:restriction-morphs} to the original morphism $\cM_1$ for $P=\{0,1\}^2$ above, we obtain a morphism $\opsum(\{0,1\}^2)[P] \to \agg_s[\nonleafs \cup \pins]$.
  It is straightforward to construct is a morphism $\trivial(P) \to \opsum(\{0,1\}^2)[P]$ respecting every leaf $i \in P$ (because any three of the linear forms $(x,a,b) \mapsto x$, $(x,a,b) \mapsto x+a$, $(x,a,b) \mapsto x+b$, $(x,a,b) \mapsto x+a+b$ are linearly independent).
  The composite of these two morphisms gives $\cM_0$ when $P$ has size $2$ or $3$.

  Alternatively, one can verify these cases directly using pictures or function definitions as above, and we leave this to the energetic reader.

  Now we consider $\cM_2,\cM_3,\dots,\cM_s$.
  For simplicity we consider $\cM_2$; the other cases are analogous.
  The shape has $\alpha^{\cM_2}(\omega;\triangle)=\omega$ for all $\omega \in P$, $\alpha^{\cM_2}(\omega;2) = \zeroi$ for all $\omega \in \{0,1\}^2$ and $\alpha^{\cM_2}(t) = \diamond$ for all non-leaves $t \in \nonleafs^{\agg_s}$.

  For $P=\{0,1\}^2$ again, the non-trivial maps are as follows.
  \begin{align*}
    \theta^{\cM_1}_{00;\diamond} = \big( x &\mapsto (0,x,0,\dots,0,0) \big)&
    \theta^{\cM_1}_{00;1}      \!=\! 
    \theta^{\cM_1}_{10;1}        = \big( x &\mapsto (x,0,\dots,0,0) \big)\\
    \theta^{\cM_1}_{01;\diamond} = \big( x &\mapsto (0,x,0,\dots,0,0) \big)&
    \theta^{\cM_1}_{01;1}      \!=\! 
    \theta^{\cM_1}_{11;1}        = \big( x &\mapsto (x,0,\dots,0,0) \big)\\
    \theta^{\cM_1}_{10;\diamond} = \big( x &\mapsto (0,x,0,\dots,0,0) \big) &
    \theta^{\cM_1}_{00;s+1}    \!=\! 
    \theta^{\cM_1}_{01;s+1}      = \big( x &\mapsto (0,x,0,\dots,0) \big)\\
    \theta^{\cM_1}_{11;\diamond} = \big( x &\mapsto (0,x,0,\dots,0,0) \big)&
    \theta^{\cM_1}_{10;s+1}    \!=\! 
    \theta^{\cM_1}_{11;s+1}      = \big( x &\mapsto (0,x,0,\dots,0) \big).
  \end{align*}
  Again we can express this as a picture, in Figure~\ref{fig:const}.
  \begin{figure}[htbp]
    \begin{center}
      \begin{tikzpicture}[
          defnode/.style={rectangle,draw,inner sep=4pt,outer sep=0pt,minimum size=15pt},
          mydot/.style={circle,fill,inner sep=0.5pt},
          leaf/.style={defnode,fill=leafgreen},
          toggle/.style={defnode,fill=leafgreen}
        ]

        \node[defnode,scale=0.8] (U) at (-5, 1) {$x$};
        \node[leaf,scale=0.6]    (UX00) at ($(U) + (225:1)$) {$x$};
        \node[leaf,scale=0.6]    (UX01) at ($(U) + (135:1)$) {$x$};
        \node[leaf,scale=0.6]    (UX10) at ($(U) + (315:1)$) {$x$};
        \node[leaf,scale=0.6]    (UX11) at ($(U) + (045:1)$) {$x$};

        \draw[-stealth] (U) -- (UX00);
        \draw[-stealth] (U) -- (UX01);
        \draw[-stealth] (U) -- (UX10);
        \draw[-stealth] (U) -- (UX11);

        \node[defnode,scale=0.6] (V00) at (0, 0) {$(0,x,0,\dots,0,0)$};
        \node[defnode,scale=0.6] (V01) at (0, 2) {$(0,x,0,\dots,0,0)$};
        \node[defnode,scale=0.6] (V10) at (3.5, 0) {$(0,x,0,\dots,0,0)$};
        \node[defnode,scale=0.6] (V11) at (3.5, 2) {$(0,x,0,\dots,0,0)$};
        \node[defnode,fill=lightgray,scale=0.5] (Ws-0) at ($(V00)!0.5!(V10)$) {$(x,0,\dots,0,0)$};
        \node[defnode,fill=lightgray,scale=0.5] (Ws-1) at ($(V01)!0.5!(V11)$) {$(x,0,\dots,0,0)$};
        \node[defnode,fill=lightgray,scale=0.5] (Wt0-) at ($(V00)!0.5!(V01)$) {$(0,x,0,\dots,0)$};
        \node[defnode,fill=lightgray,scale=0.5] (Wt1-) at ($(V10)!0.5!(V11)$) {$(0,x,0,\dots,0)$};

        \node[leaf,scale=0.6]    (WX00) at ($(V00) + (225:1)$) {$x$};
        \node[leaf,scale=0.6]    (WX01) at ($(V01) + (135:1)$) {$x$};
        \node[leaf,scale=0.6]    (WX10) at ($(V10) + (315:1)$) {$x$};
        \node[leaf,scale=0.6]    (WX11) at ($(V11) + (045:1)$) {$x$};

        \node[leaf,scale=0.6]    (W200) at ($(V00) + (315:1)$) {$(0,0,\dots,0,0)$};
        \node[leaf,scale=0.6]    (W201) at ($(V01) + (045:1)$) {$(0,0,\dots,0,0)$};
        \node[leaf,scale=0.6]    (W210) at ($(V10) + (225:1)$) {$(0,0,\dots,0,0)$};
        \node[leaf,scale=0.6]    (W211) at ($(V11) + (135:1)$) {$(0,0,\dots,0,0)$};

        \draw[-stealth] (V00) -- (Ws-0);
        \draw[-stealth] (V10) -- (Ws-0);
        \draw[-stealth] (V01) -- (Ws-1);
        \draw[-stealth] (V11) -- (Ws-1);

        \draw[-stealth] (V00) -- (Wt0-);
        \draw[-stealth] (V10) -- (Wt1-);
        \draw[-stealth] (V01) -- (Wt0-);
        \draw[-stealth] (V11) -- (Wt1-);

        \draw[-stealth] (V00) -- (WX00);
        \draw[-stealth] (V01) -- (WX01);
        \draw[-stealth] (V10) -- (WX10);
        \draw[-stealth] (V11) -- (WX11);
        \draw[-stealth] (V00) -- (W200);
        \draw[-stealth] (V01) -- (W201);
        \draw[-stealth] (V10) -- (W210);
        \draw[-stealth] (V11) -- (W211);

        \draw[dotted, -angle 60] (UX00) to[bend right=10] (WX00);
        \draw[dotted, -angle 60] (UX01) to[bend left=10] (WX01);
        \draw[dotted, -angle 60] (UX10) to[bend right=25] (WX10);
        \draw[dotted, -angle 60] (UX11) to[bend left=25] (WX11);

        \draw[dashed, -angle 60] (U) to[bend right=10]  (V00);
        \draw[dashed, -angle 60] (U) to[bend left=10]   (V01);
        \draw[dashed, -angle 60] (U) to[bend left=5] (V10);
        \draw[dashed, -angle 60] (U) to[bend right=5]  (V11);

        \draw[dashed, -angle 60] (U) to[bend left=6] (Ws-0);
        \draw[dashed, -angle 60] (U) to[bend right=6] (Ws-1);
        \draw[dashed, -angle 60] (U) to (Wt0-);
        \draw[dashed, -angle 60] (U) .. controls ++(0,2.5) .. ++(8, 2.5) .. controls ++(2, 0) .. ++(2,-1.5) .. controls ++(0,-1) ..  (Wt1-.east);
      \end{tikzpicture}
    \end{center}
    \caption{The morphism $\cM_2$ in the $\sumconst_1$ assignment of $\agg_s$.}%
    \label{fig:const}
  \end{figure}

  Now the left-hand diagram is $\const(\{0,1\}^2)$ and the right-hand diagram is $\agg_s$ with all the toggle vertices removed except for those of the form $\omega;2$.
  As before it is clear by inspection that the morphism condition~\eqref{eq:diag-morph} holds.

  When $P \ne \{0,1\}^2$, applying Remark~\ref{rem:restriction-morphs} to the morphism for $P=\{0,1\}^2$ we get a morphism $\const(\{0,1\}^2)[P] \to \agg_s[\nonleafs \cup \pins \cup S_2]$, and it is clear that $\const(\{0,1\}^2)[P] = \const(P)$, as required.
\end{proof}

The assignment $\sumconst_1$ has the ``sum'' taking place in the first tensor mode and the ``const'' in the other modes.
We can make any other mode $r \in \cR$ the distinguished ``sum'' mode by permuting indices.

\begin{corollary}%
  \label{cor:sumconst-2}
  For any $r \in [s]$ there exists an assignment $\sumconst_r$ of the gate $G=\agg_s^P$ such that:---
  \begin{itemize}
    \item $D_0=\opsum(\{0,1\}^2)$ if $P=\{0,1\}^2$ and $\trivial(P)$ otherwise;
    \item $D_r=\opsum(P)$; and
    \item $D_\ell = \const(P)$ for $\ell \in [s] \setminus \{r\}$;
  \end{itemize}
  with the natural identification of leaves $\leafs^{D_\ell} = P \cong \cP^G = \{\omega;\triangle \colon \omega \in P \}$.
\end{corollary}
\begin{proof}%
  Recall $\cR = [s]$.
  We can use the same partition $\toggles = \bigcup_{r \in \cR} S_r$, diagrams $(D_r)_{r \in \cR \cup \{0\}}$ and morphisms $(\cM_r)_{r \in \cR \cup \{0\}}$ from Lemma~\ref{lem:sum-const-assignment}, but then compose with the permutation $\cR \to \cR$ given by the transposition $1 \leftrightarrow r$.
\end{proof}

There is one further family of interesting assignment of the AggreGate, which we term $\opcross$.
This occurs when we want an AggreGate to ``do nothing'', as far as possible, and just wire its pins to be equal to each other in pairs.
It will correspond to the tensor identities
\[
  u \otimes w \otimes v_3 \otimes \dots \otimes v_s + u' \otimes w' \otimes v_3 \otimes \dots \otimes v_s =
  u \otimes w \otimes v_3 \otimes \dots \otimes v_s + u' \otimes w' \otimes v_3 \otimes \dots \otimes v_s
\]
and
\[
  u \otimes w \otimes v_3 \otimes \dots \otimes v_s + u' \otimes 0 \otimes v_3 \otimes \dots \otimes v_s =
  u \otimes w \otimes v_3 \otimes \dots \otimes v_s
\]
as well as permutations of these.
Clearly these identities are vacuous and sub-optimal, but it is necessary to have them as gate assignments nonetheless.

We first define some more auxiliary data.%
\footnote{Here ``$\crs$'' is short for ``Cross'' but the latter denotes an assignment whereas the former is a linear datum.}
\begin{definition}%
  \label{def:cross}
  Let $\tau, \tau'$ be the values $01, 10$ in some order, i.e., $\{\tau, \tau'\} = \{01,10\}$.
  The linear datum $\crs^{\tau}$ has $I=\{0,1\}^2$, $V=\FF_p^2$ and linear maps $\phi_{00}(x,y) = x$, $\phi_{11}(x,y) = y$, $\phi_{\tau}(x,y)=x$ and $\phi_{\tau'}(x,y)=y$.

  We also consider the $3$-leaf variants $\crs^{\tau}_{\strikeout{\eta}} = \crs^{\tau}\langle \eta \rangle$ and $\crs^{\tau}_{\overline{\eta}} = \crs^{\tau} [\{0,1\}^2 \setminus \{\eta\}]$, for any $\eta \in \{0,1\}^2$.
  The former has $\dim V = 1$, and is obtained form $\opcross^\tau$ by setting $x=0$ or $y=0$ (depending on the values of $\tau$ and $\eta$), and the latter has $\dim V = 2$ and is obtained by simply deleting the index $\eta$.
\end{definition}

The superscript $\tau$ means that the pair $00$ and $\tau$ are set to the same value, as are the parallel pair $\tau'$ and $11$.
Then the subscript $\strikeout{\eta}$ means ``the index $\eta$, and its twin, are forced to zero'', and $\overline{\eta}$ means ``the index $\eta$ is missing'' and its twin can take any value.

\begin{lemma}%
  \label{lem:cross}
  Suppose $s \ge 2$ and $\tau \in \{01,10\}$.
  There is an assignment $\opcross^{\tau}_{1,2}$ of the gate $\agg_s$ with diagrams:---
  \begin{itemize}
    \item $D_0 = \opsum(\{0,1\}^2)$,
    \item $D_1=D_2=\crs^{\tau}$,
    \item $D_3=\cdots=D_s=\const(\{0,1\}^2)$,
  \end{itemize}
  with the usual identification $\{0,1\}^2 \cong \leafs^{\agg_s}$.
  
  Now let $\eta \in \{0,1\}^2$ be any index and write $P = \{0,1\}^2 \setminus \{\eta\}$.
  Then there is an assignment $\opcross^{\tau,\overline{\eta}}_{1,2}$ of $\agg_s^{\overline{\eta}}$ with:---
  \begin{itemize}
    \item $D_0=\trivial(P)$,
    \item $D_1 = \crs^{\tau}_{\overline{\eta}}$,
    \item $D_2 = \crs^{\tau}_{\strikeout{\eta}}$,
    \item $D_3=\cdots=D_s = \const(P)$,
  \end{itemize}
  again with the canonical bijection $P \cong \leafs^{\agg_s^{\overline{\eta}}}$.
\end{lemma}
\begin{proof}%
  As in Definition~\ref{def:cross} we choose $\tau'$ to be whichever of $01,10$ is not equal to $\tau$, i.e., so that $\{\tau,\tau'\} = \{01,10\}$.

  We first consider $\opcross^{\tau}_{1,2}$.
  The partition is given by
  $S_1 = \{ (00;2), (\tau; 2) \}$,
  $S_2 = \{ (\tau';2), (11; 2) \}$ and
  \[
    S_r = \bigl\{ \omega;r \colon \omega \in \{0,1\}^2 \bigr\}
  \]
  for $3 \le r \le s$.
  The morphisms
  \[
   \cM_0 \colon \opsum(\{0,1\}^2) \to \agg_s[\nonleafs \cup \pins] 
  \]
  and 
  \[
    \cM_r \colon \const(\{0,1\}^2) \to \agg_s[\nonleafs \cup \pins \cup S_r] 
  \]
  for $3 \le r \le s$ are exactly those given in Lemma~\ref{lem:sum-const-assignment}, and the necessary properties were already shown there.

  We now specify $\cM_1$ and $\cM_2$.
  First consider the case $\tau=01$.
  As usual we are forced to take $\alpha^{\cM_1}(\omega;\triangle) = \omega$, $\alpha^{\cM_1}(00;2) = \alpha^{\cM_1}(01;2) = \zeroi$ and $\alpha^{\cM_1}(x) = \diamond$ for every non-leaf $x \in \nonleafs^{\agg_s}$.

  The non-trivial maps $\theta^{\cM_1}_x$ are given by
  \begin{align*}
    \theta^{\cM_1}_{00;\diamond}    \!=\! \bigl( (x,y) \mkern-3mu &\mapsto \mkern-2mu (0,x,0,\babydots,0,0) \bigr)&
    \theta^{\cM_1}_{00;1}           \!=\!                                          
    \theta^{\cM_1}_{10;1}           \!=\! \bigl( (x,y) \mkern-3mu &\mapsto \mkern-2mu (x,0,\babydots,0,0) \bigr)\\
    \theta^{\cM_1}_{01;\diamond}    \!=\! \bigl( (x,y) \mkern-3mu &\mapsto \mkern-2mu (0,x,0,\babydots,0,0) \bigr) &
    \theta^{\cM_1}_{01;1}           \!=\!                                          
    \theta^{\cM_1}_{11;1}           \!=\! \bigl( (x,y) \mkern-3mu &\mapsto \mkern-2mu (x,0,\babydots,0,0) \bigr)\\
    \theta^{\cM_1}_{10;\diamond}    \!=\! \bigl( (x,y) \mkern-3mu &\mapsto \mkern-2mu (y\!-\!x,x,0,\babydots,0,0) \bigr)&
    \theta^{\cM_1}_{00;s+1}         \!=\!                                          
    \theta^{\cM_1}_{01;s+1}         \!=\! \bigl( (x,y) \mkern-3mu &\mapsto \mkern-2mu (  0,x,0,\babydots,0) \bigr)\\
    \theta^{\cM_1}_{11;\diamond}    \!=\! \bigl( (x,y) \mkern-3mu &\mapsto \mkern-2mu (y\!-\!x,x,0,\babydots,0,0) \bigr)&
    \theta^{\cM_1}_{10;s+1}         \!=\!                                          
    \theta^{\cM_1}_{11;s+1}         \!=\! \bigl( (x,y) \mkern-3mu &\mapsto \mkern-2mu (y\!-\!x,x,0,\babydots,0) \bigr).
  \end{align*}
  We represent this pictorially in Figure~\ref{fig:cross-1}.
  \begin{figure}[htbp]
    \begin{center}
      \begin{tikzpicture}[
          defnode/.style={rectangle,draw,inner sep=4pt,outer sep=0pt,minimum size=15pt},
          mydot/.style={circle,fill,inner sep=0.5pt},
          leaf/.style={defnode,fill=leafgreen},
          toggle/.style={defnode,fill=leafgreen}
        ]

        \node[defnode,scale=0.8] (U) at (-5, 1) {$(x,y)$};
        \node[leaf,scale=0.6]    (UX00) at ($(U) + (225:1)$) {$x$};
        \node[leaf,scale=0.6]    (UX01) at ($(U) + (135:1)$) {$x$};
        \node[leaf,scale=0.6]    (UX10) at ($(U) + (315:1)$) {$y$};
        \node[leaf,scale=0.6]    (UX11) at ($(U) + (045:1)$) {$y$};

        \draw[-stealth] (U) -- (UX00);
        \draw[-stealth] (U) -- (UX01);
        \draw[-stealth] (U) -- (UX10);
        \draw[-stealth] (U) -- (UX11);

        \node[defnode,scale=0.6] (V00) at (0, 0) {$(0,x,0,\dots,0,0)$};
        \node[defnode,scale=0.6] (V01) at (0, 2) {$(0,x,0,\dots,0,0)$};
        \node[defnode,scale=0.6] (V10) at (3.5, 0) {$(y-x,x,0,\dots,0,0)$};
        \node[defnode,scale=0.6] (V11) at (3.5, 2) {$(y-x,x,0,\dots,0,0)$};
        \node[defnode,fill=lightgray,scale=0.5] (Ws-0) at ($(V00)!0.5!(V10)$) {$(x,\dots,0,0)$};
        \node[defnode,fill=lightgray,scale=0.5] (Ws-1) at ($(V01)!0.5!(V11)$) {$(x,\dots,0,0)$};
        \node[defnode,fill=lightgray,scale=0.5] (Wt0-) at ($(V00)!0.5!(V01)$) {$(0,x,\dots,0)$};
        \node[defnode,fill=lightgray,scale=0.5] (Wt1-) at ($(V10)!0.5!(V11)$) {$(y-x,x,\dots,0)$};

        \node[leaf,scale=0.6]    (WX00) at ($(V00) + (225:1)$) {$x$};
        \node[leaf,scale=0.6]    (WX01) at ($(V01) + (135:1)$) {$x$};
        \node[leaf,scale=0.6]    (WX10) at ($(V10) + (315:1)$) {$y$};
        \node[leaf,scale=0.6]    (WX11) at ($(V11) + (045:1)$) {$y$};

        \node[leaf,scale=0.6]    (W200) at ($(V00) + (315:1)$) {$(0,0,\dots,0,0)$};
        \node[leaf,scale=0.6]    (W201) at ($(V01) + (045:1)$) {$(0,0,\dots,0,0)$};

        \draw[-stealth] (V00) -- (Ws-0);
        \draw[-stealth] (V10) -- (Ws-0);
        \draw[-stealth] (V01) -- (Ws-1);
        \draw[-stealth] (V11) -- (Ws-1);

        \draw[-stealth] (V00) -- (Wt0-);
        \draw[-stealth] (V10) -- (Wt1-);
        \draw[-stealth] (V01) -- (Wt0-);
        \draw[-stealth] (V11) -- (Wt1-);

        \draw[-stealth] (V00) -- (WX00);
        \draw[-stealth] (V01) -- (WX01);
        \draw[-stealth] (V10) -- (WX10);
        \draw[-stealth] (V11) -- (WX11);
        \draw[-stealth] (V00) -- (W200);
        \draw[-stealth] (V01) -- (W201);

        \draw[dotted, -angle 60] (UX00) to[bend right=10] (WX00);
        \draw[dotted, -angle 60] (UX01) to[bend left=10] (WX01);
        \draw[dotted, -angle 60] (UX10) to[bend right=25] (WX10);
        \draw[dotted, -angle 60] (UX11) to[bend left=25] (WX11);

        \draw[dashed, -angle 60] (U) to[bend right=10]  (V00);
        \draw[dashed, -angle 60] (U) to[bend left=10]   (V01);
        \draw[dashed, -angle 60] (U) to[bend left=5] (V10);
        \draw[dashed, -angle 60] (U) to[bend right=5]  (V11);

        \draw[dashed, -angle 60] (U) to[bend left=6] (Ws-0);
        \draw[dashed, -angle 60] (U) to[bend right=6] (Ws-1);
        \draw[dashed, -angle 60] (U) to (Wt0-);
        \draw[dashed, -angle 60] (U) .. controls ++(0,2.5) .. ++(8, 2.5) .. controls ++(2, 0) .. ++(2,-1.5) .. controls ++(0,-1) ..  (Wt1-.east);
      \end{tikzpicture}
    \end{center}
    \caption{The morphism $\cM_1$ in the $\opcross^{01}$ assignment of $\agg_s$.}%
    \label{fig:cross-1}
  \end{figure}

  The maps $\theta^{\cM_2}_x$ are defined similarly:
  \begin{align*}
    \theta^{\cM_1}_{00;\diamond} \!=\! \bigl( (x,y) \mkern-3mu &\mapsto \mkern-2mu  (x\!-\!y,y,0,\babydots,0,0) \bigr)&
    \theta^{\cM_1}_{00;1}        \!=\!
    \theta^{\cM_1}_{10;1}        \!=\! \bigl( (x,y) \mkern-3mu &\mapsto \mkern-2mu  (y,0,\babydots,0,0) \bigr)\\
    \theta^{\cM_1}_{01;\diamond} \!=\! \bigl( (x,y) \mkern-3mu &\mapsto \mkern-2mu  (x\!-\!y,y,0,\babydots,0,0) \bigr)&
    \theta^{\cM_1}_{01;1}        \!=\!
    \theta^{\cM_1}_{11;1}        \!=\! \bigl( (x,y) \mkern-3mu &\mapsto \mkern-2mu  (y,0,\babydots,0,0) \bigr)\\
    \theta^{\cM_1}_{10;\diamond} \!=\! \bigl( (x,y) \mkern-3mu &\mapsto \mkern-2mu  (0,y,0,\babydots,0,0) \bigr)&
    \theta^{\cM_1}_{00;s+1}      \!=\!
    \theta^{\cM_1}_{01;s+1}      \!=\! \bigl( (x,y) \mkern-3mu &\mapsto \mkern-2mu  (x\!-\!y,y,0,\babydots,0) \bigr)\\
    \theta^{\cM_1}_{11;\diamond} \!=\! \bigl( (x,y) \mkern-3mu &\mapsto \mkern-2mu  (0,y,0,\babydots,0,0) \bigr)&
    \theta^{\cM_1}_{10;s+1}      \!=\!
    \theta^{\cM_1}_{11;s+1}      \!=\! \bigl( (x,y) \mkern-3mu &\mapsto \mkern-2mu  (0,y,0,\babydots,0)\bigr).
  \end{align*}
  and again we draw this as Figure~\ref{fig:cross-2}.
  \begin{figure}[htbp]
    \begin{center}
      \begin{tikzpicture}[
          defnode/.style={rectangle,draw,inner sep=4pt,outer sep=0pt,minimum size=15pt},
          mydot/.style={circle,fill,inner sep=0.5pt},
          leaf/.style={defnode,fill=leafgreen},
          toggle/.style={defnode,fill=leafgreen}
        ]

        \node[defnode,scale=0.8] (U) at (-5, 1) {$(x,y)$};
        \node[leaf,scale=0.6]    (UX00) at ($(U) + (225:1)$) {$x$};
        \node[leaf,scale=0.6]    (UX01) at ($(U) + (135:1)$) {$x$};
        \node[leaf,scale=0.6]    (UX10) at ($(U) + (315:1)$) {$y$};
        \node[leaf,scale=0.6]    (UX11) at ($(U) + (045:1)$) {$y$};

        \draw[-stealth] (U) -- (UX00);
        \draw[-stealth] (U) -- (UX01);
        \draw[-stealth] (U) -- (UX10);
        \draw[-stealth] (U) -- (UX11);

        \node[defnode,scale=0.6] (V00) at (0, 0) {$(x-y,y,0,\dots,0,0)$};
        \node[defnode,scale=0.6] (V01) at (0, 2) {$(x-y,y,0,\dots,0,0)$};
        \node[defnode,scale=0.6] (V10) at (3.5, 0) {$(0,y,0,\dots,0,0)$};
        \node[defnode,scale=0.6] (V11) at (3.5, 2) {$(0,y,0,\dots,0,0)$};
        \node[defnode,fill=lightgray,scale=0.5] (Ws-0) at ($(V00)!0.5!(V10)$) {$(y,\dots,0,0)$};
        \node[defnode,fill=lightgray,scale=0.5] (Ws-1) at ($(V01)!0.5!(V11)$) {$(y,\dots,0,0)$};
        \node[defnode,fill=lightgray,scale=0.5] (Wt0-) at ($(V00)!0.5!(V01)$) {$(x-y,y,\dots,0)$};
        \node[defnode,fill=lightgray,scale=0.5] (Wt1-) at ($(V10)!0.5!(V11)$) {$(0,y,\dots,0)$};

        \node[leaf,scale=0.6]    (WX00) at ($(V00) + (225:1)$) {$x$};
        \node[leaf,scale=0.6]    (WX01) at ($(V01) + (135:1)$) {$x$};
        \node[leaf,scale=0.6]    (WX10) at ($(V10) + (315:1)$) {$y$};
        \node[leaf,scale=0.6]    (WX11) at ($(V11) + (045:1)$) {$y$};

        \node[leaf,scale=0.6]    (W210) at ($(V10) + (225:1)$) {$(0,0,\dots,0,0)$};
        \node[leaf,scale=0.6]    (W211) at ($(V11) + (135:1)$) {$(0,0,\dots,0,0)$};

        \draw[-stealth] (V00) -- (Ws-0);
        \draw[-stealth] (V10) -- (Ws-0);
        \draw[-stealth] (V01) -- (Ws-1);
        \draw[-stealth] (V11) -- (Ws-1);

        \draw[-stealth] (V00) -- (Wt0-);
        \draw[-stealth] (V10) -- (Wt1-);
        \draw[-stealth] (V01) -- (Wt0-);
        \draw[-stealth] (V11) -- (Wt1-);

        \draw[-stealth] (V00) -- (WX00);
        \draw[-stealth] (V01) -- (WX01);
        \draw[-stealth] (V10) -- (WX10);
        \draw[-stealth] (V11) -- (WX11);
        \draw[-stealth] (V10) -- (W210);
        \draw[-stealth] (V11) -- (W211);

        \draw[dotted, -angle 60] (UX00) to[bend right=10] (WX00);
        \draw[dotted, -angle 60] (UX01) to[bend left=10] (WX01);
        \draw[dotted, -angle 60] (UX10) to[bend right=25] (WX10);
        \draw[dotted, -angle 60] (UX11) to[bend left=25] (WX11);

        \draw[dashed, -angle 60] (U) to[bend right=10]  (V00);
        \draw[dashed, -angle 60] (U) to[bend left=10]   (V01);
        \draw[dashed, -angle 60] (U) to[bend left=5] (V10);
        \draw[dashed, -angle 60] (U) to[bend right=5]  (V11);

        \draw[dashed, -angle 60] (U) to[bend left=6] (Ws-0);
        \draw[dashed, -angle 60] (U) to[bend right=6] (Ws-1);
        \draw[dashed, -angle 60] (U) to (Wt0-);
        \draw[dashed, -angle 60] (U) .. controls ++(0,2.5) .. ++(8, 2.5) .. controls ++(2, 0) .. ++(2,-1.5) .. controls ++(0,-1) ..  (Wt1-.east);
      \end{tikzpicture}
    \end{center}
    \caption{The morphism $\cM_2$ in the $\opcross^{01}$ assignment of $\agg_s$.}%
    \label{fig:cross-2}
  \end{figure}

  To obtain the case $\tau=10$, we reflect these pictures around the SW--NE diagonal, and exchange the roles of the indices $1$ and $s+1$ in $\FF_p^{s+1}$.
  The details are completely analogous and we omit them.

  We briefly describe the modifications needed to obtain $\opcross^{\tau,\overline{\eta}}_{1,2}$.
  We revise the partition by taking $S_2 = \{ (\tau';2),\ (11;2),\ \eta\}$, with the other parts staying the same.
  The morphisms $\cM_0$ and $\cM_3,\dots,\cM_s$ are again the same as in Lemma~\ref{lem:cross}.
  For $\cM_1$, we apply Remark~\ref{rem:restriction-morphs} to the morphism above to obtain one with domain $\crs^\tau[P]$, where $P = \{0,1\}^2 \setminus \{\eta\}$.
  For $\cM_2$, as before we compose the map above with the co-restriction map $\crs^\tau \langle \eta \rangle \to \crs^\tau$.
\end{proof}

Again, it is clear we can permute the indices $[s]$ so that the roles of the special tensor modes $(1,2)$ are played by some other pair of indices $(i,j)$.
\begin{corollary}%
  \label{cor:cross}
  For any pair of distinct indices $i,j \in [s]$ there are assignments $\opcross^\tau_{i,j}$ and $\opcross^{\tau,\overline{\eta}}_{i,j}$ analogous to those in Lemma~\ref{lem:cross} but with $D_i$, $D_j$ swapped with $D_1$, $D_2$ respectively.
\end{corollary}

\subsection{The individual AGate}%
\label{sub:agate}

We now describe the gate which is the building block of the circuit needed to prove Theorem~\ref{thm:baby-thm}.

This subsection and the next are essentially about realizing the sketch in Section~\ref{sub:repr}: that is, building a double-and-add binary multiplier to realize the identity $B(ax,y)=B(x,ay)$ for an integer $a$.
The ``individual AGate'' or ``IndAGate'' discussed here performs one step of the ``double and add'' algorithm: it performs one multiplication by two and optionally adds one register to another.
Figure~\ref{fig:bilinear} essentially contains three of these basic units.
Chaining several IndAGates together forms a ``large AGate'' or ``AGate'', which implements the full circuit sketched in Figure~\ref{fig:bilinear}, and will be discussed in the next section.

We will build the IndAGate gate out of AggreGates.
This allows us to construct assignments of the IndAGate by ``glueing together'' the assignments of the component AggreGates that we constructed in Section~\ref{sub:aggre}.
This ``glueing'' process is fairly intuitive but slightly annoying to phrase formally.

We first define the underlying diagram.
\begin{definition}%
  \label{def:agate}
  Let $s \ge 2$.
  The diagram $\smagate_s$ is described by the following picture.
  
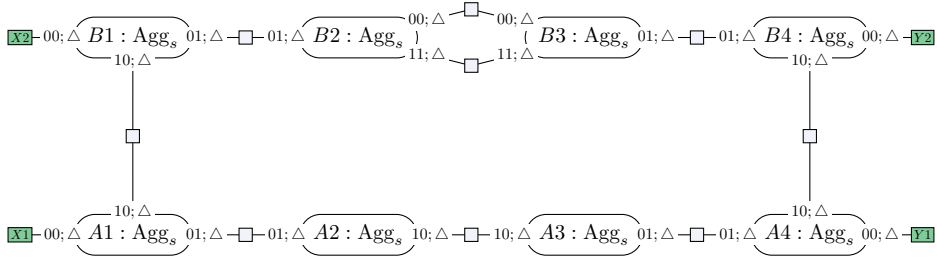
\begin{center}
    \begin{tikzpicture}[
        defnode/.style={rectangle,draw,fill=lightgray,inner sep=2pt,outer sep=0pt,minimum size=10pt},
        gatelabel/.style={rectangle, inner sep=2pt, outer sep=0pt, fill=white,scale=0.6},
        subnode/.style={rounded rectangle, draw, inner sep=5pt, outer sep=0pt,minimum size=15pt},
        leaf/.style={rectangle,draw,fill=leafgreen,inner sep=2pt,outer sep=0pt,minimum size=10pt},
        scale=0.75
      ]

      \node[subnode,scale=0.8] (A1) at (0,0) {$A1 : \agg_s$};
      \node[subnode,scale=0.8] (A2) at (4,0) {$A2 : \agg_s$};
      \node[subnode,scale=0.8] (A3) at (8,0) {$A3 : \agg_s$};
      \node[subnode,scale=0.8] (A4) at (12,0) {$A4 : \agg_s$};
      \node[subnode,scale=0.8] (B1) at (0,3.5) {$B1 : \agg_s$};
      \node[subnode,scale=0.8] (B2) at (4,3.5) {$B2 : \agg_s$};
      \node[subnode,scale=0.8] (B3) at (8,3.5) {$B3 : \agg_s$};
      \node[subnode,scale=0.8] (B4) at (12,3.5) {$B4 : \agg_s$};

      \node[leaf,scale=0.5] (X1) at ($(A1)+(-2,0)$) {$X1$};
      \node[leaf,scale=0.5] (X2) at ($(B1)+(-2,0)$) {$X2$};
      \node[leaf,scale=0.5] (Y1) at ($(A4)+(2,0)$) {$Y1$};
      \node[leaf,scale=0.5] (Y2) at ($(B4)+(2,0)$) {$Y2$};

      \node[defnode,scale=0.5] (A12) at ($(A1)!0.5!(A2)$) {};
      \node[defnode,scale=0.5] (A23) at ($(A2)!0.5!(A3)$) {};
      \node[defnode,scale=0.5] (A34) at ($(A3)!0.5!(A4)$) {};

      \node[defnode,scale=0.5] (B12) at ($(B1)!0.5!(B2)$) {};
      \node[defnode,scale=0.5] (B23a) at ($(B2)!0.5!(B3) + (0,0.5)$) {};
      \node[defnode,scale=0.5] (B23b) at ($(B2)!0.5!(B3) + (0,-0.5)$) {};
      \node[defnode,scale=0.5] (B34) at ($(B3)!0.5!(B4)$) {};

      \node[defnode,scale=0.5] (AB1) at ($(A1)!0.5!(B1)$) {};
      \node[defnode,scale=0.5] (AB4) at ($(A4)!0.5!(B4)$) {};

      \draw (A1) -- node[gatelabel, pos=0.28] {$00;\triangle$} (X1);
      \draw (A1) -- node[gatelabel, pos=0.28] {$01;\triangle$} (A12);
      \draw (A1) -- node[gatelabel, pos=0.04] {$10;\triangle$} (AB1);

      \draw (A2) -- node[gatelabel, pos=0.28] {$01;\triangle$} (A12);
      \draw (A2) -- node[gatelabel, pos=0.28] {$10;\triangle$} (A23);

      \draw (A3) -- node[gatelabel, pos=0.28] {$10;\triangle$} (A23);
      \draw (A3) -- node[gatelabel, pos=0.28] {$01;\triangle$} (A34);

      \draw (A4) -- node[gatelabel, pos=0.28] {$00;\triangle$} (Y1);
      \draw (A4) -- node[gatelabel, pos=0.28] {$01;\triangle$} (A34);
      \draw (A4) -- node[gatelabel, pos=0.04] {$10;\triangle$} (AB4);

      \draw (B1) -- node[gatelabel, pos=0.28] {$00;\triangle$} (X2);
      \draw (B1) -- node[gatelabel, pos=0.28] {$01;\triangle$} (B12);
      \draw (B1) -- node[gatelabel, pos=0.04] {$10;\triangle$} (AB1);

      \draw (B2) -- node[gatelabel, pos=0.28] {$01;\triangle$} (B12);
      \draw (B2) -- node[gatelabel, pos=0.28] {$00;\triangle$} (B23a);
      \draw (B2) -- node[gatelabel, pos=0.28] {$11;\triangle$} (B23b);

      \draw (B3) -- node[gatelabel, pos=0.28] {$01;\triangle$} (B34);
      \draw (B3) -- node[gatelabel, pos=0.28] {$00;\triangle$} (B23a);
      \draw (B3) -- node[gatelabel, pos=0.28] {$11;\triangle$} (B23b);

      \draw (B4) -- node[gatelabel, pos=0.28] {$00;\triangle$} (Y2);
      \draw (B4) -- node[gatelabel, pos=0.28] {$01;\triangle$} (B34);
      \draw (B4) -- node[gatelabel, pos=0.04] {$10;\triangle$} (AB4);
    \end{tikzpicture}
  \end{center}
  Note that the four visible leaves have been given their own new top-level labels $X1$, $X2$, $Y1$, $Y2$, which satisfy $X1 = A1;00;\triangle$, $X2 = B1;00;\triangle$, $Y1 = A4;00;\triangle$, $Y2 = B4;00;\triangle$.
\end{definition}
It is not straightforward to enumerate the leaves of $\smagate_s$ except to say that they are all the leaves of the eight copies of $\agg_s$ except those appearing as gray nodes in the diagram.

As before our first goal is to show that we can build this diagram starting from a single copy of $\aggregate_s$, and hence from $\gc_s$.

\begin{lemma}%
  \label{lem:build-agate}
  Let $s \ge 2$, and set $\gamma(X1)=((00;\triangle),0)$ and $\gamma(Y1)=((00;\triangle),1)$.
  Then $\agg_s \entails^3_{\gamma} \smagate_s$.
\end{lemma}
\begin{proof}%
  Starting with $\agg_s$, we first apply
  $\CS(\{(01;\triangle)\})$
  to obtain
  \begin{center}
    \begin{tikzpicture}[
        defnode/.style={rectangle,draw,fill=lightgray,inner sep=2pt,outer sep=0pt,minimum size=10pt},
        gatelabel/.style={rectangle, inner sep=2pt, outer sep=0pt, fill=white,scale=0.6},
        subnode/.style={rounded rectangle, draw, inner sep=5pt, outer sep=0pt,minimum size=15pt},
        leaf/.style={rectangle,draw,fill=leafgreen,inner sep=2pt,outer sep=0pt,minimum size=10pt},
        scale=0.8
      ]

      \node[subnode,scale=0.8] (A1) at (0,0) {$L : \agg_s$};
      \node[subnode,scale=0.8] (A2) at (4,0) {$R : \agg_s$};
      \node[defnode,scale=0.5] (A12) at ($(A1)!0.5!(A2)$) {};
      \draw (A1) -- node[gatelabel, pos=0.28] {$01;\triangle$} (A12);
      \draw (A2) -- node[gatelabel, pos=0.28] {$01;\triangle$} (A12);
    \end{tikzpicture}
  \end{center}
  Next we apply $\CS(\{(L;10;\triangle)\})$ to obtain
  \begin{center}
    \begin{tikzpicture}[
        defnode/.style={rectangle,draw,fill=lightgray,inner sep=2pt,outer sep=0pt,minimum size=10pt},
        gatelabel/.style={rectangle, inner sep=2pt, outer sep=0pt, fill=white,scale=0.6},
        subnode/.style={rounded rectangle, draw, inner sep=5pt, outer sep=0pt,minimum size=15pt},
        leaf/.style={rectangle,draw,fill=leafgreen,inner sep=2pt,outer sep=0pt,minimum size=10pt},
        scale=0.8
      ]

      \node[subnode,scale=0.8] (A1) at (0,0) {$L;L : \agg_s$};
      \node[subnode,scale=0.8] (A2) at (4,0) {$L;R : \agg_s$};
      \node[subnode,scale=0.8] (B1) at (0,2.5) {$R;L : \agg_s$};
      \node[subnode,scale=0.8] (B2) at (4,2.5) {$R;R : \agg_s$};

      \node[defnode,scale=0.5] (A12) at ($(A1)!0.5!(A2)$) {};
      \node[defnode,scale=0.5] (B12) at ($(B1)!0.5!(B2)$) {};
      \node[defnode,scale=0.5] (AB1) at ($(A1)!0.5!(B1)$) {};

      \draw (A1) -- node[gatelabel, pos=0.28] {$01;\triangle$} (A12);
      \draw (A1) -- node[gatelabel, pos=0.04] {$10;\triangle$} (AB1);
      \draw (A2) -- node[gatelabel, pos=0.28] {$01;\triangle$} (A12);

      \draw (B1) -- node[gatelabel, pos=0.28] {$01;\triangle$} (B12);
      \draw (B1) -- node[gatelabel, pos=0.04] {$10;\triangle$} (AB1);
      \draw (B2) -- node[gatelabel, pos=0.28] {$01;\triangle$} (B12);
    \end{tikzpicture}
  \end{center}
  Finally, we apply $\CS(\{(L;R;10;\triangle), (R;R;00;\triangle), (R;R;11;\triangle)\})$.
  After relabelling
  \begin{align*}
    L;L;L &\leadsto A1 & L;R;L &\leadsto B1 \\
    L;L;R &\leadsto A2 & L;R;R &\leadsto B2 \\
    R;L;R &\leadsto A3 & R;R;R &\leadsto B3 \\
    R;L;L &\leadsto A4 & R;R;L &\leadsto B4 
  \end{align*}
  we obtain exactly the diagram $\smagate_s$ shown above, as required.
\end{proof}

There are again a few different gate structures on the $\smagate_s$ diagram, depending on whether it is going to appear at the beginning, middle or end (or both the beginning and the end) of the circuit depicted in Figure~\ref{fig:bilinear}.

\begin{definition}%
  \label{def:agate-gate}
  We describe four gates $\smagate_s^{U}$, where $U$ is any subset of $\{\Alpha,\Omega\}$.
  (Here $\Alpha$ is pronounced ``initial'' and $\Omega$ is pronounced ``final''.)

  In all cases we take $\cR = [s]$.
  Then the pins of $\smagate_s^{U}$ are the subset of
  $\{X1,X2,Y1,Y2\}$
  where $X2$ is included when $\Alpha \notin U$ and $Y2$ is included when $\Omega \notin U$.
  In particular, $|\pins| = 4 - |U|$.

  We write $\smagate_s$, $\smagate_s^{\Alpha}$ and $\smagate_s^{\Alpha,\Omega}$ in place of $\smagate_s^{\emptyset}$, $\smagate_s^{\{\Alpha\}}$, $\smagate_s^{\{\Alpha,\Omega\}}$, etc..
\end{definition}

Indeed, in Figure~\ref{fig:bilinear},
in the middle ``block'' the top rail has connections to both the left and the right, corresponding to the pins $X2$ and $Y2$ respectively.
In the first block the leftwards connection $X2$ is not needed, and similarly in the last block the rightwards connection $Y2$ is not needed.

In the same way that the diagram $\smagate_s$ is made up of eight sub-diagrams $A1,A2,\dots,B4$ of type $\aggregate_s$, it is useful to think of the gate $\smagate_s^{U}$ as being composed of eight sub-gates.
In other words, for each sub-diagram\footnote{%
  \label{footnote:subgate}%
  Recall that induced sub-diagrams such as $\smagate_s[A1]$ are not strictly speaking copies of $\agg_s$, but rather copies of $\agg_s$ where some leaves have become non-leaves.
  For example, $\smagate_s[A1]$ is a copy of $\agg_s\bigl(\{(01;\triangle),(10;\triangle)\} \leadsto \nonleafs\bigr)$.
  When we think of $A1$ as a sub-gate, the gate structure is relative to an \emph{intact} copy of the original diagram $\agg_s$.
  This is logically the right thing to do---a pin which has been joined to another pin should still be treated as a pin, even though it is no longer a leaf---but is very awkward notationally. 
}
  $A1 \colon \agg_s$, $\dots$, $B4 \colon \agg_s$, we will assign a gate structure $\agg_s^P$ for various choices of $P \subseteq \{0,1\}^2$.

The choice is essentially forced by considering the available pins and toggles.
We want every toggle of $\smagate_s^{U}$ to be a toggle of exactly one sub-gate, so that we can partition the toggles $\toggles^{\smagate_s^{U}}$ by combining partitions of the sub-gates.
Equivalently, an element of $ij;\triangle$ should be a pin in some copy $A1,\dots,B4$ when the corresponding vertex is either one of the $4-|U|$ pins of $\smagate_s^{U}$, or was previously fused to another vertex in a joining operation. 

Specifically, we consider $A1,A2,\dots,B4$ as gates as follows:---
\begin{align*}
  A1 &: \agg_s^{\{00,01,10\}} & B1 &: \agg_s^{\{01,10\} \cup \{00 \text{ if } \Alpha \notin U\}}  \\
  A2 &: \agg_s^{\{01,10\}}    & B2 &: \agg_s^{\{00,01,11\}}  \\
  A3 &: \agg_s^{\{01,10\}}    & B3 &: \agg_s^{\{00,01,11\}}  \\
  A4 &: \agg_s^{\{00,01,10\}} & B4 &: \agg_s^{\{01,10\} \cup \{00 \text{ if } \Omega \notin U\}} 
\end{align*}

We now consider assignments.
For a non-initial gate, $\smagate_s$ or $\smagate_s^{\Omega}$, the assignment will ultimately depend on whether the relevant bit of the integer $a$ is $0$ or $1$.
Hence we will exhibit $2$ assignments, corresponding to the values $0$ and $1$ respectively.
For an initial gate $\smagate_s^{A}$ or $\smagate_s^{A,\Omega}$, we will actually give $8$ assignments, depending on the first $2$ bits of $a$ and also its sign.

We begin with the simplest case of the gate $\smagate_s$ with four pins.
As usual we first define some auxiliary linear data.
\begin{definition}%
  \label{def:ags}
  Define four linear data $\ag(i,A)$ and $\ag(i,B)$ for $i \in \{0,1\}$, all having $I = \{X1,X2,Y1,Y2\}$, $V = \FF_p^2$ and $W_j=\FF_p$ for $j \in I$, as follows.
  Arranging the diagrams as
  
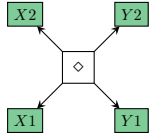
\begin{center}
    \begin{tikzpicture}[
        defnode/.style={rectangle,draw,inner sep=4pt,outer sep=0pt,minimum size=15pt},
        mydot/.style={circle,fill,inner sep=0.5pt},
        leaf/.style={defnode,fill=leafgreen},
        toggle/.style={defnode,fill=leafgreen}
      ]
      \node[defnode,scale=0.8] (U) at (0, 0) {$\diamond$};
      \node[leaf,scale=0.6]    (UX00) at ($(U) + (225:1)$) {$X1$};
      \node[leaf,scale=0.6]    (UX01) at ($(U) + (135:1)$) {$X2$};
      \node[leaf,scale=0.6]    (UX10) at ($(U) + (315:1)$) {$Y1$};
      \node[leaf,scale=0.6]    (UX11) at ($(U) + (045:1)$) {$Y2$};
      \draw[-stealth] (U) -- (UX00);
      \draw[-stealth] (U) -- (UX01);
      \draw[-stealth] (U) -- (UX10);
      \draw[-stealth] (U) -- (UX11);
    \end{tikzpicture}
  \end{center}
  the linear maps $V \to W_i$ are given respectively by
  \begin{center}
    \begin{tikzpicture}[
        defnode/.style={rectangle,draw,inner sep=4pt,outer sep=0pt,minimum size=15pt},
        mydot/.style={circle,fill,inner sep=0.5pt},
        leaf/.style={defnode,fill=leafgreen},
        toggle/.style={defnode,fill=leafgreen}
      ]
      \begin{scope}[shift={(0,0)}]
        \node[defnode,scale=0.8] (U) at (0, 0) {$(x,y)$};
        \node[leaf,scale=0.6]    (UX00) at ($(U) + (225:1)$) {$x$};
        \node[leaf,scale=0.6]    (UX01) at ($(U) + (135:1)$) {$2y$};
        \node[leaf,scale=0.6]    (UX10) at ($(U) + (315:1)$) {$x$};
        \node[leaf,scale=0.6]    (UX11) at ($(U) + (045:1)$) {$y$};
        \draw[-stealth] (U) -- (UX00);
        \draw[-stealth] (U) -- (UX01);
        \draw[-stealth] (U) -- (UX10);
        \draw[-stealth] (U) -- (UX11);

        \node["{$\ag(0,A)$}" {scale=0.9}] at (0, -2) {};
      \end{scope}
      \begin{scope}[shift={(3,0)}]
        \node[defnode,scale=0.8] (U) at (0, 0) {$(x,y)$};
        \node[leaf,scale=0.6]    (UX00) at ($(U) + (225:1)$) {$x$};
        \node[leaf,scale=0.6]    (UX01) at ($(U) + (135:1)$) {$y$};
        \node[leaf,scale=0.6]    (UX10) at ($(U) + (315:1)$) {$x$};
        \node[leaf,scale=0.6]    (UX11) at ($(U) + (045:1)$) {$2y$};
        \draw[-stealth] (U) -- (UX00);
        \draw[-stealth] (U) -- (UX01);
        \draw[-stealth] (U) -- (UX10);
        \draw[-stealth] (U) -- (UX11);

        \node["{$\ag(0,B)$}" {scale=0.9}] at (0, -2) {};
      \end{scope}
      \begin{scope}[shift={(6,0)}]
        \node[defnode,scale=0.8] (U) at (0, 0) {$(x,y)$};
        \node[leaf,scale=0.6]    (UX00) at ($(U) + (225:1)$) {$x$};
        \node[leaf,scale=0.6]    (UX01) at ($(U) + (135:1)$) {$2y-2x$};
        \node[leaf,scale=0.6]    (UX10) at ($(U) + (315:1)$) {$x$};
        \node[leaf,scale=0.6]    (UX11) at ($(U) + (045:1)$) {$y$};
        \draw[-stealth] (U) -- (UX00);
        \draw[-stealth] (U) -- (UX01);
        \draw[-stealth] (U) -- (UX10);
        \draw[-stealth] (U) -- (UX11);

        \node["{$\ag(1,A)$}" {scale=0.9}] at (0, -2) {};
      \end{scope}
      \begin{scope}[shift={(9,0)}]
        \node[defnode,scale=0.8] (U) at (0, 0) {$(x,y)$};
        \node[leaf,scale=0.6]    (UX00) at ($(U) + (225:1)$) {$x$};
        \node[leaf,scale=0.6]    (UX01) at ($(U) + (135:1)$) {$y$};
        \node[leaf,scale=0.6]    (UX10) at ($(U) + (315:1)$) {$x+2y$};
        \node[leaf,scale=0.6]    (UX11) at ($(U) + (045:1)$) {$2y$};
        \draw[-stealth] (U) -- (UX00);
        \draw[-stealth] (U) -- (UX01);
        \draw[-stealth] (U) -- (UX10);
        \draw[-stealth] (U) -- (UX11);

        \node["{$\ag(1,B)$}" {scale=0.9}] at (0, -2) {};
      \end{scope}
    \end{tikzpicture}
  \end{center}
\end{definition}

\begin{lemma}%
  \label{lem:agate-assign-1}
  Let $s \ge 2$ and $i \in \{0,1\}$.
  Then there exists an assignment $\Middle(i)$ of $\smagate_s$ with
  \begin{itemize}
    \item $D_0 = \trivial(\{X1,X2,Y1,Y2\})$;
    \item $D_1 = \ag(i,A)$;
    \item $D_2 = \ag(i,B)$; and
    \item $D_r = \const(\{X1,X2,Y1,Y2\})$ for $3 \le r \le s$.
  \end{itemize}
\end{lemma}
We give some brief explanation, again with reference to Figure~\ref{fig:bilinear} and in particular the middle block.
The value $i \in \{0,1\}$ will represent a certain bit of the integer $a$; in Figure~\ref{fig:bilinear} it is bit $4$ and $i=1$.
The operation is easiest to consider in the mode $r=2$, corresponding to the data $\ag(0,B)$ or $\ag(1,B)$, and to the second $h'$ coordinates in Figure~\ref{fig:bilinear}.
Then, the effect of the circuit is to (i) double the value $y=2h'$ on the top ``rail'' to give $4h'$, and (ii) if $i=1$, add this doubled value to the bottom rail $x=h'$ to give $5h'$.
If $i=0$, we would omit step (ii) and leave the bottom rail $x$ unchanged.
(In the language of Figure~\ref{fig:bilinear}, we would replace the third vertical line with a dotted line and its two endpoints with null gates.)

The mode $r=1$, corresponding to the linear data $\ag(0,A)$ or $\ag(1,A)$ and the first coordinate in Figure~\ref{fig:bilinear}, performs a kind of dual calculation.
This is most easily thought of from \emph{right to left}: (i) subtract the bottom-right rail value $x=h$ from the top right rail $y=-3h$, giving $-3h$, and then (ii) double the value $y-x$ to give $2(y-x)=-6h$ on the top-left rail.

These correspond to the following tensor identities when $s=2$: when $i=0$,
\[
  x_1 \otimes x_2 - (2y_1) \otimes y_2 = x_1 \otimes x_2 - y_1 \otimes (2y_2)
\]
or when $i=1$,
\[
  x_1 \otimes x_2 - (2y_1-2x_1) \otimes y_2 = x_1 \otimes (x_2+2y_2) - y_1 \otimes (2 y_2).
\]
When $s>2$, simply tensor everything with a constant value $x_3 \otimes \dots \otimes x_s$.

\begin{proof}[Proof of Lemma~\ref{lem:agate-assign-1}]%
  We begin by describing the partitions in terms of known assignments of the sub-gates $A1,\dots,B4$.
  Specifically, consider the following assignments.
  \begin{equation}
    \label{eq:agate-assignments}
    \begin{aligned}
      A1 &: \opcross^{01;\overline{11}}_{1,2} & B1 &: \opcross^{01;\overline{11}}_{1,2}  \\
      A2 &: \sumconst_2 &                       B2 &: \sumconst_1 \\
      A3 &: \sumconst_2 &                       B3 &: \sumconst_2 \\
      A4 &:
      \begin{rcases}
        \begin{dcases}
          \opcross^{01;\overline{11}}_{2,1} &\colon i=0 \\
          \sumconst_2 &\colon i=1
        \end{dcases}
      \end{rcases} &
      B4 &: \sumconst_1      .
    \end{aligned}
  \end{equation}
  We refer to the partitions defined by each of these assignments as
  \[
    \toggles^{\ell} = S^{\ell}_1 \cup \dots \cup S^{\ell}_s
  \]
  for $\ell \in \{A1,A2,\dots,B4\}$.
  Similarly we write $D_r^\ell$ and $\cM_r^\ell$ for the other data in the assignment.

  We can then define a partition of $\toggles^{\smagate_s}$ by combining these:
  that is, for $r \in [s]$,
  \[
    S_r = \bigcup_{\ell \in \{A1,A2,\dots,B4\}} S_r^{\ell}.
  \]
  Since
  \[
    \toggles^{\smagate_s} = \bigcup_{\ell \in \{A1,A2,\dots,B4\}} \toggles^\ell
  \]
  by our choice of gate structure on $A1,\dots,B4$, this is a valid partition of the toggles of $\smagate_s$.

  As usual, it is useful to summarize the construction of the morphisms $\cM_1,\cM_2,\dots$ pictorially.
  Starting with the case $i=1$ and considering
  \[
  \cM_2 \colon \ag(1,B) \to \smagate_s[\nonleafs \cup \pins \cup S_2],
  \]
  where $\cP = \cP^{\smagate_s} = \{X1,X2,Y1,Y2\}$,
  this is shown
  in Figure~\ref{fig:agate-1b}.
  We have annotated the subdiagram boxes $A1,\dots,B4$ with a (sometimes slightly abbreviated) description of what their assignment from~\eqref{eq:agate-assignments} looks like when $r=2$; i.e., with a description of the domain $D_2^\ell$ of the morphism $\cM_2^\ell$.

  \begin{figure}[htbp]
    \begin{center}
      \begin{tikzpicture}[
          defnode/.style={rectangle,fill=lightgray,draw,inner sep=2pt,outer sep=0pt,minimum size=10pt},
          gatelabel/.style={ellipse, inner sep=0.5pt, outer sep=0pt, fill=white,scale=0.6},
          subnode/.style={rounded rectangle, draw, inner sep=5pt, outer sep=0pt,minimum size=15pt},
          leaf/.style={rectangle,draw,fill=leafgreen,inner sep=2pt,outer sep=0pt,minimum size=10pt},
          scale=0.8
        ]

        \node[subnode,scale=0.8] (A1) at (0,0)    {$\crs^{01}_{\strikeout{11}}$};
        \node[subnode,scale=0.8] (A2) at (3.5,0)    {$\opsum$};
        \node[subnode,scale=0.8] (A3) at (7,0)    {$\opsum$};
        \node[subnode,scale=0.8] (A4) at (10.5,0)   {$\opsum$};
        \node[subnode,scale=0.8] (B1) at (0,3)  {$\crs^{01}_{\strikeout{11}}$};
        \node[subnode,scale=0.8] (B2) at (3.5,3)  {$\const$};
        \node[subnode,scale=0.8] (B3) at (7,3)  {$\opsum$};
        \node[subnode,scale=0.8] (B4) at (10.5,3) {$\const$};

        \node[leaf,scale=0.7] (X1) at ($(A1)+(-2,0)$) {$x$};
        \node[leaf,scale=0.7] (X2) at ($(B1)+(-2,0)$) {$y$};
        \node[leaf,scale=0.7] (Y1) at ($(A4)+(2,0)$)  {$x+2y$};
        \node[leaf,scale=0.7] (Y2) at ($(B4)+(2,0)$)  {$2y$};

        \node[defnode,scale=0.7] (A12) at ($(A1)!0.5!(A2)$) {$x$};
        \node[defnode,scale=0.7] (A23) at ($(A2)!0.5!(A3)$) {$-x$};
        \node[defnode,scale=0.7] (A34) at ($(A3)!0.5!(A4)$) {$x$};

        \node[defnode,scale=0.7] (B12) at ($(B1)!0.5!(B2)$) {$y$};
        \node[defnode,scale=0.7] (B23a) at ($(B2)!0.5!(B3) + (0,0.5)$) {$y$};
        \node[defnode,scale=0.7] (B23b) at ($(B2)!0.5!(B3) + (0,-0.5)$) {$y$};
        \node[defnode,scale=0.7] (B34) at ($(B3)!0.5!(B4)$) {$2y$};

        \node[defnode,scale=0.7] (AB1) at ($(A1)!0.5!(B1)$) {$0$};
        \node[defnode,scale=0.7] (AB4) at ($(A4)!0.5!(B4)$) {$2y$};

        \draw (A1) -- node[gatelabel, pos=0.28] {$00;\triangle$} (X1);
        \draw (A1) -- node[gatelabel, pos=0.28] {$01;\triangle$} (A12);

        \draw (A2) -- node[gatelabel, pos=0.28] {$01;\triangle$} (A12);
        \draw (A2) -- node[gatelabel, pos=0.28] {$10;\triangle$} (A23);

        \draw (A3) -- node[gatelabel, pos=0.28] {$10;\triangle$} (A23);
        \draw (A3) -- node[gatelabel, pos=0.28] {$01;\triangle$} (A34);

        \draw (A4) -- node[gatelabel, pos=0.28] {$00;\triangle$} (Y1);
        \draw (A4) -- node[gatelabel, pos=0.28] {$01;\triangle$} (A34);

        \draw (B1) -- node[gatelabel, pos=0.28] {$00;\triangle$} (X2);
        \draw (B1) -- node[gatelabel, pos=0.28] {$01;\triangle$} (B12);

        \draw (B2) -- node[gatelabel, pos=0.28] {$01;\triangle$} (B12);
        \draw (B2) -- node[gatelabel, pos=0.28] {$00;\triangle$} (B23a);
        \draw (B2) -- node[gatelabel, pos=0.28] {$11;\triangle$} (B23b);

        \draw (B3) -- node[gatelabel, pos=0.28] {$01;\triangle$} (B34);
        \draw (B3) -- node[gatelabel, pos=0.28] {$00;\triangle$} (B23a);
        \draw (B3) -- node[gatelabel, pos=0.28] {$11;\triangle$} (B23b);

        \draw (B4) -- node[gatelabel, pos=0.28] {$00;\triangle$} (Y2);
        \draw (B4) -- node[gatelabel, pos=0.28] {$01;\triangle$} (B34);

        \draw (A1) -- node[gatelabel, pos=0.08] {$10;\triangle$} (AB1);
        \draw (A4) -- node[gatelabel, pos=0.08] {$10;\triangle$} (AB4);
        \draw (B1) -- node[gatelabel, pos=0.08] {$10;\triangle$} (AB1);
        \draw (B4) -- node[gatelabel, pos=0.08] {$10;\triangle$} (AB4);
      \end{tikzpicture}
    \end{center}
    \caption{A description of the morphism $\cM_2$ for the assignment $\Middle(1)$.}%
    \label{fig:agate-1b}
  \end{figure}

  Informally, checking the validity of this picture amounts to noting by inspection that every gate is ``doing what it is supposed to do''.
  For example, the gate $A4$ in mode $r=2$ is behaving like a $\opsum$ gate, and the values on its pins do indeed satisfy $(x+2y) = (x) + (2y)$, while the gate $B4$ is behaving like a $\const$ gate and indeed all its pins have the same value $2y$.
  Finally, the gate $A1$ is crossing its $00$ value to $01$ (both $x$) and the zeroed $\strikeout{11}$ pin forces its twin $10$ to zero.

  However, we need to turn this informal check into a formal construction of a morphism $\cM_2$ with the required properties.
  We have at our disposal the various morphisms $\cM_2^\ell$ from the assignments in~\eqref{eq:agate-assignments}.
  To construct $\cM_2$, we must (i) \emph{specialize} these morphisms $\cM_2$ to the values shown in Figure~\ref{fig:agate-1b} (e.g., setting the pins of $A4$ to the specific values $x+2y$, $x$ and $2y$ for consistent variables $(x,y)$ rather than an arbitrary triple $a+b$, $a$ and $b$), and then (ii) patch these together to form a single morphism.
  Our hope is to convince the reader that this process is no more than formal nonsense, and that Figure~\ref{fig:agate-1b} captures the entire argument, but we give full details on this occasion.

  Each morphism $\cM_2^\ell$ has co-domain $\aggregate_s[\nonleafs \cup \pins^\ell \cup S_2^\ell]$.
  However, each sub-diagram\footnote{These could also be denoted $\smagate_s[A1]$, $\smagate_s[A2]$, $\dots$, $\smagate_s[B4]$: see the abuses of notation in Definition~\ref{def:labelling} and Convention~\ref{conv:labelling-other}.} $A1,A2,\dots,B4$ of $\smagate_s$ is almost an isomorphic copy of $\aggregate_s$, where ``almost'' reflects the fact that some vertices which are leaves in $\aggregate_s$ will not be leaves in some of $A1,\dots,B4$ (see Remark~\ref{rem:joining-sub}).
  Hence, using Definition~\ref{def:leaf-non-leaf} and Remark~\ref{rem:leaf-non-leaf-morph},
  we obtain morphisms:---
  \begin{align*}
    \cM_2^{A1} &\colon \crs^{01}_{\strikeout{11}}(\{01,10\} \leadsto \nonleafs)  \to A1\bigl[\nonleafs \cup \pins \cup S_2^{A1}\bigr] \\
    \cM_2^{A2} &\colon \opsum(\{01,10\})(\{01,10\} \leadsto \nonleafs)           \to A2\bigl[\nonleafs \cup \pins \cup S_2^{A2}\bigr] \\
    \cM_2^{A3} &\colon \opsum(\{01,10\})(\{01,10\} \leadsto \nonleafs)           \to A3\bigl[\nonleafs \cup \pins \cup S_2^{A3}\bigr] \\
    \cM_2^{A4} &\colon \opsum(\{00,01,10\})(\{01,10\} \leadsto \nonleafs)        \to A4\bigl[\nonleafs \cup \pins \cup S_2^{A4}\bigr] \\
    \cM_2^{B1} &\colon \crs^{01}_{\strikeout{11}}(\{01,10\} \leadsto \nonleafs)  \to B1\bigl[\nonleafs \cup \pins \cup S_2^{B1}\bigr] \\
    \cM_2^{B2} &\colon \const(\{00,01,11\})(\{00,01,11\} \leadsto \nonleafs)     \to B2\bigl[\nonleafs \cup \pins \cup S_2^{B2}\bigr] \\
    \cM_2^{B3} &\colon \opsum(\{00,01,11\})(\{00,01,11\} \leadsto \nonleafs)     \to B3\bigl[\nonleafs \cup \pins \cup S_2^{B3}\bigr] \\
    \cM_2^{B4} &\colon \const(\{00,01,10\})(\{01,10\} \leadsto \nonleafs)        \to B4\bigl[\nonleafs \cup \pins \cup S_2^{B4}\bigr] .
  \end{align*}
  For the ``specialization'' step, we compose $\cM_2^\ell$ with certain morphisms $\Theta_\ell$ with domain $\ag(1,B)$, to obtain morphisms $\Theta_\ell \circ \cM_2^\ell \colon \ag(1,B) \to \ell[\nonleafs \cup \pins^\ell \cup S_2^\ell]$, for $\ell \in \{A1,\dots,B4\}$.
  The morphisms $\Theta_\ell$ can be read off from Figure~\ref{fig:agate-1b}, but we give them explicitly in Figure~\ref{fig:eight-morphisms}.
  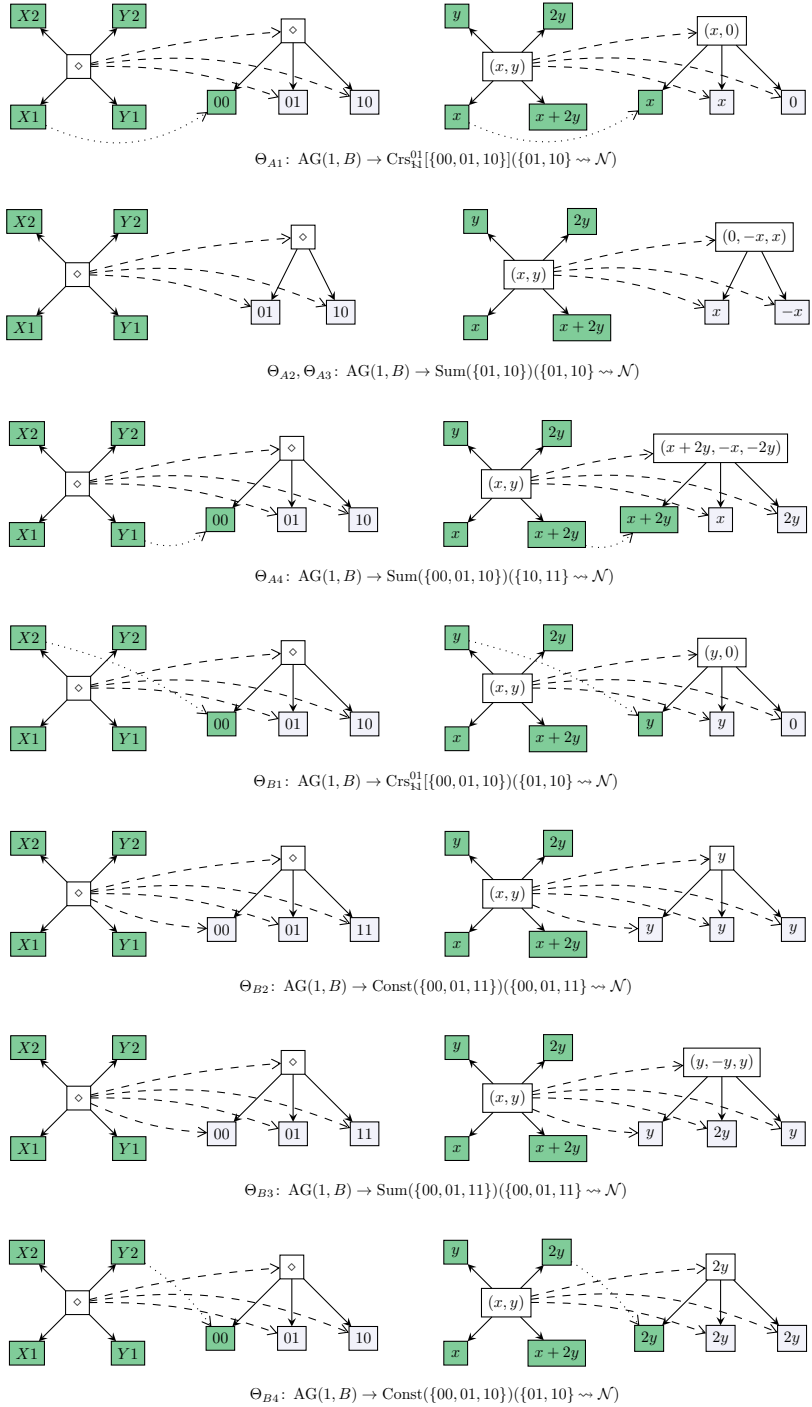
\begin{figure}[htbp]
    \begin{center}
      % PICTURE 1, A1
      \begin{tikzpicture}[
          defnode/.style={rectangle,draw,inner sep=4pt,outer sep=0pt,minimum size=15pt},
          mydot/.style={circle,fill,inner sep=0.5pt},
          nonleaf/.style={defnode,fill=lightgray},
          leaf/.style={defnode,fill=leafgreen},
          toggle/.style={defnode,fill=leafgreen},
          scale=0.95
        ]
        \node[scale=0.6] at (5, -1.3) {$\Theta_{A1} \colon \ag(1,B) \to \crs^{01}_{\strikeout{11}}[\{00,01,10\}](\{01,10\} \leadsto \nonleafs)$};
        \begin{scope}[shift={(0,0)}]
          \node[defnode,scale=0.6] (U) at (0, 0) {$\diamond$};
          \node[leaf,scale=0.6]    (UX00) at ($(U) + (225:1)$) {$X1$};
          \node[leaf,scale=0.6]    (UX01) at ($(U) + (135:1)$) {$X2$};
          \node[leaf,scale=0.6]    (UX10) at ($(U) + (315:1)$) {$Y1$};
          \node[leaf,scale=0.6]    (UX11) at ($(U) + (045:1)$) {$Y2$};
          \draw[-stealth] (U) -- (UX00);
          \draw[-stealth] (U) -- (UX01);
          \draw[-stealth] (U) -- (UX10);
          \draw[-stealth] (U) -- (UX11);

          \node[defnode,scale=0.6] (V) at (3, 0.5) {$\diamond$};
          \node[leaf,scale=0.6]    (VX00) at ($(V) + (-1,-1)$) {$00$};
          \node[nonleaf,scale=0.6]    (VX01) at ($(V) + ( 0,-1)$) {$01$};
          \node[nonleaf,scale=0.6]    (VX10) at ($(V) + ( 1,-1)$) {$10$};
          \draw[-stealth] (V) -- (VX00);
          \draw[-stealth] (V) -- (VX01);
          \draw[-stealth] (V) -- (VX10);

          \draw[dotted, -angle 60] (UX00) to[bend right=30] (VX00);
          \draw[dashed, -angle 60] (U) to[bend left=10] (VX01);
          \draw[dashed, -angle 60] (U) to[bend left=15] (VX10);
          \draw[dashed, -angle 60] (U) to[bend left=5] (V);
        \end{scope}
        \begin{scope}[shift={(6,0)}]
          \node[defnode,scale=0.6] (U) at (0, 0) {$(x,y)$};
          \node[leaf,scale=0.6]    (UX00) at ($(U) + (225:1)$) {$x$};
          \node[leaf,scale=0.6]    (UX01) at ($(U) + (135:1)$) {$y$};
          \node[leaf,scale=0.6]    (UX10) at ($(U) + (315:1)$) {$x+2y$};
          \node[leaf,scale=0.6]    (UX11) at ($(U) + (045:1)$) {$2y$};
          \draw[-stealth] (U) -- (UX00);
          \draw[-stealth] (U) -- (UX01);
          \draw[-stealth] (U) -- (UX10);
          \draw[-stealth] (U) -- (UX11);

          \node[defnode,scale=0.6] (V) at (3, 0.5) {$(x,0)$};
          \node[leaf,scale=0.6]    (VX00) at ($(V) + (-1,-1)$) {$x$};
          \node[nonleaf,scale=0.6]    (VX01) at ($(V) + ( 0,-1)$) {$x$};
          \node[nonleaf,scale=0.6]    (VX10) at ($(V) + ( 1,-1)$) {$0$};
          \draw[-stealth] (V) -- (VX00);
          \draw[-stealth] (V) -- (VX01);
          \draw[-stealth] (V) -- (VX10);

          \draw[dotted, -angle 60] (UX00) to[bend right=30] (VX00);
          \draw[dashed, -angle 60] (U) to[bend left=10] (VX01);
          \draw[dashed, -angle 60] (U) to[bend left=15] (VX10);
          \draw[dashed, -angle 60] (U) to[bend left=5] (V);
        \end{scope}
      \end{tikzpicture}
      \\[\baselineskip]
      % PICTURE 2, A2 + A3
      \begin{tikzpicture}[
          defnode/.style={rectangle,draw,inner sep=4pt,outer sep=0pt,minimum size=15pt},
          mydot/.style={circle,fill,inner sep=0.5pt},
          nonleaf/.style={defnode,fill=lightgray},
          leaf/.style={defnode,fill=leafgreen},
          toggle/.style={defnode,fill=leafgreen}
        ]
        \node[scale=0.6] at (5, -1.3) {$\Theta_{A2},\Theta_{A3} \colon \ag(1,B) \to \opsum(\{01,10\})(\{01,10\} \leadsto \nonleafs)$};
        \begin{scope}[shift={(0,0)}]
          \node[defnode,scale=0.6] (U) at (0, 0) {$\diamond$};
          \node[leaf,scale=0.6]    (UX00) at ($(U) + (225:1)$) {$X1$};
          \node[leaf,scale=0.6]    (UX01) at ($(U) + (135:1)$) {$X2$};
          \node[leaf,scale=0.6]    (UX10) at ($(U) + (315:1)$) {$Y1$};
          \node[leaf,scale=0.6]    (UX11) at ($(U) + (045:1)$) {$Y2$};
          \draw[-stealth] (U) -- (UX00);
          \draw[-stealth] (U) -- (UX01);
          \draw[-stealth] (U) -- (UX10);
          \draw[-stealth] (U) -- (UX11);

          \node[defnode,scale=0.6] (V) at (3, 0.5) {$\diamond$};
          \node[nonleaf,scale=0.6]    (VX01) at ($(V) + ( -0.5,-1)$) {$01$};
          \node[nonleaf,scale=0.6]    (VX10) at ($(V) + ( 0.5,-1)$) {$10$};
          \draw[-stealth] (V) -- (VX01);
          \draw[-stealth] (V) -- (VX10);

          \draw[dashed, -angle 60] (U) to[bend left=10] (VX01);
          \draw[dashed, -angle 60] (U) to[bend left=15] (VX10);
          \draw[dashed, -angle 60] (U) to[bend left=5] (V);
        \end{scope}
        \begin{scope}[shift={(6,0)}]
          \node[defnode,scale=0.6] (U) at (0, 0) {$(x,y)$};
          \node[leaf,scale=0.6]    (UX00) at ($(U) + (225:1)$) {$x$};
          \node[leaf,scale=0.6]    (UX01) at ($(U) + (135:1)$) {$y$};
          \node[leaf,scale=0.6]    (UX10) at ($(U) + (315:1)$) {$x+2y$};
          \node[leaf,scale=0.6]    (UX11) at ($(U) + (045:1)$) {$2y$};
          \draw[-stealth] (U) -- (UX00);
          \draw[-stealth] (U) -- (UX01);
          \draw[-stealth] (U) -- (UX10);
          \draw[-stealth] (U) -- (UX11);

          \node[defnode,scale=0.6] (V) at (3, 0.5) {$(0,-x,x)$};
          \node[nonleaf,scale=0.6]    (VX01) at ($(V) + ( -0.5,-1)$) {$x$};
          \node[nonleaf,scale=0.6]    (VX10) at ($(V) + ( 0.5,-1)$) {$-x$};
          \draw[-stealth] (V) -- (VX01);
          \draw[-stealth] (V) -- (VX10);

          \draw[dashed, -angle 60] (U) to[bend left=10] (VX01);
          \draw[dashed, -angle 60] (U) to[bend left=15] (VX10);
          \draw[dashed, -angle 60] (U) to[bend left=5] (V);
        \end{scope}
      \end{tikzpicture}
      \\[\baselineskip]
      % PICTURE 3, A4
      \begin{tikzpicture}[
          scale=0.95,
          defnode/.style={rectangle,draw,inner sep=4pt,outer sep=0pt,minimum size=15pt},
          mydot/.style={circle,fill,inner sep=0.5pt},
          nonleaf/.style={defnode,fill=lightgray},
          leaf/.style={defnode,fill=leafgreen},
          toggle/.style={defnode,fill=leafgreen}
        ]
        \node[scale=0.6] at (5, -1.3) {$\Theta_{A4} \colon \ag(1,B) \to \opsum(\{00,01,10\})(\{10,11\} \leadsto \nonleafs)$};
        \begin{scope}[shift={(0,0)}]
          \node[defnode,scale=0.6] (U) at (0, 0) {$\diamond$};
          \node[leaf,scale=0.6]    (UX00) at ($(U) + (225:1)$) {$X1$};
          \node[leaf,scale=0.6]    (UX01) at ($(U) + (135:1)$) {$X2$};
          \node[leaf,scale=0.6]    (UX10) at ($(U) + (315:1)$) {$Y1$};
          \node[leaf,scale=0.6]    (UX11) at ($(U) + (045:1)$) {$Y2$};
          \draw[-stealth] (U) -- (UX00);
          \draw[-stealth] (U) -- (UX01);
          \draw[-stealth] (U) -- (UX10);
          \draw[-stealth] (U) -- (UX11);

          \node[defnode,scale=0.6] (V) at (3, 0.5) {$\diamond$};
          \node[leaf,scale=0.6]    (VX00) at ($(V) + (-1,-1)$) {$00$};
          \node[nonleaf,scale=0.6]    (VX01) at ($(V) + ( 0,-1)$) {$01$};
          \node[nonleaf,scale=0.6]    (VX10) at ($(V) + ( 1,-1)$) {$10$};
          \draw[-stealth] (V) -- (VX00);
          \draw[-stealth] (V) -- (VX01);
          \draw[-stealth] (V) -- (VX10);

          \draw[dotted, -angle 60] (UX10) to[bend right=30] (VX00);
          \draw[dashed, -angle 60] (U) to[bend left=10] (VX01);
          \draw[dashed, -angle 60] (U) to[bend left=15] (VX10);
          \draw[dashed, -angle 60] (U) to[bend left=5] (V);
        \end{scope}
        \begin{scope}[shift={(6,0)}]
          \node[defnode,scale=0.6] (U) at (0, 0) {$(x,y)$};
          \node[leaf,scale=0.6]    (UX00) at ($(U) + (225:1)$) {$x$};
          \node[leaf,scale=0.6]    (UX01) at ($(U) + (135:1)$) {$y$};
          \node[leaf,scale=0.6]    (UX10) at ($(U) + (315:1)$) {$x+2y$};
          \node[leaf,scale=0.6]    (UX11) at ($(U) + (045:1)$) {$2y$};
          \draw[-stealth] (U) -- (UX00);
          \draw[-stealth] (U) -- (UX01);
          \draw[-stealth] (U) -- (UX10);
          \draw[-stealth] (U) -- (UX11);

          \node[defnode,scale=0.6] (V) at (3, 0.5) {$(x+2y,-x,-2y)$};
          \node[leaf,scale=0.6]    (VX00) at ($(V) + (-1,-1)$) {$x+2y$};
          \node[nonleaf,scale=0.6]    (VX01) at ($(V) + ( 0,-1)$) {$x$};
          \node[nonleaf,scale=0.6]    (VX10) at ($(V) + ( 1,-1)$) {$2y$};
          \draw[-stealth] (V) -- (VX00);
          \draw[-stealth] (V) -- (VX01);
          \draw[-stealth] (V) -- (VX10);

          \draw[dotted, -angle 60] (UX10) to[bend right=30] (VX00);
          \draw[dashed, -angle 60] (U) to[bend left=10] (VX01);
          \draw[dashed, -angle 60] (U) to[bend left=15] (VX10);
          \draw[dashed, -angle 60] (U) to[bend left=5] (V);
        \end{scope}
      \end{tikzpicture}
      \\[\baselineskip]
      % PICTURE 4, B1
      \begin{tikzpicture}[
          scale=0.95,
          defnode/.style={rectangle,draw,inner sep=4pt,outer sep=0pt,minimum size=15pt},
          mydot/.style={circle,fill,inner sep=0.5pt},
          nonleaf/.style={defnode,fill=lightgray},
          leaf/.style={defnode,fill=leafgreen},
          toggle/.style={defnode,fill=leafgreen}
        ]
        \node[scale=0.6] at (5, -1.3) {$\Theta_{B1} \colon \ag(1,B) \to \crs^{01}_{\strikeout{11}}[\{00,01,10\})(\{01,10\} \leadsto \nonleafs)$};
        \begin{scope}[shift={(0,0)}]
          \node[defnode,scale=0.6] (U) at (0, 0) {$\diamond$};
          \node[leaf,scale=0.6]    (UX00) at ($(U) + (225:1)$) {$X1$};
          \node[leaf,scale=0.6]    (UX01) at ($(U) + (135:1)$) {$X2$};
          \node[leaf,scale=0.6]    (UX10) at ($(U) + (315:1)$) {$Y1$};
          \node[leaf,scale=0.6]    (UX11) at ($(U) + (045:1)$) {$Y2$};
          \draw[-stealth] (U) -- (UX00);
          \draw[-stealth] (U) -- (UX01);
          \draw[-stealth] (U) -- (UX10);
          \draw[-stealth] (U) -- (UX11);

          \node[defnode,scale=0.6] (V) at (3, 0.5) {$\diamond$};
          \node[leaf,scale=0.6]    (VX00) at ($(V) + (-1,-1)$) {$00$};
          \node[nonleaf,scale=0.6]    (VX01) at ($(V) + ( 0,-1)$) {$01$};
          \node[nonleaf,scale=0.6]    (VX10) at ($(V) + ( 1,-1)$) {$10$};
          \draw[-stealth] (V) -- (VX00);
          \draw[-stealth] (V) -- (VX01);
          \draw[-stealth] (V) -- (VX10);

          \draw[dotted, -angle 60] (UX01) to[bend left=10] (VX00);
          \draw[dashed, -angle 60] (U) to[bend left=10] (VX01);
          \draw[dashed, -angle 60] (U) to[bend left=15] (VX10);
          \draw[dashed, -angle 60] (U) to[bend left=5] (V);
        \end{scope}
        \begin{scope}[shift={(6,0)}]
          \node[defnode,scale=0.6] (U) at (0, 0) {$(x,y)$};
          \node[leaf,scale=0.6]    (UX00) at ($(U) + (225:1)$) {$x$};
          \node[leaf,scale=0.6]    (UX01) at ($(U) + (135:1)$) {$y$};
          \node[leaf,scale=0.6]    (UX10) at ($(U) + (315:1)$) {$x+2y$};
          \node[leaf,scale=0.6]    (UX11) at ($(U) + (045:1)$) {$2y$};
          \draw[-stealth] (U) -- (UX00);
          \draw[-stealth] (U) -- (UX01);
          \draw[-stealth] (U) -- (UX10);
          \draw[-stealth] (U) -- (UX11);

          \node[defnode,scale=0.6] (V) at (3, 0.5) {$(y,0)$};
          \node[leaf,scale=0.6]    (VX00) at ($(V) + (-1,-1)$) {$y$};
          \node[nonleaf,scale=0.6]    (VX01) at ($(V) + ( 0,-1)$) {$y$};
          \node[nonleaf,scale=0.6]    (VX10) at ($(V) + ( 1,-1)$) {$0$};
          \draw[-stealth] (V) -- (VX00);
          \draw[-stealth] (V) -- (VX01);
          \draw[-stealth] (V) -- (VX10);

          \draw[dotted, -angle 60] (UX01) to[bend left=10] (VX00);
          \draw[dashed, -angle 60] (U) to[bend left=10] (VX01);
          \draw[dashed, -angle 60] (U) to[bend left=15] (VX10);
          \draw[dashed, -angle 60] (U) to[bend left=5] (V);
        \end{scope}
      \end{tikzpicture}
      \\[\baselineskip]
      % PICTURE 5, B2
      \begin{tikzpicture}[
          scale=0.95,
          defnode/.style={rectangle,draw,inner sep=4pt,outer sep=0pt,minimum size=15pt},
          mydot/.style={circle,fill,inner sep=0.5pt},
          nonleaf/.style={defnode,fill=lightgray},
          leaf/.style={defnode,fill=leafgreen},
          toggle/.style={defnode,fill=leafgreen}
        ]
        \node[scale=0.6] at (5, -1.3) {$\Theta_{B2} \colon \ag(1,B) \to \const(\{00,01,11\})(\{00,01,11\} \leadsto \nonleafs)$};
        \begin{scope}[shift={(0,0)}]
          \node[defnode,scale=0.6] (U) at (0, 0) {$\diamond$};
          \node[leaf,scale=0.6]    (UX00) at ($(U) + (225:1)$) {$X1$};
          \node[leaf,scale=0.6]    (UX01) at ($(U) + (135:1)$) {$X2$};
          \node[leaf,scale=0.6]    (UX10) at ($(U) + (315:1)$) {$Y1$};
          \node[leaf,scale=0.6]    (UX11) at ($(U) + (045:1)$) {$Y2$};
          \draw[-stealth] (U) -- (UX00);
          \draw[-stealth] (U) -- (UX01);
          \draw[-stealth] (U) -- (UX10);
          \draw[-stealth] (U) -- (UX11);

          \node[defnode,scale=0.6] (V) at (3, 0.5) {$\diamond$};
          \node[nonleaf,scale=0.6]    (VX00) at ($(V) + (-1,-1)$) {$00$};
          \node[nonleaf,scale=0.6]    (VX01) at ($(V) + ( 0,-1)$) {$01$};
          \node[nonleaf,scale=0.6]    (VX10) at ($(V) + ( 1,-1)$) {$11$};
          \draw[-stealth] (V) -- (VX00);
          \draw[-stealth] (V) -- (VX01);
          \draw[-stealth] (V) -- (VX10);

          \draw[dashed, -angle 60] (U) to[bend right=10] (VX00);
          \draw[dashed, -angle 60] (U) to[bend left=10] (VX01);
          \draw[dashed, -angle 60] (U) to[bend left=15] (VX10);
          \draw[dashed, -angle 60] (U) to[bend left=5] (V);
        \end{scope}
        \begin{scope}[shift={(6,0)}]
          \node[defnode,scale=0.6] (U) at (0, 0) {$(x,y)$};
          \node[leaf,scale=0.6]    (UX00) at ($(U) + (225:1)$) {$x$};
          \node[leaf,scale=0.6]    (UX01) at ($(U) + (135:1)$) {$y$};
          \node[leaf,scale=0.6]    (UX10) at ($(U) + (315:1)$) {$x+2y$};
          \node[leaf,scale=0.6]    (UX11) at ($(U) + (045:1)$) {$2y$};
          \draw[-stealth] (U) -- (UX00);
          \draw[-stealth] (U) -- (UX01);
          \draw[-stealth] (U) -- (UX10);
          \draw[-stealth] (U) -- (UX11);

          \node[defnode,scale=0.6] (V) at (3, 0.5) {$y$};
          \node[nonleaf,scale=0.6]    (VX00) at ($(V) + (-1,-1)$) {$y$};
          \node[nonleaf,scale=0.6]    (VX01) at ($(V) + ( 0,-1)$) {$y$};
          \node[nonleaf,scale=0.6]    (VX10) at ($(V) + ( 1,-1)$) {$y$};
          \draw[-stealth] (V) -- (VX00);
          \draw[-stealth] (V) -- (VX01);
          \draw[-stealth] (V) -- (VX10);

          \draw[dashed, -angle 60] (U) to[bend right=10] (VX00);
          \draw[dashed, -angle 60] (U) to[bend left=10] (VX01);
          \draw[dashed, -angle 60] (U) to[bend left=15] (VX10);
          \draw[dashed, -angle 60] (U) to[bend left=5] (V);
        \end{scope}
      \end{tikzpicture}
      \\[\baselineskip]
      % PICTURE 6, B3
      \begin{tikzpicture}[
          scale=0.95,
          defnode/.style={rectangle,draw,inner sep=4pt,outer sep=0pt,minimum size=15pt},
          mydot/.style={circle,fill,inner sep=0.5pt},
          nonleaf/.style={defnode,fill=lightgray},
          leaf/.style={defnode,fill=leafgreen},
          toggle/.style={defnode,fill=leafgreen}
        ]
        \node[scale=0.6] at (5, -1.3) {$\Theta_{B3} \colon \ag(1,B) \to \opsum(\{00,01,11\})(\{00,01,11\} \leadsto \nonleafs)$};
        \begin{scope}[shift={(0,0)}]
          \node[defnode,scale=0.6] (U) at (0, 0) {$\diamond$};
          \node[leaf,scale=0.6]    (UX00) at ($(U) + (225:1)$) {$X1$};
          \node[leaf,scale=0.6]    (UX01) at ($(U) + (135:1)$) {$X2$};
          \node[leaf,scale=0.6]    (UX10) at ($(U) + (315:1)$) {$Y1$};
          \node[leaf,scale=0.6]    (UX11) at ($(U) + (045:1)$) {$Y2$};
          \draw[-stealth] (U) -- (UX00);
          \draw[-stealth] (U) -- (UX01);
          \draw[-stealth] (U) -- (UX10);
          \draw[-stealth] (U) -- (UX11);

          \node[defnode,scale=0.6] (V) at (3, 0.5) {$\diamond$};
          \node[nonleaf,scale=0.6]    (VX00) at ($(V) + (-1,-1)$) {$00$};
          \node[nonleaf,scale=0.6]    (VX01) at ($(V) + ( 0,-1)$) {$01$};
          \node[nonleaf,scale=0.6]    (VX10) at ($(V) + ( 1,-1)$) {$11$};
          \draw[-stealth] (V) -- (VX00);
          \draw[-stealth] (V) -- (VX01);
          \draw[-stealth] (V) -- (VX10);

          \draw[dashed, -angle 60] (U) to[bend right=10] (VX00);
          \draw[dashed, -angle 60] (U) to[bend left=10] (VX01);
          \draw[dashed, -angle 60] (U) to[bend left=15] (VX10);
          \draw[dashed, -angle 60] (U) to[bend left=5] (V);
        \end{scope}
        \begin{scope}[shift={(6,0)}]
          \node[defnode,scale=0.6] (U) at (0, 0) {$(x,y)$};
          \node[leaf,scale=0.6]    (UX00) at ($(U) + (225:1)$) {$x$};
          \node[leaf,scale=0.6]    (UX01) at ($(U) + (135:1)$) {$y$};
          \node[leaf,scale=0.6]    (UX10) at ($(U) + (315:1)$) {$x+2y$};
          \node[leaf,scale=0.6]    (UX11) at ($(U) + (045:1)$) {$2y$};
          \draw[-stealth] (U) -- (UX00);
          \draw[-stealth] (U) -- (UX01);
          \draw[-stealth] (U) -- (UX10);
          \draw[-stealth] (U) -- (UX11);

          \node[defnode,scale=0.6] (V) at (3, 0.5) {$(y,-y,y)$};
          \node[nonleaf,scale=0.6]    (VX00) at ($(V) + (-1,-1)$) {$y$};
          \node[nonleaf,scale=0.6]    (VX01) at ($(V) + ( 0,-1)$) {$2y$};
          \node[nonleaf,scale=0.6]    (VX10) at ($(V) + ( 1,-1)$) {$y$};
          \draw[-stealth] (V) -- (VX00);
          \draw[-stealth] (V) -- (VX01);
          \draw[-stealth] (V) -- (VX10);

          \draw[dashed, -angle 60] (U) to[bend right=10] (VX00);
          \draw[dashed, -angle 60] (U) to[bend left=10] (VX01);
          \draw[dashed, -angle 60] (U) to[bend left=15] (VX10);
          \draw[dashed, -angle 60] (U) to[bend left=5] (V);
        \end{scope}
      \end{tikzpicture}
      \\[\baselineskip]
      % PICTURE 7, B4
      \begin{tikzpicture}[
          scale=0.95,
          defnode/.style={rectangle,draw,inner sep=4pt,outer sep=0pt,minimum size=15pt},
          mydot/.style={circle,fill,inner sep=0.5pt},
          nonleaf/.style={defnode,fill=lightgray},
          leaf/.style={defnode,fill=leafgreen},
          toggle/.style={defnode,fill=leafgreen}
        ]
        \node[scale=0.6] at (5, -1.3) {$\Theta_{B4} \colon \ag(1,B) \to \const(\{00,01,10\})(\{01,10\} \leadsto \nonleafs)$};
        \begin{scope}[shift={(0,0)}]
          \node[defnode,scale=0.6] (U) at (0, 0) {$\diamond$};
          \node[leaf,scale=0.6]    (UX00) at ($(U) + (225:1)$) {$X1$};
          \node[leaf,scale=0.6]    (UX01) at ($(U) + (135:1)$) {$X2$};
          \node[leaf,scale=0.6]    (UX10) at ($(U) + (315:1)$) {$Y1$};
          \node[leaf,scale=0.6]    (UX11) at ($(U) + (045:1)$) {$Y2$};
          \draw[-stealth] (U) -- (UX00);
          \draw[-stealth] (U) -- (UX01);
          \draw[-stealth] (U) -- (UX10);
          \draw[-stealth] (U) -- (UX11);

          \node[defnode,scale=0.6] (V) at (3, 0.5) {$\diamond$};
          \node[leaf,scale=0.6]    (VX00) at ($(V) + (-1,-1)$) {$00$};
          \node[nonleaf,scale=0.6]    (VX01) at ($(V) + ( 0,-1)$) {$01$};
          \node[nonleaf,scale=0.6]    (VX10) at ($(V) + ( 1,-1)$) {$10$};
          \draw[-stealth] (V) -- (VX00);
          \draw[-stealth] (V) -- (VX01);
          \draw[-stealth] (V) -- (VX10);

          \draw[dotted, -angle 60] (UX11) to[bend left=10] (VX00);
          \draw[dashed, -angle 60] (U) to[bend left=10] (VX01);
          \draw[dashed, -angle 60] (U) to[bend left=15] (VX10);
          \draw[dashed, -angle 60] (U) to[bend left=5] (V);
        \end{scope}
        \begin{scope}[shift={(6,0)}]
          \node[defnode,scale=0.6] (U) at (0, 0) {$(x,y)$};
          \node[leaf,scale=0.6]    (UX00) at ($(U) + (225:1)$) {$x$};
          \node[leaf,scale=0.6]    (UX01) at ($(U) + (135:1)$) {$y$};
          \node[leaf,scale=0.6]    (UX10) at ($(U) + (315:1)$) {$x+2y$};
          \node[leaf,scale=0.6]    (UX11) at ($(U) + (045:1)$) {$2y$};
          \draw[-stealth] (U) -- (UX00);
          \draw[-stealth] (U) -- (UX01);
          \draw[-stealth] (U) -- (UX10);
          \draw[-stealth] (U) -- (UX11);

          \node[defnode,scale=0.6] (V) at (3, 0.5) {$2y$};
          \node[leaf,scale=0.6]    (VX00) at ($(V) + (-1,-1)$) {$2y$};
          \node[nonleaf,scale=0.6]    (VX01) at ($(V) + ( 0,-1)$) {$2y$};
          \node[nonleaf,scale=0.6]    (VX10) at ($(V) + ( 1,-1)$) {$2y$};
          \draw[-stealth] (V) -- (VX00);
          \draw[-stealth] (V) -- (VX01);
          \draw[-stealth] (V) -- (VX10);

          \draw[dotted, -angle 60] (UX11) to[bend left=10] (VX00);
          \draw[dashed, -angle 60] (U) to[bend left=10] (VX01);
          \draw[dashed, -angle 60] (U) to[bend left=15] (VX10);
          \draw[dashed, -angle 60] (U) to[bend left=5] (V);
        \end{scope}
      \end{tikzpicture}
    \end{center}
    \caption{The eight morphisms from $\ag(1,B)$ to be composed with the assignment morphisms $\cM_2$ for gates A1,\dots,B4.}%
    \label{fig:eight-morphisms}
  \end{figure}
  
  Finally, we apply Lemma~\ref{lem:patching} to the collection of morphisms $\Theta_\ell \circ \cM_2^\ell$ to construct a single morphism $\cM_2 \colon \ag(1,B) \to \smagate[\nonleafs \cup \pins \cup S_2]$.
  It is clear by inspection and Figure~\ref{fig:agate-1b} that the compatibility hypothesis of Lemma~\ref{lem:patching} is satisfied: for example,
  \[
    \theta^{\Theta_{A4} \circ \cM_2^{A4}}_{10;\triangle} = \theta^{\Theta_{B4} \circ \cM_2^{B4}}_{10;\triangle} = \bigl( (x,y) \mapsto 2y \bigr).
  \]

  To complete the proof for $i=1$, it remains to construct $\cM_0$, $\cM_1$ and $\cM_3,\dots,\cM_s$.
  This time, we will content ourselves with giving the analogous pictures to Figure~\ref{fig:agate-1b} in each case, since it is then mechanical to adapt the argument above (for $\cM_2$ when $i=1$) to these other cases.
  These pictures are given in Figure~\ref{fig:agate-1aplus}.
  The construction of $\cM_4,\dots,\cM_s$ is completely analogous to the case $\cM_3$ shown.

  \begin{figure}[htbp]
    \begin{center}
      \begin{tikzpicture}[
          defnode/.style={rectangle,fill=lightgray,draw,inner sep=2pt,outer sep=0pt,minimum size=10pt},
          gatelabel/.style={rectangle, inner sep=2pt, outer sep=0pt, fill=white,scale=0.6},
          subnode/.style={rounded rectangle, draw, inner sep=5pt, outer sep=0pt,minimum size=15pt},
          leaf/.style={rectangle,draw,fill=leafgreen,inner sep=2pt,outer sep=0pt,minimum size=10pt},
          scale=0.8
        ]
        \node[scale=0.8] at (-2,1.5) {$\cM_0:$};

        \node[subnode,scale=0.8] (A1) at (0,0)    {$\trivial$};
        \node[subnode,scale=0.8] (A2) at (3.5,0)  {$\trivial$};
        \node[subnode,scale=0.8] (A3) at (7,0)    {$\trivial$};
        \node[subnode,scale=0.8] (A4) at (10.5,0) {$\trivial$};
        \node[subnode,scale=0.8] (B1) at (0,3)    {$\trivial$};
        \node[subnode,scale=0.8] (B2) at (3.5,3)  {$\trivial$};
        \node[subnode,scale=0.8] (B3) at (7,3)    {$\trivial$};
        \node[subnode,scale=0.8] (B4) at (10.5,3) {$\trivial$};

        \node[leaf,scale=0.7] (X1) at ($(A1)+(-2,0)$) {$x$};
        \node[leaf,scale=0.7] (X2) at ($(B1)+(-2,0)$) {$y$};
        \node[leaf,scale=0.7] (Y1) at ($(A4)+(2,0)$)  {$z$};
        \node[leaf,scale=0.7] (Y2) at ($(B4)+(2,0)$)  {$w$};

        \node[defnode,scale=0.7] (A12) at ($(A1)!0.5!(A2)$) {$0$};
        \node[defnode,scale=0.7] (A23) at ($(A2)!0.5!(A3)$) {$0$};
        \node[defnode,scale=0.7] (A34) at ($(A3)!0.5!(A4)$) {$0$};

        \node[defnode,scale=0.7] (B12) at ($(B1)!0.5!(B2)$) {$0$};
        \node[defnode,scale=0.7] (B23a) at ($(B2)!0.5!(B3) + (0,0.5)$) {$0$};
        \node[defnode,scale=0.7] (B23b) at ($(B2)!0.5!(B3) + (0,-0.5)$) {$0$};
        \node[defnode,scale=0.7] (B34) at ($(B3)!0.5!(B4)$) {$0$};

        \node[defnode,scale=0.7] (AB1) at ($(A1)!0.5!(B1)$) {$0$};
        \node[defnode,scale=0.7] (AB4) at ($(A4)!0.5!(B4)$) {$0$};

        \draw (A1) -- node[gatelabel, pos=0.10] {$00$} (X1);
        \draw (A1) -- node[gatelabel, pos=0.12] {$01$} (A12);

        \draw (A2) -- node[gatelabel, pos=0.12] {$01$} (A12);
        \draw (A2) -- node[gatelabel, pos=0.12] {$10$} (A23);

        \draw (A3) -- node[gatelabel, pos=0.12] {$10$} (A23);
        \draw (A3) -- node[gatelabel, pos=0.12] {$01$} (A34);

        \draw (A4) -- node[gatelabel, pos=0.12] {$00$} (Y1);
        \draw (A4) -- node[gatelabel, pos=0.12] {$01$} (A34);

        \draw (B1) -- node[gatelabel, pos=0.10] {$00$} (X2);
        \draw (B1) -- node[gatelabel, pos=0.12] {$01$} (B12);

        \draw (B2) -- node[gatelabel, pos=0.12] {$01$} (B12);
        \draw (B2) -- node[gatelabel, pos=0.12] {$00$} (B23a);
        \draw (B2) -- node[gatelabel, pos=0.12] {$11$} (B23b);

        \draw (B3) -- node[gatelabel, pos=0.12] {$01$} (B34);
        \draw (B3) -- node[gatelabel, pos=0.12] {$00$} (B23a);
        \draw (B3) -- node[gatelabel, pos=0.12] {$11$} (B23b);

        \draw (B4) -- node[gatelabel, pos=0.12] {$00$} (Y2);
        \draw (B4) -- node[gatelabel, pos=0.12] {$01$} (B34);

        \draw (A1) -- node[gatelabel, pos=0.08] {$10$} (AB1);
        \draw (A4) -- node[gatelabel, pos=0.08] {$10$} (AB4);
        \draw (B1) -- node[gatelabel, pos=0.08] {$10$} (AB1);
        \draw (B4) -- node[gatelabel, pos=0.08] {$10$} (AB4);
      \end{tikzpicture}
      \\[\baselineskip]
      \vspace{\baselineskip}
      \begin{tikzpicture}[
          defnode/.style={rectangle,fill=lightgray,draw,inner sep=2pt,outer sep=0pt,minimum size=10pt},
          gatelabel/.style={rectangle, inner sep=2pt, outer sep=0pt, fill=white,scale=0.6},
          subnode/.style={rounded rectangle, draw, inner sep=5pt, outer sep=0pt,minimum size=15pt},
          leaf/.style={rectangle,draw,fill=leafgreen,inner sep=2pt,outer sep=0pt,minimum size=10pt},
          scale=0.8
        ]
        \node[scale=0.8] at (-2,1.5) {$\cM_1:$};

        \node[subnode,scale=0.8] (A1) at (0,0)    {$\crs^{01}_{\overline{11}}$};
        \node[subnode,scale=0.8] (A2) at (3.5,0)  {$\const$};
        \node[subnode,scale=0.8] (A3) at (7,0)    {$\const$};
        \node[subnode,scale=0.8] (A4) at (10.5,0) {$\const$};
        \node[subnode,scale=0.8] (B1) at (0,3)    {$\crs^{01}_{\overline{11}}$};
        \node[subnode,scale=0.8] (B2) at (3.5,3)  {$\opsum$};
        \node[subnode,scale=0.8] (B3) at (7,3)    {$\const$};
        \node[subnode,scale=0.8] (B4) at (10.5,3) {$\opsum$};

        \node[leaf,scale=0.7] (X1) at ($(A1)+(-2,0)$) {$x$};
        \node[leaf,scale=0.7] (X2) at ($(B1)+(-2,0)$) {$2y-2x$};
        \node[leaf,scale=0.7] (Y1) at ($(A4)+(2,0)$)  {$x$};
        \node[leaf,scale=0.7] (Y2) at ($(B4)+(2,0)$)  {$y$};

        \node[defnode,scale=0.7] (A12) at ($(A1)!0.5!(A2)$) {$x$};
        \node[defnode,scale=0.7] (A23) at ($(A2)!0.5!(A3)$) {$x$};
        \node[defnode,scale=0.7] (A34) at ($(A3)!0.5!(A4)$) {$x$};

        \node[defnode,scale=0.7] (B12) at ($(B1)!0.5!(B2)$) {$2y-2x$};
        \node[defnode,scale=0.7] (B23a) at ($(B2)!0.5!(B3) + (0,0.5)$) {$y-x$};
        \node[defnode,scale=0.7] (B23b) at ($(B2)!0.5!(B3) + (0,-0.5)$) {$y-x$};
        \node[defnode,scale=0.7] (B34) at ($(B3)!0.5!(B4)$) {$y-x$};

        \node[defnode,scale=0.7] (AB1) at ($(A1)!0.5!(B1)$) {$0$};
        \node[defnode,scale=0.7] (AB4) at ($(A4)!0.5!(B4)$) {$x$};

        \draw (A1) -- node[gatelabel, pos=0.10] {$00$} (X1);
        \draw (A1) -- node[gatelabel, pos=0.12] {$01$} (A12);

        \draw (A2) -- node[gatelabel, pos=0.12] {$01$} (A12);
        \draw (A2) -- node[gatelabel, pos=0.12] {$10$} (A23);

        \draw (A3) -- node[gatelabel, pos=0.12] {$10$} (A23);
        \draw (A3) -- node[gatelabel, pos=0.12] {$01$} (A34);

        \draw (A4) -- node[gatelabel, pos=0.12] {$00$} (Y1);
        \draw (A4) -- node[gatelabel, pos=0.12] {$01$} (A34);

        \draw (B1) -- node[gatelabel, pos=0.10] {$00$} (X2);
        \draw (B1) -- node[gatelabel, pos=0.12] {$01$} (B12);

        \draw (B2) -- node[gatelabel, pos=0.12] {$01$} (B12);
        \draw (B2) -- node[gatelabel, pos=0.12] {$00$} (B23a);
        \draw (B2) -- node[gatelabel, pos=0.12] {$11$} (B23b);

        \draw (B3) -- node[gatelabel, pos=0.12] {$01$} (B34);
        \draw (B3) -- node[gatelabel, pos=0.12] {$00$} (B23a);
        \draw (B3) -- node[gatelabel, pos=0.12] {$11$} (B23b);

        \draw (B4) -- node[gatelabel, pos=0.12] {$00$} (Y2);
        \draw (B4) -- node[gatelabel, pos=0.12] {$01$} (B34);

        \draw (A1) -- node[gatelabel, pos=0.08] {$10$} (AB1);
        \draw (A4) -- node[gatelabel, pos=0.08] {$10$} (AB4);
        \draw (B1) -- node[gatelabel, pos=0.08] {$10$} (AB1);
        \draw (B4) -- node[gatelabel, pos=0.08] {$10$} (AB4);
      \end{tikzpicture}
      \\[\baselineskip]
      \vspace{\baselineskip}
      \begin{tikzpicture}[
          defnode/.style={rectangle,fill=lightgray,draw,inner sep=2pt,outer sep=0pt,minimum size=10pt},
          gatelabel/.style={rectangle, inner sep=2pt, outer sep=0pt, fill=white,scale=0.6},
          subnode/.style={rounded rectangle, draw, inner sep=5pt, outer sep=0pt,minimum size=15pt},
          leaf/.style={rectangle,draw,fill=leafgreen,inner sep=2pt,outer sep=0pt,minimum size=10pt},
          scale=0.8
        ]
        \node[scale=0.8] at (-2,1.5) {$\cM_3$};

        \node[subnode,scale=0.8] (A1) at (0,0)    {$\const$};
        \node[subnode,scale=0.8] (A2) at (3.5,0)  {$\const$};
        \node[subnode,scale=0.8] (A3) at (7,0)    {$\const$};
        \node[subnode,scale=0.8] (A4) at (10.5,0) {$\const$};
        \node[subnode,scale=0.8] (B1) at (0,3)    {$\const$};
        \node[subnode,scale=0.8] (B2) at (3.5,3)  {$\const$};
        \node[subnode,scale=0.8] (B3) at (7,3)    {$\const$};
        \node[subnode,scale=0.8] (B4) at (10.5,3) {$\const$};

        \node[leaf,scale=0.7] (X1) at ($(A1)+(-2,0)$) {$x$};
        \node[leaf,scale=0.7] (X2) at ($(B1)+(-2,0)$) {$x$};
        \node[leaf,scale=0.7] (Y1) at ($(A4)+(2,0)$)  {$x$};
        \node[leaf,scale=0.7] (Y2) at ($(B4)+(2,0)$)  {$x$};

        \node[defnode,scale=0.7] (A12) at ($(A1)!0.5!(A2)$) {$x$};
        \node[defnode,scale=0.7] (A23) at ($(A2)!0.5!(A3)$) {$x$};
        \node[defnode,scale=0.7] (A34) at ($(A3)!0.5!(A4)$) {$x$};

        \node[defnode,scale=0.7] (B12) at ($(B1)!0.5!(B2)$) {$x$};
        \node[defnode,scale=0.7] (B23a) at ($(B2)!0.5!(B3) + (0,0.5)$) {$x$};
        \node[defnode,scale=0.7] (B23b) at ($(B2)!0.5!(B3) + (0,-0.5)$) {$x$};
        \node[defnode,scale=0.7] (B34) at ($(B3)!0.5!(B4)$) {$x$};

        \node[defnode,scale=0.7] (AB1) at ($(A1)!0.5!(B1)$) {$x$};
        \node[defnode,scale=0.7] (AB4) at ($(A4)!0.5!(B4)$) {$x$};

        \draw (A1) -- node[gatelabel, pos=0.10] {$00$} (X1);
        \draw (A1) -- node[gatelabel, pos=0.12] {$01$} (A12);

        \draw (A2) -- node[gatelabel, pos=0.12] {$01$} (A12);
        \draw (A2) -- node[gatelabel, pos=0.12] {$10$} (A23);

        \draw (A3) -- node[gatelabel, pos=0.12] {$10$} (A23);
        \draw (A3) -- node[gatelabel, pos=0.12] {$01$} (A34);

        \draw (A4) -- node[gatelabel, pos=0.12] {$00$} (Y1);
        \draw (A4) -- node[gatelabel, pos=0.12] {$01$} (A34);

        \draw (B1) -- node[gatelabel, pos=0.10] {$00$} (X2);
        \draw (B1) -- node[gatelabel, pos=0.12] {$01$} (B12);

        \draw (B2) -- node[gatelabel, pos=0.12] {$01$} (B12);
        \draw (B2) -- node[gatelabel, pos=0.12] {$00$} (B23a);
        \draw (B2) -- node[gatelabel, pos=0.12] {$11$} (B23b);

        \draw (B3) -- node[gatelabel, pos=0.12] {$01$} (B34);
        \draw (B3) -- node[gatelabel, pos=0.12] {$00$} (B23a);
        \draw (B3) -- node[gatelabel, pos=0.12] {$11$} (B23b);

        \draw (B4) -- node[gatelabel, pos=0.12] {$00$} (Y2);
        \draw (B4) -- node[gatelabel, pos=0.12] {$01$} (B34);

        \draw (A1) -- node[gatelabel, pos=0.08] {$10$} (AB1);
        \draw (A4) -- node[gatelabel, pos=0.08] {$10$} (AB4);
        \draw (B1) -- node[gatelabel, pos=0.08] {$10$} (AB1);
        \draw (B4) -- node[gatelabel, pos=0.08] {$10$} (AB4);
      \end{tikzpicture}
    \end{center}
    \caption{A pictorial description of the morphisms $\cM_0$, $\cM_1$ and $\cM_3$ for the assignment $\Middle(1)$.}%
    \label{fig:agate-1aplus}
  \end{figure}

  When $i=0$, we similarly just show the revised pictures for $\cM_1$ and $\cM_2$, since those for $\cM_0$ and $\cM_3,\dots,\cM_s$ are unchanged.
  These are given in Figure~\ref{fig:agate-0}.
  \begin{figure}[htbp]
    \begin{center}
      \begin{tikzpicture}[
          defnode/.style={rectangle,fill=lightgray,draw,inner sep=2pt,outer sep=0pt,minimum size=10pt},
          gatelabel/.style={rectangle, inner sep=2pt, outer sep=0pt, fill=white,scale=0.6},
          subnode/.style={rounded rectangle, draw, inner sep=5pt, outer sep=0pt,minimum size=15pt},
          leaf/.style={rectangle,draw,fill=leafgreen,inner sep=2pt,outer sep=0pt,minimum size=10pt},
          scale=0.8
        ]
        \node[scale=0.8] at (-2,1.5) {$\cM_1$};

        \node[subnode,scale=0.8] (A1) at (0,0)    {$\crs^{01}_{\overline{11}}$};
        \node[subnode,scale=0.8] (A2) at (3.5,0)  {$\const$};
        \node[subnode,scale=0.8] (A3) at (7,0)    {$\const$};
        \node[subnode,scale=0.8] (A4) at (10.5,0) {$\crs^{01}_{\strikeout{11}}$};
        \node[subnode,scale=0.8] (B1) at (0,3)    {$\crs^{01}_{\overline{11}}$};
        \node[subnode,scale=0.8] (B2) at (3.5,3)  {$\opsum$};
        \node[subnode,scale=0.8] (B3) at (7,3)    {$\const$};
        \node[subnode,scale=0.8] (B4) at (10.5,3) {$\opsum$};

        \node[leaf,scale=0.7] (X1) at ($(A1)+(-2,0)$) {$x$};
        \node[leaf,scale=0.7] (X2) at ($(B1)+(-2,0)$) {$2y$};
        \node[leaf,scale=0.7] (Y1) at ($(A4)+(2,0)$)  {$x$};
        \node[leaf,scale=0.7] (Y2) at ($(B4)+(2,0)$)  {$y$};

        \node[defnode,scale=0.7] (A12) at ($(A1)!0.5!(A2)$) {$x$};
        \node[defnode,scale=0.7] (A23) at ($(A2)!0.5!(A3)$) {$x$};
        \node[defnode,scale=0.7] (A34) at ($(A3)!0.5!(A4)$) {$x$};

        \node[defnode,scale=0.7] (B12) at ($(B1)!0.5!(B2)$) {$2y$};
        \node[defnode,scale=0.7] (B23a) at ($(B2)!0.5!(B3) + (0,0.5)$) {$y$};
        \node[defnode,scale=0.7] (B23b) at ($(B2)!0.5!(B3) + (0,-0.5)$) {$y$};
        \node[defnode,scale=0.7] (B34) at ($(B3)!0.5!(B4)$) {$y$};

        \node[defnode,scale=0.7] (AB1) at ($(A1)!0.5!(B1)$) {$0$};
        \node[defnode,scale=0.7] (AB4) at ($(A4)!0.5!(B4)$) {$0$};

        \draw (A1) -- node[gatelabel, pos=0.10] {$00$} (X1);
        \draw (A1) -- node[gatelabel, pos=0.12] {$01$} (A12);

        \draw (A2) -- node[gatelabel, pos=0.12] {$01$} (A12);
        \draw (A2) -- node[gatelabel, pos=0.12] {$10$} (A23);

        \draw (A3) -- node[gatelabel, pos=0.12] {$10$} (A23);
        \draw (A3) -- node[gatelabel, pos=0.12] {$01$} (A34);

        \draw (A4) -- node[gatelabel, pos=0.12] {$00$} (Y1);
        \draw (A4) -- node[gatelabel, pos=0.12] {$01$} (A34);

        \draw (B1) -- node[gatelabel, pos=0.10] {$00$} (X2);
        \draw (B1) -- node[gatelabel, pos=0.12] {$01$} (B12);

        \draw (B2) -- node[gatelabel, pos=0.12] {$01$} (B12);
        \draw (B2) -- node[gatelabel, pos=0.12] {$00$} (B23a);
        \draw (B2) -- node[gatelabel, pos=0.12] {$11$} (B23b);

        \draw (B3) -- node[gatelabel, pos=0.12] {$01$} (B34);
        \draw (B3) -- node[gatelabel, pos=0.12] {$00$} (B23a);
        \draw (B3) -- node[gatelabel, pos=0.12] {$11$} (B23b);

        \draw (B4) -- node[gatelabel, pos=0.12] {$00$} (Y2);
        \draw (B4) -- node[gatelabel, pos=0.12] {$01$} (B34);

        \draw (A1) -- node[gatelabel, pos=0.08] {$10$} (AB1);
        \draw (A4) -- node[gatelabel, pos=0.08] {$10$} (AB4);
        \draw (B1) -- node[gatelabel, pos=0.08] {$10$} (AB1);
        \draw (B4) -- node[gatelabel, pos=0.08] {$10$} (AB4);
      \end{tikzpicture}
      \\[\baselineskip]
      \vspace{\baselineskip}
      \begin{tikzpicture}[
          defnode/.style={rectangle,fill=lightgray,draw,inner sep=2pt,outer sep=0pt,minimum size=10pt},
          gatelabel/.style={rectangle, inner sep=2pt, outer sep=0pt, fill=white,scale=0.6},
          subnode/.style={rounded rectangle, draw, inner sep=5pt, outer sep=0pt,minimum size=15pt},
          leaf/.style={rectangle,draw,fill=leafgreen,inner sep=2pt,outer sep=0pt,minimum size=10pt},
          scale=0.8
        ]
        \node[scale=0.8] at (-2,1.5) {$\cM_2$};

        \node[subnode,scale=0.8] (A1) at (0,0)    {$\crs^{01}_{\strikeout{11}}$};
        \node[subnode,scale=0.8] (A2) at (3.5,0)  {$\opsum$};
        \node[subnode,scale=0.8] (A3) at (7,0)    {$\opsum$};
        \node[subnode,scale=0.8] (A4) at (10.5,0) {$\crs^{01}_{\overline{11}}$};
        \node[subnode,scale=0.8] (B1) at (0,3)    {$\crs^{01}_{\strikeout{11}}$};
        \node[subnode,scale=0.8] (B2) at (3.5,3)  {$\const$};
        \node[subnode,scale=0.8] (B3) at (7,3)    {$\opsum$};
        \node[subnode,scale=0.8] (B4) at (10.5,3) {$\const$};

        \node[leaf,scale=0.7] (X1) at ($(A1)+(-2,0)$) {$x$};
        \node[leaf,scale=0.7] (X2) at ($(B1)+(-2,0)$) {$y$};
        \node[leaf,scale=0.7] (Y1) at ($(A4)+(2,0)$)  {$x$};
        \node[leaf,scale=0.7] (Y2) at ($(B4)+(2,0)$)  {$2y$};

        \node[defnode,scale=0.7] (A12) at ($(A1)!0.5!(A2)$) {$x$};
        \node[defnode,scale=0.7] (A23) at ($(A2)!0.5!(A3)$) {$-x$};
        \node[defnode,scale=0.7] (A34) at ($(A3)!0.5!(A4)$) {$x$};

        \node[defnode,scale=0.7] (B12) at ($(B1)!0.5!(B2)$) {$y$};
        \node[defnode,scale=0.7] (B23a) at ($(B2)!0.5!(B3) + (0,0.5)$) {$y$};
        \node[defnode,scale=0.7] (B23b) at ($(B2)!0.5!(B3) + (0,-0.5)$) {$y$};
        \node[defnode,scale=0.7] (B34) at ($(B3)!0.5!(B4)$) {$2y$};

        \node[defnode,scale=0.7] (AB1) at ($(A1)!0.5!(B1)$) {$0$};
        \node[defnode,scale=0.7] (AB4) at ($(A4)!0.5!(B4)$) {$2y$};

        \draw (A1) -- node[gatelabel, pos=0.10] {$00$} (X1);
        \draw (A1) -- node[gatelabel, pos=0.12] {$01$} (A12);

        \draw (A2) -- node[gatelabel, pos=0.12] {$01$} (A12);
        \draw (A2) -- node[gatelabel, pos=0.12] {$10$} (A23);

        \draw (A3) -- node[gatelabel, pos=0.12] {$10$} (A23);
        \draw (A3) -- node[gatelabel, pos=0.12] {$01$} (A34);

        \draw (A4) -- node[gatelabel, pos=0.12] {$00$} (Y1);
        \draw (A4) -- node[gatelabel, pos=0.12] {$01$} (A34);

        \draw (B1) -- node[gatelabel, pos=0.10] {$00$} (X2);
        \draw (B1) -- node[gatelabel, pos=0.12] {$01$} (B12);

        \draw (B2) -- node[gatelabel, pos=0.12] {$01$} (B12);
        \draw (B2) -- node[gatelabel, pos=0.12] {$00$} (B23a);
        \draw (B2) -- node[gatelabel, pos=0.12] {$11$} (B23b);

        \draw (B3) -- node[gatelabel, pos=0.12] {$01$} (B34);
        \draw (B3) -- node[gatelabel, pos=0.12] {$00$} (B23a);
        \draw (B3) -- node[gatelabel, pos=0.12] {$11$} (B23b);

        \draw (B4) -- node[gatelabel, pos=0.12] {$00$} (Y2);
        \draw (B4) -- node[gatelabel, pos=0.12] {$01$} (B34);

        \draw (A1) -- node[gatelabel, pos=0.08] {$10$} (AB1);
        \draw (A4) -- node[gatelabel, pos=0.08] {$10$} (AB4);
        \draw (B1) -- node[gatelabel, pos=0.08] {$10$} (AB1);
        \draw (B4) -- node[gatelabel, pos=0.08] {$10$} (AB4);
      \end{tikzpicture}
    \end{center}
    \caption{A description of the morphisms $\cM_1$ and $\cM_2$ for the assignment $\Middle(0)$.}%
    \label{fig:agate-0}
  \end{figure}
\end{proof}

This completes the picture for $\smagate_s$, but unfortunately we still have to design assignments for the gates $\smagate_s^\Alpha$, $\smagate_s^\Omega$ and $\smagate_s^{\Alpha,\Omega}$.
The case $\smagate_s^\Alpha$ is the most involved and we tackle this next.

\begin{definition}%
  \label{def:initial-data}
  Let $i,j \in \{0,1\}$ and $\fs \in \{\pm 1\}$.
  We define linear data $\ag(i,j,\fs,A)$ and $\ag(i,j,\fs,B)$ with $I = \{ X1, Y1, Y2 \}$, $W_\ell = \FF_p$ for $\ell \in I$, $V^{\ag(i,j,\fs,A)} = \FF_p^2$, $V^{\ag(i,j,\fs,B)} = \FF_p$, and with linear maps $V \to W_i$ given by
  \vspace{0.5\baselineskip}
  \begin{center}
    \begin{tikzpicture}[
        defnode/.style={rectangle,draw,inner sep=4pt,outer sep=0pt,minimum size=15pt},
        mydot/.style={circle,fill,inner sep=0.5pt},
        leaf/.style={defnode,fill=leafgreen},
        toggle/.style={defnode,fill=leafgreen}
      ]
      \begin{scope}[shift={(-5,0)}]
        \node[defnode,scale=0.8] (U) at (0, 0) {$\diamond$};
        \node[leaf,scale=0.6]    (UX00) at ($(U) + (225:1)$) {$X1$};
        \node[leaf,scale=0.6]    (UX10) at ($(U) + (315:1)$) {$Y1$};
        \node[leaf,scale=0.6]    (UX11) at ($(U) + (045:1)$) {$Y2$};
        \draw[-stealth] (U) -- (UX00);
        \draw[-stealth] (U) -- (UX10);
        \draw[-stealth] (U) -- (UX11);
      \end{scope}
      \begin{scope}[shift={(0,0)}]
        \node[defnode,scale=0.8] (U) at (0, 0) {$(x,y)$};
        \node[leaf,scale=0.6]    (UX00) at ($(U) + (225:1)$) {$\fs((i+2j)x - 2y)$};
        \node[leaf,scale=0.6]    (UX10) at ($(U) + (315:1)$) {$x$};
        \node[leaf,scale=0.6]    (UX11) at ($(U) + (045:1)$) {$y$};
        \draw[-stealth] (U) -- (UX00);
        \draw[-stealth] (U) -- (UX10);
        \draw[-stealth] (U) -- (UX11);

        \node["{$\ag(i,j,\fs,A)$}" {scale=0.9}] at (0, -2) {};
      \end{scope}
      \begin{scope}[shift={(3.5,0)}]
        \node[defnode,scale=0.8] (U) at (0, 0) {$x$};
        \node[leaf,scale=0.6]    (UX00) at ($(U) + (225:1)$) {$x$};
        \node[leaf,scale=0.6]    (UX10) at ($(U) + (315:1)$) {$\fs (i+2j)x$};
        \node[leaf,scale=0.6]    (UX11) at ($(U) + (045:1)$) {$2\fs x$};
        \draw[-stealth] (U) -- (UX00);
        \draw[-stealth] (U) -- (UX10);
        \draw[-stealth] (U) -- (UX11);

        \node["{$\ag(i,j,\fs,B)$}" {scale=0.9}] at (0, -2) {};
      \end{scope}
    \end{tikzpicture}
  \end{center}
\end{definition}

\begin{lemma}%
  \label{lem:agate-assign-2}
  Let $s \ge 2$.
  For any $i,j \in \{0,1\}$ and $\fs \in \{\pm 1\}$ there exists an assignment $\Initial(i,j,\fs)$ of $\smagate_s^{\Alpha}$ with
  \begin{itemize}
    \item $D_0 =\trivial(\{X1,Y1,Y2\})$;
    \item $D_1 = \ag(i,j,\fs,A)$;
    \item $D_2 = \ag(i,j,\fs,B)$; and
    \item $D_r = \const(\{X1,Y1,Y2\})$ for $3 \le r \le s$.
  \end{itemize}
\end{lemma}

This corresponds to a true identity in tensor algebra, namely
\[
  (\fs ((i+2j) x_1 - 2 y_1)) \otimes x_2 = x_1 \otimes (\fs (i+2j) x_2) - y_1 \otimes (2 \fs x).
\]
Again we motivate $\ag(i,j,\fs,A)$ and $\ag(i,j,\fs,B)$ with reference to Figure~\ref{fig:bilinear}, and in particular the first block.
The situation there has $i=1$, $j=0$ and $\fs=1$.

In the second coordinate, corresponding to $r=2$ and $\ag(i,j,\fs,B)$, we take $x=h$ as an input in the bottom left and seed the bottom ``accumulator'' rail with $(i+2j) h$ and the top ``powers of two'' rail with $2h$.
If $j=1$, the dotted line in Figure~\ref{fig:bilinear} would become another $(+,=)$, $(=,+)$ pair and the middle bottom rail would become $(h, 3h')$ accordingly.

In the first coordinate $r=1$, $\ag(i,j,\fs,A)$, we again reason from right to left: we take $y=-6h$ and $x=h$ as inputs in the top right and bottom right respectively, and output $\fs((i+2j)x - 2y) = 13h$ in the bottom left.

\begin{proof}%
  As in Lemma~\ref{lem:agate-assign-2}, we begin by specifying the sub-assignments on $A1,\dots,B4$.
  \begin{equation}
    \label{eq:agate-assignments-2}
    \begin{aligned}
      A1 &: 
      \begin{rcases}
        \begin{dcases}
          \opcross^{10;\overline{11}}_{1,2} &\colon i=0 \\
          \sumconst_1 &\colon i=1
        \end{dcases}
      \end{rcases} &
      B1 &: 
      \begin{rcases}
        \begin{dcases} 
          \sumconst_1 &\colon \fs = 1 \\ 
          \sumconst_2 &\colon \fs= -1
        \end{dcases}
      \end{rcases}   \\
      A2 &: 
      \begin{rcases}
        \begin{dcases} 
          \sumconst_2 &\colon \fs = 1 \\ 
          \sumconst_1 &\colon \fs= -1
        \end{dcases}
      \end{rcases}   &
      B2 &: \sumconst_1 \\
      A3 &: \sumconst_2 &
      B3 &: \sumconst_2 \\
      A4 &:
      \begin{rcases} 
        \begin{dcases} 
          \opcross^{01;\overline{11}}_{2,1} &\colon j=0  \\
          \sumconst_2 &\colon j=1
        \end{dcases} 
      \end{rcases} &                       
      B4 &: \sumconst_1   .
    \end{aligned}
  \end{equation}
  We now draw pictures do describe the morphisms $\cM_1$ and $\cM_2$ as in Lemma~\ref{lem:agate-assign-1}.
  By means of yet more creative notation we avoid drawing eight such pictures, one for each choice of $(i,j,\fs)$.
  These are shown in Figure~\ref{fig:agate-initial}.

  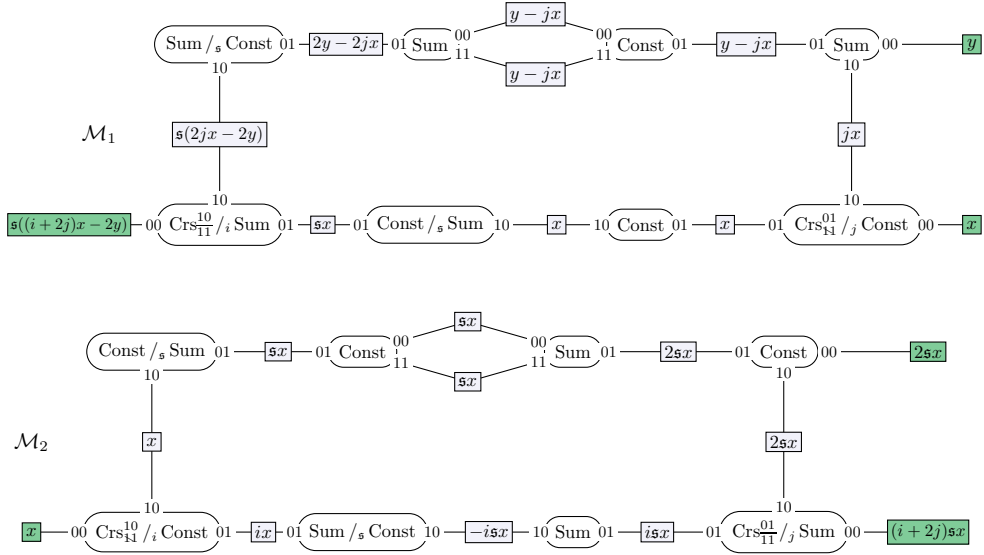
\begin{figure}[htbp]
    \begin{center}
      \begin{tikzpicture}[
          defnode/.style={rectangle,fill=lightgray,draw,inner sep=2pt,outer sep=0pt,minimum size=10pt},
          gatelabel/.style={rectangle, inner sep=2pt, outer sep=0pt, fill=white,scale=0.6},
          subnode/.style={rounded rectangle, draw, inner sep=5pt, outer sep=0pt,minimum size=15pt},
          leaf/.style={rectangle,draw,fill=leafgreen,inner sep=2pt,outer sep=0pt,minimum size=10pt},
          scale=0.8
        ]
        \node[scale=0.8] at (-2,1.5) {$\cM_1$};

        \node[subnode,scale=0.7] (A1) at (0,0)    {$\crs^{10}_{\overline{11}} /_i \opsum$};
        \node[subnode,scale=0.7] (A2) at (3.5,0)  {$\const /_{\fs} \opsum$};
        \node[subnode,scale=0.7] (A3) at (7,0)    {$\const$};
        \node[subnode,scale=0.7] (A4) at (10.5,0) {$\crs^{01}_{\strikeout{11}} /_j \const$};
        \node[subnode,scale=0.7] (B1) at (0,3)    {$\opsum /_\fs \const$};
        \node[subnode,scale=0.7] (B2) at (3.5,3)  {$\opsum$};
        \node[subnode,scale=0.7] (B3) at (7,3)    {$\const$};
        \node[subnode,scale=0.7] (B4) at (10.5,3) {$\opsum$};

        \node[leaf,scale=0.6] (X1) at ($(A1)+(-2.5,0)$) {$\fs((i+2j)x - 2y)$};
        \node[leaf,scale=0.7] (Y1) at ($(A4)+(2,0)$)  {$x$};
        \node[leaf,scale=0.7] (Y2) at ($(B4)+(2,0)$)  {$y$};

        \node[defnode,scale=0.65] (A12) at ($(A1)!0.49!(A2)$) {$\fs x$};
        \node[defnode,scale=0.7] (A23) at ($(A2)!0.6!(A3)$) {$x$};
        \node[defnode,scale=0.7] (A34) at ($(A3)!0.4!(A4)$) {$x$};

        \node[defnode,scale=0.66] (B12) at ($(B1)!0.6!(B2)$) {$2y-2jx$};
        \node[defnode,scale=0.7] (B23a) at ($(B2)!0.5!(B3) + (0,0.5)$) {$y-jx$};
        \node[defnode,scale=0.7] (B23b) at ($(B2)!0.5!(B3) + (0,-0.5)$) {$y-jx$};
        \node[defnode,scale=0.7] (B34) at ($(B3)!0.5!(B4)$) {$y-jx$};

        \node[defnode,scale=0.66] (AB1) at ($(A1)!0.5!(B1)$) {$\fs(2jx-2y)$};
        \node[defnode,scale=0.7] (AB4) at ($(A4)!0.5!(B4)$) {$jx$};

        \draw (A1) -- node[gatelabel, pos=0.10] {$00$} (X1);
        \draw (A1) -- node[gatelabel, pos=0.12] {$01$} (A12);

        \draw (A2) -- node[gatelabel, pos=0.12] {$01$} (A12);
        \draw (A2) -- node[gatelabel, pos=0.12] {$10$} (A23);

        \draw (A3) -- node[gatelabel, pos=0.12] {$10$} (A23);
        \draw (A3) -- node[gatelabel, pos=0.12] {$01$} (A34);

        \draw (A4) -- node[gatelabel, pos=0.12] {$00$} (Y1);
        \draw (A4) -- node[gatelabel, pos=0.12] {$01$} (A34);

        \draw (B1) -- node[gatelabel, pos=0.12] {$01$} (B12);

        \draw (B2) -- node[gatelabel, pos=0.12] {$01$} (B12);
        \draw (B2) -- node[gatelabel, pos=0.12] {$00$} (B23a);
        \draw (B2) -- node[gatelabel, pos=0.12] {$11$} (B23b);

        \draw (B3) -- node[gatelabel, pos=0.12] {$01$} (B34);
        \draw (B3) -- node[gatelabel, pos=0.12] {$00$} (B23a);
        \draw (B3) -- node[gatelabel, pos=0.12] {$11$} (B23b);

        \draw (B4) -- node[gatelabel, pos=0.12] {$00$} (Y2);
        \draw (B4) -- node[gatelabel, pos=0.12] {$01$} (B34);

        \draw (A1) -- node[gatelabel, pos=0.08] {$10$} (AB1);
        \draw (A4) -- node[gatelabel, pos=0.08] {$10$} (AB4);
        \draw (B1) -- node[gatelabel, pos=0.08] {$10$} (AB1);
        \draw (B4) -- node[gatelabel, pos=0.08] {$10$} (AB4);
      \end{tikzpicture}
      \\[\baselineskip]
      \vspace{\baselineskip}
      \begin{tikzpicture}[
          defnode/.style={rectangle,fill=lightgray,draw,inner sep=2pt,outer sep=0pt,minimum size=10pt},
          gatelabel/.style={rectangle, inner sep=2pt, outer sep=0pt, fill=white,scale=0.6},
          subnode/.style={rounded rectangle, draw, inner sep=5pt, outer sep=0pt,minimum size=15pt},
          leaf/.style={rectangle,draw,fill=leafgreen,inner sep=2pt,outer sep=0pt,minimum size=10pt},
          scale=0.8
        ]
        \node[scale=0.8] at (-2,1.5) {$\cM_2$};

        \node[subnode,scale=0.7] (A1) at (0,0)    {$\crs^{10}_{\strikeout{11}} /_i \const$};
        \node[subnode,scale=0.7] (A2) at (3.5,0)  {$\opsum /_{\fs} \const$};
        \node[subnode,scale=0.7] (A3) at (7,0)    {$\opsum$};
        \node[subnode,scale=0.7] (A4) at (10.5,0) {$\crs^{01}_{\overline{11}} /_j \opsum$};
        \node[subnode,scale=0.7] (B1) at (0,3)    {$\const /_\fs \opsum$};
        \node[subnode,scale=0.7] (B2) at (3.5,3)  {$\const$};
        \node[subnode,scale=0.7] (B3) at (7,3)    {$\opsum$};
        \node[subnode,scale=0.7] (B4) at (10.5,3) {$\const$};

        \node[leaf,scale=0.7] (X1) at ($(A1)+(-2,0)$) {$x$};
        \node[leaf,scale=0.65] (Y1) at ($(A4)+(2.4,0)$)  {$(i+2j) \fs x$};
        \node[leaf,scale=0.7] (Y2) at ($(B4)+(2.4,0)$)  {$2\fs x$};

        \node[defnode,scale=0.7] (A12) at ($(A1)!0.53!(A2)$) {$ix$};
        \node[defnode,scale=0.7] (A23) at ($(A2)!0.6!(A3)$) {$-i\fs x$};
        \node[defnode,scale=0.7] (A34) at ($(A3)!0.4!(A4)$) {$i \fs x$};

        \node[defnode,scale=0.7] (B12) at ($(B1)!0.6!(B2)$) {$\fs x$};
        \node[defnode,scale=0.7] (B23a) at ($(B2)!0.5!(B3) + (0,0.5)$) {$\fs x$};
        \node[defnode,scale=0.7] (B23b) at ($(B2)!0.5!(B3) + (0,-0.5)$) {$\fs x$};
        \node[defnode,scale=0.7] (B34) at ($(B3)!0.5!(B4)$) {$2 \fs x$};

        \node[defnode,scale=0.7] (AB1) at ($(A1)!0.5!(B1)$) {$x$};
        \node[defnode,scale=0.7] (AB4) at ($(A4)!0.5!(B4)$) {$2 \fs x$};

        \draw (A1) -- node[gatelabel, pos=0.10] {$00$} (X1);
        \draw (A1) -- node[gatelabel, pos=0.12] {$01$} (A12);

        \draw (A2) -- node[gatelabel, pos=0.12] {$01$} (A12);
        \draw (A2) -- node[gatelabel, pos=0.12] {$10$} (A23);

        \draw (A3) -- node[gatelabel, pos=0.12] {$10$} (A23);
        \draw (A3) -- node[gatelabel, pos=0.12] {$01$} (A34);

        \draw (A4) -- node[gatelabel, pos=0.12] {$00$} (Y1);
        \draw (A4) -- node[gatelabel, pos=0.12] {$01$} (A34);

        \draw (B1) -- node[gatelabel, pos=0.12] {$01$} (B12);

        \draw (B2) -- node[gatelabel, pos=0.12] {$01$} (B12);
        \draw (B2) -- node[gatelabel, pos=0.12] {$00$} (B23a);
        \draw (B2) -- node[gatelabel, pos=0.12] {$11$} (B23b);

        \draw (B3) -- node[gatelabel, pos=0.12] {$01$} (B34);
        \draw (B3) -- node[gatelabel, pos=0.12] {$00$} (B23a);
        \draw (B3) -- node[gatelabel, pos=0.12] {$11$} (B23b);

        \draw (B4) -- node[gatelabel, pos=0.12] {$00$} (Y2);
        \draw (B4) -- node[gatelabel, pos=0.12] {$01$} (B34);

        \draw (A1) -- node[gatelabel, pos=0.08] {$10$} (AB1);
        \draw (A4) -- node[gatelabel, pos=0.08] {$10$} (AB4);
        \draw (B1) -- node[gatelabel, pos=0.08] {$10$} (AB1);
        \draw (B4) -- node[gatelabel, pos=0.08] {$10$} (AB4);
      \end{tikzpicture}
    \end{center}
    \caption{A description of the morphisms $\cM_1$ and $\cM_2$ for the assignments $\Initial(i,j,\fs)$.}%
    \label{fig:agate-initial}
  \end{figure}

  The extra layer of notation $- /_i -$, $- /_j -$, $- /_\fs -$ means that the left option holds if $i=0$, $j=0$ or $\fs =1$ respectively, and the right option holds if $i=1$, $j=1$ or $\fs=-1$.
  We may check for each gate individually that it is doing what it should for each choice of $i$, $j$ and $\fs$.

  For $\cM_0$ and $\cM_3,\dots,\cM_s$ we may apply Remark~\ref{rem:restriction-morphs} to the morphisms in Lemma~\ref{lem:agate-assign-1} to remove the vertex $X2$ from both sides.
  We omit the details.
\end{proof}

Finally, we build some assignments for the gates $\smagate_s^{\Omega}$ and $\smagate_s^{\Alpha,\Omega}$.

\begin{definition}%
  \label{def:final-agate-assignments}
  For $i \in \{0,1\}$ we write $\ag(i,A)^\Omega$ for the linear datum or diagram $\ag(i,A) \langle Y2 \rangle$, and $\ag(i,B)^\Omega$ for $\ag(i,B)[\{X1,X2,Y1\}]$.

  For $i,j \in \{0,1\}$ and $\fs \in \{\pm 1\}$ we similarly write $\ag(i,j,\fs,A)^\Omega$ to denote $\ag(i,j,\fs,A)\langle Y2 \rangle$ and $\ag(i,j,\fs,B)^\Omega$ to denote $\ag(i,j,\fs,B)[\{X1,Y1\}]$.
\end{definition}

\begin{lemma}%
  \label{lem:agate-final-assignments}
  For $i \in \{0,1\}$ there is an assignment $\Middle^\Omega(i)$ of the gate $\smagate_s^{\Omega}$ with
  \begin{itemize}
    \item $D_0 = \trivial(\{X1, X2, Y1\})$;
    \item $D_1 = \ag(i,A)^\Omega$;
    \item $D_2 = \ag(i,B)^\Omega$; and
    \item $D_3 = \cdots = D_s = \const(\{X1, X2, Y1\})$.
  \end{itemize}
  For $i,j \in \{0,1\}$ and $\fs \in \{\pm 1\}$ there is an assignment $\Initial^{\Omega}(i,j,\fs)$ of the gate $\smagate_s^{\Alpha,\Omega}$ with
  \begin{itemize}
    \item $D_0 = \trivial(\{X1, Y1\})$;
    \item $D_1 = \ag(i,j,\fs,A)^\Omega$;
    \item $D_2 = \ag(i,j,\fs,B)^\Omega$; and
    \item $D_3 = \cdots = D_s = \const(\{X1, Y1\})$.
  \end{itemize}
\end{lemma}
\begin{proof}%
  The partitions $S_1,\dots,S_s$ in each case are identical to those in the corresponding partition constructed in Lemma~\ref{lem:agate-assign-1} or Lemma~\ref{lem:agate-assign-2} without the $\Omega$ decoration, except that $B4;00;\triangle$ is added to $S_1$ in each case.

  Alternatively, we could say that the sub-assignments of $A1,\dots,B4$ are verbatim the same as those in~\eqref{eq:agate-assignments} and~\eqref{eq:agate-assignments-2} respectively. 
  Since these are applied to the gate structure $\agg_s^{\{01,10\}}$ at B4 in place of $\agg_s^{\{00,01,10\}}$, this slight modification to the partitions arises naturally.

  Mercifully we may construct the morphisms $\cM_r$ in each case directly from those in Lemma~\ref{lem:agate-assign-1} or Lemma~\ref{lem:agate-assign-2} without repeating any of those arguments.
  In each case, the new $\cM_1$ is obtained from the old one by composing with the co-restriction map $\ag(i,A)\langle Y2 \rangle \to \ag(i,A)$ or $\ag(i,j,\fs,A)\langle Y2 \rangle \to \ag(i,j,\fs,A)$.
  The other maps $\cM_r$ are obtained from the old ones using Remark~\ref{rem:restriction-morphs} to remove the vertex $Y2$.
\end{proof}

\subsection{The large AGate}%
\label{sub:big-agate}

We are finally in a position to combine lots of individual gates $\smagate_s$ into one large one which is capable of multiplying by arbitrary numbers, implementing Figure~\ref{fig:bilinear} and its generalizations in full.

\begin{definition}%
  \label{def:large-agate}
  Let $k \ge 1$ be an integer.
  The diagram $\agate_s^k$ is defined by the following picture.
  
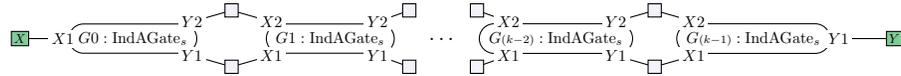
\begin{center}
    \begin{tikzpicture}[
        defnode/.style={rectangle,draw,fill=lightgray,inner sep=2pt,outer sep=0pt,minimum size=10pt},
        gatelabel/.style={rectangle, inner sep=2pt, outer sep=0pt, fill=white,scale=0.6},
        subnode/.style={rounded rectangle, draw, inner sep=5pt, outer sep=0pt,minimum size=15pt},
        leaf/.style={rectangle,draw,fill=leafgreen,inner sep=2pt,outer sep=0pt,minimum size=10pt},
        scale=0.75
      ]

      \node[subnode,scale=0.6] (G0) at (0,0) {$G0 : \smagate_s$};
      \node[subnode,scale=0.6] (G1) at (3.5,0) {$G1 : \smagate_s$};
      \node[scale=0.8] (Gdots) at (5.5,0) {$\dots$};
      \node[subnode,scale=0.6] (Gpen) at (7.5,0) {$G{\scriptstyle(k-2)} : \smagate_s$};
      \node[subnode,scale=0.6] (Glast) at (11,0) {$G{\scriptstyle(k-1)} : \smagate_s$};

      \node[leaf,scale=0.5] (X) at ($(G0)+(-2,0)$) {$X$};
      \node[leaf,scale=0.5] (Y) at ($(Glast)+(2.5,0)$) {$Y$};

      \node[defnode,scale=0.5] (G01a) at ($(G0)!0.5!(G1) + (0,0.5)$) {};
      \node[defnode,scale=0.5] (G01b) at ($(G0)!0.5!(G1) + (0,-0.5)$) {};
      \node[defnode,scale=0.5] (G1-a) at ($(G1)!0.7!(Gdots) + (0,0.5)$) {};
      \node[defnode,scale=0.5] (G1-b) at ($(G1)!0.7!(Gdots) + (0,-0.5)$) {};

      \node[defnode,scale=0.5] (G-2a) at ($(Gdots)!0.3!(Gpen) + (0,0.5)$) {};
      \node[defnode,scale=0.5] (G-2b) at ($(Gdots)!0.3!(Gpen) + (0,-0.5)$) {};
      \node[defnode,scale=0.5] (Gpla) at ($(Gpen)!0.5!(Glast) + (0,0.5)$) {};
      \node[defnode,scale=0.5] (Gplb) at ($(Gpen)!0.5!(Glast) + (0,-0.5)$) {};

      \draw (G0)    -- node[gatelabel, pos=0.15] {$X1$} (X);
      \draw (Glast) -- node[gatelabel, pos=0.15] {$Y1$} (Y);

      \draw (G0)    -- node[gatelabel, pos=0.2] {$Y2$} (G01a);
      \draw (G0)    -- node[gatelabel, pos=0.2] {$Y1$} (G01b);
      \draw (G1)    -- node[gatelabel, pos=0.2] {$X2$} (G01a);
      \draw (G1)    -- node[gatelabel, pos=0.2] {$X1$} (G01b);

      \draw (G1)    -- node[gatelabel, pos=0.2] {$Y2$} (G1-a);
      \draw (G1)    -- node[gatelabel, pos=0.2] {$Y1$} (G1-b);
      \draw (Gpen)  -- node[gatelabel, pos=0.2] {$X2$} (G-2a);
      \draw (Gpen)  -- node[gatelabel, pos=0.2] {$X1$} (G-2b);

      \draw (Gpen)  -- node[gatelabel, pos=0.2] {$Y2$} (Gpla);
      \draw (Gpen)  -- node[gatelabel, pos=0.2] {$Y1$} (Gplb);
      \draw (Glast) -- node[gatelabel, pos=0.2] {$X2$} (Gpla);
      \draw (Glast) -- node[gatelabel, pos=0.2] {$X1$} (Gplb);
    \end{tikzpicture}
  \end{center}
  Again we have assigned redundant labels $X$ to $G0;X1$ and $Y$ to $G(k-1);Y1$. 
\end{definition}

It is fairly straightforward to build this chain of IndAGates, when $k$ is a power of $2$, starting from one IndAGate (and hence from an AggreGate or a $\gc$ gate) by repeatedly applying $\CS(\{Y1,Y2\})$.
The only subtlety is that the resulting right-hand copy is reflected in the north-south axis.
An IndAGate is reflection-symmetric so this does not impact the resulting diagram, but it does mean the labels $A1,\dots,B4$ are inverted.

\begin{lemma}%
  \label{lem:build-large-agate}
  Let $m \ge 0$ and write $k = 2^m$.
  Define $\gamma(X)=(X1,0)$ and $\gamma(Y)=(X1,1)$, unless $m=0$ in which case $\gamma(Y)=(Y1,0)$.
  Then $\smagate_s \entails^{m}_\gamma \agate_s^k$.
\end{lemma}
\begin{proof}%
  When $m=0$, $\agate_s^1$ is the same diagram as $\smagate_s$, up to relabelling, so the statement is trivial.

  We then note that $\agate_s^{2k} = \agate_s^k +_{\{(G(k-1);Y1), (G(k-1);Y2)\}} \agate_s^k$, except that the labellings do not agree: specifically, we have to map
  \begin{align*}
    L;Gi &\leadsto Gi \\
    R;Gi;Aj &\leadsto G(2k-1-i);A(5-j) \\
    R;Gi;Bj &\leadsto G(2k-1-i);B(5-j) \\
    L;X &\leadsto X \\
    R;X &\leadsto Y.
  \end{align*}
  We deduce $\agate_s^k \entails^1_{\gamma'} \agate_s^{2k}$ where $\gamma'$ is the map $(L;\ell) \mapsto (\ell,0)$, $(R;\ell) \mapsto (\ell,1)$ composed with the relabelling above.
  By induction on $m$, and keeping track of the $\gamma$ functions, the result follows.
\end{proof}

The gate structure is, this time, easy to define.
\begin{definition}%
  \label{def:large-agate-gate}
  The gate $\agate_s^k$ consists of the diagram $\agate_s^k$, the modes $\cR = [s]$, and the pins $\cP = \{ X, Y \}$, with all other leaves of $\agate_s^k$ being toggles.
\end{definition}

Again it is useful to identify the sub-diagrams $G0,\dots,G(k-1)$ as gates in their own right (noting Footnote~\ref{footnote:subgate}).
The rule is that, for $0 \le i \le k-1$,
\[
  Gi \colon
  \begin{cases}
    \smagate &\colon i \ne 0,\, i \ne k-1 \\
    \smagate^{\Alpha} &\colon i = 0,\, i \ne k-1 \\
    \smagate^{\Omega} &\colon i \ne 0,\, i = k-1 \\
    \smagate^{A,\Omega} &\colon i = 0,\, i = k-1.
  \end{cases}
\]
Clearly the last option is only available if $k=1$.
Again, this choice is determined by the requirement that assigning the toggles of $\agate_s^k$ should be the same thing as assigning the toggles of all the component gates; that is,
\[
  \toggles^{\agate_s^k} = \bigcup_{i=0}^{k-1} \toggles^{Gi}.
\]

We may now state an assignment result that summarizes all our work on AGates.

\begin{definition}%
  \label{def:lag}
  Let $a$ be any integer.
  Define two linear data $\lag(a,A)$ and $\lag(a,B)$ with $I=\{X,Y\}$, $V=W_1=W_2=\FF_p$ and linear maps
  \vspace{0.5\baselineskip}
  \begin{center}
    \begin{tikzpicture}[
        defnode/.style={rectangle,draw,inner sep=4pt,outer sep=0pt,minimum size=15pt},
        mydot/.style={circle,fill,inner sep=0.5pt},
        leaf/.style={defnode,fill=leafgreen},
        toggle/.style={defnode,fill=leafgreen}
      ]
      \begin{scope}[shift={(-4,0)}]
        \node[defnode,scale=0.6] (U) at (0, 0) {$\diamond$};
        \node[leaf,scale=0.6]    (UX) at ($(U) + (225:1)$) {$X$};
        \node[leaf,scale=0.6]    (UY) at ($(U) + (315:1)$) {$Y$};
        \draw[-stealth] (U) -- (UX);
        \draw[-stealth] (U) -- (UY);
      \end{scope}
      \begin{scope}[shift={(0,0)}]
        \node[defnode,scale=0.6] (U) at (0, 0) {$x$};
        \node[leaf,scale=0.6]    (UX) at ($(U) + (225:1)$) {$ax$};
        \node[leaf,scale=0.6]    (UY) at ($(U) + (315:1)$) {$x$};
        \draw[-stealth] (U) -- (UX);
        \draw[-stealth] (U) -- (UY);

        \node["{$\lag(a,A)$}" {scale=0.9}] at (0, -1.7) {};
      \end{scope}
      \begin{scope}[shift={(3,0)}]
        \node[defnode,scale=0.6] (U) at (0, 0) {$x$};
        \node[leaf,scale=0.6]    (UX) at ($(U) + (225:1)$) {$x$};
        \node[leaf,scale=0.6]    (UY) at ($(U) + (315:1)$) {$ax$};
        \draw[-stealth] (U) -- (UX);
        \draw[-stealth] (U) -- (UY);

        \node["{$\lag(a,B)$}" {scale=0.9}] at (0, -1.7) {};
      \end{scope}
    \end{tikzpicture}
  \end{center}
\end{definition}

\begin{lemma}%
  \label{lem:agate-master}
  Suppose $s \ge 2$, and $k \ge 1$ and $a$ are integers with $|a| < 2^{k+1}$.  
  Then there is an assignment $\abilin_s^k(a)$ of $\agate_s^k$ with $D_0 = \trivial(\{X,Y\})$, $D_1 = \lag(a, A)$, $D_2 = \lag(a, B)$ and $D_r = \const(\{X,Y\})$ for $3 \le r \le s$.
\end{lemma}
\begin{proof}%
  As in the proofs of Lemma~\ref{lem:agate-assign-1} and Lemma~\ref{lem:agate-assign-2}, we obtain assignments of $G0,\dots,G(k-1)$ by the results in Section~\ref{sub:agate} and then argue, with the aid of suitable pictures, that these can be combined to yield the morphisms $\cM_0,\cM_1,\dots,\cM_s$.  

  We write $\fs \in \{\pm 1\}$ for the sign of $a$ (if $a=0$ then either value works), and $a_0,a_1,\dots,a_{k} \in \{0,1\}$ for the binary digits of $|a|$ (extended with zeros if necessary), so that
  \[
    a = \fs \left( \sum_{\ell=0}^{k} 2^\ell a_\ell \right).
  \]
  We then consider the following assignments from Section~\ref{sub:agate}:
  \[
    G0 : \begin{rcases}
      \begin{dcases}
        \Initial(a_0, a_1, \fs) &\colon k > 1 \\
        \Initial^{\Omega}(a_0, a_1, \fs) &\colon k = 1.
      \end{dcases}
    \end{rcases}
  \]
  and for $1 \le \ell \le k-1$,
  \[
    Gi : \begin{rcases}
      \begin{dcases}
        \Middle(a_{\ell+1}) &\colon \ell \ne k-1 \\
        \Middle^{\Omega}(a_{\ell+1}) &\colon \ell = k - 1.
      \end{dcases}
    \end{rcases}
  \]
  As usual, these provide us with partitions $S_1^\ell,\dots,S_s^\ell$ and morphisms $\cM_0^\ell,\dots,\cM_s^\ell$. 
  The domains $D_0^\ell,\dots,D_s^\ell$ may be read off from the statements of Lemma~\ref{lem:agate-assign-1}, Lemma~\ref{lem:agate-assign-2} and Lemma~\ref{lem:agate-final-assignments}.
  We again define the partition $S_1,\dots,S_s$ by taking unions $S_r = \bigcup_{\ell=0}^{k-1} S_r^\ell$ over the  
  component gates.

  We first analyse $\cM_2$.
  Define the integers $t_1,\dots,t_k$ by
  \[
    t_\ell = \fs \left( \sum_{j=0}^\ell 2^j a_j \right)
  \]
  so that $t_k=a$.
  Once again the heart of the proof is a picture, Figure~\ref{fig:big-agate-b}.
  \begin{figure}[htbp]
  \begin{center}
    \begin{tikzpicture}[
        defnode/.style={rectangle,draw,fill=lightgray,inner sep=3pt,outer sep=0pt,minimum size=10pt},
        gatelabel/.style={rectangle, inner sep=2pt, outer sep=0pt, fill=white,scale=0.6},
        subnode/.style={rounded rectangle, draw, inner sep=5pt, outer sep=0pt,minimum size=15pt},
        leaf/.style={rectangle,draw,fill=leafgreen,inner sep=3pt,outer sep=0pt,minimum size=10pt},
        scale=0.87
      ]

      \node[subnode,scale=0.6] (G0) at (0,0) {$\ag(a_0,a_1,\fs,B)$};
      \node[subnode,scale=0.6] (G1) at (3,0) {$\ag(a_2,B)$};
      \node[scale=0.8]         (Gdots) at (5,0) {$\dots$};
      \node[subnode,scale=0.6] (Gpen) at (7.3,0) {$\ag(a_{k-1},B)$};
      \node[subnode,scale=0.6] (Glast) at (10.3,0) {$\ag^\Omega(a_{k},B)$};

      \node[leaf,scale=0.5] (X) at ($(G0)+(-1.8,0)$) {$x$};
      \node[leaf,scale=0.5] (Y) at ($(Glast)+(1.7,0)$) {$ax$};

      \node[defnode,scale=0.5] (G01a) at ($(G0)!0.5!(G1)      + (0, 0.7)$) {$2\fs x$};
      \node[defnode,scale=0.5] (G01b) at ($(G0)!0.5!(G1)      + (0,-0.7)$) {$t_1 x$};
      \node[defnode,scale=0.5] (G1-a) at ($(G1)!0.7!(Gdots)   + (0, 0.7)$) {$4 \fs x$};
      \node[defnode,scale=0.5] (G1-b) at ($(G1)!0.7!(Gdots)   + (0,-0.7)$) {$t_2 x$};

      \node[defnode,scale=0.5] (G-2a) at ($(Gdots)!0.3!(Gpen) + (0, 0.7)$) {$2^{k-2} \fs x$};
      \node[defnode,scale=0.5] (G-2b) at ($(Gdots)!0.3!(Gpen) + (0,-0.7)$) {$t_{k-2} x$};
      \node[defnode,scale=0.5] (Gpla) at ($(Gpen)!0.5!(Glast) + (0, 0.7)$) {$2^{k-1} \fs x$};
      \node[defnode,scale=0.5] (Gplb) at ($(Gpen)!0.5!(Glast) + (0,-0.7)$) {$t_{k-1} x$};

      \draw (G0)    -- node[gatelabel, pos=0.15] {$X1$} (X);
      \draw (Glast) -- node[gatelabel, pos=0.15] {$Y1$} (Y);

      \draw (G0)    -- node[gatelabel, pos=0.2] {$Y2$} (G01a);
      \draw (G0)    -- node[gatelabel, pos=0.2] {$Y1$} (G01b);
      \draw (G1)    -- node[gatelabel, pos=0.2] {$X2$} (G01a);
      \draw (G1)    -- node[gatelabel, pos=0.2] {$X1$} (G01b);

      \draw (G1)    -- node[gatelabel, pos=0.2] {$Y2$} (G1-a);
      \draw (G1)    -- node[gatelabel, pos=0.2] {$Y1$} (G1-b);
      \draw (Gpen)  -- node[gatelabel, pos=0.2] {$X2$} (G-2a);
      \draw (Gpen)  -- node[gatelabel, pos=0.2] {$X1$} (G-2b);

      \draw (Gpen)  -- node[gatelabel, pos=0.2] {$Y2$} (Gpla);
      \draw (Gpen)  -- node[gatelabel, pos=0.2] {$Y1$} (Gplb);
      \draw (Glast) -- node[gatelabel, pos=0.2] {$X2$} (Gpla);
      \draw (Glast) -- node[gatelabel, pos=0.2] {$X1$} (Gplb);
    \end{tikzpicture}
  \end{center}
  \caption{The pictorial summary of the morphism $\cM_2$ from $\abilin_s^k(a)$.}%
  \label{fig:big-agate-b}
  \end{figure}
  
  To be precise, the four pins of $G\ell$ for $1 \le \ell \le k-1$ carry values
  \[
    \begin{aligned}
      X1 &\colon t_\ell x & Y1 &\colon t_{\ell+1} x \\
      X2 &\colon \fs 2^\ell x & Y2 &\colon \fs 2^{\ell+1} x
    \end{aligned}
  \]
  only omitting $Y2$ when $\ell=k-1$.
  In particular the $Y$ pins from $G\ell$ coincide with the $X$ pins of $G(\ell+1)$ where applicable, as they must.

  We should verify that each sub-gate $G\ell$ is ``doing what it supposed to do'', insofar as its pins are a specialization $\Theta_\ell \circ \cM_2^\ell$ of the morphism $\cM_2^\ell$.
  In the case $\ell=0$, $k>1$, the values on the pins are exactly those in $\ag(i,j,\fs,B)$ (see Definition~\ref{def:initial-data}) and $\Theta_0$ is simply
  
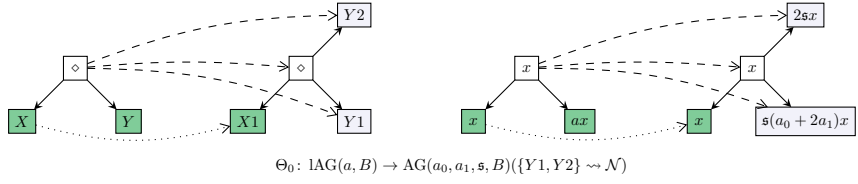
\begin{center}
    \begin{tikzpicture}[
        defnode/.style={rectangle,draw,inner sep=4pt,outer sep=0pt,minimum size=15pt},
        mydot/.style={circle,fill,inner sep=0.5pt},
        leaf/.style={defnode,fill=leafgreen},
        nonleaf/.style={defnode,fill=lightgray},
        toggle/.style={defnode,fill=leafgreen}
      ]
      \node[scale=0.6] at (5, -1.3) {$\Theta_{0} \colon \lag(a,B) \to \ag(a_0,a_1,\fs,B)(\{Y1,Y2\} \leadsto \nonleafs)$};
        \begin{scope}[shift={(0,0)}]
          \node[defnode,scale=0.6] (U) at (0, 0) {$\diamond$};
          \node[leaf,scale=0.6]    (UXX) at ($(U) + (225:1)$) {$X$};
          \node[leaf,scale=0.6]    (UXY) at ($(U) + (315:1)$) {$Y$};
          \draw[-stealth] (U) -- (UXX);
          \draw[-stealth] (U) -- (UXY);

          \node[defnode,scale=0.6] (V) at (3, 0) {$\diamond$};
          \node[leaf,scale=0.6]    (VXX1) at ($(V) + (225:1)$) {$X1$};
          \node[nonleaf,scale=0.6]    (VXY1) at ($(V) + (315:1)$) {$Y1$};
          \node[nonleaf,scale=0.6]    (VXY2) at ($(V) + (045:1)$) {$Y2$};
          \draw[-stealth] (V) -- (VXX1);
          \draw[-stealth] (V) -- (VXY1);
          \draw[-stealth] (V) -- (VXY2);

          \draw[dashed, -angle 60] (U) to[bend left=5] (V);
          \draw[dotted, -angle 60] (UXX) to[bend right=15] (VXX1);
          \draw[dashed, -angle 60] (U) to[bend left=10] (VXY1);
          \draw[dashed, -angle 60] (U) to[bend left=10] (VXY2);
        \end{scope}
        \begin{scope}[shift={(6,0)}]
          \node[defnode,scale=0.6] (U) at (0, 0) {$x$};
          \node[leaf,scale=0.6]    (UXX) at ($(U) + (225:1)$) {$x$};
          \node[leaf,scale=0.6]    (UXY) at ($(U) + (315:1)$) {$ax$};
          \draw[-stealth] (U) -- (UXX);
          \draw[-stealth] (U) -- (UXY);

          \node[defnode,scale=0.6] (V) at (3, 0) {$x$};
          \node[leaf,scale=0.6]    (VXX1) at ($(V) + (225:1)$) {$x$};
          \node[nonleaf,scale=0.6]    (VXY1) at ($(V) + (315:1)$) {$\fs(a_0+2a_1)x$};
          \node[nonleaf,scale=0.6]    (VXY2) at ($(V) + (045:1)$) {$2 \fs x$};
          \draw[-stealth] (V) -- (VXX1);
          \draw[-stealth] (V) -- (VXY1);
          \draw[-stealth] (V) -- (VXY2);

          \draw[dashed, -angle 60] (U) to[bend left=5] (V);
          \draw[dotted, -angle 60] (UXX) to[bend right=15] (VXX1);
          \draw[dashed, -angle 60] (U) to[bend left=10] (VXY1);
          \draw[dashed, -angle 60] (U) to[bend left=10] (VXY2);
        \end{scope}
    \end{tikzpicture}
  \end{center}
  where we note that $\fs(a_0+2a_1) = t_1$.
  For $1 \le \ell < k-1$, the morphism $\Theta_\ell$ is given by
  
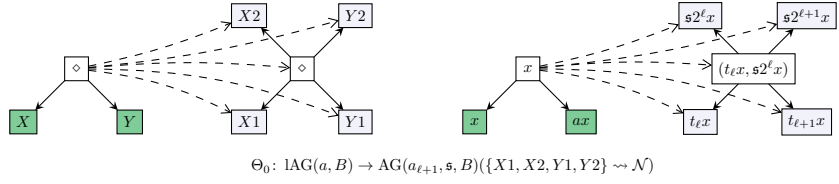
\begin{center}
    \begin{tikzpicture}[
        defnode/.style={rectangle,draw,inner sep=4pt,outer sep=0pt,minimum size=15pt},
        mydot/.style={circle,fill,inner sep=0.5pt},
        leaf/.style={defnode,fill=leafgreen},
        nonleaf/.style={defnode,fill=lightgray},
        toggle/.style={defnode,fill=leafgreen}
      ]
      \node[scale=0.6] at (5, -1.3) {$\Theta_{0} \colon \lag(a,B) \to \ag(a_{\ell+1},\fs,B)(\{X1,X2,Y1,Y2\} \leadsto \nonleafs)$};
        \begin{scope}[shift={(0,0)}]
          \node[defnode,scale=0.6] (U) at (0, 0) {$\diamond$};
          \node[leaf,scale=0.6]    (UXX) at ($(U) + (225:1)$) {$X$};
          \node[leaf,scale=0.6]    (UXY) at ($(U) + (315:1)$) {$Y$};
          \draw[-stealth] (U) -- (UXX);
          \draw[-stealth] (U) -- (UXY);

          \node[defnode,scale=0.6] (V) at (3, 0) {$\diamond$};
          \node[nonleaf,scale=0.6]    (VXX1) at ($(V) + (225:1)$) {$X1$};
          \node[nonleaf,scale=0.6]    (VXX2) at ($(V) + (135:1)$) {$X2$};
          \node[nonleaf,scale=0.6]    (VXY1) at ($(V) + (315:1)$) {$Y1$};
          \node[nonleaf,scale=0.6]    (VXY2) at ($(V) + (045:1)$) {$Y2$};
          \draw[-stealth] (V) -- (VXX1);
          \draw[-stealth] (V) -- (VXX2);
          \draw[-stealth] (V) -- (VXY1);
          \draw[-stealth] (V) -- (VXY2);

          \draw[dashed, -angle 60] (U) to[bend left=5] (V);
          \draw[dashed, -angle 60] (U) to[bend left=10] (VXX1);
          \draw[dashed, -angle 60] (U) to[bend right=5] (VXX2);
          \draw[dashed, -angle 60] (U) to[bend left=10] (VXY1);
          \draw[dashed, -angle 60] (U) to[bend right=5] (VXY2);
        \end{scope}
        \begin{scope}[shift={(6,0)}]
          \node[defnode,scale=0.6] (U) at (0, 0) {$x$};
          \node[leaf,scale=0.6]    (UXX) at ($(U) + (225:1)$) {$x$};
          \node[leaf,scale=0.6]    (UXY) at ($(U) + (315:1)$) {$ax$};
          \draw[-stealth] (U) -- (UXX);
          \draw[-stealth] (U) -- (UXY);

          \node[defnode,scale=0.6] (V) at (3, 0) {$(t_\ell x, \fs 2^\ell x)$};
          \node[nonleaf,scale=0.6]    (VXX1) at ($(V) + (225:1)$) {$t_\ell x$};
          \node[nonleaf,scale=0.6]    (VXX2) at ($(V) + (135:1)$) {$\fs 2^\ell x$};
          \node[nonleaf,scale=0.6]    (VXY1) at ($(V) + (315:1)$) {$t_{\ell+1} x$};
          \node[nonleaf,scale=0.6]    (VXY2) at ($(V) + (045:1)$) {$\fs 2^{\ell+1} x$};
          \draw[-stealth] (V) -- (VXX1);
          \draw[-stealth] (V) -- (VXX2);
          \draw[-stealth] (V) -- (VXY1);
          \draw[-stealth] (V) -- (VXY2);

          \draw[dashed, -angle 60] (U) to[bend left=5] (V);
          \draw[dashed, -angle 60] (U) to[bend left=10] (VXX1);
          \draw[dashed, -angle 60] (U) to[bend right=5] (VXX2);
          \draw[dashed, -angle 60] (U) to[bend left=10] (VXY1);
          \draw[dashed, -angle 60] (U) to[bend right=5] (VXY2);
        \end{scope}
    \end{tikzpicture}
  \end{center}
  and the key consistency checks are that $2 (\fs 2^\ell x) = \fs 2^{\ell+1} x$ and that $t_\ell x + 2 a_{\ell+1}(\fs 2^\ell x) = t_{\ell+1} x$.
  Finally for $\ell=k-1$ we make a small modification:
  
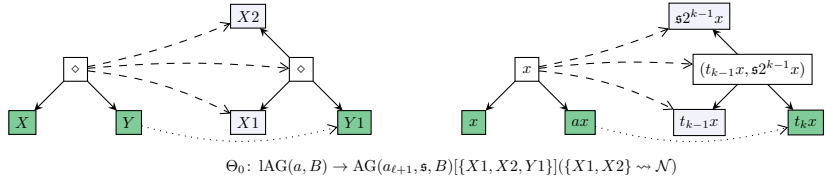
\begin{center}
    \begin{tikzpicture}[
        defnode/.style={rectangle,draw,inner sep=4pt,outer sep=0pt,minimum size=15pt},
        mydot/.style={circle,fill,inner sep=0.5pt},
        leaf/.style={defnode,fill=leafgreen},
        nonleaf/.style={defnode,fill=lightgray},
        toggle/.style={defnode,fill=leafgreen}
      ]
      \node[scale=0.6] at (5, -1.3) {$\Theta_{0} \colon \lag(a,B) \to \ag(a_{\ell+1},\fs,B)[\{X1,X2,Y1\}](\{X1,X2\} \leadsto \nonleafs)$};
        \begin{scope}[shift={(0,0)}]
          \node[defnode,scale=0.6] (U) at (0, 0) {$\diamond$};
          \node[leaf,scale=0.6]    (UXX) at ($(U) + (225:1)$) {$X$};
          \node[leaf,scale=0.6]    (UXY) at ($(U) + (315:1)$) {$Y$};
          \draw[-stealth] (U) -- (UXX);
          \draw[-stealth] (U) -- (UXY);

          \node[defnode,scale=0.6] (V) at (3, 0) {$\diamond$};
          \node[nonleaf,scale=0.6]    (VXX1) at ($(V) + (225:1)$) {$X1$};
          \node[nonleaf,scale=0.6]    (VXX2) at ($(V) + (135:1)$) {$X2$};
          \node[leaf,scale=0.6]    (VXY1) at ($(V) + (315:1)$) {$Y1$};
          \draw[-stealth] (V) -- (VXX1);
          \draw[-stealth] (V) -- (VXX2);
          \draw[-stealth] (V) -- (VXY1);

          \draw[dashed, -angle 60] (U) to[bend left=5] (V);
          \draw[dashed, -angle 60] (U) to[bend left=10] (VXX1);
          \draw[dashed, -angle 60] (U) to[bend right=5] (VXX2);
          \draw[dotted, -angle 60] (UXY) to[bend right=15] (VXY1);
        \end{scope}
        \begin{scope}[shift={(6,0)}]
          \node[defnode,scale=0.6] (U) at (0, 0) {$x$};
          \node[leaf,scale=0.6]    (UXX) at ($(U) + (225:1)$) {$x$};
          \node[leaf,scale=0.6]    (UXY) at ($(U) + (315:1)$) {$ax$};
          \draw[-stealth] (U) -- (UXX);
          \draw[-stealth] (U) -- (UXY);

          \node[defnode,scale=0.6] (V) at (3, 0) {$(t_{k-1} x, \fs 2^{k-1} x)$};
          \node[nonleaf,scale=0.6]    (VXX1) at ($(V) + (225:1)$) {$t_{k-1} x$};
          \node[nonleaf,scale=0.6]    (VXX2) at ($(V) + (135:1)$) {$\fs 2^{k-1} x$};
          \node[leaf,scale=0.6]    (VXY1) at ($(V) + (315:1)$) {$t_{k} x$};
          \draw[-stealth] (V) -- (VXX1);
          \draw[-stealth] (V) -- (VXX2);
          \draw[-stealth] (V) -- (VXY1);

          \draw[dashed, -angle 60] (U) to[bend left=5] (V);
          \draw[dashed, -angle 60] (U) to[bend left=10] (VXX1);
          \draw[dashed, -angle 60] (U) to[bend right=5] (VXX2);
          \draw[dotted, -angle 60] (UXY) to[bend right=15] (VXY1);
        \end{scope}
    \end{tikzpicture}
  \end{center}
  using the observation that $a=t_k$.
  Note that in each case these morphisms were entirely forced by Figure~\ref{fig:big-agate-b}.
  Applying Lemma~\ref{lem:patching} to $\Theta_\ell \circ \cM_2^\ell$, as in the proof of Lemma~\ref{lem:agate-assign-1}, we get the required morphism $\cM_2$.

  When $k=1$ these pictures are not quite correct but in this case the construction is very direct from Lemma~\ref{lem:agate-final-assignments}.

  Again we cut corners for $\cM_0$, $\cM_1$ and $\cM_3,\dots,\cM_s$ by showing the analogue of Figure~\ref{fig:big-agate-b} and leaving it to the reader to imagine or construct the maps $\Theta_\ell$.
  
  We first consider $\cM_1$.
  We define integers $u_0,u_1,\dots,u_k$ by
  \[
    u_\ell = \sum_{i=\ell}^k 2^{i-\ell} a_i
  \]
  which obey the recurrence $u_\ell = a_\ell + 2 u_{\ell+1}$ for $0 \le \ell < k$.
  Note also that $\fs u_0 = a$.
  Then Figure~\ref{fig:big-agate-a} describes the morphism $\Theta_1$.
  Again to be precise, the pins on $G\ell$ for $1 \le \ell \le k-1$ carry values
  \[
    \begin{aligned}
      X1 &\colon x & Y1 &\colon x \\
      X2 &\colon {-2} u_{\ell+1} x & Y2 &\colon {-2} u_{\ell+2} x
    \end{aligned}
  \]
  but omitting $Y2$ for $\ell=k-1$.
  
  \begin{figure}[htbp]
    \begin{center}
      \begin{tikzpicture}[
          defnode/.style={rectangle,draw,fill=lightgray,inner sep=3pt,outer sep=0pt,minimum size=10pt},
          gatelabel/.style={rectangle, inner sep=2pt, outer sep=0pt, fill=white,scale=0.6},
          subnode/.style={rounded rectangle, draw, inner sep=5pt, outer sep=0pt,minimum size=15pt},
          leaf/.style={rectangle,draw,fill=leafgreen,inner sep=3pt,outer sep=0pt,minimum size=10pt},
          scale=0.87
        ]
        \node[subnode,scale=0.6] (G0) at (0,0) {$\ag(a_0,a_1,\fs,A)$};
        \node[subnode,scale=0.6] (G1) at (3,0) {$\ag(a_2,A)$};
        \node[scale=0.8] (Gdots) at (5.0,0) {$\dots$};
        \node[subnode,scale=0.6] (Gpen) at (7.3,0) {$\ag(a_{k-1},A)$};
        \node[subnode,scale=0.6] (Glast) at (10.3,0) {$\ag(a_k, A)$};

        \node[leaf,scale=0.5] (X) at ($(G0)+(-1.8,0)$) {$ax$};
        \node[leaf,scale=0.5] (Y) at ($(Glast)+(1.7,0)$) {$x$};

        \node[defnode,scale=0.5] (G01a) at ($(G0)!0.5!(G1)      + (0, 0.7)$) {$-2 u_2 x$};
        \node[defnode,scale=0.5] (G01b) at ($(G0)!0.5!(G1)      + (0,-0.7)$) {$x$};
        \node[defnode,scale=0.5] (G1-a) at ($(G1)!0.7!(Gdots)   + (0, 0.7)$) {$-2 u_3 x$};
        \node[defnode,scale=0.5] (G1-b) at ($(G1)!0.7!(Gdots)   + (0,-0.7)$) {$x$};

        \node[defnode,scale=0.5] (G-2a) at ($(Gdots)!0.3!(Gpen) + (0, 0.7)$) {$-2 u_{k-1} x$};
        \node[defnode,scale=0.5] (G-2b) at ($(Gdots)!0.3!(Gpen) + (0,-0.7)$) {$x$};
        \node[defnode,scale=0.5] (Gpla) at ($(Gpen)!0.5!(Glast) + (0, 0.7)$) {$-2 u_k x$};
        \node[defnode,scale=0.5] (Gplb) at ($(Gpen)!0.5!(Glast) + (0,-0.7)$) {$x$};

        \draw (G0)    -- node[gatelabel, pos=0.15] {$X1$} (X);
        \draw (Glast) -- node[gatelabel, pos=0.15] {$Y1$} (Y);

        \draw (G0)    -- node[gatelabel, pos=0.2] {$Y2$} (G01a);
        \draw (G0)    -- node[gatelabel, pos=0.2] {$Y1$} (G01b);
        \draw (G1)    -- node[gatelabel, pos=0.2] {$X2$} (G01a);
        \draw (G1)    -- node[gatelabel, pos=0.2] {$X1$} (G01b);

        \draw (G1)    -- node[gatelabel, pos=0.2] {$Y2$} (G1-a);
        \draw (G1)    -- node[gatelabel, pos=0.2] {$Y1$} (G1-b);
        \draw (Gpen)  -- node[gatelabel, pos=0.2] {$X2$} (G-2a);
        \draw (Gpen)  -- node[gatelabel, pos=0.2] {$X1$} (G-2b);

        \draw (Gpen)  -- node[gatelabel, pos=0.2] {$Y2$} (Gpla);
        \draw (Gpen)  -- node[gatelabel, pos=0.2] {$Y1$} (Gplb);
        \draw (Glast) -- node[gatelabel, pos=0.2] {$X2$} (Gpla);
        \draw (Glast) -- node[gatelabel, pos=0.2] {$X1$} (Gplb);
      \end{tikzpicture}
    \end{center}
  \caption{The pictorial summary of the morphism $\cM_1$ from $\abilin_s^k(a)$.}%
  \label{fig:big-agate-a}
  \end{figure}

  The pictures for $\cM_0$ and $\cM_3$ (with $\cM_4,\dots,\cM_s$ being analogous to the latter) are straightforward by comparison and are shown in Figure~\ref{fig:big-agate-other}.

  \begin{figure}[htbp]
  \begin{center}
    \begin{tikzpicture}[
        defnode/.style={rectangle,draw,fill=lightgray,inner sep=3pt,outer sep=0pt,minimum size=10pt},
        gatelabel/.style={rectangle, inner sep=2pt, outer sep=0pt, fill=white,scale=0.6},
        subnode/.style={rounded rectangle, draw, inner sep=5pt, outer sep=0pt,minimum size=15pt},
        leaf/.style={rectangle,draw,fill=leafgreen,inner sep=3pt,outer sep=0pt,minimum size=10pt},
        scale=0.87
      ]
      \node[scale=0.8] at (-3,0) {$\cM_0:$};

      \node[subnode,scale=0.6] (G0) at (0,0) {$\trivial$};
      \node[subnode,scale=0.6] (G1) at (2.5,0) {$\trivial$};
      \node[scale=0.8] (Gdots) at (4.5,0) {$\dots$};
      \node[subnode,scale=0.6] (Gpen) at (6.5,0) {$\trivial$};
      \node[subnode,scale=0.6] (Glast) at (9.0,0) {$\trivial$};

      \node[leaf,scale=0.5] (X) at ($(G0)+(-1.7,0)$) {$x$};
      \node[leaf,scale=0.5] (Y) at ($(Glast)+(1.7,0)$) {$y$};

      \node[defnode,scale=0.5] (G01a) at ($(G0)!0.5!(G1)      + (0, 0.7)$) {$0$};
      \node[defnode,scale=0.5] (G01b) at ($(G0)!0.5!(G1)      + (0,-0.7)$) {$0$};
      \node[defnode,scale=0.5] (G1-a) at ($(G1)!0.7!(Gdots)   + (0, 0.7)$) {$0$};
      \node[defnode,scale=0.5] (G1-b) at ($(G1)!0.7!(Gdots)   + (0,-0.7)$) {$0$};

      \node[defnode,scale=0.5] (G-2a) at ($(Gdots)!0.3!(Gpen) + (0, 0.7)$) {$0$};
      \node[defnode,scale=0.5] (G-2b) at ($(Gdots)!0.3!(Gpen) + (0,-0.7)$) {$0$};
      \node[defnode,scale=0.5] (Gpla) at ($(Gpen)!0.5!(Glast) + (0, 0.7)$) {$0$};
      \node[defnode,scale=0.5] (Gplb) at ($(Gpen)!0.5!(Glast) + (0,-0.7)$) {$0$};

      \draw (G0)    -- node[gatelabel, pos=0.15] {$X1$} (X);
      \draw (Glast) -- node[gatelabel, pos=0.15] {$Y1$} (Y);

      \draw (G0)    -- node[gatelabel, pos=0.2] {$Y2$} (G01a);
      \draw (G0)    -- node[gatelabel, pos=0.2] {$Y1$} (G01b);
      \draw (G1)    -- node[gatelabel, pos=0.2] {$X2$} (G01a);
      \draw (G1)    -- node[gatelabel, pos=0.2] {$X1$} (G01b);

      \draw (G1)    -- node[gatelabel, pos=0.2] {$Y2$} (G1-a);
      \draw (G1)    -- node[gatelabel, pos=0.2] {$Y1$} (G1-b);
      \draw (Gpen)  -- node[gatelabel, pos=0.2] {$X2$} (G-2a);
      \draw (Gpen)  -- node[gatelabel, pos=0.2] {$X1$} (G-2b);

      \draw (Gpen)  -- node[gatelabel, pos=0.2] {$Y2$} (Gpla);
      \draw (Gpen)  -- node[gatelabel, pos=0.2] {$Y1$} (Gplb);
      \draw (Glast) -- node[gatelabel, pos=0.2] {$X2$} (Gpla);
      \draw (Glast) -- node[gatelabel, pos=0.2] {$X1$} (Gplb);
    \end{tikzpicture}
    \\[1\baselineskip]
    \begin{tikzpicture}[
        defnode/.style={rectangle,draw,fill=lightgray,inner sep=3pt,outer sep=0pt,minimum size=10pt},
        gatelabel/.style={rectangle, inner sep=2pt, outer sep=0pt, fill=white,scale=0.6},
        subnode/.style={rounded rectangle, draw, inner sep=5pt, outer sep=0pt,minimum size=15pt},
        leaf/.style={rectangle,draw,fill=leafgreen,inner sep=3pt,outer sep=0pt,minimum size=10pt},
        scale=0.87
      ]
      \node[scale=0.8] at (-3,0) {$\cM_3:$};

      \node[subnode,scale=0.6] (G0) at (0,0) {$\const$};
      \node[subnode,scale=0.6] (G1) at (2.5,0) {$\const$};
      \node[scale=0.8] (Gdots) at (4.5,0) {$\dots$};
      \node[subnode,scale=0.6] (Gpen) at (6.5,0) {$\const$};
      \node[subnode,scale=0.6] (Glast) at (9.0,0) {$\const$};

      \node[leaf,scale=0.5] (X) at ($(G0)+(-1.7,0)$) {$x$};
      \node[leaf,scale=0.5] (Y) at ($(Glast)+(1.7,0)$) {$x$};

      \node[defnode,scale=0.5] (G01a) at ($(G0)!0.5!(G1)      + (0, 0.7)$) {$x$};
      \node[defnode,scale=0.5] (G01b) at ($(G0)!0.5!(G1)      + (0,-0.7)$) {$x$};
      \node[defnode,scale=0.5] (G1-a) at ($(G1)!0.7!(Gdots)   + (0, 0.7)$) {$x$};
      \node[defnode,scale=0.5] (G1-b) at ($(G1)!0.7!(Gdots)   + (0,-0.7)$) {$x$};

      \node[defnode,scale=0.5] (G-2a) at ($(Gdots)!0.3!(Gpen) + (0, 0.7)$) {$x$};
      \node[defnode,scale=0.5] (G-2b) at ($(Gdots)!0.3!(Gpen) + (0,-0.7)$) {$x$};
      \node[defnode,scale=0.5] (Gpla) at ($(Gpen)!0.5!(Glast) + (0, 0.7)$) {$x$};
      \node[defnode,scale=0.5] (Gplb) at ($(Gpen)!0.5!(Glast) + (0,-0.7)$) {$x$};

      \draw (G0)    -- node[gatelabel, pos=0.15] {$X1$} (X);
      \draw (Glast) -- node[gatelabel, pos=0.15] {$Y1$} (Y);

      \draw (G0)    -- node[gatelabel, pos=0.2] {$Y2$} (G01a);
      \draw (G0)    -- node[gatelabel, pos=0.2] {$Y1$} (G01b);
      \draw (G1)    -- node[gatelabel, pos=0.2] {$X2$} (G01a);
      \draw (G1)    -- node[gatelabel, pos=0.2] {$X1$} (G01b);

      \draw (G1)    -- node[gatelabel, pos=0.2] {$Y2$} (G1-a);
      \draw (G1)    -- node[gatelabel, pos=0.2] {$Y1$} (G1-b);
      \draw (Gpen)  -- node[gatelabel, pos=0.2] {$X2$} (G-2a);
      \draw (Gpen)  -- node[gatelabel, pos=0.2] {$X1$} (G-2b);

      \draw (Gpen)  -- node[gatelabel, pos=0.2] {$Y2$} (Gpla);
      \draw (Gpen)  -- node[gatelabel, pos=0.2] {$Y1$} (Gplb);
      \draw (Glast) -- node[gatelabel, pos=0.2] {$X2$} (Gpla);
      \draw (Glast) -- node[gatelabel, pos=0.2] {$X1$} (Gplb);
    \end{tikzpicture}
  \end{center}
  \caption{The pictorial summary of the morphisms $\cM_0$ and $\cM_3$ from $\abilin_s^k(a)$.}%
  \label{fig:big-agate-other}
  \end{figure}
\end{proof}

Again we have treated the indices $1,2 \in \cR$ specially in Lemma~\ref{lem:agate-master}, but we can always permute $\cR$ to make other indices special.

\begin{corollary}%
  \label{cor:agate-master}
  For integers $s \ge 1$, $k \ge 1$ and $a$ with $|a|<2^{k+1}$, and distinct indices $i,j \in [s]$, there is an assignment $\abilin_{s;i,j}^k(a)$ of $\agate_s^k$ with $D_0 = \trivial(\{X,Y\})$, $D_i = \lag(a,A)$, $D_j = \lag(a,B)$ and $D_r = \const(\{X,Y\})$ for all $r \in [s] \setminus \{i,j\}$.
\end{corollary}

\subsection{Finishing the proof}%
\label{sub:finish-repr}

We now complete the proof of Theorem~\ref{thm:baby-thm}.
By far the most complicated ingredient is Lemma~\ref{lem:agate-master} in the special case $s=2$.

We recall the proof of Theorem~\ref{thm:baby-thm} proceeds via Theorem~\ref{thm:baby-gc}.
To get started, we make a standard observation that the $U^3$-norm datum has Cauchy--Schwarz complexity $2$.

\begin{lemma}%
  \label{lem:u3-cs-complex}
  The linear datum $U^3$ has Cauchy--Schwarz complexity $2$ at index $000$, in the sense of Definition~\ref{def:cs-complexity}.
  Hence we have $\diagram(U^3) \entails^0_\gamma \diagram(\gc_2)$ where $\gamma(\triangle) = (000, 0)$.
\end{lemma}
\begin{proof}%
  To see the first statement, we partition $\{0,1\}^3 \setminus \{000\}$ into three parts $S_1, S_2, S_3$ based on the first position that is a $1$; that is,
  \[
    S_i = \bigl\{ \omega \in \{0,1\}^3 \colon \omega_i=1,\ \omega_j = 0\ \forall j<i \bigr\}.
  \]
  Using linear maps
  \begin{align*}
    \mu_1(x) &= (x,-x,0,0) \\
    \mu_2(x) &= (x,0,-x,0) \\
    \mu_3(x) &= (x,0,0,-x)
  \end{align*}
  Definition~\ref{def:cs-complexity} holds.
  The second statement follows by Proposition~\ref{prop:diag-cs-complexity}.
\end{proof}

\begin{proof}[Proof of Theorem~\ref{thm:baby-gc}]%
  Fix $m \ge 0$ and $k=2^m$ such that $|a| < 2^{k+1}$.
  Combining Lemma~\ref{lem:gc-aggre}, Lemma~\ref{lem:build-agate} and Lemma~\ref{lem:build-large-agate} for $s=2$, we have $\gc_2 \entails^{m+5}_{\gamma_1} \agate_2^{k}$ where $\gamma_1(X) = (\triangle, 0)$, $\gamma_1(Y) = (\triangle, 1)$.
  Hence, $U^3 \entails^{m+5}_{\gamma_2} \agate_2^k$ where $\gamma_2(X) = (000,0)$ and $\gamma_2(Y) = (000,1)$.

  We make the following claim.
  \begin{claim}
    Let $\Psi$ be the linear datum from Lemma~\ref{lem:gc-cs-baby} \uppar{which depends on $a$}.
    Then there is a diagram morphism $\Theta \colon \diagram(\Psi) \to \agate_2^k$ with $\alpha^\Theta(X) = 1$ and $\alpha^\Theta(Y) = 2$, respecting $X$ and $Y$.
  \end{claim}
  \begin{proof}[Proof of claim]%
    By Proposition~\ref{prop:adjunction} it is equivalent to find a morphism of linear data $\Psi \to \datum(\agate_2^k)$ with the same properties.
    By Lemma~\ref{lem:gc-cs-baby} and Remark~\ref{rem:many-morphs}, it suffices to find a partition $\leafs^{\agate_2^k} = \{X,Y\} \cup S_1 \cup S_2$ and morphisms of linear data
    \begin{align*}
      \cM_0 \colon \trivial(\{X,Y\}) &\to \datum(\agate_2^k)[\{X,Y\}] \\
      \cM_1 \colon \lag(a, A) &\to \datum(\agate_2^k)[\{X,Y\} \cup S_1] \\
      \cM_2 \colon \lag(a, B) &\to \datum(\agate_2^k)[\{X,Y\} \cup S_2]
    \end{align*}
    with $\alpha^{\cM_r}(X) = X$, $\alpha^{\cM_r}(Y) = Y$ for $r=0,1,2$, and $\alpha^{\cM_r}(x) = \zeroi$ for $r=1,2$ and $x \in S_r$, all respecting $X,Y$.
    But by Lemma~\ref{lem:agate-master}, we have such a partition and morphisms of diagrams
    \begin{align*}
      \cM_0 \colon \trivial(\{X,Y\}) &\to \agate_2^k[\{X,Y\}] \\
      \cM_1 \colon \lag(a, A) &\to \agate_2^k[\{X,Y\} \cup S_1] \\
      \cM_2 \colon \lag(a, B) &\to \agate_2^k[\{X,Y\} \cup S_2]
    \end{align*}
    with the same properties, and by Proposition~\ref{prop:adjunction} again and Proposition~\ref{prop:sub-sub} this is equivalent.
  \end{proof}
  
  It follows that $\agate_2^k \entails^0_{\gamma_3} \Psi$ where $\gamma_3(X) = (1,0)$ and $\gamma_3(Y) = (2,0)$.
  The conclusion of Theorem~\ref{thm:baby-gc} follows, with $L = m+5$.
\end{proof}

As noted in Section~\ref{sub:baby-gc}, Theorem~\ref{thm:baby-thm} follows from Theorem~\ref{thm:baby-gc}, by Lemma~\ref{lem:gc-bound-baby} (and Corollary~\ref{cor:entails-entails}, and Corollary~\ref{cor:log-entails}).

\section{Some stashing theory}%
\label{sec:stashing}

Before proving Theorem~\ref{thm:main}, we must establish some useful results related to ``stashing'', as discussed in Section~\ref{sub:true-summary} (see also Remark~\ref{rem:stashing} and Example~\ref{ex:cache-example}).

\subsection{Discarding stashed structure}%
\label{sub:stashing-theory-1}

As described in Example~\ref{ex:cache-example}, the notion of ``stashing'' corresponds in the language of diagrams to joinings $C +_{i \leftrightarrow j} D$ at a single vertex, or equivalently to diagrams with a cut-vertex of in-degree $2$ and out-degree $0$.

Recall that a multilinear average $\Lambda(\cdots)$ of shape $C +_{i \leftrightarrow j} D$ can be thought of as a multilinear average $\Lambda(\cdots)$ of shape $C$, where we assume the function $f_i$ is a dual function of shape $(D,j)$.
We have stated that we can always simply discard this extra information about $f_i$, and revert to treating it as an arbitrary $1$-bounded function.
In the language of diagrams, this looks like ``collapsing'' the diagram $C +_{i \leftrightarrow j} D$ back onto $C$.
We now state this formally.

\begin{proposition}%
  \label{prop:discard-dual}
  Let $C$, $D$ be diagrams, let $i \in \leafs^C$ and $j \in \leafs^D$, and suppose $V^C_i = V^D_j$.
  Further assume $D$ is surjective at $j$ \uppar{see Definition~\ref{def:diag-surj}}.

  Then there is a morphism $\Theta \colon C \to \bigl( C +_{i \leftrightarrow j} D \bigr)$ with the property that \uppar{i} for all $x \in \leafs^C \setminus \{i\}$, $\alpha^{\Theta}(L;x) = x$ and $\Theta$ respects $(L;x)$, and \uppar{ii} for all $y \in \leafs^D$, $\alpha^\Theta(R;y) = i$. 

  Alternatively, $ \bigl( C +_{i \leftrightarrow j} D \bigr) \entails^0_\gamma C$ where $\gamma(x) = ((L;x),0)$ for every $x \in \leafs^C \setminus \{i\}$.
\end{proposition}
The surjectivity hypothesis is necessary, as otherwise the dual function interpretation does not even make sense.
\begin{proof}
  For brevity we write $\wt C = C +_{i \leftrightarrow j} D$.
  Let $z \in X^C$ be the unique parent of $i$ in $\cG^C$, meaning $z \in \nonleafs^C$ and $zi \in E^C$.

  Consider the space $V^{\datum(D)}$ from~\eqref{eq:limit-space} and write $\pi_x \colon V^{\datum(D)} \to V^D_x$ for the projection maps as in~\eqref{eq:projections}.
  By hypothesis, $\pi_j$ is surjective.
  Hence it has a right inverse $\psi \colon V^D_j \to V^{\datum(D)}$, i.e., $\pi_j \circ \psi = \id_{V^D_j}$.
  Recall $V^D_j = V^C_i$.

  We define the ``collapsing'' morphism $\Theta \colon C \to \wt C$ as follows:---
  \begin{itemize}
    \item for every $x \in X^C$, $x \ne i$, we take $\alpha(L;x) = x$ and $\theta_{(L;x)} = \id_{V^C_x}$;
    \item for all remaining leaves $(R;\ell)$ of $\wt C$ we take $\alpha(R;\ell) = i$ and $\theta_{R;\ell} = \pi_\ell \circ \psi$;
    \item for all remaining non-leaves $(R;x)$ of $\wt C$ (including the special case $(L;i)=(R;j)$) we take
      $\alpha(R;x) = z$ and
        $\theta_{(R;x)} = \pi_x \circ \psi \circ \phi^C_{zi}$.
  \end{itemize}

  To check this is a morphism we must consider~\eqref{eq:diag-morph} for each kind of edge in $\wt C$ in turn.
  The edges $(L;x)(L;y)$ where $y \ne j$ are straightforward because $\theta_{(L;x)}$ and $\theta_{(L;y)}$ are both the identity.
  For the specific edge $(L;z)(L;i)$ we have
  \[
    \theta_{(L;i)} = \theta_{(R;j)} = \pi_j \circ \psi \circ \phi^C_{zi} = \phi^C_{zi}
  \]
  by our choice of $\psi$, and so~\eqref{eq:diag-morph} becomes
  \[
    \theta_{(L;i)} \circ \phi^C_{zz} = \phi^C_{zi} = \phi^{\wt C}_{(L;z)(L;i)} \circ \theta_{(L;z)}
  \]
  as required, recalling Convention~\ref{convention:ast-revenge}.

  For the remaining cases we recall the property $\phi^D_{xy} \circ \pi_x = \pi_y$ for $xy \in E^D$, by the definition~\eqref{eq:limit-space} of $V^{\datum(D)}$ (see~\eqref{eq:pi-compat}).
  For two non-leaves $(R;x),(R;y) \in \nonleafs^{\wt C}$ with $xy \in E^D$, we have
  \[
    \theta_{(R;y)} \circ \phi^C_{zz} = \pi_y \circ \bigl( \psi \circ \phi^C_{zi} \bigr)  
    = \phi^D_{xy} \circ \pi_x \circ \bigl( \psi \circ \phi^C_{zi} \bigr)
    =\phi^{\wt C}_{(R;x)(R;y)} \circ \theta_{(R;x)} 
  \]
  as required.
  If $(R;x) \in \nonleafs^{\wt C}$, $(R;y) \in \leafs^{\wt C}$ and $xy \in E^D$, we similarly have
  \[
    \theta_{(R;y)} \circ \phi^C_{zi} = \pi_y \circ \bigl( \psi \circ \phi^C_{zi} \bigr)
    =\phi^D_{xy} \circ \pi_x \circ \bigl( \psi \circ \phi^C_{zi} \bigr)
    =\phi^{\wt C}_{(R;x)(R;y)} \circ \theta_{(R;x)} 
  \]
  for the same reason.

  Hence $\Theta$ is a morphism. 
  By definition it respects every vertex $(L;x)$ for $x \in \leafs^C \setminus \{i\}$, and the $\entails$ statement follows.
\end{proof}

We briefly record a consequence of Proposition~\ref{prop:discard-dual}.

\begin{corollary}%
  \label{cor:join-cs-complex}
  Suppose $C$, $D$ are diagrams, $i \in \leafs^C$, $j \in \leafs^D$ and $V_i^C=V_j^D$.
  Suppose further that $D$ is surjective at $j$.

  If $x \in \leafs^C \setminus \{i\}$ and $C$ has Cauchy--Schwarz complexity at most $s$ at $x$, then $C +_{i \leftrightarrow j} D$ has Cauchy--Schwarz complexity at most $s$ at $(L;x)$.
\end{corollary}
\begin{proof}%
  Write $\wt C = C +_{i \leftrightarrow j} D$ and $W = V_x^C = V_{L;x}^{\wt C}$.
  Let $\Theta \colon C \to \wt C$ be the morphism in Proposition~\ref{prop:discard-dual}.
  By Proposition~\ref{prop:diag-cs-complexity} we have a morphism $\Xi \colon \gc_s(W) \to C$ with $\alpha^\Theta(x) = \triangle$ respecting $x$.
  Then the result follows by applying Proposition~\ref{prop:diag-cs-complexity} to $\Theta \circ \Xi \colon \gc_s(W) \to \wt C$.
\end{proof}

\subsection{Revealing stashed structure}%
\label{sub:reveal}

The other natural stashing statement is that we can ignore the extra hypothesis that $f_i$ is a dual function of shape $(D,j)$ for as long as we like, performing Cauchy--Schwarz or morphism steps as if $f_i$ were an arbitrary function $1$-bounded function, and then later dramatically reveal that every copy of $f_i$ now present was a dual function all along.
Again we can state this within the diagram formalism.

We first fix some notation for discussing iterated joinings. 
\begin{convention}%
  Suppose $C$ and $D_1,\dots,D_m$ are diagrams, and $\beta_r$ is a partial matching between $J_r \subseteq \leafs^C$ and $J'_r \subseteq \leafs^{D_r}$ for each $r \in [m]$ such that $V^C_i = V^{D_r}_j$ for all $(i,j) \in \beta_r$.
  Hence $C +_{\beta_r} D_r$ is defined for each $r \in [m]$.

  Now suppose that $(J_r)_{r \in [m]}$ are disjoint subsets of $\leafs^C$.
  We can define the \emph{iterated joining}
  \[
    \bigg( \dots \Big( \big(C +_{\beta_1}^{,R_1} D_1 \big) +_{\beta_2}^{,R_2} D_2 \Big) \dots +_{\beta_m}^{,R_m} D_m \bigg).
  \]
  Note this makes sense, because the labelling convention $+_{\beta_1}^{,R_1}$ means that vertex labels of $C$ are still vertex labels of $C +_{\beta_1}^{,R_1} D_1$, etc..

  We simplify the notation in two ways. 
  First, we use the convention that $+_{\beta_r}^{,R_r}$ is \emph{left-associative}: that is, if we write
  \[
    C +_{\beta_1}^{,R_1} D_1 +_{\beta_2}^{,R_2} D_2 \dots +_{\beta_m}^{,R_m} D_m
  \]
  then the arrangement of parentheses above is implied.
  Second, we avoid the ellipsis by writing
  \[
    C \bigplus_{r=1}^m {}^{,R_r}_{\beta_r} D_r
  \]
  to mean the same expression.
\end{convention}
If for simplicity $\beta_r = \{(i_r,j_r)\}$ (which is the case we care about for stashing) this is just a roundabout way of describing the diagram:
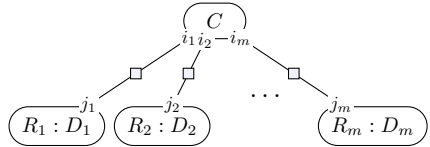
\vspace{0.5\baselineskip}
\begin{center}
  \begin{tikzpicture}[
    defnode/.style={rectangle,fill=lightgray,draw,inner sep=2pt,outer sep=0pt,minimum size=10pt},
    gatelabel/.style={rectangle, inner sep=1pt, outer sep=0pt, fill=white,scale=0.7},
    subnode/.style={rounded rectangle, draw, inner sep=5pt, outer sep=0pt,minimum size=15pt}]
    \node[subnode,scale=0.8] (V) at (0,0) {\phantom{--}$C$\phantom{--}};

    \node[subnode,scale=0.8] (W1) at (-2.1, -1.4) {$R_1 : D_1$};
    \node[subnode,scale=0.8] (W2) at (-0.7, -1.4) {$R_2 : D_2$};
    \node (dots) at (0.7, -1) {$\dots$};
    \node[subnode,scale=0.8] (Wm) at (2.1, -1.4) {$R_m : D_m$};

    \node[defnode,scale=0.4] (U1) at ($(V)!0.5!(W1)$) {};
    \node[defnode,scale=0.4] (U2) at ($(V)!0.5!(W2)$) {};
    \node[defnode,scale=0.4] (Um) at ($(V)!0.5!(Wm)$) {};

    \draw (V) -- node[pos=0.06, gatelabel] {$i_1$} (U1);
    \draw (V) -- node[pos=0.06, gatelabel] {$i_2$} (U2);
    \draw (V) -- node[pos=0.06, gatelabel] {$i_m$} (Um);

    \draw (W1) -- node[pos=0.06, gatelabel] {$j_1$} (U1);
    \draw (W2) -- node[pos=0.06, gatelabel] {$j_2$} (U2);
    \draw (Wm) -- node[pos=0.06, gatelabel] {$j_m$} (Um);
  \end{tikzpicture}
\end{center}
In particular it is clear from this picture that reordering the $+_{\beta_r}^{,R_r}$ terms does not alter the diagram or the labelling.
Hence we can use the notation $\bigplus_{i \in S}$ without specifying an ordering on $S$.

We now state the stashing / revealing result discussed above.
\begin{proposition}%
  \label{prop:awesome-stashing}
  Suppose $C_1$, $C_2$ are two diagrams and $C_1 \entails^M_{\gamma} C_2$.
  Suppose $D$ is another diagram, $i \in \leafs^{C_1}$, $j \in \leafs^D$ are leaves, $V^{C_1}_i = V^D_j$ and
  $D$ is surjective at $j$.
  Finally let $Y \subseteq \leafs^{C_2}$ be a subset such that $\gamma(r)$ is defined and equal to $(i,0)$ or $(i,1)$ for all $r \in Y$.  Then
  \[
    \left( C_1 +_{i \leftrightarrow j}^{,R} D  \right) \entails^M_{\gamma'} \left( C_2 \bigplus_{r \in Y} {}_{r \leftrightarrow j}^{,R_r} D \right).
  \]
  Here $\gamma'$ is defined as follows: for any $r \in Y$ and $x \in \leafs^D \setminus \{j\}$ we have $\gamma'(R_r;x) = ((R;x), t)$ where $\gamma(r) = (i,t)$.
  For any $y \in \leafs^{C_2}$ such that $\gamma(y) = (z,t)$ for some $z \in \leafs^{C_1} \setminus \{i\}$ we have $\gamma'(y) = (z,t)$.
\end{proposition}
A typical application will be phrased as follows.
We start by making some argument $C_0 \entails \left( C_1 +_{i \leftrightarrow j} D \right)$.
We will then say we ``stash'' $D$ and continue by reasoning about $C_1$, until we have shown $C_1 \entails C_2$.
Finally by Proposition~\ref{prop:awesome-stashing} we ``reveal'' copies of $D$ to deduce $C_0 \entails \left( C_2 \bigplus_{r \in Y} {}_{r \leftrightarrow j}^{,R_r} D \right)$.

The statement in Proposition~\ref{prop:awesome-stashing} is the one we will use in applications, but for the purposes of induction we will prove a slightly stronger version where the single vertex $i$ is replaced by a set of vertices.

\begin{lemma}%
  \label{lem:awesome-stashing}
  Suppose $C_1$, $C_2$ are two diagrams with $C_1 \entails^M_{\gamma} C_2$, that $D$ is another diagram, $j \in \leafs^D$ and $D$ is surjective at $j$.
  Also let $Z \subseteq \leafs^{C_1}$ and $Y \subseteq \leafs^{C_2}$ be subsets such that $V^{C_1}_i = V^D_j$ for all $i \in Z$, and $Y \subseteq \gamma^{-1}(Z \times \{0,1\})$.
  Then
  \[
    \left( C_1 \bigplus_{i \in Z} {}_{i \leftrightarrow j}^{,R_i} D  \right) \entails^M_{\gamma'} \left( C_2 \bigplus_{r \in Y} {}_{r \leftrightarrow j}^{,R_r} D \right).
  \]
  Here $\gamma'$ is as follows.
  Whenever $r \in Y$ and $x \in \leafs^D \setminus \{j\}$ we have $\gamma'(R_r;x) = ((R_i;x), t)$, where $\gamma(r) = (i,t)$.
  Whenever $y \in \leafs^{C_2}$ and $\gamma(y) = (z,t)$ for some $z \in \leafs^{C_1} \setminus Z$, we have $\gamma'(y) = (z,t)$.
\end{lemma}

Clearly Proposition~\ref{prop:awesome-stashing} is the case $Z=\{i\}$, up to minor relabelling.

\begin{proof}[Proof of Lemma~\ref{lem:awesome-stashing}]%
  We first make the following observation.
  \begin{claim}%
    \label{claim:awesome-claim0}
    Whenever $C$ is a diagram and $S' \subseteq S \subseteq \leafs^C$ are leaves such that $V^C_i = V^D_j$ for all $i \in S$, we have
    \[
      \left( C \bigplus_{i \in S} {}_{i \leftrightarrow j}^{,R_i} D \right) \entails^0_{\gamma} \left( C \bigplus_{i \in S'} {}_{i \leftrightarrow j}^{,R_i} D \right) 
    \]
    where $\gamma(i;x) = ((i;x),0)$ for all $i \in S'$ and $x \in \leafs^D \setminus \{j\}$, and $\gamma(y) = (y,0)$ for all $y \in \leafs^C \setminus S$.
  \end{claim}
  \begin{proof}[Proof of claim]%
    This follows by applying Proposition~\ref{prop:discard-dual} repeatedly for each $i \in S \setminus S'$ in turn.
  \end{proof}
  By applying the claim to either $C=C_1$ or $C=C_2$, we are free to make $Z$ smaller or $Y$ larger subject to the constraint $Y \subseteq \gamma^{-1}(Z \times \{0,1\})$.
  That is, without loss of generality we may assume that (a) for every $i \in Z$ there is at least one $r \in Y$ such that $\gamma(r)=(i,0)$ or $(i,1)$ (or else by Claim~\ref{claim:awesome-claim0} we reduce to the same statement for $Z' = Z \setminus \{i\}$), and (b) $Y = \gamma^{-1}(Z \times \{0,1\})$ (or else by Claim~\ref{claim:awesome-claim0} we reduce to the same statement with this larger choice of $Y$).

  Throughout we write
  \begin{align*}
    \wt C_1 &= C_1 \bigplus_{i \in Z} {}_{i \leftrightarrow j}^{,R_i} D &
    \wt C_2 &= C_2 \bigplus_{r \in Y} {}_{r \leftrightarrow j}^{,R_r} D  .
  \end{align*}
  We first consider the case that $C_1 \entails C_2$ by a single $\morph$ step.
  \begin{claim}%
    \label{claim:awesome-claim1}
    Suppose that $D$ and $j$ are as in the statement, that $\Theta \colon C_2 \to C_1$ is a morphism, and that $Z \subseteq \leafs^{C_1}$ and $Y \subseteq \leafs^{C_2}$ are subsets such that \uppar{i} $V_i^{C_1} = V_j^D$ for all $i \in Z$, \uppar{ii} $\Theta$ respects every $i \in Z$, and \uppar{iii} $Y = \alpha^{\Theta}(Z)$.

    Then there exists a morphism
    $
      \wt \Theta \colon \wt C_2 \to \wt C_1 
    $
    such that:---
    \begin{itemize}
      \item for all $y \in \leafs^{C_1} \setminus Z$ we have $\alpha^{\wt \Theta}(y) = \alpha^{\Theta}(y)$, and if $\Theta$ respects $y$ then so does $\wt \Theta$;
      \item for all $i \in Z$ and $x \in \leafs^D \setminus \{j\}$ we have $\alpha^{\wt \Theta}(R_i;x) = \bigl(R_{\alpha^{\Theta}(i)};x\bigr)$, and $\wt \Theta$ respects $R_i;x$.
    \end{itemize}
  \end{claim}
  Note that (ii), (iii) follow from our assumptions (a), (b) above in the context of a single morphism step.
  \begin{proof}[Proof of claim]%
    The original copies of $C_1$, $C_2$ in $\wt C_1$, $\wt C_2$ appear as
    \begin{align*}
      \wt C_1\bigl[X^{C_1}\bigr] &= C_1(Z \leadsto \nonleafs) &
      \wt C_2\bigl[X^{C_2}\bigr] &= C_2(Y \leadsto \nonleafs)
    \end{align*}
    (see Remark~\ref{rem:joining-sub}), so by Remark~\ref{rem:leaf-non-leaf-morph} applied to $\Theta \colon C_2 \to C_1$ we get a morphism $\Theta \colon \wt C_2\bigl[X^{C_2}\bigr] \to \wt C_1\bigl[X^{C_1}\bigr]$.
    We then compose this with the restriction $\wt C_2 \to \wt C_2[X^{C_2}]$ to give a morphism $\wt C_2 \to \wt C_1\bigl[X^{C_1}\bigr]$.

    For each pair $i \in Z$, $r \in Y$ with $\alpha^{\Theta}(i) = r$,
    we have
    \[
      \wt C_1[R_i] = D(j \leadsto \nonleafs) = \wt C_2[R_r].
    \]
    Composing with the restriction $\wt C_2 \to \wt C_2[R_r]$ gives a morphism $\wt C_2 \to \wt C_1[R_i]$.

    Finally we note that all these morphisms $\wt C_2 \to \wt C_1[-]$ coincide on the overlap vertices $i=(R_i;j)$ for $i \in Z$ (because $\Theta$ respects $i$), so Lemma~\ref{lem:patching} applies and gives a morphism $\wt \Theta \colon \wt C_2 \to \wt C_1$ with the properties claimed.
  \end{proof}
  
  Next we consider the possibility that $C_1 \entails C_2$ by a single $\CS(S)$ step.
  \begin{claim}%
    \label{claim:awesome-claim2}
    Suppose $S, Z \subseteq \leafs^{C_1}$ are disjoint subsets, $C_2 = C_1 +_{S}^{A,B} C_1$, and $Y = \{ (A;i), (B;i) \colon i \in Z \}$.
    Then $\wt C_1 \entails^1_{\gamma'} \wt C_2$.

    Here
    $\gamma'(A;y) = (y,0)$ and $\gamma'(B;y) = (y,1)$ for any $y \in \leafs^{C_1} \setminus (S \cup Z)$, and
    $\gamma'(R_{A;i};x) = ((R_i;x), 0)$ and $\gamma'(R_{B;i};x) = ((R_i;x),1)$ for any $i \in Z$ and $x \in \leafs^D \setminus \{j\}$.
  \end{claim}
  Again, the hypothesis that $Z \cap S = \emptyset$ is forced by assumption (a) and the identity of $Y$ is forced by assumption (b).
  \begin{proof}[Proof of claim]%
    Up to relabelling, we have
    $
      \wt C_2 = \wt C_1 +_{S} \wt C_2
    $.
    Hence the claimed $\entails$ follows exactly from $\CS(S)$ and relabelling.
  \end{proof}

  To complete the proof we argue by induction on the length of the sequence of $\CS$ and $\morph$ steps in the ``proof'' of $C_1 \entails C_2$, as implied by Definition~\ref{def:logic-notation-revenge}(c') (see Definition~\ref{def:logic-notation}(c)).
  It is clear that discarding information as in Definition~\ref{def:logic-notation-revenge}(d') has no effect.

  The case of zero steps vacuous, and Claim~\ref{claim:awesome-claim1} and Claim~\ref{claim:awesome-claim2} handle the case of one step.
  If $C_1 \entails_{\gamma_0}^{M_1} C_3$ and $C_3 \entails_{\gamma_1}^{M_2} C_2$ where each $\entails$ has fewer steps than the original $C_1 \entails C_2$, then we define
  \[
    \wt Y = \gamma_0^{-1}(Z \times \{0,1\}) \subseteq \leafs^{C_3}
  \]
  and apply the statement recursively to $(C_1, C_3, Z, \wt Y)$ and $(C_3, C_2, \wt Y, Y)$, giving
  \[
    \left( C_1 \bigplus_{i \in Z} {}_{i \leftrightarrow j}^{,R_i} D \right) 
    \entails_{\gamma_0'}^{M_1}
    \left( C_3 \bigplus_{\ell \in \wt Y} {}_{\ell \leftrightarrow j}^{,R_\ell} D \right) 
    \entails_{\gamma_1'}^{M_2}
    \left( C_2 \bigplus_{r \in Y} {}_{r \leftrightarrow j}^{,R_r} D \right)
  \]
  for functions $\gamma_0'$, $\gamma_1'$ whose composite is $\gamma'$.
  This completes the proof.
\end{proof}

\subsection{Surjectivity}%
\label{sub:surj}

To apply Proposition~\ref{prop:discard-dual} or Proposition~\ref{prop:awesome-stashing} we have to verify that $D$ is surjective at $j$.
In practice, the diagrams we build tend to be surjective everywhere, but this is usually not obvious.
We now give some tools to justify such statements.

The following fact is fundamental.
\begin{lemma}%
  \label{lem:surj-fact-1}
  If $C$, $D$ are two diagrams, $\Theta \colon C \to D$ is a morphism, $i \in X^D$, $\theta^\Theta_i \colon V^C_{\alpha^{\Theta}_i} \to V^D_i$ is surjective and $C$ is surjective at $\alpha^{\Theta}(i)$, then $D$ is surjective at $i$.
\end{lemma}
The hypothesis that $\theta^\Theta_i$ is surjective is certainly satisfied when $\Theta$ respects $i$.
\begin{proof}%
  Consider $V^{\datum(C)}$, $V^{\datum(D)}$ as in~\eqref{eq:limit-space} and write $\pi_{\alpha^\Theta(i)}^C \colon V^{\datum(C)} \to V^C_{\alpha^\Theta(i)}$ and $\pi^D_i \colon V^{\datum(D)} \to V^D_i$ for the projection maps as in~\eqref{eq:projections}.
  By Proposition~\ref{prop:diagram-morphism}, or more correctly by~\eqref{eq:theta-tilde} in its proof, we get a morphism $\wt \theta \colon V^{\datum(C)} \to V^{\datum(D)}$ such that
  \[
    \pi_i \circ \wt \theta = \theta^{\Theta}_i \circ \pi_{\alpha(i)} .
  \]
  By hypothesis the right-hand side is surjective, and therefore so is $\pi_i$.
\end{proof}

It follows that joining diagrams at a single vertex preserves surjectivity provided the joining itself is surjective.
\begin{corollary}%
  \label{cor:join-surj}
  Suppose $C$, $D$ are diagrams, $i \in \leafs^C$ and $j \in \leafs^D$ are leaves with $V^C_i=V^D_j$, $x \in X^C \setminus \{i\}$ is another vertex, $C$ is surjective at $x$ and $D$ is surjective at $j$.
  Then $C +_{i \leftrightarrow j} D$ is surjective at $(L;x)$.
\end{corollary}
\begin{proof}%
  This is immediate from Lemma~\ref{lem:surj-fact-1} and Proposition~\ref{prop:discard-dual}.
\end{proof}

Moreover, if we already have a hypothesis about the Cauchy--Schwarz complexity of a diagram $D$ at $i$, this implies it is surjective at $i$.
\begin{lemma}%
  \label{lem:cs-implies-surj}
  If $D$ is a diagram, $i \in \leafs^D$ is a leaf and $s_{\cs}(D, i) < \infty$, then $D$ is surjective at $i$.
\end{lemma}
\begin{proof}%
  By Proposition~\ref{prop:diag-cs-complexity} we get a morphism $\Theta \colon \gc_s(V^D_i) \to D$ with $\alpha^\Theta(i) = \triangle$ respecting $i$, for some finite $s$.
  It is clear by inspection that $\gc_s(V^D_i)$ is surjective at $\triangle$.
  The result now follows by Lemma~\ref{lem:surj-fact-1}.
\end{proof}

Many other general statements follow easily from Lemma~\ref{lem:surj-fact-1} (for instance, applying a $\CS(S)$ step preserves surjectivity) but these are the only ones we will need.

\section{A proof of Theorem~\ref{thm:main}}%
\label{sec:thm-main}

We will now prove Theorem~\ref{thm:main-ext}---and hence Theorem~\ref{thm:main}---via Theorem~\ref{thm:main-gc}, by constructing suitable circuits.
The circuit $\agate$ constructed in Section~\ref{sub:big-agate} will be a key component of these.

\subsection{The structure of the proof}%
\label{sub:main-outline}

We begin by stating some key lemmas that encode the overall layout of the proof.
From now on we fix a system of linear forms $\Phi = (\phi_i)_{i \in I}$ where $\phi_i \colon \FF_p^d \to \FF_p$.
(In Theorem~\ref{thm:main-gc} we used $I=[k]$ but we drop this requirement here.)
We may also think of $\Phi$ as a linear datum.

We will assume $(\phi_i)_{i \in I}$ are \emph{pairwise} linearly independent as elements of $(\FF_p^d)^\ast$.
Indeed if $\phi_i,\phi_j$ are linearly dependent then so are $\phi_i^{\otimes t},\phi_j^{\otimes t}$ for any $t \ge 1$, so $s(\Phi) = \infty$, contrary to the hypotheses of Theorem~\ref{thm:main-ext} or Theorem~\ref{thm:main}.

It follows that $s_{\cs}(\Phi,i) \le |I|-2$ for each $i \in I$, as discussed in Section~\ref{sub:background}: we can partition $(\phi_j)_{j \ne i}$ into $|I|-1$ singletons, and $\phi_i \notin \spn(\phi_j)$ for $j \ne i$.

Note that we can assume $\spn_{i \in I} \phi_i = (\FF_p^d)^\ast$, or in other words that $\Phi$ is non-degenerate (see Remark~\ref{rem:surj-degen}), or else we can reduce the value of $d$ without changing the value $\Lambda_\Phi(\cdots)$.
In particular we may freely assume that $d \le |I|$.

Suppose $i_0 \in I$ is the form of interest.
If $s(\Phi, i_0) = 0$ then $s_{\cs}(\Phi, i_0) = 0$ and everything is trivial, so we may assume $s(\Phi, i_0) \ge 1$.
For brevity we write $I' = I \setminus \{i_0\}$ for the remaining indices.

The initial phase of the argument is to use Cauchy--Schwarz complexity to replace each $(\phi_j)_{j \in I'}$ by a $\gc_{t_j}$-dual function, where $t_j = s_{\cs}(\Phi, j) \le |I|-2$.
We will use the following lemma.

\begin{lemma}%
  \label{lem:cs-dual}
  Suppose $D$ is a diagram, $i \in \leafs^D$ is a leaf and $s_{\cs}(D, i) \le t$.
  Then $D \entails^1_{\gamma} \left( D +_{i \leftrightarrow \triangle}^{,R_i} \gc_t(V^D_i) \right)$, where $\gamma$ is the function $\gamma(j) = (j,1)$ for all $j \in \leafs^D \setminus \{i\}$.
\end{lemma}

Applying this $|I|-1$ times to each $j \in I \setminus \{i_0\}$ will give a diagram
\begin{equation}
  \label{eq:boost-phi}
  D\bigl((t_j)_{j \in I'}\bigr) := \Phi \bigplus_{j \in I'} {}_{j \leftrightarrow \triangle}^{,R_j} \gc_{t_j}.
\end{equation}
By Corollary~\ref{cor:join-cs-complex}, replacing $\Phi$ with $\Phi +_{j \leftrightarrow \triangle} \gc_{t_j}$ does not affect the Cauchy--Schwarz complexity of the remaining vertices $j' \in I'$, $j' \ne j$, and so on, so the hypothesis of Lemma~\ref{lem:cs-dual} holds at each step.

The key lemma for the whole proof is the following.
\begin{lemma}[Key lemma]%
  \label{lem:key-lem}
  Suppose $\Phi=(\phi_i)_{i \in I}$ is a system of linear forms as above, $i_0 \in I^\Phi$, $s$ is an integer, $2 \le s \le |I|-2$, and $(t_j)_{j \in I'}$ are integers with $1 \le t_j \le s$.
  Suppose moreover that $s(\Phi,i_0) \le s-1$; that is, $\phi_{i_0}^{\otimes s}$ does not lie in the span of $\bigl((\phi_j)^{\otimes s}\bigr)_{j \in I'}$.
  
  Let $D = D((t_j)_{j \in I'})$ be the diagram~\eqref{eq:boost-phi}.
  Then $D \entails^M_{\triangle \mapsto i_0} \gc_{s-1}$, where $M = O\bigl(|I| (\log \lvert I \rvert + \log \log (10 L))\bigr)$ and $L$ is the value defined in Theorem~\ref{thm:main}.
\end{lemma}

In terms of functions, this would imply that
\[
  \lvert \Lambda_{\Phi}((f_i)_{i \in I}) \rvert \le \|f_{i_0}\|_{U^{s}}^{2^{-M-s}}
\]
whenever $(f_j)_{j \ne i_0}$ are $\gc_{t_j}$-dual functions.

If the hypothesis holds for $t_j = s_{\cs}(\Phi, j)$ and $s-1=s(\Phi,i_0)$, we can stop at this point.
This is the case in Example~\ref{ex:gw-ex}.
If not, we have to use Lemma~\ref{lem:key-lem} iteratively on \emph{other indices} $j$ than the one we care about to achieve \emph{degree reduction}: i.e., to improve some ``stashed $\gc_t$'' hypothesis to ``stashed $\gc_{t-1}$'', and so on.

\begin{corollary}%
  \label{cor:key-cor}
  Let $\Phi$ and $i_0$ be as in Lemma~\ref{lem:key-lem}.
  Let $j_0 \in I^\Phi$ be another index.
  Let $s$ be an integer, $2 \le s \le |I|-2$, and $(t_j)_{j \in I'}$ a tuple of integers, with $t_j \ge 1$ for all $j$ and $t_j \le s$ for all $j \ne j_0$.

  Suppose that $s(\Phi, j_0) \le s-1$.
  Define $(t'_j)_{j \in I'}$ by $t'_{j} = s-1$ if $j=j_0$ and $t'_j=t_j$ otherwise.
  Write $D = D\bigl((t_j)_{j \in I'}\bigr)$ and $D' = D\bigl((t'_j)_{j \in I'}\bigr)$ for the diagrams in~\eqref{eq:boost-phi}.
  Then
  $D \entails_{i_0 \mapsto i_0}^{M+2} D'$
  where $M$ has the same form as Lemma~\ref{lem:key-lem}.
\end{corollary}

Corollary~\ref{cor:key-cor} follows from Lemma~\ref{lem:key-lem} and judicious applications of stashing and Lemma~\ref{lem:cs-dual}.

Given this, we can deduce Theorem~\ref{thm:main-gc}.
\begin{proof}[Proof of Theorem~\ref{thm:main-gc} given these results]%
  By applying Lemma~\ref{lem:cs-dual} $|I|-1$ times as suggested above we deduce
  \[
    \Phi \entails^{|I|-1}_{i_0 \mapsto i_0} D\bigl((t_j)_{j \in I'}\bigr)
  \]
  where $t_j = s_{\cs}(\Phi, j) \le |I|-2$.

  Let $(t'_j)_{j \in I'}$ be the tuple $t'_j = \min\bigl(s(\Phi, i_0) + 1, t_j\bigr)$.
  We claim that
  \[
    \Phi \entails^{M'}_{i_0 \mapsto i_0} D\bigl((t'_j)_{j \in I'}\bigr)
  \]
  where $M' = O\bigl(|I|^3 (\log |I| + \log \log (10 L))\bigr)$.

  Indeed, so long as $s = \max_{j \in I'} t_j > s(\Phi, i_0)+1$, we can iteratively select $j_0$ with $t_{j_0}=s$ maximal and apply Corollary~\ref{cor:key-cor} to replace it with $t_{j_0}-1$.
  By hypothesis $s(\Phi, j_0) \le s(\Phi) \le s(\Phi, i_0) + 1$, so the hypothesis of Corollary~\ref{cor:key-cor} holds.
  Overall we apply Corollary~\ref{cor:key-cor} at most
  \[
    \sum_{j \in I'} (t_j - t'_j) \le (|I|-1)(|I|-2)
  \]
  times, which gives the bound on $M'$.

  Finally, Lemma~\ref{lem:key-lem} then implies
  $
    \Phi \entails^{M+M'}_{\triangle \mapsto i_0} \gc_{s(\Phi,i_0)}
  $
  where $M$ is the value from Lemma~\ref{lem:key-lem}.
\end{proof}

Now, Theorem~\ref{thm:main-ext} follows, by Corollary~\ref{cor:entails-entails}, Corollary~\ref{cor:log-entails} and Lemma~\ref{lem:gc-bound}.

We first prove Lemma~\ref{lem:cs-dual} and deduce Corollary~\ref{cor:key-cor} from Lemma~\ref{lem:key-lem}, since these are good illustrations of the stashing method.
After that we turn to the much harder problem of building a circuit to prove Lemma~\ref{lem:key-lem}.

\subsection{Some stashing-type proofs}%
\label{sub:stashing-applications}

We now prove Lemma~\ref{lem:cs-dual}.

\begin{proof}[Proof of Lemma~\ref{lem:cs-dual}]%
  Write $W = V^D_i$.
  By Proposition~\ref{prop:diag-cs-complexity} we have a morphism $\Theta \colon \gc_t(W) \to D$ with $\alpha^\Theta(i) = \triangle$ respecting $i$.
  In particular, $D \entails^0_{i \mapsto (\triangle,0)} \gc_s(W)$.

  By $\CS(\{i\})$ we have $D \entails^1_{\gamma_1} \bigl( D +_{\{i\}} D \bigr)$, where $\gamma_1(L;x) = \gamma_1(R;x) = x$ for all $x \in \leafs^D \setminus \{i\}$.
  We can think of the right-hand copy of $D$ as ``stashed'' and apply the fact from the first paragraph to the left copy.
  Stated formally, we invoke Proposition~\ref{prop:awesome-stashing} on $D \entails^0_{i \mapsto (\triangle,0)} \gc_s(W)$ to obtain
  \[
    \biggl(D +_{i \leftrightarrow i}^{,R} D\biggr) \entails^0_{\gamma_2} \left( \gc_s(W) +_{\triangle \leftrightarrow i}^{,R} D \right).
  \]
  Here $\gamma_2$ sends $(R;x) \mapsto ((R;x),1)$ for all $x \in \leafs^D \setminus \{i\}$.
  Note that by Lemma~\ref{lem:cs-implies-surj} we know that $D$ is surjective at $i$.

  Combining with the statement $D \entails^1_{\gamma_1} \left( D +_{\{i\}} D \right)$ and
  relabelling,\footnote{``Relabelling'' here includes switching our perspective so that the stashed copy of $D$ with label $R$ becomes the ``main part'', while what used to be the main part $\gc_s$ becomes the ``stashed part''.} this is what we wanted to prove.
\end{proof}

If it mattered, we could also prove the same statement with $\gamma(j) = (j,0)$, simply by swapping the roles of the left and right copies in $D +_{\{i\}} D$.

We now do the deduction of Corollary~\ref{cor:key-cor}.
First we note a useful lemma.

\begin{lemma}%
  \label{lem:gc-morph}
  For any finite-dimensional space $W$ and $1 \le t \le s$, there is a diagram morphism $\gc_s(W) \to \gc_t(W)$ with $\alpha(x)=x$ for all $x \in X^{\gc_t}$, respecting $\triangle$.
  In particular, $\gc_t \entails^0_{\triangle \mapsto \triangle} \gc_s$.
\end{lemma}
\begin{proof}%
  We have $s_{\cs}(\gc_t(W), \triangle) \le t$, e.g.\ by Lemma~\ref{lem:gc-cs} applied to the identity morphism $\gc_t(W) \to \gc_t(W)$.
  Hence $s_{\cs}(\gc_t(W), \triangle) \le s$.
  The morphism $\gc_s(W) \to \gc_t(W)$ is then supplied by Proposition~\ref{prop:diag-cs-complexity}(ii).

  If we prefer to be explicit, we can use a linear map $\theta \colon \FF_p^{s+1} \to \FF_p^{t+1}$ given by
  \[
    (a_1,a_2,\dots,a_{s+1}) \mapsto \bigl(a_1,a_2,\dots,a_t,a_{t+1}+a_{t+2}+\cdots+a_{s+1}\bigr).
  \]
  This determines the other linear maps uniquely; we omit the details.
\end{proof}

\begin{proof}[Proof of Corollary~\ref{cor:key-cor} given Lemma~\ref{lem:key-lem}]%
  Let $D_0 = D\bigl((t_j)_{j \in I'}\bigr)$.
  First, we discard the stashed $\gc_{t_{j_0}}$ at $j_0$: by Proposition~\ref{prop:discard-dual} $D_0 \entails^0_{i_0 \mapsto i_0} D_1$ where
  \[
    D_1 = \Phi \bigplus_{j \in I \setminus \{i_0,j_0\}} {}_{j \leftrightarrow \triangle}^{,R_j} \gc_{t_j}.
  \]
  It is clear by inspection that $\gc_{t_j}$ is surjective at $\triangle$.

  Next, we can apply $\CS(\{j_0\})$ to obtain $D_2 = D_1 +_{\{j_0\}}^{,A} D_1$.
  Certainly $D_1 \entails^1_{(A;i_0) \mapsto (i_0,1)} D_2$.
  We again stash the right-hand copy of $D_1$ and continue to reason about $D_1$.

  We would like to apply Lemma~\ref{lem:key-lem} to $D_1$ using $j_0$ in place of $i_0$, to show a statement of the form $D_1 \entails_{\triangle \mapsto j_0} \gc_{s-1}$.
  The problem is that we are missing a copy of $\gc_s$ at $i_0$.\footnote{We cannot adjust the definition of $D((t_j)_{j \in I'})$ to include a copy of $\gc_s$ at $i_0$, because then we would lose all memory of the function $f_{i_0}$ that we actually care about.
    Some copy of the original vertex $i_0$ must be respected at all times.
  Currently, that golden copy is $(A,i_0)$, hidden in the stashed part, so it is safe to discard the copy at $i_0$.}
  However we can reclaim this with Lemma~\ref{lem:cs-dual} and the following claim.

  \begin{claim}%
    \label{claim:cs-complex}
    Let $t = \max \bigl\{ t_j \colon j \in I \setminus \{i_0,j_0\} \bigr\}$.
    Then $D_1$ has Cauchy--Schwarz complexity at most $t$ at $i_0$.
  \end{claim}
  \begin{proof}[Proof of claim]%
    There are many ways to argue this but all implicitly use a partition $\leafs^{D_1} \setminus \{i_0\} = S_1 \cup \dots \cup S_{t+1}$ such as
    \begin{align*}
      S_1 &= \{j_0\} \cup \bigl\{ (R_j;1) \colon j \in I \setminus \{i_0,j_0\}\bigr\}; \\
      S_r &= \bigl\{ (R_j;r) \colon j \in I \setminus \{i_0,j_0\},\ t_j \ge r-1\bigr\} \ \ (2 \le r \le t+1).
    \end{align*}
    To justify this choice, we construct a morphism $\Theta \colon \gc_t \to D_1$ with $\alpha^\Theta(i_0) = \triangle$ and $\alpha^\Theta(x) = r$ for each $x \in S_r$, respecting $i_0$.
    This suffices by Proposition~\ref{prop:diag-cs-complexity}.

    Note that since $\phi_{i_0} \notin \spn(\phi_{j_0})$, there exists $u \in V^\Phi$ such that $\phi_{j_0}(u) = 0$ but $\phi_{i_0}(u) = 1$.
    Write $c_j = \phi_j(u)$ for $j \in I$.

    We build $\Theta$ by patching on each part of $D_1$.
    We have $D_1[I] \cong \Phi\bigl(I \setminus \{i_0,j_0\} \!\leadsto\! \nonleafs\bigr)$.
    Define a morphism $\Theta_0 \colon \gc_t \to D_1[I]$ by $\alpha^{\Theta_0}(i_0) = \triangle$, $\alpha^{\Theta_0}(j_0) = 1$ and $\alpha^{\Theta_0}(j) = \diamond$ for all other $j \in I$, with
    \begin{align*}
      \theta^{\Theta_0}_\diamond \colon \FF_p^{t+1} &\to V^\Phi \\
      (a_1,\dots,a_{t+1}) &\mapsto (a_1+\cdots+a_{t+1}) u.
    \end{align*}
    The other maps are (necessarily) $\theta^{\Theta_0}_{i_0} = \id_{\FF_p}$, $\theta^{\Theta_0}_{j_0} \colon \FF_p^t \to \FF_p$ the zero map, and $\theta^{\Theta_0}_j \colon \FF_p^{t+1} \to \FF_p$ given by $\phi_j \circ \theta^{\Theta_0}_\diamond$ for all $j \in I \setminus \{i_0,j_0\}$.
    Then~\eqref{eq:diag-morph} holds for the edges $\diamond i_0$, $\diamond j_0$ by our choice of $u$, and for the other edges by definition of $\theta_j^{\Theta}$.

    Now for each $j \in I \setminus \{i_0,j_0\}$ we have $D_1[R_j] \cong \gc_{t_j}(\triangle \leadsto \nonleafs)$.
    To define a morphism $\Theta_j \colon \gc_t \to D_1[R_j]$,
    we first define $\Xi_j \colon \gc_t \to \gc_t(\triangle \leadsto \nonleafs)$ by:---
    \begin{align*}
      \alpha^{\Xi_j}(r) &= r \ : \  r \in [t+1] & \alpha^{\Xi_j}(\triangle) = \alpha^{\Xi_j}(\diamond) &= \diamond
    \end{align*}
    and
    \begin{align*}
      \theta^{\Xi_j}_\triangle &= \phi_j \circ \theta^{\Theta_0}_\diamond = \bigl( (a_1,\dots,a_{t+1}) \mapsto c_j (a_1+\cdots+a_{t+1}) \bigr) \\
      \theta^{\Xi_j}_x &= \bigl( v \mapsto c_j v \bigr) \ \  \forall x \ne \triangle.
    \end{align*}
    It is routine to verify~\eqref{eq:diag-morph}, and we note $\theta^{\Xi_j}_\triangle = \theta^{\Theta_0}_j$.
    By Lemma~\ref{lem:gc-morph} we have a morphism $\gc_t \to \gc_{t_j}$ respecting $\triangle$, and by Remark~\ref{rem:leaf-non-leaf-morph} we obtain a morphism $\gc_t(\triangle \leadsto \nonleafs) \to D_1[R_j]$.
    We let $\Theta_j \colon \gc_t \to D[R_j]$ be the composite of $\Xi_j$ with this morphism.
    It follows that $\theta^{\Theta_j}_\triangle = \theta^{\Theta_0}_j$.
    
    Patching together $\Theta_0$ and $(\Theta_j)_{j \in I \setminus \{i_0,j_0\}}$ by Lemma~\ref{lem:patching} gives a morphism $\Theta \colon \gc_t \to D_1$ with the required properties.
  \end{proof}
  By the claim and Lemma~\ref{lem:cs-dual} we deduce that
  \[
    D_1 \entails^1_{j_0 \mapsto j_0} \left( D_1 +_{i_0 \leftrightarrow \triangle}^{,R_{i_0}} \gc_s \right) =: D_3.
  \]
  Given the hypothesis $s(\Phi, j_0) \le s-1$,
  we can now apply Lemma~\ref{lem:key-lem} with the roles of the indices $i_0$, $j_0$ reversed to conclude
  $D_3 \entails^M_{\triangle \mapsto j_0} \gc_{s-1}$.
  Combining these last two statements, $D_1 \entails^{M+1}_{\triangle \mapsto j_0} \gc_{s-1}$.

  Finally we reveal the stashed copy of $D_1$.
  By Proposition~\ref{prop:awesome-stashing}, $D_1 \entails^{M+1}_{\triangle \mapsto j_0} \gc_{s-1}$ implies
  \[
    D_2 = \left( D_1 +_{j_0 \leftrightarrow j_0}^{,A} D_1 \right) \entails^{M+1}_{(A;i_0) \mapsto (A;i_0)} \left( \gc_{t-1} +_{\triangle \leftrightarrow j_0}^{,A} D_1 \right).
  \]
  We need to verify that $D_1$ is surjective at $j_0$: indeed $\Phi$ is surjective at $j_0$ by inspection and hence $D_1$ is by repeated application of Corollary~\ref{cor:join-surj}.

  The right-hand side is the diagram we want, up to relabelling.
  Combining with the statements $D_0 \entails^0_{i_0 \mapsto i_0} D_1$ and $D_1 \entails^1_{A;i_0 \mapsto i_0} D_2$ at the start, we obtain the result.
\end{proof}

\subsection{The Bridge gate}%
\label{sub:bridge}

The next few subsections introduce some further gates we will need in Section~\ref{sub:stargates} to prove Lemma~\ref{lem:key-lem}.
We start with a very boring one.

\begin{definition}%
  \label{def:bridge-gate}
  For $s \ge 1$, the diagram $\bridge_s$ is $\gc_s +_{\{s+1\}}^{L,R} \gc_s$.
  We also add labels $X$ for $L;\triangle$ and $Y$ for $R;\triangle$ as shown below.
  \begin{center}
    \begin{tikzpicture}[
        defnode/.style={rectangle,fill=lightgray,draw,inner sep=2pt,outer sep=0pt,minimum size=10pt},
        gatelabel/.style={rectangle, inner sep=2pt, outer sep=0pt, fill=white,scale=0.6},
        subnode/.style={rounded rectangle, draw, inner sep=5pt, outer sep=0pt,minimum size=15pt},
        leaf/.style={defnode,fill=leafgreen},
        scale=0.8
      ]

      \node[subnode,scale=0.9] (VL) at (0,0) {$L \,:\, \gc_s$};
      \node[subnode,scale=0.9] (VR) at (5,0) {$R \,:\, \gc_s$};

      \node[defnode,scale=0.5] (W) at ($(VL)!0.5!(VR)$) {};

      \node[leaf,scale=0.7] (X) at ($(VL)+(-2,0)$) {$X$};
      \node[leaf,scale=0.7] (Y) at ($(VR)+( 2,0)$) {$Y$};

      \draw (VL) -- node[gatelabel, pos=0.17] {$s+1$} (W);
      \draw (VR) -- node[gatelabel, pos=0.17] {$s+1$} (W);
      \draw (VL) -- node[gatelabel, pos=0.12] {$\triangle$} (X);
      \draw (VR) -- node[gatelabel, pos=0.12] {$\triangle$} (Y);
    \end{tikzpicture}
  \end{center}
\end{definition}

It is clear we can build this from $\gc_s$ with a single $\CS(\{s+1\})$.
\begin{definition}%
  \label{def:bridge-gate-gate}
  The gate $\bridge_s$ consists of the diagram $\bridge_s$ with $\cR = [s]$ and $\cP = \{X,Y\}$.
\end{definition}

There is only one interesting assignment of $\bridge_s$, which does nothing: that is, it sets the values at $X$ and $Y$ to be equal in every mode.

\begin{lemma}%
  \label{lem:bridge-ass}
  There is an assignment $\boring$ of $\bridge_s$ with $D_0 = \trivial(\{X,Y\})$ and $D_r = \const(\{X,Y\})$ for $1 \le r \le s$.
\end{lemma}
\begin{proof}%
  The partition $\toggles = S_1 \cup \dots \cup S_s$ is given by $S_r = \{ (L;r), (R;r) \}$ for $r \in [s]$.

  For $r \in [s+1]$ we define $\tau_r \colon \FF_p \to \FF_p^{s+1}$ as in the proof of Lemma~\ref{lem:gc-cs} by
  \[
    \tau_r(z) = (0, \dots, 0, z, 0, \dots, 0)
  \]
  where the $z$ is in position $r$ on the right.
  For $1 \le r \le s$ we can now define $\cM_r \colon \const(\{X,Y\}) \to \bridge_s[\nonleafs \cup \pins \cup S_r]$ by the picture
  
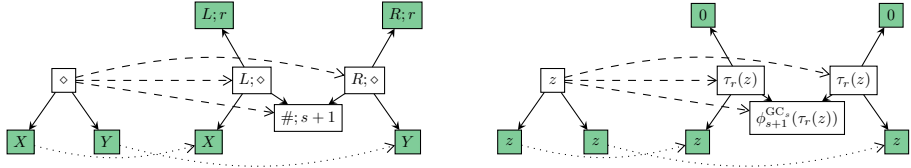
\begin{center}
    \begin{tikzpicture}[
        defnode/.style={rectangle,draw,inner sep=4pt,outer sep=0pt,minimum size=15pt},
        mydot/.style={circle,fill,inner sep=0.5pt},
        leaf/.style={defnode,fill=leafgreen},
        nonleaf/.style={defnode,fill=lightgray},
        toggle/.style={defnode,fill=leafgreen}
      ]
      \begin{scope}[shift={(0,0)}]
        \node[defnode,scale=0.6] (U) at (0, 0) {$\diamond$};
        \node[leaf,scale=0.6]    (UXX) at ($(U) + (235:1)$) {$X$};
        \node[leaf,scale=0.6]    (UXY) at ($(U) + (305:1)$) {$Y$};
        \draw[-stealth] (U) -- (UXX);
        \draw[-stealth] (U) -- (UXY);

        \node[defnode,scale=0.6] (VL) at (2.5, 0) {$L;\diamond$};
        \node[defnode,scale=0.6] (VR) at (4, 0) {$R;\diamond$};
        \node[defnode,scale=0.6] (V-) at ($(VL)!0.5!(VR)+(0,-0.5)$) {$\#;s+1$};

        \node[leaf,scale=0.6]    (AL) at ($(VL) + (120:1)$) {$L;r$};
        \node[leaf,scale=0.6]    (AR) at ($(VR) + (60:1)$) {$R;r$};
        \node[leaf,scale=0.6]    (X) at ($(VL) + (235:1)$) {$X$};
        \node[leaf,scale=0.6]    (Y) at ($(VR) + (305:1)$) {$Y$};

        \draw[-stealth] (VL) -- (V-);
        \draw[-stealth] (VL) -- (AL);
        \draw[-stealth] (VL) -- (X);
        \draw[-stealth] (VR) -- (V-);
        \draw[-stealth] (VR) -- (AR);
        \draw[-stealth] (VR) -- (Y);

        \draw[dashed, -angle 60] (U) to[bend left=0] (VL);
        \draw[dashed, -angle 60] (U) to[bend left=14] (VR);
        \draw[dashed, -angle 60] (U) to[bend right=0] (V-);
        \draw[dotted, -angle 60] (UXX) to[bend right=15] (X);
        \draw[dotted, -angle 60] (UXY) to[bend right=15] (Y);
      \end{scope}
      \begin{scope}[shift={(6.5,0)}]
        \node[defnode,scale=0.6] (U) at (0, 0) {$z$};
        \node[leaf,scale=0.6]    (UXX) at ($(U) + (235:1)$) {$z$};
        \node[leaf,scale=0.6]    (UXY) at ($(U) + (305:1)$) {$z$};
        \draw[-stealth] (U) -- (UXX);
        \draw[-stealth] (U) -- (UXY);

        \node[defnode,scale=0.6] (VL) at (2.5, 0) {$\tau_r(z)$};
        \node[defnode,scale=0.6] (VR) at (4, 0) {$\tau_r(z)$};
        \node[defnode,scale=0.6] (V-) at ($(VL)!0.5!(VR)+(0,-0.5)$) {$\phi^{\gc_s}_{s+1}(\tau_r(z))$};

        \node[leaf,scale=0.6]    (AL) at ($(VL) + (120:1)$) {$0$};
        \node[leaf,scale=0.6]    (AR) at ($(VR) + (60:1)$) {$0$};
        \node[leaf,scale=0.6]    (X) at ($(VL) + (235:1)$) {$z$};
        \node[leaf,scale=0.6]    (Y) at ($(VR) + (305:1)$) {$z$};

        \draw[-stealth] (VL) -- (V-);
        \draw[-stealth] (VL) -- (AL);
        \draw[-stealth] (VL) -- (X);
        \draw[-stealth] (VR) -- (V-);
        \draw[-stealth] (VR) -- (AR);
        \draw[-stealth] (VR) -- (Y);

        \draw[dashed, -angle 60] (U) to[bend left=0] (VL);
        \draw[dashed, -angle 60] (U) to[bend left=14] (VR);
        \draw[dashed, -angle 60] (U) to[bend right=0] (V-);
        \draw[dotted, -angle 60] (UXX) to[bend right=15] (X);
        \draw[dotted, -angle 60] (UXY) to[bend right=15] (Y);
      \end{scope}
    \end{tikzpicture}
  \end{center}
  We have use the properties $\phi^{\gc_s}_\triangle \circ \tau_r = \id_{\FF_p}$ and $\phi^{\gc_s}_r \circ \tau_r =0$.

  Similarly, we define $\cM_0 \colon \trivial(\{X,Y\}) \to \bridge_s[\nonleafs \cup \{X,Y\}]$ by
  \begin{center}
    \begin{tikzpicture}[
        defnode/.style={rectangle,draw,inner sep=4pt,outer sep=0pt,minimum size=15pt},
        mydot/.style={circle,fill,inner sep=0.5pt},
        leaf/.style={defnode,fill=leafgreen},
        nonleaf/.style={defnode,fill=lightgray},
        toggle/.style={defnode,fill=leafgreen}
      ]
      \begin{scope}[shift={(0,0)}]
        \node[defnode,scale=0.6] (U) at (0, 0) {$\diamond$};
        \node[leaf,scale=0.6]    (UXX) at ($(U) + (235:1)$) {$X$};
        \node[leaf,scale=0.6]    (UXY) at ($(U) + (305:1)$) {$Y$};
        \draw[-stealth] (U) -- (UXX);
        \draw[-stealth] (U) -- (UXY);

        \node[defnode,scale=0.6] (VL) at (2.5, 0) {$L;\diamond$};
        \node[defnode,scale=0.6] (VR) at (4, 0) {$R;\diamond$};
        \node[defnode,scale=0.6] (V-) at ($(VL)!0.5!(VR)+(0,-0.5)$) {$\#;s+1$};

        \node[leaf,scale=0.6]    (X) at ($(VL) + (235:1)$) {$X$};
        \node[leaf,scale=0.6]    (Y) at ($(VR) + (305:1)$) {$Y$};

        \draw[-stealth] (VL) -- (V-);
        \draw[-stealth] (VL) -- (X);
        \draw[-stealth] (VR) -- (V-);
        \draw[-stealth] (VR) -- (Y);

        \draw[dashed, -angle 60] (U) to[bend left=0] (VL);
        \draw[dashed, -angle 60] (U) to[bend left=14] (VR);
        \draw[dashed, -angle 60] (U) to[bend right=0] (V-);
        \draw[dotted, -angle 60] (UXX) to[bend right=15] (X);
        \draw[dotted, -angle 60] (UXY) to[bend right=15] (Y);
      \end{scope}
      \begin{scope}[shift={(6.5,0)}]
        \node[defnode,scale=0.6] (U) at (0, 0) {$(x,y)$};
        \node[leaf,scale=0.6]    (UXX) at ($(U) + (235:1)$) {$x$};
        \node[leaf,scale=0.6]    (UXY) at ($(U) + (305:1)$) {$y$};
        \draw[-stealth] (U) -- (UXX);
        \draw[-stealth] (U) -- (UXY);

        \node[defnode,scale=0.6] (VL) at (2.5, 0) {$\tau_{s+1}(x)$};
        \node[defnode,scale=0.6] (VR) at (4, 0) {$\tau_{s+1}(y)$};
        \node[defnode,scale=0.6] (V-) at ($(VL)!0.5!(VR)+(0,-0.5)$) {$0$};

        \node[leaf,scale=0.6]    (X) at ($(VL) + (235:1)$) {$x$};
        \node[leaf,scale=0.6]    (Y) at ($(VR) + (305:1)$) {$y$};

        \draw[-stealth] (VL) -- (V-);
        \draw[-stealth] (VL) -- (X);
        \draw[-stealth] (VR) -- (V-);
        \draw[-stealth] (VR) -- (Y);

        \draw[dashed, -angle 60] (U) to[bend left=0] (VL);
        \draw[dashed, -angle 60] (U) to[bend left=14] (VR);
        \draw[dashed, -angle 60] (U) to[bend right=0] (V-);
        \draw[dotted, -angle 60] (UXX) to[bend right=15] (X);
        \draw[dotted, -angle 60] (UXY) to[bend right=15] (Y);
      \end{scope}
    \end{tikzpicture}
  \end{center}
  using the properties $\phi_\triangle^{\gc_s} \circ \tau_{s+1} = \id_{\FF_p}$ and $\phi_{s+1}^{\gc_s} \circ \tau_{s+1} = 0$.
\end{proof}

The reader may wonder what need we have for a gate that does nothing.
The answer is that these objects are useful for the process of \emph{building} complicated gates and diagrams: they allow us to keep spare copies of $\gc_s$ floating around to be used later.
The purpose of this gate definition is to show that copies that remain unused do not do any harm.

\subsection{The super $\agate$}%
\label{sub:super-agate}

In Section~\ref{sub:big-agate} we defined a gate $\agate_s^{k}$ that encoded a tensor identity such as
\[
  (a v_1) \otimes v_2 \otimes \dots \otimes v_s = v_1 \otimes (a v_2) \otimes \dots \otimes v_s.
\]
Specifically this captures multilinearity in a single pair of indices.
We now strap a few of these together to encode multilinearity in all the indices simultaneously: that is, to encode a tensor identity
\[
  (a_1 v_1) \otimes (a_2 v_2) \otimes \dots \otimes (a_s v_s) = (a_1 a_2 \dots a_s) v_1 \otimes v_2 \otimes \dots \otimes v_s.
\]

\begin{definition}%
  \label{def:super-agate}
  The diagram $\supera_s^{k,m}$ consists of $m$ copies of $\agate_s^k$ glued together, as follows.
  \begin{center}
    \begin{tikzpicture}[
        defnode/.style={rectangle,draw,fill=lightgray,inner sep=2pt,outer sep=0pt,minimum size=10pt},
        gatelabel/.style={rectangle, inner sep=2pt, outer sep=0pt, fill=white,scale=0.6},
        subnode/.style={rounded rectangle, draw, inner sep=5pt, outer sep=0pt,minimum size=15pt},
        leaf/.style={rectangle,draw,fill=leafgreen,inner sep=2pt,outer sep=0pt,minimum size=10pt},
        scale=0.75
      ]

      \node[subnode,scale=0.6] (G0) at (0,0) {$H0 : \agate^k_s$};
      \node[subnode,scale=0.6] (G1) at (3.2,0) {$H1 : \agate^k_s$};
      \node[subnode,scale=0.6] (Gpen) at (7.8,0) {$H{\scriptstyle(m-2)} : \agate^k_s$};
      \node[scale=0.8] (Gdots) at ($(G1.east)!0.5!(Gpen.west)$) {$\dots$};
      \node[subnode,scale=0.6] (Glast) at (12,0) {$H{\scriptstyle(m-1)} : \agate^k_s$};

      \node[leaf,scale=0.5] (X) at ($(G0)+(-2,0)$) {$X$};
      \node[leaf,scale=0.5] (Y) at ($(Glast)+(2.2,0)$) {$Y$};

      \node[defnode,scale=0.5] (G01) at ($(G0)!0.5!(G1)$) {};
      \node[defnode,scale=0.5] (G1-) at ($(G1.east)!0.6!(Gdots)$) {};

      \node[defnode,scale=0.5] (G-2) at ($(Gdots)!0.4!(Gpen.west)$) {};
      \node[defnode,scale=0.5] (Gpl) at ($(Gpen)!0.5!(Glast)$) {};

      \draw (G0)    -- node[gatelabel, pos=0.15] {$X$} (X);
      \draw (Glast) -- node[gatelabel, pos=0.15] {$Y$} (Y);

      \draw (G0)    -- node[gatelabel, pos=0.2] {$Y$} (G01);
      \draw (G1)    -- node[gatelabel, pos=0.2] {$X$} (G01);

      \draw (G1)    -- node[gatelabel, pos=0.2] {$Y$} (G1-);
      \draw (Gpen)  -- node[gatelabel, pos=0.2] {$X$} (G-2);

      \draw (Gpen)  -- node[gatelabel, pos=0.2] {$Y$} (Gpl);
      \draw (Glast) -- node[gatelabel, pos=0.2] {$X$} (Gpl);
    \end{tikzpicture}
  \end{center}
  Again note we have introduced redundant labels $X=H0;X$ and $Y=H(m-1);Y$.
\end{definition}

Again this is a gate with two pins.
\begin{definition}%
  \label{def:super-agate-gate}
  The gate $\supera_s^{k,m}$ consists of the diagram $\supera_s^{k,m}$ together with $\cR = [s]$ and $\pins=\{X,Y\}$.
\end{definition}
We also consider the sub-diagrams $H0,\dots,H{\scriptstyle (m-1)}$ to be gates with gate structure $\agate_s^k$ (again noting Footnote~\ref{footnote:subgate}).

It is straightforward to build such a gate starting from a single $\agate_s^k$, at least when $m$ is a power of $2$.
\begin{lemma}%
  \label{lem:build-super-agate}
  For $s \ge 2$, $k \ge 1$ and $m=2^M$ for $M \ge 0$, we have $\agate_s^k \entails^M_\gamma \supera_s^{k,m}$, where $\gamma(X)=(X,0)$ and $\gamma(Y)=(X,1)$, unless $M=0$ in which case $\gamma(Y)=(Y,0)$.
\end{lemma}
\begin{proof}%
  When $M=0$ the gates $\agate_s^k$ and $\supera_s^{k,1}$ are identical up to relabelling, and just as in Lemma~\ref{lem:build-large-agate} we argue that
  \[
    \supera_s^{k,m} \entails^1 \left(\supera_s^{k,m} +_{Y \leftrightarrow Y} \supera_s^{k,m}\right)
  \]
  where the right-hand side is the same as $\supera_s^{k,2m}$ up to relabelling.
  (As before, this relabelling is non-trivial because we have to reflect the right-hand copy.)
  The result follows by induction on $M$.
\end{proof}

The assignment of $\supera_s^{k,m}$ we will use is the following.

\begin{lemma}%
  \label{lem:super-assignment}
  Suppose $s \ge 2$, $k,m \ge 1$ and $s-1 \ge m$.
  Let $a_2,\dots,a_s$ be integers with $|a_i| < 2^{k+1}$, and write $\alpha = \prod_{i=2}^n a_i$.
  Then there is an assignment $\amulti_s^{k,m}(a_2,\dots,a_n)$ of $\supera_s^{k,m}$ with:---
  \begin{itemize}
    \item $D_0 = \trivial(\{X,Y\})$,
    \item $D_r = \lag(a_r, A)$ for $2 \le r \le s$,
    \item $D_1 = \lag(\alpha, B)$.
  \end{itemize}
\end{lemma}
\begin{proof}%
  We will build the assignments out of the following sub-assignments.
  For $0 \le i \le s-2$, set
  \[
    Hi \colon \abilin_{s;i+2,1}^k(a_{i+2});
  \]
  see Corollary~\ref{cor:agate-master}.
  For $s-2 < i \le m-1$, set $Hi \colon \abilin_{s;1,2}^k(1)$.
  The partition $S_1,\dots,S_s$ is the union of the partitions from these sub-assignments, as usual.

  In other words, the sub-gate $Hi$ is responsible for moving the factor $a_{i+2}$ from tensor mode $i+2$ to tensor mode $1$. 
  If we have left-over gates $Hi$, $i>s-2$ they do nothing.

  To make this rigorous we must construct morphism $\cM_0,\dots,\cM_s$ from the morphisms $\cM_0^{(i)},\dots,\cM_s^{(i)}$ for each $0 \le i \le m-1$ obtained from Corollary~\ref{cor:agate-master} with these assignments.
  Each morphism $\cM_1,\dots,\cM_s$ has a pictorial summary of the form
  \begin{center}
    \begin{tikzpicture}[
        defnode/.style={rectangle,draw,fill=lightgray,inner sep=2pt,outer sep=0pt,minimum size=10pt},
        gatelabel/.style={rectangle, inner sep=2pt, outer sep=0pt, fill=white,scale=0.6},
        subnode/.style={rounded rectangle, draw, inner sep=5pt, outer sep=0pt,minimum size=15pt},
        leaf/.style={rectangle,draw,fill=leafgreen,inner sep=2pt,outer sep=0pt,minimum size=10pt},
        scale=0.75
      ]

      \node[subnode,scale=0.6] (G0) at (0,0) {$H0 : \agate^k_s$};
      \node[subnode,scale=0.6] (G1) at (3.2,0) {$H1 : \agate^k_s$};
      \node[subnode,scale=0.6] (Gpen) at (7.8,0) {$H(m-2) : \agate^k_s$};
      \node[scale=0.8] (Gdots) at ($(G1.east)!0.5!(Gpen.west)$) {$\dots$};
      \node[subnode,scale=0.6] (Glast) at (12,0) {$H(m-1) : \agate^k_s$};

      \node[leaf,scale=0.5] (X) at ($(G0)+(-2,0)$) {$t_0 x$};
      \node[leaf,scale=0.5] (Y) at ($(Glast)+(2.2,0)$) {$t_m x$};

      \node[defnode,scale=0.5] (G01) at ($(G0)!0.5!(G1)$) {$t_1 x$};
      \node[defnode,scale=0.5] (G1-) at ($(G1.east)!0.6!(Gdots)$) {$t_2 x$};

      \node[defnode,scale=0.5] (G-2) at ($(Gdots)!0.4!(Gpen.west)$) {$t_{m-2} x$};
      \node[defnode,scale=0.5] (Gpl) at ($(Gpen)!0.5!(Glast)$) {$t_{m-1} x$};

      \draw (G0)    -- node[gatelabel, pos=0.15] {$X$} (X);
      \draw (Glast) -- node[gatelabel, pos=0.15] {$Y$} (Y);

      \draw (G0)    -- node[gatelabel, pos=0.2] {$Y$} (G01);
      \draw (G1)    -- node[gatelabel, pos=0.2] {$X$} (G01);

      \draw (G1)    -- node[gatelabel, pos=0.2] {$Y$} (G1-);
      \draw (Gpen)  -- node[gatelabel, pos=0.2] {$X$} (G-2);

      \draw (Gpen)  -- node[gatelabel, pos=0.2] {$Y$} (Gpl);
      \draw (Glast) -- node[gatelabel, pos=0.2] {$X$} (Gpl);
    \end{tikzpicture}
  \end{center}
  where $(t_0,\dots,t_m)$ is a different tuple of integers depending on which morphism $\cM_r$ is being considered.
  Specifically:---
  \begin{itemize}
    \item for $r=1$ we take $t_0=1$, $t_1=a_2$, $t_2=a_2 a_3$, and so on, up to $t_s=\cdots=t_m=a_2 a_3 \dots a_s = \alpha$; that is,
      \[
        t_i = \prod_{j=2}^{\min(i+1,s)} a_j;
      \]
    \item for $2 \le r \le s$ we take $t_0=t_1=\cdots=t_{r-2}=a_r$ and $t_{r-1}=t_r=\cdots=t_m=1$.
  \end{itemize}
  For each value of $r$, each gate $Hi$ is then ``doing what it is supposed to do'', namely acting as a gate $\lag(a_{i+2}, B)$ (if $r=1$), $\lag(a_{i+2}, A)$ (if $r=i+2$) or $\const$ otherwise.
  It is therefore routine, in the spirit of previous proofs, to compose the morphisms $\cM_r^{(i)}$ with a suitable rescaling morphism $D_r \to D_r^{(i)}$ (with all linear maps given by $v \mapsto c v$, for some $c \in \FF_p$) and then combine them using Lemma~\ref{lem:patching} to get $\cM_r$.

  Finally, $\cM_0$ is summarized similarly by the following picture.
  \begin{center}
    \begin{tikzpicture}[
        defnode/.style={rectangle,draw,fill=lightgray,inner sep=2pt,outer sep=0pt,minimum size=10pt},
        gatelabel/.style={rectangle, inner sep=2pt, outer sep=0pt, fill=white,scale=0.6},
        subnode/.style={rounded rectangle, draw, inner sep=5pt, outer sep=0pt,minimum size=15pt},
        leaf/.style={rectangle,draw,fill=leafgreen,inner sep=2pt,outer sep=0pt,minimum size=10pt},
        scale=0.75
      ]

      \node[subnode,scale=0.6] (G0) at (0,0) {$H0 : \agate^k_s$};
      \node[subnode,scale=0.6] (G1) at (3.2,0) {$H1 : \agate^k_s$};
      \node[subnode,scale=0.6] (Gpen) at (7.8,0) {$H(m-2) : \agate^k_s$};
      \node[scale=0.8] (Gdots) at ($(G1.east)!0.5!(Gpen.west)$) {$\dots$};
      \node[subnode,scale=0.6] (Glast) at (12,0) {$H(m-1) : \agate^k_s$};

      \node[leaf,scale=0.5] (X) at ($(G0)+(-2,0)$) {$x$};
      \node[leaf,scale=0.5] (Y) at ($(Glast)+(2.2,0)$) {$y$};

      \node[defnode,scale=0.5] (G01) at ($(G0)!0.5!(G1)$) {$0$};
      \node[defnode,scale=0.5] (G1-) at ($(G1.east)!0.6!(Gdots)$) {$0$};

      \node[defnode,scale=0.5] (G-2) at ($(Gdots)!0.4!(Gpen.west)$) {$0$};
      \node[defnode,scale=0.5] (Gpl) at ($(Gpen)!0.5!(Glast)$) {$0$};

      \draw (G0)    -- node[gatelabel, pos=0.15] {$X$} (X);
      \draw (Glast) -- node[gatelabel, pos=0.15] {$Y$} (Y);

      \draw (G0)    -- node[gatelabel, pos=0.2] {$Y$} (G01);
      \draw (G1)    -- node[gatelabel, pos=0.2] {$X$} (G01);

      \draw (G1)    -- node[gatelabel, pos=0.2] {$Y$} (G1-);
      \draw (Gpen)  -- node[gatelabel, pos=0.2] {$X$} (G-2);

      \draw (Gpen)  -- node[gatelabel, pos=0.2] {$Y$} (Gpl);
      \draw (Glast) -- node[gatelabel, pos=0.2] {$X$} (Gpl);
    \end{tikzpicture}
  \end{center}
\end{proof}

\subsection{The big AggreGate}%
\label{sub:big-agg}

In a similar spirit to the previous section, we already have a gate $\aggregate$ that can encode multilinear tensor identities such as
\[
  x \otimes v_2 \otimes \dots \otimes v_s -
  y \otimes v_2 \otimes \dots \otimes v_s - 
  z \otimes v_2 \otimes \dots \otimes v_s +
  w \otimes v_2 \otimes \dots \otimes v_s =0
\]
whenever $x-y-z+w=0$, but want to go further and define a gate that encodes multilinear identities with more summands: i.e., if $\sum_{i=0}^{m-1} (-1)^i x_i = 0$ then
\[
  \sum_{i=0}^{m-1} (-1)^i x_i \otimes v_2 \otimes \dots \otimes v_s = 0.
\]
We do this using a ``big $\aggregate$''.

\begin{definition}%
  \label{def:bigagg}
  Let $m \ge 4$ be an even integer and let $s \ge 1$.
  The diagram $\bigagg_s^m$ is the one shown in Figure~\ref{fig:bigagg}.
  \begin{figure}[htbp]
    \begin{center}
    \begin{tikzpicture}[
        defnode/.style={rectangle,draw,fill=lightgray,inner sep=2pt,outer sep=0pt,minimum size=10pt},
        gatelabel/.style={rectangle, inner sep=2pt, outer sep=0pt, fill=white,scale=0.4},
        subnode/.style={rounded rectangle, draw, inner sep=5pt, outer sep=0pt,minimum size=15pt},
        leaf/.style={rectangle,draw,fill=leafgreen,inner sep=2pt,outer sep=0pt,minimum size=10pt},
        scale=0.75
      ]
      \node[subnode,scale=0.6] (G0) at (0  :3.5) {$H0 \colon \agg_s$};
      \node[subnode,scale=0.6] (G1) at (50 :2.5) {$H1 \colon \agg_s$};
      \node[subnode,scale=0.6] (G2) at (130:2.5) {$H2 \colon \agg_s$};
      \node[subnode,scale=0.6] (G3) at (180:3.5) {$H3 \colon \agg_s$};
      \node[scale=0.8]      (Gdots) at (230:2.5) {$\dots$};
      \node[subnode,scale=0.6] (Gp) at (310:2.5) {$H{\scriptstyle(m-1)} \colon \agg_s$};

      \node[leaf,scale=0.5] (X0) at ($1.5*(G0)$) {$X0$};
      \node[leaf,scale=0.5] (X1) at ($1.5*(G1)$) {$X1$};
      \node[leaf,scale=0.5] (X2) at ($1.5*(G2)$) {$X2$};
      \node[leaf,scale=0.5] (X3) at ($1.5*(G3)$) {$X3$};
      \node[leaf,scale=0.5] (Xp) at ($1.5*(Gp)$) {$X{\scriptstyle(m-1)}$};

      \node[defnode,scale=0.5] (G01) at ($0.4*(G0) + 0.4*(G1)$) {};
      \node[defnode,scale=0.5] (G12) at ($0.4*(G1) + 0.4*(G2)$) {};
      \node[defnode,scale=0.5] (G23) at ($0.4*(G2) + 0.4*(G3)$) {};
      \node[defnode,scale=0.5] (G34) at ($0.4*(G3) + 0.4*(Gdots)$) {};
      \node[defnode,scale=0.5] (G45) at ($0.4*(Gdots) + 0.4*(Gp)$) {};
      \node[defnode,scale=0.5] (G50) at ($0.4*(Gp) + 0.4*(G0)$) {};

      \draw (G0)    -- node[gatelabel, pos=0.15] {$00;\triangle$} (X0);
      \draw (G1)    -- node[gatelabel, pos=0.15] {$00;\triangle$} (X1);
      \draw (G2)    -- node[gatelabel, pos=0.15] {$00;\triangle$} (X2);
      \draw (G3)    -- node[gatelabel, pos=0.15] {$00;\triangle$} (X3);
      \draw (Gp)    -- node[gatelabel, pos=0.15] {$00;\triangle$} (Xp);

      \draw (G0)    -- node[gatelabel, pos=0.15] {$11;\triangle$} (G01);
      \draw (G1)    -- node[gatelabel, pos=0.11] {$11;\triangle$} (G01);
      \draw (G1)    -- node[gatelabel, pos=0.15] {$10;\triangle$} (G12);
      \draw (G2)    -- node[gatelabel, pos=0.15] {$10;\triangle$} (G12);
      \draw (G2)    -- node[gatelabel, pos=0.11] {$11;\triangle$} (G23);
      \draw (G3)    -- node[gatelabel, pos=0.15] {$11;\triangle$} (G23);
      \draw (G3)    -- node[gatelabel, pos=0.15] {$10;\triangle$} (G34);
      \draw (Gp)    -- node[gatelabel, pos=0.15] {$11;\triangle$} (G45);
      \draw (Gp)    -- node[gatelabel, pos=0.11] {$10;\triangle$} (G50);
      \draw (G0)    -- node[gatelabel, pos=0.15] {$10;\triangle$} (G50);
    \end{tikzpicture}
  \end{center}
  \caption{The diagram $\bigagg_s^m$.}%
  \label{fig:bigagg}
  \end{figure}
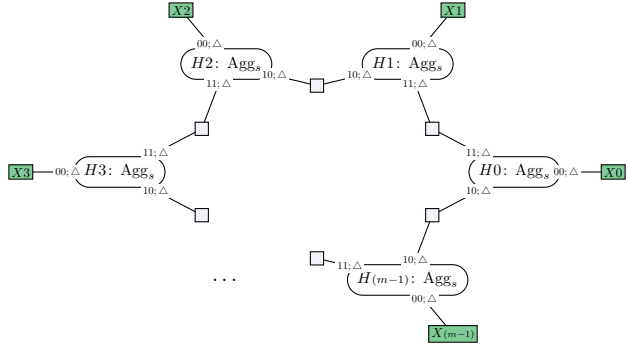
  
  In other words, we have diagrams $H\ell \colon \agg_s$ for $0 \le \ell < m$, with overlapping vertices $H\ell;11;\triangle = H(\ell+1);11;\triangle$ and $H(\ell+1);10;\triangle = H(\ell+2);10;\triangle$ for $0 \le \ell < m$ even (with $\ell+1$, $\ell+2$ taken modulo $m$).
  We also label $X\ell = H\ell;00;\triangle$.
\end{definition}

\begin{definition}%
  \label{def:bigagg-gate}
   The gate $\bigagg_s^{m}$ consists of the diagram $\bigagg_s^m$ together with $\cR = [s]$ and $\cP = \{Xi \colon 0 \le i \le m-1\}$.
\end{definition}

We identify the sub-gates $Hi$ for $0 \le i \le m-1$ with $\agg_s^{\overline{01}}$.

\begin{remark}%
  \label{rem:build-bigagg}
  We do not actually need a lemma telling us we can build a $\bigagg_s^m$ starting from an $\aggregate_s$ when $m=2^M \ge 4$ is a power of $2$, but for future reference it will be useful to describe the process.
  We would start with one $\agg_s$ and apply $\CS(\{11;\triangle\})$ to obtain
  
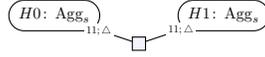
\begin{center}
  \begin{tikzpicture}[
      defnode/.style={rectangle,draw,fill=lightgray,inner sep=2pt,outer sep=0pt,minimum size=10pt},
      gatelabel/.style={rectangle, inner sep=2pt, outer sep=0pt, fill=white,scale=0.4},
      subnode/.style={rounded rectangle, draw, inner sep=5pt, outer sep=0pt,minimum size=15pt},
      leaf/.style={rectangle,draw,fill=leafgreen,inner sep=2pt,outer sep=0pt,minimum size=10pt},
      scale=0.75
    ]
    \node[subnode,scale=0.6] (G0) at (0,0) {$H0 \colon \agg_s$};
    \node[subnode,scale=0.6] (G1) at (3,0) {$H1 \colon \agg_s$};
  
    \node[defnode,scale=0.5] (G01) at (1.5,-0.5) {};
  
    \draw (G0)    -- node[gatelabel, pos=0.15] {$11;\triangle$} (G01);
    \draw (G1)    -- node[gatelabel, pos=0.11] {$11;\triangle$} (G01);
  \end{tikzpicture}
  \end{center}
  We for $j=1,2,\dots,M-2$ we then repeatedly apply $\CS(\{H{\scriptstyle(2^j-1)};10;\triangle\})$ and relabel, to obtain a chain:
    
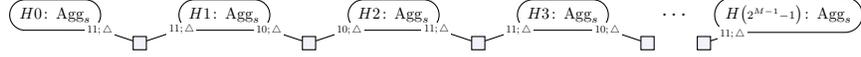
\begin{center}
    \begin{tikzpicture}[
        defnode/.style={rectangle,draw,fill=lightgray,inner sep=2pt,outer sep=0pt,minimum size=10pt},
        gatelabel/.style={rectangle, inner sep=2pt, outer sep=0pt, fill=white,scale=0.4},
        subnode/.style={rounded rectangle, draw, inner sep=5pt, outer sep=0pt,minimum size=15pt},
        leaf/.style={rectangle,draw,fill=leafgreen,inner sep=2pt,outer sep=0pt,minimum size=10pt},
        scale=0.75
      ]
      \node[subnode,scale=0.6] (G0) at (0,0) {$H0 \colon \agg_s$};
      \node[subnode,scale=0.6] (G1) at (3,0) {$H1 \colon \agg_s$};
      \node[subnode,scale=0.6] (G2) at (6,0) {$H2 \colon \agg_s$};
      \node[subnode,scale=0.6] (G3) at (9,0) {$H3 \colon \agg_s$};
      \node[scale=0.8]      (Gdots) at (11,0) {$\dots$};
      \node[subnode,scale=0.6] (Gp) at (13,0) {$H{\scriptstyle\bigl(2^{M-1}-1\bigr)} \colon \agg_s$};

      \node[defnode,scale=0.5] (G01) at ($0.5*(G0) + 0.5*(G1)    + (0,-0.5)$) {};
      \node[defnode,scale=0.5] (G12) at ($0.5*(G1) + 0.5*(G2)    + (0,-0.5)$) {};
      \node[defnode,scale=0.5] (G23) at ($0.5*(G2) + 0.5*(G3)    + (0,-0.5)$) {};
      \node[defnode,scale=0.5] (G34) at (10.5,-0.5) {};
      \node[defnode,scale=0.5] (G45) at (11.5,-0.5) {};

      \draw (G0)    -- node[gatelabel, pos=0.15] {$11;\triangle$} (G01);
      \draw (G1)    -- node[gatelabel, pos=0.11] {$11;\triangle$} (G01);
      \draw (G1)    -- node[gatelabel, pos=0.15] {$10;\triangle$} (G12);
      \draw (G2)    -- node[gatelabel, pos=0.15] {$10;\triangle$} (G12);
      \draw (G2)    -- node[gatelabel, pos=0.11] {$11;\triangle$} (G23);
      \draw (G3)    -- node[gatelabel, pos=0.15] {$11;\triangle$} (G23);
      \draw (G3)    -- node[gatelabel, pos=0.15] {$10;\triangle$} (G34);
      \draw (Gp)    -- node[gatelabel, pos=0.15] {$11;\triangle$} (G45);
    \end{tikzpicture}
  \end{center}
  Finally we apply $\CS\bigl(\bigl\{(H0;10;\triangle), (H({\scriptstyle2^{M-1}-1});10;\triangle)\bigr\}\bigr)$ to close the loop, and relabel a final time.
\end{remark}

There is again one interesting assignment, analogous to the $\sumconst$ assignment for the gate $\aggregate$.

\begin{definition}%
  \label{def:bigsum}
  For an even integer $m \ge 2$, the linear datum $\bigsum(m)$ has index set $I=\{0,1,\dots,m-1\}$, $W_i =\FF_p$ for each $i \in I$, and
  \[
    V = \Bigl\{ (x_i)_{i \in I} \colon {\textstyle \sum_{i \in I} } (-1)^i x_i = 0  \Bigr\} \subseteq \FF_p^I
  \]
  and linear maps $\phi_i \colon V \to W_i$ the natural projections $(x_j)_{j \in I} \mapsto x_i$.
\end{definition}

\begin{lemma}%
  \label{lem:bigagg-assign}
  There is an assignment $\sumconst_1$ of $\bigagg_s^{m}$ with diagrams $D_0 = \trivial(\pins)$, $D_1 = \bigsum(m)$ and $D_r=\const(\pins)$ for $2 \le r \le s$.
  As necessary we use the obvious bijection $\{0,1,\dots,m-1\} \leftrightarrow \pins^{\bigagg_s^m}$.
\end{lemma}
\begin{proof}%
  We assign each sub-gate $Hi$ with $\sumconst_1$ (see Lemma~\ref{lem:sum-const-assignment}).
  We set $S_r = \bigcup_{i=0}^{m-1} S_r^{(i)}$ to be the union of the sub-partitions in the usual way.
  Also let
  \[
    \cM_1^{(i)} \colon \opsum(\pins^{Hi}) \to \agg_s
  \]
  and
  \[
    \cM_r^{(i)} \colon \const(\pins^{Hi}) \to \agg_s
  \]
  for $2 \le r \le s$ be the morphisms obtained from Lemma~\ref{lem:sum-const-assignment}.

  For $2 \le r \le s$, the morphisms $\cM_r^{(i)}$ already have
  \[
    \theta^{\cM_r^{(i)}}_{11;\triangle} = \theta^{\cM_r^{(i)}}_{10;\triangle}  = \id_{\FF_p}
  \]
  for every $i \in \{0,\dots,m-1\}$, and so we can directly apply Lemma~\ref{lem:patching} (and Remark~\ref{rem:leaf-non-leaf-morph}) to obtain %the required morphism $\cM_r \colon \const(S) \to \bigagg_s^{m,S}$.
  the required morphism $\cM_r \colon \const(\cP) \to \bigagg_s^{m}$.

  For $\cM_1$, we now specify the overlap maps $\theta^{\cM_1}_{Hi;10;\triangle} \colon V^{\bigsum(m)} \to \FF_p$ and $\theta^{\cM_1}_{Hi;11;\triangle} \colon V^{\bigsum(m)} \to \FF_p$.
  We claim there is a unique choice such that:---
  \begin{enumerate}[label=(\alph*)]
    \item $\theta^{\cM_1}_{H0;10;\triangle} = 0$;
    \item for $0 \le i \le m-1$, $\theta^{\cM_1}_{Hi;10;\triangle} = \theta^{\cM_1}_{Hi;11;\triangle} + \phi^{\bigsum(m)}_i$;
    \item we have $\theta^{\cM_1}_{Hi;11;\triangle} = \theta^{\cM_1}_{H(i+1);11;\triangle}$ and
      $\theta^{\cM_1}_{H(i+1);10;\triangle} = \theta^{\cM_1}_{H(i+2);10;\triangle}$ for all
      even integers $0 \le i < m$, with additions $i+1$, $i+2$ taken modulo $m$.
  \end{enumerate}
  Indeed, (c) is necessary as $H{\scriptstyle (i+1)};10;\triangle = H{\scriptstyle (i+2)};10;\triangle$ and $Hi;11;\triangle = H{\scriptstyle (i+1)};11;\triangle$ are the same vertices in $\bigagg_s^m$.
  We can give an explicit choice of maps by
  \begin{align*}
    \theta^{\cM_1}_{Hi;10;\triangle}\bigl((x_j)_{j \in I}\bigr) &= -\sum_{0 \le j < i} (-1)^j x_j &
    \theta^{\cM_1}_{Hi;11;\triangle}\bigl((x_j)_{j \in I}\bigr) &= -\sum_{0 \le j \le i} (-1)^j x_j
  \end{align*}
  for $i$ even and
  \begin{align*}
    \theta^{\cM_1}_{Hi;11;\triangle}\bigl((x_j)_{j \in I}\bigr) &= -\sum_{0 \le j < i } (-1)^j x_j &
    \theta^{\cM_1}_{Hi;10;\triangle}\bigl((x_j)_{j \in I}\bigr) &= -\sum_{0 \le j \le i} (-1)^j x_j
  \end{align*}
  for $i$ odd.
  One can verify that (a)--(c) holds for this choice.
  Crucially, the required fact $\theta^{\cM_1}_{H(m-1);10;\triangle} = \theta^{\cM_1}_{H0;10;\triangle} = 0$
  holds because $\sum_{0 \le i \le m-1} (-1)^i x_i = 0$ on $V^{\bigsum(m)}$.

  In particular, (b) dictates that each gate $Hi$ is ``doing what it is supposed to'', in this case acting as a $\opsum$ gate.
  Formally, this means we can find specialization morphisms $\Theta_i \colon \bigsum(m) \to \opsum(\{00,10,11\})(\{10,11\} \leadsto \nonleafs)$ such that the composites $\Theta_i \circ \cM_1$ agree at vertices $Hi;10;\triangle$ and $Hi;11;\triangle$ with the maps given above, and at vertices $Xi$ with the maps $\phi_i^{\bigsum(m)}$.
  As usual Lemma~\ref{lem:patching} then gives the required $\cM_1$.

  Finally, for $\cM_0$ we can simply specialize the morphisms $\cM_0^{(i)}$ to be zero at vertices $Hi;10;\triangle$ and $Hi;11;\triangle$ (indeed, starting with a morphism $\trivial(\pins) \to \agg_s$ we can make every pin be whatever we want).
  Then these specialized morphisms are compatible and Lemma~\ref{lem:patching} again gives $\cM_0$.
\end{proof}

\subsection{StarGates and Lemma~\ref{lem:key-lem}}%
\label{sub:stargates}

We finally describe the gate that will actually prove Lemma~\ref{lem:key-lem}, and thereby Theorem~\ref{thm:main}.
We refer to the set-up from Section~\ref{sub:main-outline}, in particular the linear datum $\Phi$.

\begin{definition}%
  \label{def:stargate}
  We fix parameters $s \ge 2$ and\footnote{%
    The values $k,m$ pertain to the parameters of gates $\supera_s^{k,m}$; see Definition~\ref{def:super-agate}.
  The value $n$ denotes the number of ``spokes'' of the StarGate, which correspond to rows of the matrix in~\eqref{eq:matrix-eq}.}
   $k,m,n \ge 1$.
  Let $\Phi=(\phi_i)_{i \in I}$ be a system of linear forms, $i_0 \in I$ and $I' = I \setminus \{i_0\}$.

  The diagram $\stargate(s,k,m,n,\Phi,i_0)$ in the case $n=8$, $I'=\{i_1,\dots,i_5\}$ is shown in Figure~\ref{fig:stargate}.
  In general it has the following component sub-gates / sub-diagrams.
  \begin{itemize}
    \item For each \emph{spoke} $\ell \in \{0,1,\dots,n-1\}$ there is a \emph{$\Phi$-gate} $F \ell \colon \Phi$.
    \item For each \emph{form} $j \in I'$, there is a gate $\bag j \colon \bigagg_s^{n}$ (at the center).
    \item For each form $j \in I'$ and each spoke $\ell \in \Sigma$ there is a super AGate $A(j,\ell) \colon \supera_s^{k,m}$, and a bridge $\opbr(j,\ell) \colon \bridge_s$.
  \end{itemize}
  The connections between these gates may be summarized as follows.
  For each spoke $\ell$ and each form $j \in I'$, the vertices $F\ell;j = \opbr(j,\ell);X$, $\opbr(j,\ell);Y = A(j,\ell);X$, and $A(j,\ell);Y = \bag j;\ell$ are identified (and become non-leaves).

  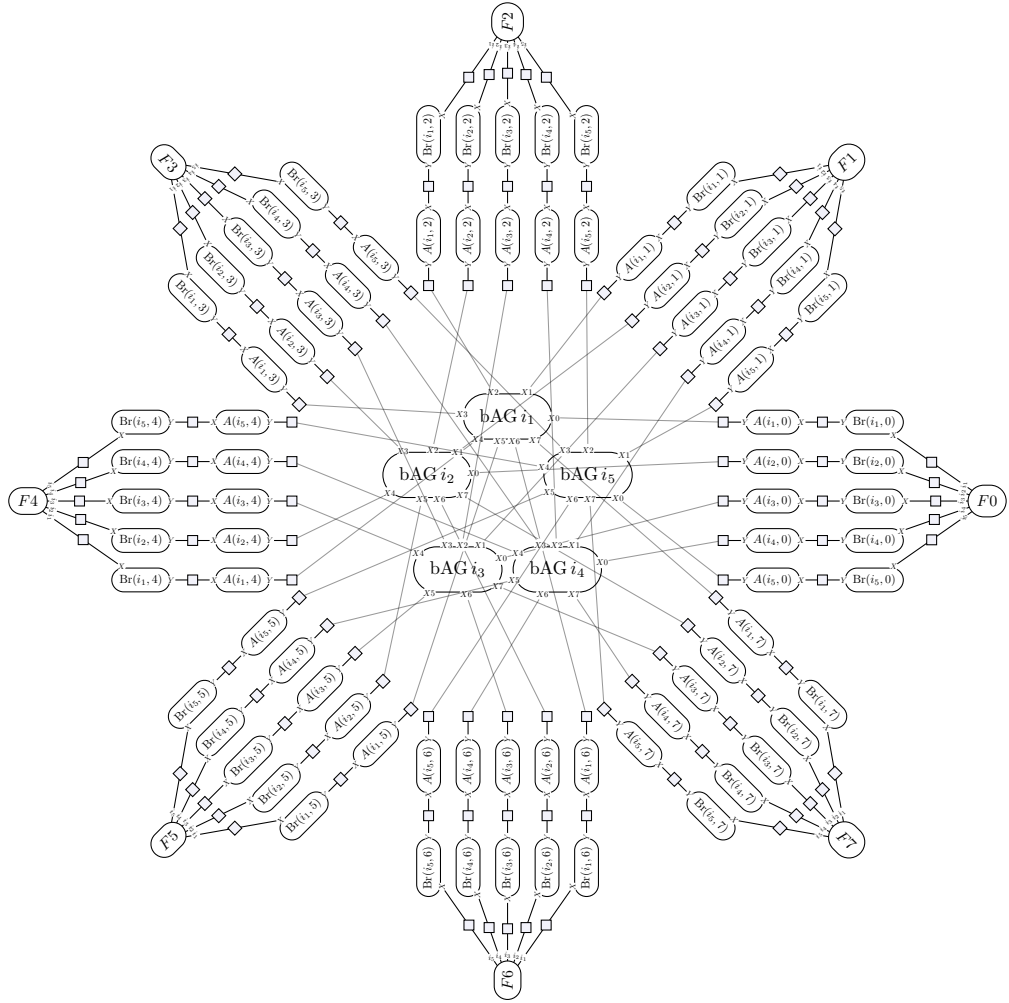
\begin{figure}[htbp]
    \begin{center}
    \begin{tikzpicture}[
        defnode/.style={rectangle,draw,fill=lightgray,inner sep=2pt,outer sep=0pt,minimum size=10pt},
        gatelabel/.style={circle, inner sep=0pt, outer sep=0pt, fill=white,scale=0.4},
        subnode/.style={rounded rectangle, draw, inner sep=5pt, outer sep=0pt,minimum size=15pt},
        leaf/.style={rectangle,draw,fill=leafgreen,inner sep=2pt,outer sep=0pt,minimum size=10pt},
        scale=0.75
      ]
      \foreach \i in {7,0,1,2} {
        \coordinate (origin) at (45*\i:6.5);
        \begin{scope}[rotate={45*\i}, shift={(origin)}]
          \node[transform shape, subnode,scale=0.8, minimum height=20pt] (F\i) at (2,0) {$F\i$};
        
          \node[transform shape, subnode,scale=0.6] (B\i1) at (0, 1.4) {$\opbr(i_1,\i)$};
          \node[transform shape, subnode,scale=0.6] (B\i2) at (0, 0.7) {$\opbr(i_2,\i)$};
          \node[transform shape, subnode,scale=0.6] (B\i3) at (0, 0.0) {$\opbr(i_3,\i)$};
          \node[transform shape, subnode,scale=0.6] (B\i4) at (0,-0.7) {$\opbr(i_4,\i)$};
          \node[transform shape, subnode,scale=0.6] (B\i5) at (0,-1.4) {$\opbr(i_5,\i)$};

          \node[transform shape, subnode,scale=0.6] (A\i1) at (-1.8, 1.4) {$A(i_1,\i)$};
          \node[transform shape, subnode,scale=0.6] (A\i2) at (-1.8, 0.7) {$A(i_2,\i)$};
          \node[transform shape, subnode,scale=0.6] (A\i3) at (-1.8, 0.0) {$A(i_3,\i)$};
          \node[transform shape, subnode,scale=0.6] (A\i4) at (-1.8,-0.7) {$A(i_4,\i)$};
          \node[transform shape, subnode,scale=0.6] (A\i5) at (-1.8,-1.4) {$A(i_5,\i)$};

          \foreach \j in {1,...,5} {
            \draw (F\i)    -- node[gatelabel, pos=0.10, rotate=90, transform shape] {$i_{\j}$} node[midway, defnode, transform shape,scale=0.5] {} node[gatelabel,pos=0.97, transform shape] {$X$} (B\i\j);
            \draw (B\i\j)    -- node[gatelabel, pos=0.03, transform shape] {$Y$} node[midway, defnode, transform shape,scale=0.5] {} node[gatelabel,pos=0.97, transform shape] {$X$} (A\i\j);
            \node[transform shape, defnode, scale=0.5] (Ap\i\j)  at ($(A\i\j.west) + (-0.4, 0) $) {};
            \draw (A\i\j) -- node[gatelabel, pos=0.03, transform shape] {$Y$} (Ap\i\j);
          }
        \end{scope}
      }
      \foreach \i in {3,4,5,6} {
        \coordinate (origin) at (45*\i:6.5);
        \begin{scope}[rotate={45*\i}, shift={(origin)}]
          \node[transform shape, rotate=180, subnode,scale=0.8] (F\i) at (2,0) {$F\i$};
        
          \node[transform shape, rotate=180, subnode,scale=0.6] (B\i1) at (0, 1.4) {$\opbr(i_1,\i)$};
          \node[transform shape, rotate=180, subnode,scale=0.6] (B\i2) at (0, 0.7) {$\opbr(i_2,\i)$};
          \node[transform shape, rotate=180, subnode,scale=0.6] (B\i3) at (0, 0.0) {$\opbr(i_3,\i)$};
          \node[transform shape, rotate=180, subnode,scale=0.6] (B\i4) at (0,-0.7) {$\opbr(i_4,\i)$};
          \node[transform shape, rotate=180, subnode,scale=0.6] (B\i5) at (0,-1.4) {$\opbr(i_5,\i)$};

          \node[transform shape, rotate=180, subnode,scale=0.6] (A\i1) at (-1.8, 1.4) {$A(i_1,\i)$};
          \node[transform shape, rotate=180, subnode,scale=0.6] (A\i2) at (-1.8, 0.7) {$A(i_2,\i)$};
          \node[transform shape, rotate=180, subnode,scale=0.6] (A\i3) at (-1.8, 0.0) {$A(i_3,\i)$};
          \node[transform shape, rotate=180, subnode,scale=0.6] (A\i4) at (-1.8,-0.7) {$A(i_4,\i)$};
          \node[transform shape, rotate=180, subnode,scale=0.6] (A\i5) at (-1.8,-1.4) {$A(i_5,\i)$};

          \foreach \j in {1,...,5} {
            \draw (F\i)    -- node[gatelabel, pos=0.10, rotate=90, transform shape] {$i_{\j}$} node[midway, defnode, transform shape,scale=0.5] {} node[gatelabel,pos=0.97, rotate=180, transform shape] {$X$} (B\i\j);
            \draw (B\i\j)    -- node[gatelabel, pos=0.03, rotate=180, transform shape] {$Y$} node[midway, defnode, transform shape,scale=0.5] {} node[gatelabel,pos=0.97, rotate=180, transform shape] {$X$} (A\i\j);
            \node[transform shape, defnode, scale=0.5] (Ap\i\j)  at ($(A\i\j.east) + (-0.4, 0)$) {};
            \draw (A\i\j) -- node[gatelabel, pos=0.03, rotate=180, transform shape] {$Y$} (Ap\i\j);
          }
        \end{scope}
      }
      \foreach \i in {1,...,5} {
        \node[subnode, scale=0.7, inner sep = 8pt, outer sep = 1pt] (bag\i) at (18+72*\i:1.5) {$\bag i_\i$};
        \foreach  \j in {0,...,7} {
          \draw[opacity=0.4] (bag\i) -- node[gatelabel, pos=0.00, opacity=1, scale=0.8] {$X{\j}$} (Ap\j\i);
          }
        }
    \end{tikzpicture}
  \end{center}
  \caption{The diagram $\stargate$, for $n=8$ and $I=\{i_0,\dots,i_5\}$.}%
  \label{fig:stargate}
  \end{figure}
\end{definition}
The term ``$\stargate$'' is motivated by Figure~\ref{fig:stargate}.

\begin{remark}%
  \label{rem:phi-gate}
  We have already assigned each sub-diagram a gate structure in Definition~\ref{def:stargate}, with the exception of the sub-diagrams $F\ell$ for $0 \le \ell \le n-1$, where there as yet is no formal notion a ``$\Phi$-gate''.
  For consistency of notation we give each of the copies of $\Phi$ a gate structure, with $\cR^{F\ell} = [s]$
  and
  \[
    \pins^{F\ell} = \begin{dcases}
      \begin{rcases}
        I &\colon \ell=0 \\
        I' &\colon \ell \ne 0.
      \end{rcases}
    \end{dcases}
  \]
\end{remark}

It is true, but not obvious, that we can build this gate starting from the configuration in Lemma~\ref{lem:key-lem}.

\begin{lemma}%
  \label{lem:build-stargate}
  Let $\Phi$, $I$, $i_0$, $I'$, $(t_j)_{j \in I'}$ and be as in the statement of Lemma~\ref{lem:key-lem}, and let $D = D((t_j)_{j \in I'})$ be the diagram from~\eqref{eq:boost-phi}.
  Suppose $s$ is an integer with $s \ge \max_{j \in I'} t_j$.

  Suppose also that $k,m \ge 1$ and $n \ge 4$ are powers of $2$.
  Then $D \entails^M_{F0;i_0 \mapsto i_0} \stargate(t, k,m,n,\Phi,i_0)$, where
  \[
    M = \lvert I' \rvert \bigl(9 + \log_2 k + \log_2 m\bigr) + \log_2 n.
  \]
\end{lemma}
The proof is covered in Section~\ref{sub:stargate-build}.

We can put a gate structure on $\stargate$ itself, albeit a rather artificial one, compatible with Remark~\ref{rem:phi-gate} in the sense that $\toggles$ is the union of the toggles of all sub-gates.
\begin{definition}
  The gate $\stargate(s, k, m, n, \Phi, i_0)$ consists of the corresponding diagram $\stargate$ with $\cR = [s]$ and $\pins = \{ (F0; i_0)\}$.
\end{definition}

The desired ``assignment'' of $\stargate$ is the following.

\begin{lemma}%
  \label{lem:assign-stargate}
  Suppose $\Phi$ is a system of linear forms with index set $I$, $i_0 \in I$, $I' = I \setminus \{i_0\}$, $s \ge 2$ is an integer and $s(\Phi,i_0) \le s-1$.

  Moreover suppose that $V^\Phi = \FF_p^d$ and $\phi_i(x_1,\dots,x_d) = a_{i,1} x_1 + \cdots + a_{i,d} x_d$ for integers $a_{i,j}$ with $\lvert a_{i,j} \rvert \le L$.
  Also suppose that $k,m \ge 1$ and $n \ge 4$ are integers such that $2L^2 < 2^{k+1}$, $s-1 \le m$, $n$ is even and
  \[
    n \ge \binom{d+s-1}{s}.
  \]
  Write $G = \stargate(s,k,m,n,\Phi,i_0)$.

  Then there exists an assignment of $G$ with $\cM_r = \trivial(\{i_0\})$ for each $r$, $0 \le r \le s$.
  Equivalently, as a diagram, $G$ has Cauchy--Schwarz complexity at most $s-1$ at $F0;i_0$.
\end{lemma}
The last two sentences are indeed equivalent by Proposition~\ref{prop:diag-cs-complexity}.
Note $\cM_0$ is rather irrelevant in this case: given $\cM_1 \colon \trivial(\{i_0\}) \to \stargate[\nonleafs \cup \pins \cup S_1]$ we can define $\cM_0 \colon \trivial(\{i_0\}) \to \stargate[\nonleafs \cup \pins]$ immediately by restriction.

We cover the proof in Section~\ref{sub:stargate-assign}.

\begin{proof}[Proof of Lemma~\ref{lem:key-lem} given Lemma~\ref{lem:build-stargate} and Lemma~\ref{lem:assign-stargate}]
  Let $k,m,n$ be powers of two which are as small as possible while obeying the conditions in Lemma~\ref{lem:assign-stargate}.
  Specifically we can take $k = O(\log (10 L))$, $m = O(s)$ and $n = O(d^s)$.

  By Lemma~\ref{lem:build-stargate} we have
  \[
    D\bigl((t_j)_{j \in I'}\bigr) \entails^M_{F0;i_0 \mapsto i_0} \stargate(s, k, m, n, \Phi, i_0)
  \]
  where
  \[
    M = \lvert I' \rvert \bigl( 9 + \log_2 k + \log_2 m \bigr) + \log_2 n = O \bigl( \lvert I \rvert (\log \log (10 L) + \log \lvert I \rvert ) \bigr)
  \]
  since $s,d \le \lvert I \rvert$ by our standing hypotheses.
  By Lemma~\ref{lem:assign-stargate} and Proposition~\ref{prop:diag-cs-complexity} we deduce
  \[
    \stargate(s, k, m, n, \Phi, i_0) \entails^0_{\triangle \mapsto F0;i_0} \gc_{s-1}.
  \]
  Combining these gives the result.
\end{proof}

\subsection{Building a StarGate}%
\label{sub:stargate-build}

In this section we prove Lemma~\ref{lem:build-stargate}.
We start with some useful observations.

\begin{claim}%
  \label{claim:claim0}
  We have:---
  \begin{enumerate}[label=(\roman*)]
    \item $\bridge_s \entails^0_{\triangle \mapsto Y} \gc_s$;
    \item $\bridge_s$ is surjective at $X$ and $Y$;
    \item $\supera_s^{k,m}$ is surjective at $X$ and $Y$;
    \item $\aggregate_s$ is surjective at $\gamma;\triangle$ for any $\gamma \in \{0,1\}^2$.
  \end{enumerate}
\end{claim}
\begin{proof}[Proof of claim]%
  Part (i) follows from Proposition~\ref{prop:discard-dual} and the definition $\bridge_s = \gc_s +_{\{s+1\}} \gc_s$.
  It is clear by inspection that $\gc_s$ is surjective at $s+1$.

  Part (ii) is true by Corollary~\ref{cor:join-surj}.
  For (iii), we apply Lemma~\ref{lem:surj-fact-1} to the morphism $\cM_0 \colon \trivial(\{X,Y\}) \to \supera_s^{k,m}$ from Lemma~\ref{lem:super-assignment}.
  Similarly for (iv) we apply Lemma~\ref{lem:surj-fact-1} to the morphism $\cM_0$ from Lemma~\ref{lem:sum-const-assignment} for $P=\{0,1\}^2$ (say), noting by inspection that $\opsum(\{0,1\}^2)$ is surjective.
\end{proof}

The following claims break down the process of building a StarGate.
\begin{claim}%
  \label{claim:claim1}
  For $s \ge 1$, let $D_1 = \bridge_s +_{Y \leftrightarrow \triangle} \gc_s$.
  Then $\gc_s \entails^2_{(L;X) \mapsto \triangle} D_1$.
\end{claim}
\begin{proof}[Proof of claim]%
  By $\CS(\{s+1\})$ we have $\gc_s \entails^1_{X \mapsto \triangle} \bridge_s$, as observed in Section~\ref{sub:bridge}.
  Similarly $\bridge_s \entails^1_{L;X \mapsto X} \bigl( \bridge_s +_{Y \leftrightarrow Y} \bridge_s \bigr)$ by $\CS(\{Y\})$.

  Applying Proposition~\ref{prop:awesome-stashing} to Claim~\ref{claim:claim0}(i) we deduce that
  \[
    \left( \bridge_s +_{Y \leftrightarrow Y}^{,R} \bridge_s \right) \entails^0_{R;X \mapsto R;X}
    \left( \gc_s +_{\triangle \leftrightarrow Y}^{,R} \bridge_s \right).
  \]
  We need that $\bridge_s$ is surjective at $Y$, which is Claim~\ref{claim:claim0}(ii).

  The right-hand side is the same as $D_1$ up to relabelling, so combining the $\entails$ statements gives the claim.
  A schematic is shown in Figure~\ref{fig:claim1}.
\end{proof}

\begin{figure}[htbp]
\begin{NoHyper}
\begin{center}
  \begin{tikzpicture}[
      defnode/.style={rectangle,fill=lightgray,draw,inner sep=2pt,outer sep=0pt,minimum size=10pt},
      gatelabel/.style={rectangle, inner sep=2pt, outer sep=0pt, fill=white,scale=0.6},
      subnode/.style={rounded rectangle, draw, inner sep=5pt, outer sep=0pt,minimum size=15pt},
      leaf/.style={defnode,fill=leafgreen},
      scale=0.8
    ]

    \begin{scope}[shift={(0,0)}]
      \node[subnode, scale=0.8] (V) at (0,0) {$\gc_s$};
      \node[leaf,scale=0.5] (Vp) at (-1,0) {};
      \draw (V.west) -- node[pos=0.0,gatelabel] {$\triangle$} (Vp);
    \end{scope}
    \node at (1,0) {$\leadsto$};
    \node[scale=0.5] at (1,0.5) {$\CS(\{s+1\})$};
    \begin{scope}[shift={(2.7,0)}]
      \node[subnode, scale=0.8] (V1) at (0,0) {$\gc_s$};
      \node[leaf,scale=0.5] (Vp) at (-1,0) {};
      \node[subnode, scale=0.8] (V2) at (3,0) {$\gc_s$};
      \draw (V1.west) -- node[pos=0.0,gatelabel] {$\triangle$} (Vp);
      \draw (V1.east) -- node[pos=0.13,gatelabel] {$s+1$} node[midway,defnode,scale=0.6] {} node[pos=0.87,gatelabel] {$s+1$} (V2.west);
    \end{scope}
    \node at (1,-2) {$\leadsto$};
    \node[scale=0.5] at (1,-1.5) {$\CS(\{Y\})$};
    \begin{scope}[shift={(2.7,-2)}]
      \node[subnode, scale=0.8] (V1) at (0,0) {$\gc_s$};
      \node[leaf,scale=0.5] (Vp) at (-1,0) {};
      \node[subnode, scale=0.8] (V2) at (3,0) {$\gc_s$};
      \node[subnode, scale=0.8] (V3) at (6,0) {$\gc_s$};
      \node[subnode, scale=0.8] (V4) at (9,0) {$\gc_s$};
      \draw (V1.west) -- node[pos=0.0,gatelabel] {$\triangle$} (Vp);
      \draw (V1.east) -- node[pos=0.13,gatelabel] {$s+1$} node[midway,defnode,scale=0.6] {} node[pos=0.87,gatelabel] {$s+1$} (V2.west);
      \draw (V2.east) -- node[pos=0.03,gatelabel] {$\triangle$} node[midway,defnode,scale=0.6] {} node[pos=0.97,gatelabel] {$\triangle$} (V3.west);
      \draw (V3.east) -- node[pos=0.13,gatelabel] {$s+1$} node[midway,defnode,scale=0.6] {} node[pos=0.87,gatelabel] {$s+1$} (V4.west);
    \end{scope}
    \node at (1,-4) {$\leadsto$};
    \node[scale=0.5,align=center] at (1,-3.4) {Prop~\ref{prop:discard-dual}+Prop~\ref{prop:awesome-stashing}\\+Claim~\ref{claim:claim0}(i)};
    \begin{scope}[shift={(2.7,-4)}]
      \node[subnode, scale=0.8] (V1) at (0,0) {$\gc_s$};
      \node[leaf,scale=0.5] (Vp) at (-1,0) {};
      \node[subnode, scale=0.8] (V2) at (3,0) {$\gc_s$};
      \node[subnode, scale=0.8] (V3) at (6,0) {$\gc_s$};
      \draw (V1.west) -- node[pos=0.0,gatelabel] {$\triangle$} (Vp);
      \draw (V1.east) -- node[pos=0.13,gatelabel] {$s+1$} node[midway,defnode,scale=0.6] {} node[pos=0.87,gatelabel] {$s+1$} (V2.west);
      \draw (V2.east) -- node[pos=0.03,gatelabel] {$\triangle$} node[midway,defnode,scale=0.6] {} node[pos=0.97,gatelabel] {$\triangle$} (V3.west);
    \end{scope}
  \end{tikzpicture}
\end{center}
\end{NoHyper}
\caption{A schematic of the proof of Claim~\ref{claim:claim1}.}%
\label{fig:claim1}
\end{figure}

\begin{claim}%
  \label{claim:claim2}
  Define
  \[
    D_2 = \supera_s^{k,m} +_{X \leftrightarrow Y}^{,A} \bridge_s +_{Y \leftrightarrow 00;\triangle}^{,B} \aggregate_s .
  \]
  Then we have $\gc_s \entails_{A;X \mapsto \triangle}^{9+\log_2 k + \log_2 m} D_2$.
  Also, $D_2$ is surjective at $A;X$.
\end{claim}
\begin{proof}[Proof of claim]%
  Set $M = 5 + \log_2 k + \log_2 m$.
  Combining Lemma~\ref{lem:gc-aggre}, Lemma~\ref{lem:build-agate}, Lemma~\ref{lem:build-large-agate} and Lemma~\ref{lem:build-super-agate}, we have $\gc_s \entails^{M}_{X,Y \mapsto \triangle} \supera_s^{k,m}$.
  Now by Proposition~\ref{prop:awesome-stashing}, we deduce
  \[
    \left( \gc_s +^{,R}_{\triangle \leftrightarrow Y} \bridge_s \right) \entails^{M}_{A;X \mapsto R;X} \left( \supera_s^{k,m} +^{,A}_{X \leftrightarrow Y} \bridge_s +^{,B}_{Y \leftrightarrow Y} \bridge_s \right).
  \]
  Again we use Claim~\ref{claim:claim0}(ii).
  We temporarily write $C_0$ for the diagram on the right-hand side.

  Also write $C_1 = \supera_s^{k,m} +_{X \leftrightarrow Y}^{,A} \bridge_s$.
  By Claim~\ref{claim:claim0}(i) and Lemma~\ref{lem:gc-aggre} we have
  \[
    \bridge_s \entails^0_{\triangle \mapsto Y} \gc_s \entails^{2}_{00;\triangle \mapsto \triangle} \aggregate_s
  \]
  and so by Proposition~\ref{prop:awesome-stashing},
  \[
    \left( \bridge_s +_{Y \leftrightarrow Y}^{,R} \, C_1 \right) \entails^{2}_{R;A;X \mapsto R;A;X} \left( \aggregate_s +_{00;\triangle \leftrightarrow Y}^{,R} \, C_1 \right).
  \]
  We need to check that $C_1$ is surjective at $Y$, which holds by Claim~\ref{claim:claim0}(iii), Claim~\ref{claim:claim0}(ii) and Corollary~\ref{cor:join-surj}.

  Up to relabelling, the right-hand side is $D_2$ and the left-hand side is $C_0$.
  Furthermore, $\gc_s +_{\triangle \leftrightarrow Y}^{,R} \bridge_s$ is the same up to relabelling as the diagram $D_1$ in Claim~\ref{claim:claim1}.
  Recalling Claim~\ref{claim:claim1}, we have shown $\gc_s \entails D_1 \entails C_0 \entails D_2$,
  which proves the first statement.

  Again we give a schematic, in Figure~\ref{fig:claim2}.
  \begin{figure}[htbp]
  \begin{NoHyper}
  \begin{center}
    \begin{tikzpicture}[
        defnode/.style={rectangle,fill=lightgray,draw,inner sep=2pt,outer sep=0pt,minimum size=10pt},
        gatelabel/.style={rectangle, inner sep=2pt, outer sep=0pt, fill=white,scale=0.6},
        subnode/.style={rounded rectangle, draw, inner sep=5pt, outer sep=0pt,minimum size=15pt},
        leaf/.style={defnode,fill=leafgreen},
        scale=0.8
      ]

      \begin{scope}[shift={(0,0)}]
        \node[subnode, scale=0.8] (V) at (0,0) {$\gc_s$};
        \node[leaf,scale=0.5] (Vp) at (-1,0) {};
        \draw (V.west) -- node[pos=0.0,gatelabel] {$\triangle$} (Vp);
      \end{scope}
      \node at (1,0) {$\leadsto$};
      \node[scale=0.5] at (1,0.5) {Claim~\ref{claim:claim1}};
      \begin{scope}[shift={(2.7,0)}]
        \node[subnode, scale=0.8] (V1) at (0,0) {$\gc_s$};
        \node[leaf,scale=0.5] (Vp) at (-1,0) {};
        \node[subnode, scale=0.8] (V2) at (3,0) {$\gc_s$};
        \node[subnode, scale=0.8] (V3) at (6,0) {$\gc_s$};
        \draw (V1.west) -- node[pos=0.0,gatelabel] {$\triangle$} (Vp);
        \draw (V1.east) -- node[pos=0.13,gatelabel] {$s+1$} node[midway,defnode,scale=0.6] {} node[pos=0.87,gatelabel] {$s+1$} (V2.west);
        \draw (V2.east) -- node[pos=0.03,gatelabel] {$\triangle$} node[midway,defnode,scale=0.6] {} node[pos=0.97,gatelabel] {$\triangle$} (V3.west);
      \end{scope}
      \node at (0,-2) {$\leadsto$};
      \node[scale=0.5,align=center] at (0,-1.5) {Lem~\ref{lem:build-agate}+Lem~\ref{lem:build-large-agate}\\+Lem~\ref{lem:build-super-agate}+Prop~\ref{prop:awesome-stashing}};
      \begin{scope}[shift={(1.7,-2)}]
        \node[subnode, scale=0.8] (V1) at (0,0) {$\gc_s$};
        \node[leaf,scale=0.5] (Vp) at (-1,0) {};
        \node[subnode, scale=0.8] (V2) at (3,0) {$\gc_s$};
        \node[subnode, scale=0.8] (V3) at (6,0) {$\supera_s^{k,m}$};
        \node[subnode, scale=0.8] (V4) at (9,0) {$\gc_s$};
        \node[subnode, scale=0.8] (V5) at (12,0) {$\gc_s$};
        \draw (V1.west) -- node[pos=0.0,gatelabel] {$\triangle$} (Vp);
        \draw (V1.east) -- node[pos=0.13,gatelabel] {$s+1$} node[midway,defnode,scale=0.6] {} node[pos=0.87,gatelabel] {$s+1$} (V2.west);
        \draw (V2.east) -- node[pos=0.03,gatelabel] {$\triangle$} node[midway,defnode,scale=0.6] {} node[pos=0.97,gatelabel] {$X$} (V3.west);
        \draw (V3.east) -- node[pos=0.03,gatelabel] {$Y$} node[midway,defnode,scale=0.6] {} node[pos=0.97,gatelabel] {$\triangle$} (V4.west);
        \draw (V4.east) -- node[pos=0.13,gatelabel] {$s+1$} node[midway,defnode,scale=0.6] {} node[pos=0.87,gatelabel] {$s+1$} (V5.west);
      \end{scope}
      \node at (0,-4) {$\leadsto$};
      \node[scale=0.5,align=center] at (0,-3.5) {Prop~\ref{prop:awesome-stashing}+\\Claim~\ref{claim:claim0}(i)+Lem~\ref{lem:gc-aggre}};
      \begin{scope}[shift={(1.7,-4)}]
        \node[subnode, scale=0.8] (V1) at (0,0) {$\gc_s$};
        \node[leaf,scale=0.5] (Vp) at (-1,0) {};
        \node[subnode, scale=0.8] (V2) at (3,0) {$\gc_s$};
        \node[subnode, scale=0.8] (V3) at (6,0) {$\supera_s^{k,m}$};
        \node[subnode, scale=0.8] (V4) at (9.8,0) {$\agg_s$};
        \draw (V1.west) -- node[pos=0.0,gatelabel] {$\triangle$} (Vp);
        \draw (V1.east) -- node[pos=0.13,gatelabel] {$s+1$} node[midway,defnode,scale=0.6] {} node[pos=0.87,gatelabel] {$s+1$} (V2.west);
        \draw (V2.east) -- node[pos=0.03,gatelabel] {$\triangle$} node[midway,defnode,scale=0.6] {} node[pos=0.97,gatelabel] {$X$} (V3.west);
        \draw (V3.east) -- node[pos=0.03,gatelabel] {$Y$} node[midway,defnode,scale=0.6] {} node[pos=0.87,gatelabel] {$00;\triangle$} (V4.west);
      \end{scope}
    \end{tikzpicture}
  \end{center}
  \end{NoHyper}
  \caption{A schematic of the proof of Claim~\ref{claim:claim2}.}%
  \label{fig:claim2}
  \end{figure}
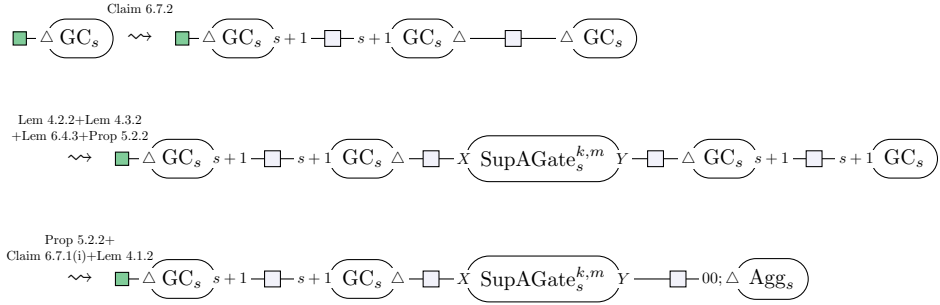

  For the surjectivity statement, we combine Claim~\ref{claim:claim0}(i,iii,iv) with repeated applications of Corollary~\ref{cor:join-surj}.
\end{proof}

Recall the diagram $D=D((t_j)_{j \in I'})$ from~\eqref{eq:boost-phi}.
Our next step is to apply Claim~\ref{claim:claim2} a total of $|I'|$ times, to replace each copy $\gc_{t_j}$ in $D$, one at a time, with a copy of the diagram $D_2$ from Claim~\ref{claim:claim2}.

\begin{claim}%
  \label{claim:claim3}
  Let
  \[
    D_3 = \Phi \bigplus_{j \in I'} {}_{j \leftrightarrow A;X}^{R_j} D_2.
  \]
  Then for $M = 9+\log_2 k + \log_2 m$, we have
  $
    D \entails_{i_0 \mapsto i_0}^{|I'|M} D_3
  $.
\end{claim}
We show this pictorially, in Figure~\ref{fig:claim3}, in the special case $I=\{i_0,i_1,\dots,i_5\}$.
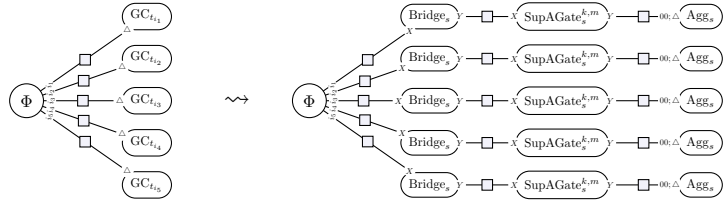
\begin{figure}[htbp]
\begin{NoHyper}
\begin{center}
  \begin{tikzpicture}[
      defnode/.style={rectangle,fill=lightgray,draw,inner sep=2pt,outer sep=0pt,minimum size=10pt},
      gatelabel/.style={circle, inner sep=1pt, outer sep=0pt, fill=white,scale=0.4},
      subnode/.style={rounded rectangle, draw, inner sep=5pt, outer sep=0pt,minimum size=15pt},
      leaf/.style={defnode,fill=leafgreen},
      scale=0.8
    ]
    \begin{scope}[shift={(0,0)}]
      \node[transform shape, subnode,scale=0.8, minimum height=20pt] (F) at (-2,0) {$\Phi$};

      \node[transform shape, subnode,scale=0.6] (B1) at (0, 1.4) {$\gc_{t_{i_1}}$};
      \node[transform shape, subnode,scale=0.6] (B2) at (0, 0.7) {$\gc_{t_{i_2}}$};
      \node[transform shape, subnode,scale=0.6] (B3) at (0, 0.0) {$\gc_{t_{i_3}}$};
      \node[transform shape, subnode,scale=0.6] (B4) at (0,-0.7) {$\gc_{t_{i_4}}$};
      \node[transform shape, subnode,scale=0.6] (B5) at (0,-1.4) {$\gc_{t_{i_5}}$};
      \foreach \j in {1,...,5} {
        \draw (F)    -- node[gatelabel, pos=0.10, inner sep=0pt, rotate=90, transform shape] {$i_{\j}$} node[midway, defnode, transform shape,scale=0.5] {} node[gatelabel,pos=0.97, transform shape] {$\triangle$} (B\j);
      }
    \end{scope}
    \node at (1.5,0) {$\leadsto$};
    \begin{scope}[shift={(4.7,0)}]
      \node[transform shape, subnode,scale=0.8, minimum height=20pt] (F) at (-2,0) {$\Phi$};

      \node[transform shape, subnode,scale=0.6] (B1) at (0, 1.4) {$\bridge_s$};
      \node[transform shape, subnode,scale=0.6] (B2) at (0, 0.7) {$\bridge_s$};
      \node[transform shape, subnode,scale=0.6] (B3) at (0, 0.0) {$\bridge_s$};
      \node[transform shape, subnode,scale=0.6] (B4) at (0,-0.7) {$\bridge_s$};
      \node[transform shape, subnode,scale=0.6] (B5) at (0,-1.4) {$\bridge_s$};

      \node[transform shape, subnode,scale=0.6] (A1) at (2.2, 1.4) {$\supera_s^{k,m}$};
      \node[transform shape, subnode,scale=0.6] (A2) at (2.2, 0.7) {$\supera_s^{k,m}$};
      \node[transform shape, subnode,scale=0.6] (A3) at (2.2, 0.0) {$\supera_s^{k,m}$};
      \node[transform shape, subnode,scale=0.6] (A4) at (2.2,-0.7) {$\supera_s^{k,m}$};
      \node[transform shape, subnode,scale=0.6] (A5) at (2.2,-1.4) {$\supera_s^{k,m}$};

      \node[transform shape, subnode,scale=0.6] (C1) at (4.5, 1.4) {$\agg_s$};
      \node[transform shape, subnode,scale=0.6] (C2) at (4.5, 0.7) {$\agg_s$};
      \node[transform shape, subnode,scale=0.6] (C3) at (4.5, 0.0) {$\agg_s$};
      \node[transform shape, subnode,scale=0.6] (C4) at (4.5,-0.7) {$\agg_s$};
      \node[transform shape, subnode,scale=0.6] (C5) at (4.5,-1.4) {$\agg_s$};

      \foreach \j in {1,...,5} {
        \draw (F)    -- node[gatelabel, inner sep=0pt, pos=0.10, rotate=90, transform shape] {$i_{\j}$} node[midway, defnode, transform shape,scale=0.5] {} node[gatelabel,pos=0.97, transform shape] {$X$} (B\j);
        \draw (B\j)    -- node[gatelabel, pos=0.03, transform shape] {$Y$} node[midway, defnode, transform shape,scale=0.5] {} node[gatelabel,pos=0.97, transform shape] {$X$} (A\j);
        \draw (A\j) -- node[gatelabel, pos=0.03, transform shape] {$Y$} node[midway, defnode, scale=0.5, transform shape] {} node[gatelabel,pos=0.87,transform shape] {$00;\triangle$} (C\j);
      }
    \end{scope}
  \end{tikzpicture}
\end{center}
\end{NoHyper}
\caption{A pictorial depiction of Claim~\ref{claim:claim3}.}%
\label{fig:claim3}
\end{figure}

\begin{proof}[Proof of claim]%
  For any subset $S \subseteq I'$ define
  \[
    D_3^{(S)} = \Phi \bigplus_{j \in S} {}_{j \leftrightarrow A;X}^{R j} D_2
    \bigplus_{j \in I' \setminus S} {}_{j \leftrightarrow \triangle}^{R j} \gc_{t_j}.
  \]
  The we claim $D \entails_{i_0 \mapsto i_0}^{|S|M} D_3^{(S)}$ for any $S$, which suffices by taking $S=I'$.
  Indeed, if $S=\emptyset$ this is trivial, and if $S=S'\cup\{j_1\}$ then
  by Claim~\ref{claim:claim2} and Lemma~\ref{lem:gc-morph} we have $\gc_{t_j} \entails^M_{A;X \mapsto \triangle} D_2$, so by Proposition~\ref{prop:awesome-stashing} we have
  \[
    \bigl( \gc_{t_j} +_{\triangle \leftrightarrow j_1}^R C \bigr)
    \entails^M_{R;i_0 \mapsto R;i_0}
    \bigl( D_2 +_{A;X; \leftrightarrow j_1}^R C \bigr)
  \]
  where
  \[
    C = \Phi \bigplus_{j \in S'} {}_{j \leftrightarrow A;X}^{R j} D_2 
    \bigplus_{j \in I' \setminus S} {}_{j \leftrightarrow \triangle}^{R j} \gc_{t_j} .
  \]
  Up to relabelling the left-hand side is $D_3^{(S')}$ and the right-hand side is $D_3^{(S)}$.
  Hence the claim follows, by induction on $|S|$.
\end{proof}
  
We have now built one of the outer pieces in Figure~\ref{fig:stargate}, except with some extra $\aggregate_s$ gates replacing the central part.
In the final stage of the construction, we clone this piece $n$ ways, while building the gates $\bigagg_s^n$ in the center.

\begin{claim}%
  \label{claim:claim4}
  For $D_3$ the diagram in Claim~\ref{claim:claim3}, we have
  \[
    D_3 \entails^{\log_2 n}_{F0;i_0 \mapsto i_0} \stargate(s,k,m,n,\Phi,i_0).
  \]
\end{claim}
\begin{proof}[Proof of claim]%
  The key idea is to apply the steps in Remark~\ref{rem:build-bigagg} to \emph{all the gates $R_j;B \colon \agg_s$ simultaneously}.
  That is, we perform the following steps:---
  \begin{itemize}
    \item apply $\CS(\{R_j;B;11;\triangle \colon j \in I'\})$;
    \item rename $L \leadsto H0$ and $R \leadsto H1$ as in Remark~\ref{rem:build-bigagg};
    \item for $r=1,2,\dots,\log_2 n - 2$, apply
      \[
        \CS\bigl(\bigl\{ H{\scriptstyle (2^r-1)};R_j;B;10;\triangle \colon j \in I' \bigr\}\bigr)
      \]
      and relabel $L;H\ell \mapsto H\ell$, $R;H\ell \leadsto H{\scriptstyle(2^{r+1}-\ell-1)}$;
    \item finally apply
      \[
        \CS\bigl(\bigl\{ H\ell;R_j;B;10;\triangle \colon j \in I' , \ell \in \{0, n/2-1\}  \bigr\}\bigr)
      \]
      to close the loop, and relabel $L;H\ell \leadsto H\ell$, $R;H\ell \leadsto H{\scriptstyle(n-\ell-1)}$.
  \end{itemize}
  Finally we relabel again to match the labels in Definition~\ref{def:stargate}.
  To be explicit, this goes as follows: for each spoke $\ell \in \{0,1,\dots,n-1\}$ and form $j \in I'$, we would set
  \begin{itemize}
    \item $H\ell;R_j;A \leadsto \opbr(j,\ell)$,
    \item $H\ell;R_j;B \leadsto (\bag j;H\ell)$,
    \item $H\ell;R_j;\ast \leadsto A(j,\ell)$,
    \item $H\ell;\ast \leadsto F\ell$
  \end{itemize}
  where ``$\ast$'' does not include indices previously reassigned.

  Verifying that the resulting diagram is indeed $\stargate(s,k,m,n,\Phi,i_0)$ is now an exercise in unpacking the definitions.
\end{proof}

Lemma~\ref{lem:build-stargate} follows immediately from Claim~\ref{claim:claim3} and Claim~\ref{claim:claim4}. 

\subsection{Assigning a StarGate}%
\label{sub:stargate-assign}

We finish by proving Lemma~\ref{lem:assign-stargate}.

The first step is to do some multilinear algebra, following the outline in Section~\ref{sub:true-summary}.
We begin by choosing a good basis $e_1,\dots,e_d$ for $\FF_p^d = V^\Phi$.

\begin{claim}%
  \label{claim:basis}
  There exists a basis $e_1,\dots,e_d$ for $\FF_p^d$ such that
  $\phi_{i_0}(e_\ell) = 0$ whenever $2 \le \ell \le d$, $\phi_{i_0}(e_1) \ne 0$,
  and moreover the values $\phi_{j}(e_\ell)$ for $j \in I$, $\ell \in [d]$ have representatives in $\ZZ$ with absolute value at most $2 L^2$.
\end{claim}
\begin{proof}[Proof of claim]%
  This is a standard exercise in linear algebra but we take some care to control the coefficients.
  Since $\phi_{i_0} \ne 0$, by permuting the standard basis $\eps_1,\dots,\eps_d$ of $\FF_p^d$ (which does not affect $L$) we may assume WLOG that
  $\phi_{i_0}(\eps_1) \ne 0$ in $\FF_p$.
  Then set $e_1 = \eps_1$, $\lambda = \phi_{i_0}(e_1) \ne 0$ and
  \[
    e_\ell = \lambda \eps_\ell - \phi_{i_0}(e_\ell) \eps_1
  \]
  for $2 \le \ell \le d$.

  We have $\phi_{i_0}(e_\ell)=0$.
  Since all values $\phi_{j}(\eps_\ell)$ including $\lambda$ have integer representatives at most $L$ in absolute value,
  it follows that 
  \[
    \phi_j(e_\ell) = \lambda \phi_j(\eps_\ell) - \phi_{i_0}(e_\ell) \phi_j(\eps_1)
  \]
  has an integer representative which is at most $2 L^2$ in absolute value, for all $j \in I$ and $2 \le \ell \le d$.
  In the case $\ell=1$ it is clear $\phi_j(e_1) = \phi_j(\eps_1)$ has an integer representative at most $L$ in absolute value.
\end{proof}

Next we translate the hypothesis $s(\Phi, i_0) \le s-1$ into a concrete statement in this basis.
\begin{claim}%
  \label{claim:complexity}
  Assume $s(\Phi, i_0) \le s-1$.
  Writing
  \[
    \Omega = \bigl\{ \tau = (\tau_1,\tau_2,\dots, \tau_s) \colon 1 \le \tau_1 \le \cdots \le \tau_s \le d \bigr\}
  \]
  there are coefficients $(\beta_{\tau})_{\tau \in \Omega} \in \FF_p$ such that
  \begin{equation}%
    \label{eq:beta-i0}
    \beta_{(1,1,\dots,1)} \phi_{i_0}(e_1) = 1
  \end{equation}
  and for all $j \in I'$,
  \begin{equation}%
    \label{eq:beta-prop}
    \sum_{\tau \in \Omega} \beta_{\tau} \prod_{r=1}^s \phi_j(e_{\tau_r}) = 0.
  \end{equation}
\end{claim}
\begin{proof}%
  By hypothesis we have
  \[
    \phi_{i_0}^{\otimes s} \notin \spn \left( \phi_j^{\otimes s} \colon j \in I' \right)
  \]
  as elements of $\bigl((\FF_p^d)^\ast\bigr)^{\otimes s} \cong \bigl((\FF_p^d)^{\otimes s}\bigr)^\ast$.
  By duality, there exists an element $v \in (\FF_p^d)^{\otimes s}$ such that
  $
    \phi_{i_0}^{\otimes s}(v) = 1
  $
  but
  $
    \phi_{j}^{\otimes s}(v) = 0
  $
  for all $j \in I'$.
  Expanding $v$ in the basis $e_1,\dots,e_d$, we have
  \[
    v = \sum_{\ell_1,\dots,\ell_s \in [d]} b_{\ell_1,\dots,\ell_s} (e_{\ell_1} \otimes \dots \otimes e_{\ell_s})
  \]
  for some coefficients $b_{\ell_1,\dots,\ell_s} \in \FF_p$.
  For $j \in I$, we have
  \[
    \phi_j^{\otimes s}(v) = \sum_{\ell_1,\dots,\ell_s \in [d]} b_{\ell_1,\dots,\ell_s} \prod_{r \in [s]} \phi_j(e_{\ell_r}).
  \]
  When $j=i_0$, since $\phi_{i_0}(e_\ell)=0$ for $\ell \ne 1$, every term vanishes except $\ell_1=\cdots=\ell_s=1$, so $b_{1,1,\dots,1} \phi_{i_0}(e_1)^s = 1$.
  When $j \ne i$ we deduce
  \[
    \sum_{\ell_1,\dots,\ell_s \in [d]} b_{\ell_1,\dots,\ell_s} \prod_{r \in [s]} \phi_j(e_{\ell_r}) = 0.
  \]

  Finally set
  \[
    \beta_{\tau} = \phi_{i_0}(e_1)^{s-1} \sum_{(\ell_1,\dots,\ell_s) \sim \tau} b_{\ell_1,\dots,\ell_s}
  \]
  where $(\ell_1,\dots,\ell_s) \sim (\tau_1,\dots,\tau_s)$ means that the two sequences are permutations of each other.
  Then~\eqref{eq:beta-i0} and~\eqref{eq:beta-prop} follow from the corresponding properties of $b_{\ell_1,\dots,\ell_s}$.
\end{proof}

We fix an enumeration $\Omega = \{\tau^{0},\dots,\tau^{N-1}\}$ where $N = \lvert \Omega \rvert = \binom{d+s-1}{s}$, meaning $N \le n$ by hypothesis, and where $\tau^0 = (1,1,\dots,1)$.

We now describe the claimed assignment of $\stargate(s,k,m,n,\Phi,i_0)$, henceforth abbreviated to $\stargate$.
The following values are key: for $j \in I$, $0 \le \ell \le N-1$ and $r \in [s]$, we pick an integer $a_{r}^{(j,\ell)}$ such that
\[
  a_{r}^{(j,\ell)} \bmod{p} = \phi_j(e_{\tau^\ell_r})
\]
and $|a_r^{(j,\ell)}| \le 2 L^2 < 2^{k+1}$, which is possible by Claim~\ref{claim:basis}.
When $N \le \ell \le n-1$ we set $a_r^{(j,\ell)} = 0$.

We consider assignments of the sub-gates of $\stargate$ as follows.
\begin{itemize}
  \item For $j \in I'$, take $\bag j \colon \sumconst_1$.
  \item For $j \in I'$ and $0 \le \ell \le n-1$, take $\opbr(j, \ell) \colon \boring_s$.
  \item For $j \in I'$ and $0 \le \ell \le n-1$,
    take $A(j, \ell) \colon \amulti_s^{k,m} \bigl(a^{(j,\ell)}_2,\dots,a^{(j,\ell)}_s\bigr)$.
\end{itemize}
Using Lemma~\ref{lem:bigagg-assign}, Lemma~\ref{lem:bridge-ass} and Lemma~\ref{lem:super-assignment}, we obtain the data $(S_r^{(G)})_{r \in [s]}$ and $(\cM_r^{(G)})_{r \in [s]}$ for each gate $G$.
(We ignore $\cM_0^{(G)}$.)
We recall that the domains $D_r^{(G)}$ are as follows.
\begin{itemize}
  \item $D_1^{(\bag j)} = \bigsum_n$ and $D_r^{(\bag j)} = \const(\{0,\dots,n-1\})$ for $2 \le r \le s$;
  \item $D_r^{(\opbr(j,\ell))} = \const(\{X,Y\})$ for $r \in [s]$;
  \item $D_r^{(A(j,\ell))} = \lag(a^{(j,\ell)}_r, A)$ for $2 \le r \le s$, and $D_1^{(A(j,\ell))} = \lag\bigl(\prod_{r=2}^s a^{(j,\ell)}_r, B\bigr)$.
\end{itemize}

In the remaining sub-diagrams $F\ell$ for $0 \le \ell \le n-1$, the only toggles are the leaves $F\ell;i_0$ for $\ell \ne 0$.
We partition these as follows.
Given $\ell$, $1 \le \ell \le N-1$, we choose the least $r \in [s]$ such that $\tau^\ell_r \ne 1$ (which exists, as $\tau^\ell \ne \tau^0 = (1,\dots,1)$).
Then set $S_r^{(F\ell)} = \{ F\ell;i_0 \}$ and $S_{r'}^{(F\ell)} = \emptyset$ for $r' \ne r$.
For $N \le \ell \le n-1$, set $S_1^{(F\ell)} = \{ F\ell; i_0\}$ and $S_r^{(F\ell)} = \emptyset$ for $r \ne 1$.

For notational compatibility we set $D_r^{(F\ell)} = \Phi$ and $\cM_r^{(F\ell)}$ the identity morphism $\Phi \to \Phi$, for $0 \le \ell \le n-1$ and $r \in [s]$.

As usual we set $S_r = \bigcup_{G \in \Gamma} S_r^{(G)}$, where
\begin{equation}%
  \label{eq:gates}
  \Gamma =
  \begin{aligned}[t]
    &\bigl\{ F\ell \colon 0 \le \ell \le n-1\bigr\} \cup \bigl\{ \opbr(j,\ell) \colon j \in I', 0 \le \ell \le n-1\bigr\} \\
    &\cup \bigl\{ A(j,\ell) \colon j \in I', 0 \le \ell \le n-1 \bigr\} \cup \bigl\{ \bag j \colon j \in I'\bigr\}
  \end{aligned}
\end{equation}
is the collection of sub-gates.
Then $(S_r)_{r \in [s]}$ is a valid partition of $\leafs^{\stargate} \setminus \{F0;i_0\}$.

As is now familiar, we construct the required morphisms
\[
  \cM_r \colon \trivial(\{i_0\}) \to \stargate\bigl[\nonleafs \cup \{F0;i_0\} \cup S_r\bigr]
\]
for $r \in [s]$ by piecing together the existing morphisms $\cM_r^{(G)}$.

\begin{claim}%
  \label{claim:ass-claim}
  There exist morphisms $\Theta_r^{(G)}$ for $r \in [s]$ and $G \in \Gamma$, where
  \[
    \Theta_r^{(G)} \colon \trivial(\{i_0\}) \to D_r^{(G)}\bigl(Z^{(G)} \leadsto \nonleafs\bigr)
  \]
  and where $Z^{(G)} \subseteq \leafs^{G}$ is $I'$ if $G=F\ell$ and $Z^{(G)} = \leafs^G$ for all other $G$, such that the following holds.

  Write $\fM_r^{(G)} = \Theta_r^{(G)} \circ \cM_r^{(G)}$ \uppar{using Remark~\ref{rem:leaf-non-leaf-morph}}, which are morphisms
  \[
    \trivial(\{i_0\}) \to G\bigl[\nonleafs \cup \pins^G \cup S_r^{(G)}\bigr].
  \]
  Then for each $j \in I'$ and $0 \le \ell \le n-1$:---
  \begin{equation}%
    \begin{aligned}
      \theta^{\fM_1^{(F\ell)}}_{j} &= \theta^{\fM_1^{(\opbr(j,\ell))}}_X = x \mapsto (-1)^\ell \beta_{\tau^\ell} a_1^{(j,\ell)} x \\
      \theta^{\fM_1^{(\opbr(j,\ell))}}_Y &= \theta^{\fM_1^{(A(j,\ell))}}_X = x \mapsto (-1)^\ell \beta_{\tau^\ell} a_1^{(j,\ell)} x \\
      \theta^{\fM_1^{(A(j,\ell))}}_Y &= \theta^{\fM_1^{(\bag j)}}_{X\ell} = x \mapsto  (-1)^\ell \beta_{\tau^\ell} \left( \prod_{r=1}^s a_{r}^{(j,\ell)} \right) x
    \end{aligned}
    \label{eq:compat-maps}
  \end{equation}
  and for $2 \le r \le s$:---
  \begin{equation}%
    \begin{aligned}
      \theta^{\fM_r^{(F\ell)}}_{j} &= \theta^{\fM_r^{(\opbr(j,\ell))}}_X = x \mapsto \lambda a_{r}^{(j,\ell)} x \\
      \theta^{\fM_r^{(\opbr(j,\ell))}}_Y &= \theta^{\fM_r^{(A(j,\ell))}}_X = x \mapsto \lambda a_{r}^{(j,\ell)} x \\
      \theta^{\fM_r^{(A(j,\ell))}}_Y &= \theta^{\fM_r^{(\bag j)}}_{X\ell} = x \mapsto \lambda x
    \end{aligned}
    \label{eq:compat-maps-2}
  \end{equation}
  where $\lambda = \phi_{i_0}(e_1)^{-1} \in \FF_p$ \uppar{which exists by Claim~\ref{claim:basis}}.

  Finally, $\theta^{\fM_r^{(F0)}}_{i_0} = \id_{\FF_p}$ for all $r \in [s]$, and
  $\alpha^{\Theta_r^{(F\ell)}}(i_0) = \zeroi$ whenever $F\ell;i_0 \in S_r$.
\end{claim}

In particular,~\eqref{eq:compat-maps} and~\eqref{eq:compat-maps-2} imply that for each $r$ the morphisms $\fM_r^{(G)}$ agree wherever they overlap.
Hence by Lemma~\ref{lem:patching}, we can patch $\fM_r^{(G)}$ together to obtain morphisms
\[
  \cM_r \colon \trivial(\{i_0\}) \to \stargate\bigl[\nonleafs \cup \{F0;i_0\} \cup S_r\bigr]
\]
respecting $F0;i_0$ and with $\alpha^{\cM_r}(x) = \zeroi$ for all $x \in S_r$.
As noted in Section~\ref{sub:stargates}, we can define $\cM_0$ by composing $\cM_1$ with the restriction map
\[
  \stargate[\nonleafs \cup \{F0;i_0\} \cup S_1] \to \stargate[\nonleafs \cup \{F0;i_0\}],
\]
or similarly for any other $\cM_r$, $r \ge 1$.
Hence we get the assignment claimed, and Claim~\ref{claim:ass-claim} completes the proof of Lemma~\ref{lem:assign-stargate}.

We note informally why each gate is doing what it is supposed to do relative the maps specified in Claim~\ref{claim:ass-claim} at the shared pins.\footnote{
  It would be visually clearer to place these values at the vertices of Figure~\ref{fig:stargate}, as in previous cases such as Figure~\ref{fig:agate-initial}, but this is typographically impractical in this case.
  The reader can of course draw their own such picture.}
\begin{itemize}
  \item Certainly the $\bridge$ gates always behave as $\const(\{X,Y\})$.
  \item The vertices $F\ell;j$ have maps
    \[
      x \mapsto C_\ell a_r^{(j,\ell)} x = \phi_j(C_\ell x e_{\tau_r^\ell})
    \]
    for some $C_\ell \in \FF_p$, which are indeed $\phi_j$ composed with the map $\FF_p \to V^\Phi$, $x \mapsto C_\ell x e_{\tau_r^\ell}$, as required by what we informally call a ``$\Phi$-gate''.
  \item The gates $A(j,\ell)$ do indeed have $\theta_X = a_r^{(j,\ell)} \theta_Y$ when $r=2,\dots,s$ and
    $\theta_Y = \left( \prod_{r=2}^s a_r^{(j,\ell)} \right) \theta_X$ when $r=1$, as required by the linear data $\lag(a_r^{j,\ell},A)$ and $\lag\left(\prod_{r=1}^s a_r^{j,\ell},B\right)$.
  \item The gates $\bag j$ clearly behave as $\const(I')$ when $r \ne 1$, and
    \[
      \sum_{\ell=0}^{n-1} \beta_{\tau^\ell} \left( \prod_{r=1}^s a_r^{(j,\ell)} \right) = 0
    \]
    by~\eqref{eq:beta-prop} so when $r=1$, $\bag j$ behaves like a $\bigsum$ gate.
\end{itemize}
Finally, the fact that when $F\ell;i_0 \in S_r$ we have $\tau_r^\ell \ne 1$, so $\phi_{i_0}(e_{\tau_r^\ell} ) = 0$ and hence
  $\phi_{i_0}(C_\ell x e_{\tau_r^\ell}) = 0$,
allows us to set $\alpha^{\Theta_r^{(F\ell)}}(i_0) = \zeroi$.

We now make this slightly more formal.
\begin{proof}[Proof of claim]%
  We again specify the morphisms $\Theta_r^{(G)}$ pictorially.
  When $r=1$ these are shown in Figure~\ref{fig:ass-claim-1}, while those for $2 \le r \le s$ are shown in Figure~\ref{fig:ass-claim-2}.
  \begin{figure}
    \begin{center}
      \begin{tikzpicture}[
          scale=0.95,
          defnode/.style={rectangle,draw,inner sep=4pt,outer sep=0pt,minimum size=15pt},
          mydot/.style={circle,fill,inner sep=0.5pt},
          nonleaf/.style={defnode,fill=lightgray},
          leaf/.style={defnode,fill=leafgreen},
          toggle/.style={defnode,fill=leafgreen}
        ]
        \node[scale=0.6] at (5, -1.5) {$\Theta_1^{\opbr(j,\ell)} \colon \trivial(\{i_0\}) \to \const(\{X,Y\})(\{X,Y\} \leadsto \nonleafs)$};
        \begin{scope}[shift={(0,0)}]
          \node[defnode,scale=0.6] (U) at (0, 0) {$\diamond$};
          \node[leaf,scale=0.6]    (Ui) at ($(U) + (0,-1)$) {$i_0$};
          \draw[-stealth] (U) -- (Ui);

          \node[defnode,scale=0.6] (V) at (3, 0) {$\diamond$};
          \node[nonleaf,scale=0.6]    (VX) at ($(V) + (-0.6,-1)$) {$X$};
          \node[nonleaf,scale=0.6]    (VY) at ($(V) + ( 0.6,-1)$) {$Y$};
          \draw[-stealth] (V) -- (VX);
          \draw[-stealth] (V) -- (VY);

          \draw[dashed, -angle 60] (U) to[bend right=10] (VX);
          \draw[dashed, -angle 60] (U) to[bend left=10] (VY);
          \draw[dashed, -angle 60] (U) to (V);
        \end{scope}
        \begin{scope}[shift={(6,0)}]
          \node[defnode,scale=0.6] (U) at (0, 0) {$x$};
          \node[leaf,scale=0.6]    (Ui) at ($(U) + (0,-1)$) {$x$};
          \draw[-stealth] (U) -- (Ui);

          \node[defnode,scale=0.6] (V) at (3, 0) {$(-1)^\ell \beta_{\tau^\ell} a_1^{(j,\ell)} x$};
          \node[nonleaf,scale=0.6]    (VX) at ($(V) + (-1.0,-1)$) {$(-1)^\ell \beta_{\tau^\ell} a_1^{(j,\ell)} x$};
          \node[nonleaf,scale=0.6]    (VY) at ($(V) + ( 1.0,-1)$) {$(-1)^\ell \beta_{\tau^\ell} a_1^{(j,\ell)} x$};
          \draw[-stealth] (V) -- (VX);
          \draw[-stealth] (V) -- (VY);

          \draw[dashed, -angle 60] (U) to[bend right=10] (VX);
          \draw[dashed, -angle 60] (U) to[bend left=10] (VY);
          \draw[dashed, -angle 60] (U) to (V);
        \end{scope}
      \end{tikzpicture}
      \\[0.5\baselineskip]
      \begin{tikzpicture}[
          scale=0.95,
          defnode/.style={rectangle,draw,inner sep=4pt,outer sep=0pt,minimum size=15pt},
          mydot/.style={circle,fill,inner sep=0.5pt},
          nonleaf/.style={defnode,fill=lightgray},
          leaf/.style={defnode,fill=leafgreen},
          toggle/.style={defnode,fill=leafgreen}
        ]
        \node[scale=0.6] at (5, -1.5) {$\Theta_1^{A(j,\ell)} \colon \trivial(\{i_0\}) \to \lag \left( \prod_{r=2}^s a_r^{(j,\ell)}, B \right)(\{X,Y\} \leadsto \nonleafs)$};
        \begin{scope}[shift={(0,0)}]
          \node[defnode,scale=0.6] (U) at (0, 0) {$\diamond$};
          \node[leaf,scale=0.6]    (Ui) at ($(U) + (0,-1)$) {$i_0$};
          \draw[-stealth] (U) -- (Ui);

          \node[defnode,scale=0.6] (V) at (3, 0) {$\diamond$};
          \node[nonleaf,scale=0.6]    (VX) at ($(V) + (-0.6,-1)$) {$X$};
          \node[nonleaf,scale=0.6]    (VY) at ($(V) + ( 0.6,-1)$) {$Y$};
          \draw[-stealth] (V) -- (VX);
          \draw[-stealth] (V) -- (VY);

          \draw[dashed, -angle 60] (U) to[bend right=10] (VX);
          \draw[dashed, -angle 60] (U) to[bend left=10] (VY);
          \draw[dashed, -angle 60] (U) to (V);
        \end{scope}
        \begin{scope}[shift={(6,0)}]
          \node[defnode,scale=0.6] (U) at (0, 0) {$x$};
          \node[leaf,scale=0.6]    (Ui) at ($(U) + (0,-1)$) {$x$};
          \draw[-stealth] (U) -- (Ui);

          \node[defnode,scale=0.6] (V) at (3.5, 0) {$(-1)^\ell \beta_{\tau^\ell} a_1^{(j,\ell)} x$};
          \node[nonleaf,scale=0.6]    (VX) at ($(V) + (-1.3,-1)$) {$(-1)^\ell \beta_{\tau^\ell} a_1^{(j,\ell)} x$};
          \node[nonleaf,scale=0.6]    (VY) at ($(V) + ( 1.6,-1)$) {$(-1)^\ell \beta_{\tau^\ell} \left( \prod_{r=1}^s a_r^{(j,\ell)} \right) x$};
          \draw[-stealth] (V) -- (VX);
          \draw[-stealth] (V) -- (VY);

          \draw[dashed, -angle 60] (U) to[bend right=10] (VX);
          \draw[dashed, -angle 60] (U) to[bend left=5] (VY);
          \draw[dashed, -angle 60] (U) to (V);
        \end{scope}
      \end{tikzpicture}
      \\[0.5\baselineskip]
      \begin{tikzpicture}[
          scale=0.95,
          defnode/.style={rectangle,draw,inner sep=4pt,outer sep=0pt,minimum size=15pt},
          mydot/.style={circle,fill,inner sep=0.5pt},
          nonleaf/.style={defnode,fill=lightgray},
          leaf/.style={defnode,fill=leafgreen},
          toggle/.style={defnode,fill=leafgreen}
        ]
        \node[scale=0.6] at (5, -1.5) {$\Theta_1^{F0} \colon \trivial(\{i_0\}) \to \Phi(I' \leadsto \nonleafs)$};
        \begin{scope}[shift={(0,0)}]
          \node[defnode,scale=0.6] (U) at (0, 0) {$\diamond$};
          \node[leaf,scale=0.6]    (Ui) at ($(U) + (0,-1)$) {$i_0$};
          \draw[-stealth] (U) -- (Ui);

          \node[defnode,scale=0.6]    (V) at (3, 0) {$\diamond$};
          \node[leaf,scale=0.6]       (VX) at ($(V) + (-0.8,-1)$) {$i_0$};
          \node[scale=0.6]            (Vd1) at ($(V) + (0.5,-1)$) {$\dots$};
          \node[nonleaf,scale=0.6]    (VY) at ($(V) + ( 1,-1)$) {$j$};
          \node[scale=0.6]            (Vd2) at ($(V) + ( 1.5,-1)$) {$\dots$};
          \draw[-stealth] (V) -- (VX);
          \draw[-stealth] (V) -- (VY);

          \draw[dotted, -angle 60] (Ui) to[bend right=10] (VX);
          \draw[dashed, -angle 60] (U) to[bend left=10] (VY);
          \draw[dashed, -angle 60] (U) to (V);
        \end{scope}
        \begin{scope}[shift={(6,0)}]
          \node[defnode,scale=0.6] (U) at (0, 0) {$x$};
          \node[leaf,scale=0.6]    (Ui) at ($(U) + (0,-1)$) {$x$};
          \draw[-stealth] (U) -- (Ui);

        \node[defnode,scale=0.6]    (V) at (3, 0) {$\beta_{\tau^0} x e_{1}$};
        \node[leaf,scale=0.6]       (VX) at ($(V) + (-0.8,-1)$) {$x$};
          \node[scale=0.6]            (Vd1) at ($(V) + (0.5,-1)$) {$\dots$};
          \node[nonleaf,scale=0.6]    (VY) at ($(V) + ( 1.5,-1)$) {$\beta_{\tau^0} a_1^{(j,0)} x$};
          \node[scale=0.6]            (Vd2) at ($(V) + ( 2.5,-1)$) {$\dots$};
          \draw[-stealth] (V) -- (VX);
          \draw[-stealth] (V) -- (VY);

          \draw[dotted, -angle 60] (Ui) to[bend right=10] (VX);
          \draw[dashed, -angle 60] (U) to[bend left=10] (VY);
          \draw[dashed, -angle 60] (U) to (V);
        \end{scope}
      \end{tikzpicture}
      \\[0.5\baselineskip]
      \begin{tikzpicture}[
          scale=0.95,
          defnode/.style={rectangle,draw,inner sep=4pt,outer sep=0pt,minimum size=15pt},
          mydot/.style={circle,fill,inner sep=0.5pt},
          nonleaf/.style={defnode,fill=lightgray},
          leaf/.style={defnode,fill=leafgreen},
          toggle/.style={defnode,fill=leafgreen}
        ]
        \node[scale=0.6] at (5, -1.5) {$\Theta_1^{F\ell} \colon \trivial(\{i_0\}) \to \Phi(I' \leadsto \nonleafs)$, $\ell \ne 0$, $F\ell;i_0 \in S_1$};
        \begin{scope}[shift={(0,0)}]
          \node[defnode,scale=0.6] (U) at (0, 0) {$\diamond$};
          \node[leaf,scale=0.6]    (Ui) at ($(U) + (0,-1)$) {$i_0$};
          \draw[-stealth] (U) -- (Ui);

          \node[defnode,scale=0.6]    (V) at (3, 0) {$\diamond$};
          \node[leaf,scale=0.6]       (VX) at ($(V) + (-0.8,-1)$) {$i_0$};
          \node[scale=0.6]            (Vd1) at ($(V) + (0.5,-1)$) {$\dots$};
          \node[nonleaf,scale=0.6]    (VY) at ($(V) + ( 1,-1)$) {$j$};
          \node[scale=0.6]            (Vd2) at ($(V) + ( 1.5,-1)$) {$\dots$};
          \draw[-stealth] (V) -- (VX);
          \draw[-stealth] (V) -- (VY);

          \draw[dashed, -angle 60] (U) to[bend left=10] (VY);
          \draw[dashed, -angle 60] (U) to (V);
        \end{scope}
        \begin{scope}[shift={(6,0)}]
          \node[defnode,scale=0.6] (U) at (0, 0) {$x$};
          \node[leaf,scale=0.6]    (Ui) at ($(U) + (0,-1)$) {$x$};
          \draw[-stealth] (U) -- (Ui);

        \node[defnode,scale=0.6]    (V) at (3, 0) {$(-1)^\ell \beta_{\tau^\ell} x e_{\tau_1^\ell}$};
        \node[leaf,scale=0.6]       (VX) at ($(V) + (-0.8,-1)$) {$0$};
          \node[scale=0.6]            (Vd1) at ($(V) + (0.3,-1)$) {$\dots$};
          \node[nonleaf,scale=0.6]    (VY) at ($(V) + ( 1.5,-1)$) {$(-1)^\ell \beta_{\tau^\ell} a_1^{(j,\ell)} x$};
          \node[scale=0.6]            (Vd2) at ($(V) + ( 2.7,-1)$) {$\dots$};
          \draw[-stealth] (V) -- (VX);
          \draw[-stealth] (V) -- (VY);

          \draw[dashed, -angle 60] (U) to[bend left=5] (VY);
          \draw[dashed, -angle 60] (U) to (V);
        \end{scope}
      \end{tikzpicture}
      \\[0.5\baselineskip]
      \begin{tikzpicture}[
          scale=0.95,
          defnode/.style={rectangle,draw,inner sep=4pt,outer sep=0pt,minimum size=15pt},
          mydot/.style={circle,fill,inner sep=0.5pt},
          nonleaf/.style={defnode,fill=lightgray},
          leaf/.style={defnode,fill=leafgreen},
          toggle/.style={defnode,fill=leafgreen}
        ]
        \node[scale=0.6] at (5, -1.5) {$\Theta_1^{F\ell} \colon \trivial(\{i_0\}) \to \Phi(I' \leadsto \nonleafs)[I']$, $\ell \ne 0$, $F\ell;i_0 \notin S_1$};
        \begin{scope}[shift={(0,0)}]
          \node[defnode,scale=0.6] (U) at (0, 0) {$\diamond$};
          \node[leaf,scale=0.6]    (Ui) at ($(U) + (0,-1)$) {$i_0$};
          \draw[-stealth] (U) -- (Ui);

          \node[defnode,scale=0.6]    (V) at (3, 0) {$\diamond$};
          \node[scale=0.6]            (Vd1) at ($(V) + (0.5,-1)$) {$\dots$};
          \node[nonleaf,scale=0.6]    (VY) at ($(V) + ( 1,-1)$) {$j$};
          \node[scale=0.6]            (Vd2) at ($(V) + ( 1.5,-1)$) {$\dots$};
          \draw[-stealth] (V) -- (VY);

          \draw[dashed, -angle 60] (U) to[bend left=10] (VY);
          \draw[dashed, -angle 60] (U) to (V);
        \end{scope}
        \begin{scope}[shift={(6,0)}]
          \node[defnode,scale=0.6] (U) at (0, 0) {$x$};
          \node[leaf,scale=0.6]    (Ui) at ($(U) + (0,-1)$) {$x$};
          \draw[-stealth] (U) -- (Ui);

        \node[defnode,scale=0.6]    (V) at (3, 0) {$(-1)^\ell \beta_{\tau^\ell} x e_{\tau_1^\ell}$};
          \node[scale=0.6]            (Vd1) at ($(V) + (0.3,-1)$) {$\dots$};
          \node[nonleaf,scale=0.6]    (VY) at ($(V) + ( 1.5,-1)$) {$(-1)^\ell \beta_{\tau^\ell} a_1^{(j,\ell)} x$};
          \node[scale=0.6]            (Vd2) at ($(V) + ( 2.7,-1)$) {$\dots$};
          \draw[-stealth] (V) -- (VY);

          \draw[dashed, -angle 60] (U) to[bend left=5] (VY);
          \draw[dashed, -angle 60] (U) to (V);
        \end{scope}
      \end{tikzpicture}
      \\[0.5\baselineskip]
      \begin{tikzpicture}[
          scale=0.95,
          defnode/.style={rectangle,draw,inner sep=4pt,outer sep=0pt,minimum size=15pt},
          mydot/.style={circle,fill,inner sep=0.5pt},
          nonleaf/.style={defnode,fill=lightgray},
          leaf/.style={defnode,fill=leafgreen},
          toggle/.style={defnode,fill=leafgreen}
        ]
        \node[scale=0.6] at (5, -1.5) {$\Theta_1^{\bag j} \colon \trivial(\{i_0\}) \to \bigsum_n(\{X0,\dots,X(n-1)\} \leadsto \nonleafs)$};
        \begin{scope}[shift={(0,0)}]
          \node[defnode,scale=0.6] (U) at (0, 0) {$\diamond$};
          \node[leaf,scale=0.6]    (Ui) at ($(U) + (0,-1)$) {$i_0$};
          \draw[-stealth] (U) -- (Ui);

          \node[defnode,scale=0.6]    (V) at (3, 0) {$\diamond$};
          \node[scale=0.6]    (Vd1) at ($(V) + ( -0.7,-1)$) {$\dots$};
          \node[nonleaf,scale=0.6]    (Vi) at ($(V) + ( 0,-1)$) {$X\ell$};
          \node[scale=0.6]            (Vd2) at ($(V) + ( 0.7,-1)$) {$\dots$};
          \draw[-stealth] (V) -- (Vi);

          \draw[dashed, -angle 60] (U) to[bend left=10] (Vi);
          \draw[dashed, -angle 60] (U) to (V);
        \end{scope}
        \begin{scope}[shift={(6,0)}]
          \node[defnode,scale=0.6] (U) at (0, 0) {$x$};
          \node[leaf,scale=0.6]    (Ui) at ($(U) + (0,-1)$) {$x$};
          \draw[-stealth] (U) -- (Ui);

          \node[defnode,scale=0.6]    (V) at (3, 0) {$\left( (-1)^\ell \beta_{\tau^\ell} \left( \prod_{r=1}^s a_r^{(j,\ell)} \right) x\right)_{0 \le \ell \le n-1}$};
          \node[scale=0.6]    (Vd1) at ($(V) + ( -1.5,-1)$) {$\dots$};
          \node[nonleaf,scale=0.6]    (Vi) at ($(V) + ( -0,-1)$) {$(-1)^\ell \beta_{\tau^\ell} \left( \prod_{r=1}^s a_r^{(j,\ell)} \right) x$};
          \node[scale=0.6]    (Vd2) at ($(V) + ( 1.5,-1)$) {$\dots$};
          \draw[-stealth] (V) -- (Vi);

          \draw[dashed, -angle 60] (U) to[bend right=10] (Vi);
          \draw[dashed, -angle 60] (U) to (V);
        \end{scope}
      \end{tikzpicture}
    \end{center}
    \caption{The morphisms $\Theta_1^{(G)}$ from Claim~\ref{claim:ass-claim}.}%
    \label{fig:ass-claim-1}
  \end{figure}
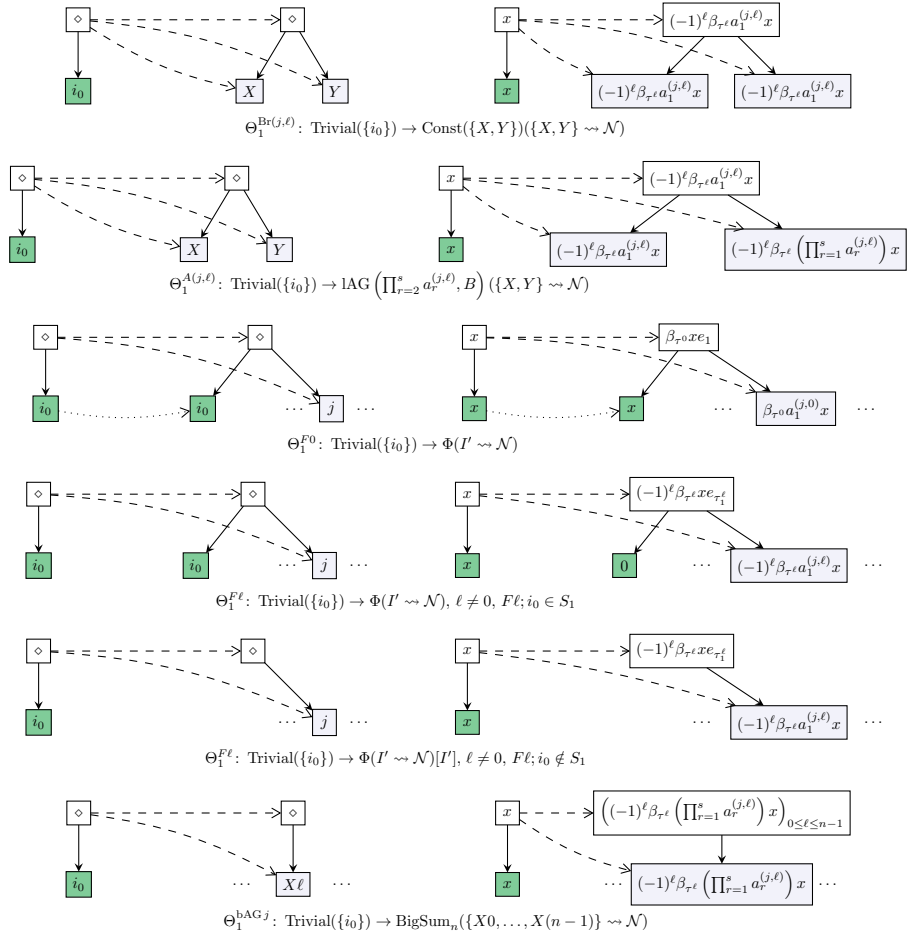

  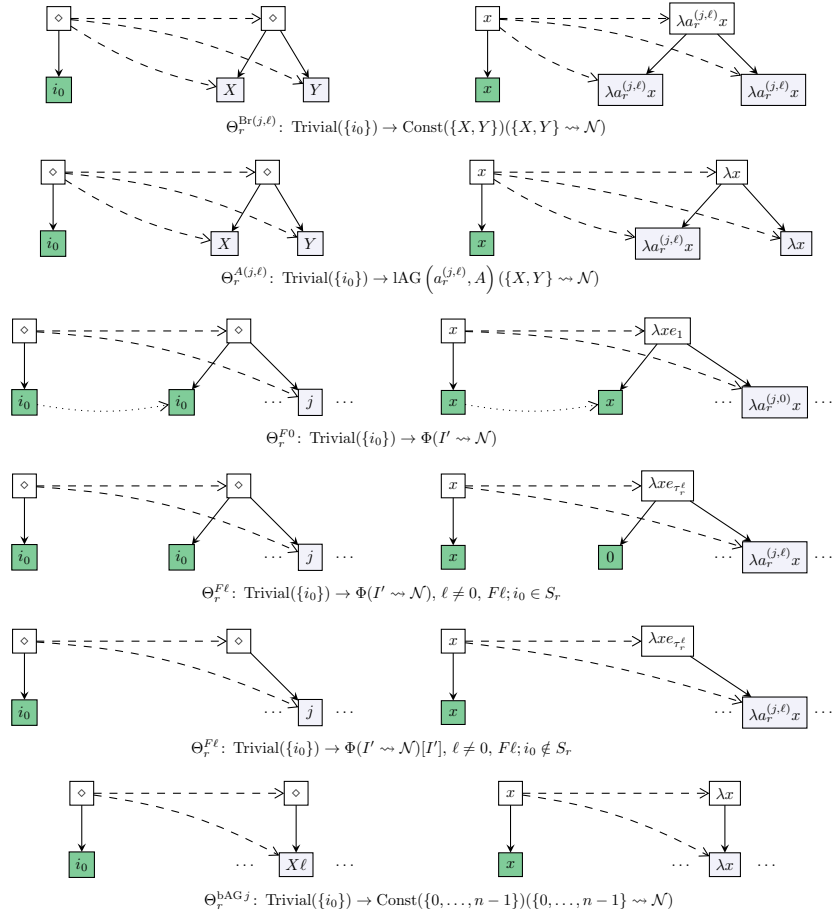
\begin{figure}
    \begin{center}
      \begin{tikzpicture}[
          scale=0.95,
          defnode/.style={rectangle,draw,inner sep=4pt,outer sep=0pt,minimum size=15pt},
          mydot/.style={circle,fill,inner sep=0.5pt},
          nonleaf/.style={defnode,fill=lightgray},
          leaf/.style={defnode,fill=leafgreen},
          toggle/.style={defnode,fill=leafgreen}
        ]
        \node[scale=0.6] at (5, -1.5) {$\Theta_r^{\opbr(j,\ell)} \colon \trivial(\{i_0\}) \to \const(\{X,Y\})(\{X,Y\} \leadsto \nonleafs)$};
        \begin{scope}[shift={(0,0)}]
          \node[defnode,scale=0.6] (U) at (0, 0) {$\diamond$};
          \node[leaf,scale=0.6]    (Ui) at ($(U) + (0,-1)$) {$i_0$};
          \draw[-stealth] (U) -- (Ui);

          \node[defnode,scale=0.6] (V) at (3, 0) {$\diamond$};
          \node[nonleaf,scale=0.6]    (VX) at ($(V) + (-0.6,-1)$) {$X$};
          \node[nonleaf,scale=0.6]    (VY) at ($(V) + ( 0.6,-1)$) {$Y$};
          \draw[-stealth] (V) -- (VX);
          \draw[-stealth] (V) -- (VY);

          \draw[dashed, -angle 60] (U) to[bend right=10] (VX);
          \draw[dashed, -angle 60] (U) to[bend left=10] (VY);
          \draw[dashed, -angle 60] (U) to (V);
        \end{scope}
        \begin{scope}[shift={(6,0)}]
          \node[defnode,scale=0.6] (U) at (0, 0) {$x$};
          \node[leaf,scale=0.6]    (Ui) at ($(U) + (0,-1)$) {$x$};
          \draw[-stealth] (U) -- (Ui);

          \node[defnode,scale=0.6] (V) at (3, 0) {$\lambda a_r^{(j,\ell)} x$};
          \node[nonleaf,scale=0.6]    (VX) at ($(V) + (-1.0,-1)$) {$\lambda a_r^{(j,\ell)} x$};
          \node[nonleaf,scale=0.6]    (VY) at ($(V) + ( 1.0,-1)$) {$\lambda a_r^{(j,\ell)} x$};
          \draw[-stealth] (V) -- (VX);
          \draw[-stealth] (V) -- (VY);

          \draw[dashed, -angle 60] (U) to[bend right=10] (VX);
          \draw[dashed, -angle 60] (U) to[bend left=10] (VY);
          \draw[dashed, -angle 60] (U) to (V);
        \end{scope}
      \end{tikzpicture}
      \\[0.5\baselineskip]
      \begin{tikzpicture}[
          scale=0.95,
          defnode/.style={rectangle,draw,inner sep=4pt,outer sep=0pt,minimum size=15pt},
          mydot/.style={circle,fill,inner sep=0.5pt},
          nonleaf/.style={defnode,fill=lightgray},
          leaf/.style={defnode,fill=leafgreen},
          toggle/.style={defnode,fill=leafgreen}
        ]
        \node[scale=0.6] at (5, -1.5) {$\Theta_r^{A(j,\ell)} \colon \trivial(\{i_0\}) \to \lag \left(a_r^{(j,\ell)}, A \right)(\{X,Y\} \leadsto \nonleafs)$};
        \begin{scope}[shift={(0,0)}]
          \node[defnode,scale=0.6] (U) at (0, 0) {$\diamond$};
          \node[leaf,scale=0.6]    (Ui) at ($(U) + (0,-1)$) {$i_0$};
          \draw[-stealth] (U) -- (Ui);

          \node[defnode,scale=0.6] (V) at (3, 0) {$\diamond$};
          \node[nonleaf,scale=0.6]    (VX) at ($(V) + (-0.6,-1)$) {$X$};
          \node[nonleaf,scale=0.6]    (VY) at ($(V) + ( 0.6,-1)$) {$Y$};
          \draw[-stealth] (V) -- (VX);
          \draw[-stealth] (V) -- (VY);

          \draw[dashed, -angle 60] (U) to[bend right=10] (VX);
          \draw[dashed, -angle 60] (U) to[bend left=10] (VY);
          \draw[dashed, -angle 60] (U) to (V);
        \end{scope}
        \begin{scope}[shift={(6,0)}]
          \node[defnode,scale=0.6] (U) at (0, 0) {$x$};
          \node[leaf,scale=0.6]    (Ui) at ($(U) + (0,-1)$) {$x$};
          \draw[-stealth] (U) -- (Ui);

          \node[defnode,scale=0.6] (V) at (3.5, 0) {$\lambda x$};
          \node[nonleaf,scale=0.6]    (VX) at ($(V) + (-0.9,-1)$) {$\lambda a_r^{(j,\ell)} x$};
          \node[nonleaf,scale=0.6]    (VY) at ($(V) + ( 0.9,-1)$) {$\lambda x$};
          \draw[-stealth] (V) -- (VX);
          \draw[-stealth] (V) -- (VY);

          \draw[dashed, -angle 60] (U) to[bend right=10] (VX);
          \draw[dashed, -angle 60] (U) to[bend left=5] (VY);
          \draw[dashed, -angle 60] (U) to (V);
        \end{scope}
      \end{tikzpicture}
      \\[0.5\baselineskip]
      \begin{tikzpicture}[
          scale=0.95,
          defnode/.style={rectangle,draw,inner sep=4pt,outer sep=0pt,minimum size=15pt},
          mydot/.style={circle,fill,inner sep=0.5pt},
          nonleaf/.style={defnode,fill=lightgray},
          leaf/.style={defnode,fill=leafgreen},
          toggle/.style={defnode,fill=leafgreen}
        ]
        \node[scale=0.6] at (5, -1.5) {$\Theta_r^{F0} \colon \trivial(\{i_0\}) \to \Phi(I' \leadsto \nonleafs)$};
        \begin{scope}[shift={(0,0)}]
          \node[defnode,scale=0.6] (U) at (0, 0) {$\diamond$};
          \node[leaf,scale=0.6]    (Ui) at ($(U) + (0,-1)$) {$i_0$};
          \draw[-stealth] (U) -- (Ui);

          \node[defnode,scale=0.6]    (V) at (3, 0) {$\diamond$};
          \node[leaf,scale=0.6]       (VX) at ($(V) + (-0.8,-1)$) {$i_0$};
          \node[scale=0.6]            (Vd1) at ($(V) + (0.5,-1)$) {$\dots$};
          \node[nonleaf,scale=0.6]    (VY) at ($(V) + ( 1,-1)$) {$j$};
          \node[scale=0.6]            (Vd2) at ($(V) + ( 1.5,-1)$) {$\dots$};
          \draw[-stealth] (V) -- (VX);
          \draw[-stealth] (V) -- (VY);

          \draw[dotted, -angle 60] (Ui) to[bend right=10] (VX);
          \draw[dashed, -angle 60] (U) to[bend left=10] (VY);
          \draw[dashed, -angle 60] (U) to (V);
        \end{scope}
        \begin{scope}[shift={(6,0)}]
          \node[defnode,scale=0.6] (U) at (0, 0) {$x$};
          \node[leaf,scale=0.6]    (Ui) at ($(U) + (0,-1)$) {$x$};
          \draw[-stealth] (U) -- (Ui);

        \node[defnode,scale=0.6]    (V) at (3, 0) {$\lambda x e_{1}$};
        \node[leaf,scale=0.6]       (VX) at ($(V) + (-0.8,-1)$) {$x$};
          \node[scale=0.6]            (Vd1) at ($(V) + (0.8,-1)$) {$\dots$};
          \node[nonleaf,scale=0.6]    (VY) at ($(V) + ( 1.5,-1)$) {$\lambda a_r^{(j,0)} x$};
          \node[scale=0.6]            (Vd2) at ($(V) + ( 2.2,-1)$) {$\dots$};
          \draw[-stealth] (V) -- (VX);
          \draw[-stealth] (V) -- (VY);

          \draw[dotted, -angle 60] (Ui) to[bend right=10] (VX);
          \draw[dashed, -angle 60] (U) to[bend left=10] (VY);
          \draw[dashed, -angle 60] (U) to (V);
        \end{scope}
      \end{tikzpicture}
      \\[0.5\baselineskip]
      \begin{tikzpicture}[
          scale=0.95,
          defnode/.style={rectangle,draw,inner sep=4pt,outer sep=0pt,minimum size=15pt},
          mydot/.style={circle,fill,inner sep=0.5pt},
          nonleaf/.style={defnode,fill=lightgray},
          leaf/.style={defnode,fill=leafgreen},
          toggle/.style={defnode,fill=leafgreen}
        ]
        \node[scale=0.6] at (5, -1.5) {$\Theta_r^{F\ell} \colon \trivial(\{i_0\}) \to \Phi(I' \leadsto \nonleafs)$, $\ell \ne 0$, $F\ell;i_0 \in S_r$};
        \begin{scope}[shift={(0,0)}]
          \node[defnode,scale=0.6] (U) at (0, 0) {$\diamond$};
          \node[leaf,scale=0.6]    (Ui) at ($(U) + (0,-1)$) {$i_0$};
          \draw[-stealth] (U) -- (Ui);

          \node[defnode,scale=0.6]    (V) at (3, 0) {$\diamond$};
          \node[leaf,scale=0.6]       (VX) at ($(V) + (-0.8,-1)$) {$i_0$};
          \node[scale=0.6]            (Vd1) at ($(V) + (0.5,-1)$) {$\dots$};
          \node[nonleaf,scale=0.6]    (VY) at ($(V) + ( 1,-1)$) {$j$};
          \node[scale=0.6]            (Vd2) at ($(V) + ( 1.5,-1)$) {$\dots$};
          \draw[-stealth] (V) -- (VX);
          \draw[-stealth] (V) -- (VY);

          \draw[dashed, -angle 60] (U) to[bend left=10] (VY);
          \draw[dashed, -angle 60] (U) to (V);
        \end{scope}
        \begin{scope}[shift={(6,0)}]
          \node[defnode,scale=0.6] (U) at (0, 0) {$x$};
          \node[leaf,scale=0.6]    (Ui) at ($(U) + (0,-1)$) {$x$};
          \draw[-stealth] (U) -- (Ui);

        \node[defnode,scale=0.6]    (V) at (3, 0) {$\lambda x e_{\tau_r^\ell}$};
        \node[leaf,scale=0.6]       (VX) at ($(V) + (-0.8,-1)$) {$0$};
          \node[scale=0.6]            (Vd1) at ($(V) + (0.8,-1)$) {$\dots$};
          \node[nonleaf,scale=0.6]    (VY) at ($(V) + ( 1.5,-1)$) {$\lambda a_r^{(j,\ell)} x$};
          \node[scale=0.6]            (Vd2) at ($(V) + ( 2.2,-1)$) {$\dots$};
          \draw[-stealth] (V) -- (VX);
          \draw[-stealth] (V) -- (VY);

          \draw[dashed, -angle 60] (U) to[bend left=5] (VY);
          \draw[dashed, -angle 60] (U) to (V);
        \end{scope}
      \end{tikzpicture}
      \\[0.5\baselineskip]
      \begin{tikzpicture}[
          scale=0.95,
          defnode/.style={rectangle,draw,inner sep=4pt,outer sep=0pt,minimum size=15pt},
          mydot/.style={circle,fill,inner sep=0.5pt},
          nonleaf/.style={defnode,fill=lightgray},
          leaf/.style={defnode,fill=leafgreen},
          toggle/.style={defnode,fill=leafgreen}
        ]
        \node[scale=0.6] at (5, -1.5) {$\Theta_r^{F\ell} \colon \trivial(\{i_0\}) \to \Phi(I' \leadsto \nonleafs)[I']$, $\ell \ne 0$, $F\ell;i_0 \notin S_r$};
        \begin{scope}[shift={(0,0)}]
          \node[defnode,scale=0.6] (U) at (0, 0) {$\diamond$};
          \node[leaf,scale=0.6]    (Ui) at ($(U) + (0,-1)$) {$i_0$};
          \draw[-stealth] (U) -- (Ui);

          \node[defnode,scale=0.6]    (V) at (3, 0) {$\diamond$};
          \node[scale=0.6]            (Vd1) at ($(V) + (0.5,-1)$) {$\dots$};
          \node[nonleaf,scale=0.6]    (VY) at ($(V) + ( 1,-1)$) {$j$};
          \node[scale=0.6]            (Vd2) at ($(V) + ( 1.5,-1)$) {$\dots$};
          \draw[-stealth] (V) -- (VY);

          \draw[dashed, -angle 60] (U) to[bend left=10] (VY);
          \draw[dashed, -angle 60] (U) to (V);
        \end{scope}
        \begin{scope}[shift={(6,0)}]
          \node[defnode,scale=0.6] (U) at (0, 0) {$x$};
          \node[leaf,scale=0.6]    (Ui) at ($(U) + (0,-1)$) {$x$};
          \draw[-stealth] (U) -- (Ui);

        \node[defnode,scale=0.6]    (V) at (3, 0) {$\lambda x e_{\tau_r^\ell}$};
          \node[scale=0.6]            (Vd1) at ($(V) + (0.8,-1)$) {$\dots$};
          \node[nonleaf,scale=0.6]    (VY) at ($(V) + ( 1.5,-1)$) {$\lambda a_r^{(j,\ell)} x$};
          \node[scale=0.6]            (Vd2) at ($(V) + ( 2.2,-1)$) {$\dots$};
          \draw[-stealth] (V) -- (VY);

          \draw[dashed, -angle 60] (U) to[bend left=5] (VY);
          \draw[dashed, -angle 60] (U) to (V);
        \end{scope}
      \end{tikzpicture}
      \\[0.5\baselineskip]
      \begin{tikzpicture}[
          scale=0.95,
          defnode/.style={rectangle,draw,inner sep=4pt,outer sep=0pt,minimum size=15pt},
          mydot/.style={circle,fill,inner sep=0.5pt},
          nonleaf/.style={defnode,fill=lightgray},
          leaf/.style={defnode,fill=leafgreen},
          toggle/.style={defnode,fill=leafgreen}
        ]
        \node[scale=0.6] at (5, -1.5) {$\Theta_r^{\bag j} \colon \trivial(\{i_0\}) \to \const(\{0,\dots,n-1\})(\{0,\dots,n-1\} \leadsto \nonleafs)$};
        \begin{scope}[shift={(0,0)}]
          \node[defnode,scale=0.6] (U) at (0, 0) {$\diamond$};
          \node[leaf,scale=0.6]    (Ui) at ($(U) + (0,-1)$) {$i_0$};
          \draw[-stealth] (U) -- (Ui);

          \node[defnode,scale=0.6]    (V) at (3, 0) {$\diamond$};
          \node[scale=0.6]    (Vd1) at ($(V) + ( -0.7,-1)$) {$\dots$};
          \node[nonleaf,scale=0.6]    (Vi) at ($(V) + ( 0,-1)$) {$X\ell$};
          \node[scale=0.6]            (Vd2) at ($(V) + ( 0.7,-1)$) {$\dots$};
          \draw[-stealth] (V) -- (Vi);

          \draw[dashed, -angle 60] (U) to[bend left=10] (Vi);
          \draw[dashed, -angle 60] (U) to (V);
        \end{scope}
        \begin{scope}[shift={(6,0)}]
          \node[defnode,scale=0.6] (U) at (0, 0) {$x$};
          \node[leaf,scale=0.6]    (Ui) at ($(U) + (0,-1)$) {$x$};
          \draw[-stealth] (U) -- (Ui);

          \node[defnode,scale=0.6]    (V) at (3, 0) {$\lambda x$};
          \node[scale=0.6]    (Vd1) at ($(V) + ( -0.6,-1)$) {$\dots$};
          \node[nonleaf,scale=0.6]    (Vi) at ($(V) + ( -0,-1)$) {$\lambda x$};
          \node[scale=0.6]    (Vd2) at ($(V) + ( 0.6,-1)$) {$\dots$};
          \draw[-stealth] (V) -- (Vi);

          \draw[dashed, -angle 60] (U) to[bend left=10] (Vi);
          \draw[dashed, -angle 60] (U) to (V);
        \end{scope}
      \end{tikzpicture}
    \end{center}
    \caption{The morphisms $\Theta_r^{(G)}$, $2 \le r \le s$, from Claim~\ref{claim:ass-claim}.}%
    \label{fig:ass-claim-2}
  \end{figure}
  
  For every morphism $\Theta_r^{(G)}$ we take $\alpha(x) = \diamond$ unless $G = F\ell$ and $x=i_0$, in which case $\alpha(x) = \zeroi$ for $\ell \ne 0$ and $\alpha(x) = i_0$ for $\ell=0$.
  
  In some cases a brief justification is necessary that these are valid morphisms, similar to the informal discussion above.
  In the case of $\Theta_1^{\bag j}$, for the linear map
  \begin{align*}
    \theta_\diamond^{\Theta_1^{\bag j}} \colon \FF_p &\to V^{\bigsum_n} \\
      x &\mapsto \left( (-1)^\ell \beta_{\tau^\ell} a_1^{(j,\ell)} x \right)_{0 \le \ell \le n-1}.
  \end{align*}
  to make sense we need the $\pm 1$ sum of the tuple on the right to be zero, i.e.,
  \[
    \sum_{\ell=0}^{n-1} \beta_{\tau^\ell} \left( \prod_{r=1}^s a_r^{(j,\ell)} \right) = 0.
  \]
  This holds by~\eqref{eq:beta-prop} and $a_r^{(j,\ell)} \bmod p = \phi_j(e_{\tau_r^\ell})$.

  For the maps $\Theta_r^{(F0)}$ we need to check that
  \[
    \phi_{i_0}(\beta_{\tau^0} e_1) = 1
  \]
  when $r=1$, which holds by~\eqref{eq:beta-i0}, and for $2 \le r \le s$ that
  \[
    \phi_{i_0}(\lambda e_1) = 1
  \]
  which holds by definition of $\lambda$.

  Similarly, for $\Theta_r^{(F\ell)}$ when $\ell \ne 0$ and $F\ell;i_0 \in S_r$, we need to verify that
  \[
    \phi_{i_0}(e_{\tau_r^\ell}) = 0
  \]
  which holds because $\tau_r^\ell \ne 1$ (by construction of $S_r^{(F\ell)}$) and $\phi_{i_0}(e_i) = 0$ for $i \ne 1$ (by Claim~\ref{claim:basis}).
  
Recall that the morphisms $\cM_r^{(G)}$ respect all of the vertices $x$ appearing in~\eqref{eq:compat-maps} and~\eqref{eq:compat-maps-2}, so $\theta^{\cM_r^{(G)}}_x$ are the identity map for these $x$.
  Hence the maps $\theta^{\fM_r^{(G)}}_x$ coincide with the corresponding maps in Figure~\ref{fig:ass-claim-1} and Figure~\ref{fig:ass-claim-2}.
  For example, since $\alpha^{\cM_1^{(A(j,\ell))}}(Y) = Y$, the second row of Figure~\ref{fig:ass-claim-1} gives
  \[
    \theta^{\fM_1^{(A(j,\ell))}}_Y = x \mapsto (-1)^\ell \beta_{\tau^\ell} \left( \prod_{r=1}^s a_r^{(j,\ell)} \right) x
  \]
  as required by~\eqref{eq:compat-maps}.
The other parts of~\eqref{eq:compat-maps},\eqref{eq:compat-maps-2}, and the condition $\theta_{i_0}^{\fM_r^{(F0)}} = \id_{\FF_p}$, follow similarly.
\end{proof}

This completes the proof of Theorem~\ref{thm:main-ext} and Theorem~\ref{thm:main}.

\section{Conjectures}%
\label{sec:conjectures}

In Section~\ref{sub:background} we began by stating a very general question: if a putative inequality has a corresponding 100\% functional equation statement which is (a) true, or (b) has an elementary proof, when can we turn this into a Cauchy--Schwarz proof of the inequality?

We now revisit this question and attempt to make it precise.

\subsection{Cases to avoid}%
\label{sub:bad-conj}

Some caution needs to be taken when phrasing this precisely.
For example, consider the following statement: for some $C>0$ and any $1$-bounded function $f \colon \FF_p^n \to \CC$,
\[
  \left\lvert \EE_{x,y \in \FF_p^n} f(x) f(y) \overline{f(x+y)^2} \right\rvert \ge \left\lvert \EE_{x \in \FF_p^n} f(x) \right\rvert^{C}.
\]
The 100\% analogue states that if $F \colon \FF_p^n \to \RR/\ZZ$ is a constant function then
\[
  \forall x,y \in \FF_p^n:\  F(x)+F(y) = 2 F(x+y)
\]
which is clearly true.
However, the statement itself is clearly false: for example, take $f=1_S$ for $S \subseteq \FF_p^n$ a non-empty sum-free set.

We can see what has gone wrong by writing down an elementary and Cauchy--Schwarz friendly proof of the 100\% functional equation version.
Given the hypothesis $\forall a\in A \colon F(a) = A$, by a duplication step we have $\forall a,b \in \FF_p^n \colon F(a)=F(b)$ or equivalently 
$
  \forall a,b \in \FF_p^n \colon F(a) - F(b) = 0
$.
By another duplication step,
\[
  \forall a,b,a',b' \in \FF_p^n \colon F(a) - F(b) = F(a') - F(b').
\]
By specialization we can set $a=x$, $b'=y$, $a'=b=x+y$.
Phrased in the language of Section~\ref{sub:linear-datum}, by $\CS(\emptyset)$, $\CS(\emptyset)$ followed by a $\morph$ step we have
\[
  \trivial(\{A\}) \entails^1 \trivial(\{A,B\}) \entails^1 \trivial(\{A,B,A',B'\}) \entails^0 \Phi
\]
where $\Phi$ is the system of linear forms $(x,y,x+y,x+y)$.

However, the statement that comes out of this and Corollary~\ref{cor:log-entails} is that
there exist functions $f_1,\dots,f_4$ which are all translates of $f$, such that
\[
  \left\lvert \EE_{x,y \in \FF_p^n} f_1(x) f_2(y) \overline{f_3(x+y)} \overline{f_4(x+y)} \right\rvert \ge \left\lvert \EE_{x \in \FF_p^n} f(x) \right\rvert^{4}.
\]
This weaker statement is true and does not conflict with the example above.

In other words, we \emph{can} turn the 100\% proof into a Cauchy--Schwarz proof, but it is not a proof of the (naive, false) inequality we first wrote down.
Instead it is necessary and natural to introduce additional translates, as in the proof of Proposition~\ref{prop:morphism}.

In the same way, these extra translates save us from trying to prove a polynomial bound in Roth's theorem over $\FF_p$ for $p$ large, which would also be false by Behrend's construction~\cite{behrend}.

\subsection{A precise phrasing}%
\label{sub:precise}

It therefore turns out that the framework of Section~\ref{sub:linear-datum} is already well-suited for phrasing a conjecture that is not obviously false.
The most ambitious version is the following.

\begin{conjecture}[Bold conjecture]%
  \label{conj:bold-conjecture}
  Let $\Phi$, $\Psi$ be linear data and let $\gamma \colon I^\Psi \to I^\Phi \times \{0,1\}$ be a partial function.
  Suppose the following holds: for any tuple of functions $(F_i)_{i \in I^\Phi}$, $F_i \colon W_i^\Phi \to \RR/\ZZ$, and $A \in \RR/\ZZ$ such that
  \[
    \forall v \in V^\Phi \colon \sum_{i \in I^\Phi} F_i\bigl(\phi^\Phi_i(v)\bigr) = A
  \]
  there exists a tuple of functions $(F'_j)_{j \in I^\Psi}$ and $B \in \RR/\ZZ$ such that
  \[
    \forall u \in V^\Psi \colon \sum_{j \in I^\Psi} F'_j\bigl(\phi^\Psi_j(u)\bigr) = B
  \]
  and such that whenever $\gamma(j) = (i,\ell)$ is defined, $F'_j = (-1)^{\ell} F_i$.

  Then for some $M \ge 0$, $\Phi \entails_{\gamma}^M \Psi$.
\end{conjecture}
In particular, the conclusion implies an inequality statement by Corollary~\ref{cor:log-entails}.

Variants are possible: we could restrict to systems of linear forms $W_i^\Phi = \FF_p = W_j^\Psi$, or we could weaken the hypothesis $F'_j = (-1)^{\ell} F_i$ to stating that $F'_j$ is a translate of $(-1)^\ell F_i$.
It could also perhaps make a difference to replace the domain $W_i^\Phi$ of $F_i$ with $(W_i^\Phi)^n$ for an arbitrary $n \ge 1$.

The author is not aware of any counterexamples even to this very strong conjecture.
Conjecture~\ref{conj:asymmetric} corresponds to the cases where $\Psi = \gc_s$ or equivalently $\Psi=U^{s+1}$.
Hence Example~\ref{ex:bad-asymmetric} gives an open case of Conjecture~\ref{conj:bold-conjecture}.

\begin{remark}%
  \label{rem:logic}
  It is notable that this conjecture has more in common with logic than analysis.
  Essentially we are asking for a completeness theorem for the not-quite-first-order-language defined implicitly by $\entails$. 
  Indeed, the hypothesis states roughly that $\Phi \models \Psi$: i.e., every ``model'' $(F_i)_{i \in I^\Phi}$ of $\Phi$ gives a corresponding ``model'' $(F'_j)_{j \in I^\Psi}$ of $\Psi$.
  The desired conclusion is that this implies $\Phi \entails \Psi$, i.e., that there is a syntactic ``proof'' of $\Psi$ assuming $\Phi$.
\end{remark}

\subsection{Requiring an elementary argument}%
\label{sub:elem-arg}

We originally asked for less than Conjecture~\ref{conj:bold-conjecture} in that we assumed we already had \emph{some} elementary proof of the functional equation identity, albeit not one of Cauchy--Schwarz type.

It is natural to ask whether any true functional equation identity of this sort has an elementary proof.
Certainly Conjecture~\ref{conj:bold-conjecture} implies the answer is yes.
As such issues are beyond the author's expertise, we formulate a weaker conjecture where this question is factored out, by conditioning on the existence of an elementary proof.

There are many possibly ways to encode formally what an elementary proof means in this setting, but one guess is the following.

\begin{definition}%
  \label{def:elem}
  We define a relation $\Phi {\entails'}^k_\gamma \Psi$ by adding another inference rule (B) to Definition~\ref{def:logic-notation}, which states the following.

    Suppose we already have $\Phi {\entails'}^{k_0}_{\gamma_0} \Psi_0$ and $\Phi {\entails'}^{k_1}_{\gamma_1} \Psi_1$.
    Suppose also that $J_0 \subseteq I^{\Psi_0}$ and $J_1 \subseteq I^{\Psi_1}$ are subsets and $\beta$ is a partial matching between $J_0$ and $J_1$ such that $W^{\Psi_0}_{i_0} = W^{\Psi_1}_{i_1}$ for all $(i_0,i_1) \in \beta$.

    Finally, suppose that for all $(i_0,i_1) \in \beta$, $\gamma_0(i_0)$ and $\gamma_1(i_1)$ are both defined and $\gamma_0(i_0) = \gamma_1(i_1)$.

    Then $\Phi {\entails'}^{k_0+k_1}_{\gamma} \bigl( \Psi_0 +_\beta \Psi_1 \bigr)$, where $\gamma$ is defined by $\gamma(L;i_0) = \gamma_0(i_0)$ and $\gamma(R;i_1) = (j,\ell+1\bmod2)$ if $\gamma_1(i_1) = (j,\ell)$.
\end{definition}
We call this an ``$\opr{Eliminate}$'' step.
It allows us to use the previously missing logic that if
\[
  \forall u \in V^{\Psi_0}:\ [E_1](u) = \sum_{i \in J_0} F'_j(\phi^{\Psi_0}_i(u))
\]
and
\[
  \forall v \in V^{\Psi_1}:\ [E_2](v) = \sum_{j \in J_1} F'_j(\phi^{\Psi_1}_j(v))
\]
for some expressions $[E_1]$ and $[E_2]$, then for all pairs $(u,v)$ such that $\phi^{\Psi_0}_i(u)=\phi^{\Psi_1}_j(v)$ for all $(i,j) \in \beta$ (i.e., the right-hand sides are demonstrably equal) we have $[E_1](u) = [E_2](v)$.

Then a weaker version of Conjecture~\ref{conj:bold-conjecture} is the following.

\begin{conjecture}[Slightly less bold conjecture]%
  \label{conj:less-bold}
  If $\Phi {\entails'}^k_\gamma \Psi$ then $\Phi {\entails}^{k'}_\gamma \Psi$ for some $k' \ge 0$.
\end{conjecture}

When $|J_0|=|J_1|=1$, the process of mimicking the step (B) using only Cauchy--Schwarz is essentially what ``stashing'' is all about.
These techniques could potentially extend to cases where the restrictions $\Psi_0[J_0]$ and $\Psi_1[J_1]$ are ``trivial'', i.e., when the map $V \mapsto \bigoplus_{j \in J} W_j$ is surjective.
However, there appear to be genuine difficulties when trying to extend this to eliminating non-trivial sub-expressions using only Cauchy--Schwarz.

Of course, there are other inequalities than Cauchy--Schwarz and it would also be fruitful to interpret an $\opr{Eliminate}$ step as an inequality in some other way, ideally still with polynomial bounds.
This possibility was discussed implicitly in Footnote~\ref{footnote:ultra}.

It would also be interesting if Conjecture~\ref{conj:less-bold} were false in an essential way.

\subsection{A special case}%
\label{sub:special}

There is an attractive special case of either Conjecture~\ref{conj:bold-conjecture} or Conjecture~\ref{conj:less-bold} which is in a way dual to the Gowers--Wolf problem:
instead of asking when $\Phi \entails \gc_s$, we ask when $\gc_s \entails \Phi$.

In terms of functional equations, we are assuming that a function $F \colon \FF_p^n \to \RR/\ZZ$ is a (generalized) polynomial of degree at most $s$, and wish to deduce an identity
\[
  \forall v \in (V^{\Phi})^n:\ \sum_{i \in I^\Phi} (-1)^{\fs(i)} F\bigl(\phi_i^\Phi(v)\bigr) = 0
\]
where $\fs(i) \in \{0,1\}$ determine the signs.
It is a relatively straightforward exercise in multilinear algebra to compute whether or not such an identity should hold.

\begin{conjecture}%
  \label{conj:maker}
  Let $\Phi$ be is a linear datum with $W_i^\Phi = \FF_p$ for all $i \in I^\Phi$ \uppar{i.e., a system of linear forms}.
  Suppose $s \ge 1$ is given, and there is a function $\fs \colon I \to \{0,1\}$ such that
  \[
    \sum_{i \in I} (-1)^{\fs(i)} \left( \phi_i^\Phi \right)^{\otimes t} = 0
  \]
  as elements of $\bigl((V^\Phi)^\ast\bigr)^{\otimes t}$, for all $t$, $0 \le t \le s$.
  The case $t=0$ just says that $\sum_{i \in I} (-1)^{\fs(i)} = 0$.
  
  Then for some $k \ge 0$, $\gc_s \entails^k_{\gamma} \Phi$, where $\gamma(i) = (\triangle, \fs(i))$ for all $i \in I$.
\end{conjecture}

As well as being a natural case of Conjecture~\ref{conj:bold-conjecture}, this is also motivated by the fact that an affirmative answer could be helpful in proving Conjecture~\ref{conj:asymmetric}.

\bibliography{refs}
\bibliographystyle{alpha}

\end{document}